\documentclass[reqno,10pt]{amsart}
\usepackage{kotex}
\usepackage[margin=1.1in]{geometry}
\usepackage{color, hyperref}
\usepackage{amsmath,amsfonts,amssymb,amsthm, physics}
\usepackage{mathtools}
\usepackage{commath}
\usepackage{cleveref}
\usepackage{tikz}
\newtheorem{theorem}{Theorem}[section]
\newtheorem{lemma}[theorem]{Lemma}

\newtheorem{corollary}{Corollary}[section]
\theoremstyle{definition}
\newtheorem{definition}[theorem]{Definition}

\usepackage{ulem}

\theoremstyle{remark}

\numberwithin{equation}{section}

\let\hide\iffalse

\newcommand{\bq}{\begin{equation}}
\newcommand{\eq}{\end{equation}}

\newcommand{\reqnomode}{\tagsleft@false}

\allowdisplaybreaks

\begin{document}
\title[Large amplitude solution with Large external potential ]{On the Large Amplitude Solution of the Boltzmann equation with Large External Potential and Boundary Effects} 

\author{Jong-in Kim}
\address{Department of Mathematics, Pohang University of Science and Technology, South Korea }
\email{kimjim@postech.ac.kr}
\author{Donghyun Lee}
\address{Department of Mathematics, Pohang University of Science and Technology, South Korea }
\email{donglee@postech.ac.kr}




\begin{abstract}
	The Boltzmann equation is a fundamental equation in kinetic theory that describes the motion of rarefied gases. In this study, we examine the Boltzmann equation within a $C^{1}$ bounded domain, subject to a large external potential $\Phi(x)$ and diffuse reflection boundary conditions. Initially, we prove the asymptotic stability of small perturbations near the local Maxwellian $\mu_{E}(x,v)$. Subsequently, we demonstrate the asymptotic stability of large amplitude solutions with initial data that is arbitrarily large in (weighted) $L^{\infty}$, but sufficiently small in the sense of relative entropy. Specifically, we extend the results for large amplitude solutions of the Boltzmann equation (with or without external potential) \cite{DHWZ2019,DKL2020,DW2019,KLP2022} to scenarios involving significant external potentials \cite{ChanBELP,WWGS2019} under diffuse reflection boundary conditions.
\end{abstract}

\maketitle
\tableofcontents

\section{Introduction}
The Boltzmann equation with an external potential $\Phi$ is 
\begin{align} \label{BE}
	\partial_t F+v \cdot \nabla_x F-\nabla_x \Phi(x) \cdot \nabla_v F =Q(F,F),\quad  F(0,x,v) = F_0(x,v), 
\end{align}
where $F(t,x,v)$ is a distribution function for the gas particles at time $t$, a position $x\in \Omega$, and a velocity $v\in \mathbb{R}^3$ and the external force $\Phi(x)$ is a given function which depends only on the spatial variable $x \in \Omega$. The collision operator $Q$ is the bilinear form
\begin{align}
	Q(F_1,F_2):=\int_{\mathbb{R}^3} \int_{\mathbb{S}^2}B(v-u, \omega) \left[F_1(u')F_2(v')-F_1(u)F_2(v)\right] d\omega du,
\end{align} 
where the post-collision velocity pair $(v',u')$ and the pre-collision velocity pair $(v,u)$ satisfy the relation
\begin{align*}
	u'=u+[(v-u)\cdot \omega]\omega, \quad v'= v-[(v-u)\cdot \omega]\omega
\end{align*} 
with $\omega \in \mathbb{S}^2$, according to the conservation of momentum and energy of two particles
\begin{align*}
	u+v = u'+v', \quad |u|^2+|v|^2 = |u'|^2+|v'|^2.
\end{align*}
The collision kernel $B$ for the hard potential model with angular cutoff is of the form
\begin{align*}
	B(v-u,\omega) = |v-u|^\gamma b(\cos\theta),
\end{align*}
where $0 \le \gamma \le 1$, $\cos \theta = \frac{(v-u) \cdot \omega}{|v-u|}$, and $0\le b(\cos\theta) \le C_b|\cos \theta| $ for some constant $C_b$.
For the angular cutoff case, we can write the collision operator $Q$ as
\begin{align*}
	Q(F_1,F_2)&=\int_{\mathbb{R}^3} \int_{\mathbb{S}^2}B(v-u, \omega) F_1(u')F_2(v')d\omega du -\int_{\mathbb{R}^3} \int_{\mathbb{S}^2}B(v-u, \omega)F_1(u)F_2(v) d\omega du\\
	&=: Q_+(F_1,F_2) -Q_-(F_1,F_2),
\end{align*}
where $Q_+(F_1,F_2)$ and $Q_-(F_1,F_2)$ are the gain term and the loss term, respectively.\\
For a given external potential $\Phi$, the equation \eqref{BE} has a local maxwellian $\mu_E(x,v) = e^{-\Phi(x)}e^{-\frac{|v|^2}{2}} =  e^{-\Phi(x)}\mu(v)$. The boundary condition for equation \eqref{BE} is given by
\begin{align} \label{DRBC}
	F(t,x,v)|_{\gamma_-} = c_\mu \mu(v) \int_{n(x)\cdot v' >0} F(t,x,v')\{n(x)\cdot v'\}dv',
\end{align}
where $c_\mu$ is given by
\begin{align*}
	c_\mu \int_{n(x)\cdot v' >0} \mu(v')\{n(x)\cdot v'\}dv' =1.
\end{align*}
We consider the perturbation by
\begin{align*}
	F(t,x,v)= \mu_E(x,v) + \mu_E^{\frac{1}{2}}(x,v)f(t,x,v).
\end{align*}
Then we can derive the perturbed equation as follows:
\begin{align} \label{PBE}
	\partial_t f + v\cdot \nabla_x f - \nabla_x \Phi(x) \cdot \nabla_vf + e^{-\Phi(x)}Lf = e^{-\frac{\Phi(x)}{2}}\Gamma(f,f)
\end{align}
with the following boundary condition
\begin{align} \label{PDRBC}
	f(t,x,v)|_{\gamma_-} = c_\mu \mu^{\frac{1}{2}}(v)\int_{n(x)\cdot v' >0} f(t,x,v')\mu^{\frac{1}{2}}(v')\{n(x)\cdot v'\}dv'.
\end{align}
Here the linear operator $L$ is
\begin{align*}
	Lf = -\frac{1}{\sqrt{\mu}}\left\{Q(\mu,\sqrt{\mu}f )+Q(\sqrt{\mu}f,\mu) \right\} = \nu(v)f -Kf,
\end{align*}
with the collision frequency $\nu(v) = \int_{\mathbb{R}^3}\int_{\mathbb{S}^2} B(v-u,\omega)\mu(u) d\omega du \sim (1+|v|)^\gamma$. Note that $\nu(v)$ has the greatest lower bound, denoted by $\nu_0$. It is well-known that the operator $L$ has a kernel
\begin{align} \label{kernelL}
	\text{Ker}(L) = \text{span}\left\{\mu^{1/2}, v_1\mu^{1/2}, v_2\mu^{1/2}, v_3\mu^{1/2}, \frac{|v|^2-3}{\sqrt{6}} \mu^{1/2} \right\}.
\end{align}
The nonlinear operator $\Gamma$ is
\begin{align*}
	\Gamma(f_1,f_2)= \frac{1}{\sqrt{\mu}}Q(\sqrt{\mu}f_1, \sqrt{\mu}f_2 )&= \frac{1}{\sqrt{\mu}}Q_+(\sqrt{\mu}f_1, \sqrt{\mu}f_2 )-\frac{1}{\sqrt{\mu}}Q_-(\sqrt{\mu}f_1, \sqrt{\mu}f_2 )\\
	&=: \Gamma_+(f_1,f_2)-\Gamma_-(f_1,f_2).
\end{align*}
\indent For the diffuse reflection boundary condition \eqref{DRBC}, it is easy to check mass conservation of \eqref{BE}. Therefore, we may assume 
\begin{align} \label{massconserv}
	\int_{\Omega \cross \mathbb{R}^3} f(t,x,v) \mu^{\frac{1}{2}}_E(x,v)dxdv = 0 \quad \text{for all } t\ge 0.
\end{align}
by imposing initial data $F_0$ such that
\begin{align} \label{initialmass}
	\int_{\Omega \cross \mathbb{R}^3} F_0(x,v) dxdv = 	\int_{\Omega \cross \mathbb{R}^3} \mu_{E} dxdv.
\end{align}
\indent Without loss of generality, we may assume that $\Phi \in C^3(\bar{\Omega})$ satisfies $ \Phi(x) \ge 0$ for $x \in \Omega$. Otherwise, we can replace $\Phi$ by $\Phi+C \ge 0$ for some constant $C$. Then we have the bounds for $\Phi$ on $\bar{\Omega}$ as follows: 
\begin{align*}
	0 \le \Phi(x) \le \|\Phi\|_\infty \quad \text{for}\ x \in \bar{\Omega}.
\end{align*} 

For general initial data $F_0(x,v)\geq 0$ with bounded physical quantities, the global existence of renormalized solution of the Boltzmann equation is well-known by the seminal work DiPerna-Lions \cite{DLCP1989}. However, important properties such as uniqueness, conservation, and convergence to equilibrium for that solution are not known. Of course, it should be noted that there are better results such as \cite{MouhotL1} for the spatially homogeneous Boltzmann equation. \\
\indent On the other hand, regarding the convergence to equilibrium, Desvillettes-Villani \cite{DV} proved that if there exists a unique global solution satisfying properties such as smoothness and a Gaussian lower bound, then that solution converges to the global equilibrium. \\
\indent Rigorous mathematical results concerning the existence, uniqueness, and convergence to equilibrium are, as of now, known for cases where the initial data is sufficiently close to the equilibrium state. In this direction, Ukai \cite{Ukai} first solved the problem in periodic box when perturbation $f_0 = \frac{F_0-\mu}{\sqrt{\mu}}$ is sufficiently small in some weighted Sobolev spaces. And subsequently, the theory of the Boltzmann equation in the high-order Sobolev space framework was significantly developed by Guo \cite{GuoVPB, GuoVMB, GuoLandau} in the case of periodic box problems. We also refer to \cite{SGMR2012, SGSR2004, SGED2008} for other works in this framework. \\ 
\indent For general bounded domain problems, unfortunately, the methodologies utilizing such high order Sobolev spaces have not been very helpful. Instead, Guo \cite{GDCB2010} provided a comprehensive proof by adopting an $L^2$-$L^\infty$ bootstrap argument, establishing global well-posedness and asymptotic stability in (weighted) $L^{\infty}$ for small perturbations. The methodology of this paper has evolved in various ways. \cite{GDCB2010} restricted their results to real analytic uniformly convex domains with specular reflection boundary conditions. However, this real analytic condition was later extended to general $C^3$ domains in Kim-Lee \cite{KimLee}, and the problem has also been addressed for some domains with non-convex boundaries in  Ko-Kim-Lee \cite{KimLeeNonconvex,KKL}. We also refer to  Briant-Guo \cite{Guo-Briant} for Maxwell boundary condition with polynomial tail. \\
\indent On the other hand, in Duan-Huang-Wang-Yang \cite{DHWY2017}, the authors successfully replaced the initial data $L^{\infty}$ smallness condition with an $L^{p}$ type smallness condition, while maintaining global well-posedness and asymptotic stability in the $L^{\infty}$ space. Research in this direction has also been conducted for boundary condition problems. We refer Duan-Wang \cite{DW2019} for diffuse boundary conditions and Duan-Ko-Lee \cite{DKL2020} for specular reflection boundary conditions in $C^3$ convex domains. We also refer \cite{BKLY_BGK,DHWZ2019,KLP2022} for related works. \\
\indent For the Boltzmann equation with external forces, there are not many results when general boundary conditions are imposed. This is because analyzing the characteristics of the Hamiltonian used to construct 
$L^\infty$ solutions is quite challenging. Small perturbation problems for the Boltzmann equation with \textit{non-self-consistent} external potentials have been studied in \cite{ChanBELP,KimLee,YSEB2018}. 
G. Wang-Y. Wang \cite{WWGS2019} also extends the result of \cite{ChanBELP} to a class of large oscillation initial data in the $L^\infty$ space under the smallness assumption of initial data in the $L^2$ sense. In this paper, we extend the result of \cite{WWGS2019} to the diffuse boundary condition problem in general $C^{1}$ domain under the assumption that initial data has sufficiently small relative entropy. \\
 \indent Meanwhile, for the Vlasov-Poisson-Boltzmann (VPB) equation with a self-consistent external force, the problem with some boundary conditions was studied in \cite{YCRV2021,CKLVPB, HCQL2020}. In particular, it was essential to use the fact that, near equilibrium, the perturbation of characteristics due to the self-consistent force can be sufficiently small. From this perspective, the question of whether large amplitude solutions of the Boltzmann equation with external forces can exhibit global well-posedness and global stability is a very intriguing topic. Naturally, the most intriguing problem would be the results concerning large amplitude solutions of the VPB (Vlasov-Poisson-Boltzmann) equation under boundary conditions. However, this is currently considered a very difficult problem. Even for the Boltzmann equation without boundary conditions, the result remains unknown. In this paper, we investigate large amplitude solutions of the Boltzmann equation under given time-independent external forces and diffuse boundary conditions as a starting point for research in this direction. This will undoubtedly be an important step towards solving the VPB problem with large amplitude solutions.

\subsection{Notation} \label{notation}

 We denote the closure of $\Omega$ by $\bar{\Omega}$.
 We describe the notations for function spaces we shall use in this paper. As a convention, we denote the following function spaces for $p \in [1,\infty]$,
\begin{align} \label{fntspace1}
	L^\infty_t=L^\infty([0,\infty)) ,\quad L^p_{x,v}=L^{p}(\Omega \cross \mathbb{R}^3), \quad L^p_x=L^p(\Omega), \quad L^p_v=L^p(\mathbb{R}^3).
\end{align}
Similarly, we denote the Sobolev space $H^{k}_x$ as $H^k(\Omega)$, which is equivalent to $W^{k,2}(\Omega)$.
In particular, we abbreviate the norm in the space $L^\infty(\Omega)$ as $\|\cdot\|_\infty:=\|\cdot\|_{L^\infty_x}$.  We also set $(f,g)_{L^2_v} = \int_{\mathbb{R}^3} f(v)g(v)dv$ the inner product in $L^2(\mathbb{R}^3)$.
 For a positive Lebesgue measurable function $w$ on $\Omega \cross \mathbb{R}^3$, we define the weighted space $L^\infty_{x,v}(w)$ given by the norm
\begin{align} \label{winftynorm}
	\|f\|_{L^\infty_{x,v}(w)} := \|wf\|_{L^\infty_{x,v}}= \sup_{(x,v)\in \Omega \cross \mathbb{R}^3} \left[w(x,v)|f(x,v)| \right].
\end{align}
We also define the space $L^2_\gamma = L^2(\partial\Omega\cross \mathbb{R}^3)$ with the norm
\begin{align*}
	\|f\|_{L^2_\gamma}:= \left[\int_{\gamma}|f(x,v)|^2|n(x) \cdot v| dS(x)dv\right]^{\frac{1}{2}}.
\end{align*}
Usually, 
we denote the norm on the boundary $\gamma_+$ in $L^2$ by
\begin{align} \label{L2bnorm}
	\|f\|_{L_{\gamma_+}^2}:=\left[\int_{\gamma_+}|f(x,v)|^2\{n(x) \cdot v\} dS(x)dv\right]^{\frac{1}{2}}.
\end{align}
If not specifically mentioned, $C_a$ or $C(a)$ is the generic positive constant depending on $a$, while $C_0, C_1, C_2, \cdots$ denote some specific positive constants.\\
$\partial_{x_i}$ and $\partial_{v_i}$ mean the partial derivative with respect to $x_i$ and $v_i$, respectively. We also abbreviate $\partial_{ij}=\partial_{x_i}\partial_{x_j}$.\\
We define the function $\mathbf{1}_A(x)$, which is 1 on $A$ and 0 otherwise. We will frequently use the following notations for summation:
\begin{align*}
	\sum_{i}^{M}:=\sum_{i=1}^{M} \quad \text{and} \quad \sum_{I}^{(M)^3}:=\sum_{i_1=1}^{M}\sum_{i_2=1}^{M}\sum_{i_3=1}^{M},
\end{align*}
where $M$ is an integer and $I=(i_1,i_2,i_3)$ is a tuple.

\bigskip

\section{Domain and Characteristic}
\subsection{Domain and Back-time cycles}
\label{SDBC}
Throughout this paper, we assume $\Omega:=\{x\in \mathbb{R}^3 : \xi(x) <0\}$ is connected and bounded, where $\xi(x)$ is a $C^1$ function. Suppose that $\nabla_x \xi(x) \not= 0$ at the boundary $\partial \Omega=\{x: \xi(x) =0\}$, and the outward normal vector at $x \in \partial \Omega$ is given by $n(x) = \frac{\nabla_x \xi(x)}{|\nabla_x \xi(x)|}$. We denote the phase boundary in the space $\Omega \cross \mathbb{R}^3$ by $\gamma:=\partial \Omega \cross \mathbb{R}^3$. We decompose $\gamma$ into the outgoing set $\gamma_+$, the grazing set $\gamma_0$, and the incoming set $\gamma_-$ :
\begin{align*}
	&\gamma_+=\{(x,v) \in \partial \Omega \cross \mathbb{R}^3 : n(x) \cdot v>0 \},\\
	&\gamma_0=\{(x,v) \in \partial \Omega \cross \mathbb{R}^3 : n(x) \cdot v=0 \},\\
	&\gamma_-=\{(x,v) \in \partial \Omega \cross \mathbb{R}^3 : n(x) \cdot v<0 \}.
\end{align*}

Given $(t,x,v) \in [0,\infty) \cross \Omega \cross \mathbb{R}^3$, let $[X(s;t,x,v),V(s;t,x,v)]=[X_{\mathbf{cl}}(s;t,x,v),V_{\mathbf{cl}}(s;t,x,v)]$ be the position and velocity of a particle at time $s$, which was at $(t,x,v)$. The backward characteristic $[X(s;t,x,v),V(s;t,x,v)]$ for the Boltzmann equation \eqref{BE} is determined by the Hamiltonian ODE:
\begin{align} \label{Hacon}
	\frac{dX(s;t,x,v)}{ds} = V(s;t,x,v), \quad \frac{dV(s;t,x,v)}{ds} = -\nabla_x \Phi(X(s;t,x,v))
\end{align}
with $[X(t;t,x,v),V(t;t,x,v)] = [x,v]$.\\
Unless otherwise stated about the characteristic in this paper, we abbreviate $X(s):=X(s;t,x,v)$ and $V(s):=V(s;t,x,v)$. \begin{definition} \label{Backwardexit}
	(Backward exit time) For $(t,x,v) \in [0,\infty) \cross \Omega \cross \mathbb{R}^3$, we define its backward exit time $t_\mathbf{b}(x,v)>0$ to be the last moment at which the back-time characteristic curve $[X(s;t,x,v),V(s;t,x,v)]$ remains in the interior of $\Omega$ : 
	\begin{align*}
		t_\mathbf{b}(x,v) = \sup\{\tau>0 : X(s;t,x,v) \in \Omega \ \text{for all }t-\tau \le s \le t \}.
	\end{align*}
 For any $(x,v)$, we use $t_\mathbf{b}(x,v)$ whenever it is well-defined. We have $X(t-t_\mathbf{b};t,x,v)\in \partial \Omega$ and $\xi(X(t-t_\mathbf{b};t,x,v)) = 0$. We also define
	\begin{align*}
		x_\mathbf{b}(x,v)=X(t-t_\mathbf{b}(x,v);t,x,v) \in \partial \Omega, \quad v_\mathbf{b}(x,v)=V(t-t_\mathbf{b}(x,v);t,x,v) .
	\end{align*}
\end{definition}
It is well-known that $t_\mathbf{b}(x,v)$ is lower semicontinuous.\\
For each $x \in \partial \Omega$, we introduce the velocity space for the outgoing particles
\begin{align*}
	\mathcal{V}(x):= \{v\in \mathbb{R}^3 : n(x) \cdot v >0\}.
\end{align*}
Fix any point $(x,v) \notin \gamma_0 \cup \gamma_-$, and let $(t_0,x_0,v_0) = (t,x,v)$ and $k \ge 0$. For $v_{k+1} \in \mathcal{V}_{k+1}:=\{v\in \mathbb{R}^3 : n(x_{k+1}) \cdot v >0\}$, we define the back-time cycle as
\begin{align*}
	&X_{\mathbf{cl}}(s;t,x,v) = \sum_k \mathbf{1}_{[t_{k+1},t_k)}(s) X_k(s), \\
	&V_{\mathbf{cl}}(s;t,x,v) = \sum_k \mathbf{1}_{[t_{k+1},t_k)}(s) V_k(s),
\end{align*}
with
\begin{align}\label{timecycle}
	(t_{k+1},x_{k+1},v_{k+1})= (t_{k+1}(x_k,v_k),x_{k+1},v_{k+1}) := (t_k-t_\mathbf{b}(x_k,v_k), x_\mathbf{b}(x_k,v_k),v_{k+1})
\end{align}
and $X_k(s)$ and $V_k(s)$ satisfy the followings
\begin{align} \label{Xtrajcycle}
	&X_k(s) = x_k-\int_{s}^{t_k}V_k(\tau)d\tau,\\ \label{Vtrajcycle}
	&V_k(s) = v_k+\int_{s}^{t_k} \nabla_x \Phi(X_k(\tau))d\tau.
\end{align}
We note that each of $v_j$ are independent variables, and $t_k, x_k$ depend on $t_j, x_j, v_j$ for $1\le j \le k-1$. However, the phase space $\mathcal{V}_j$ depends on $(t,x,v,v_1,v_2,..., v_{j-1})$.\\
We use the follwing notations. $[X'(s'),V'(s')]$ means the position and velocity of a particle at time $s'$, which was at $(s,x',v')$. Similarly, $[X_l(s),V_l(s)]$ and $[X'_{l'}(s'),V'_{l'}(s')]$ present the position and velocity of a particle at time $s$ and $s'$, respectively, where the particle was at $(t_l,x_l,v_l)$ and $(t'_{l'}, x'_{l'},v'_{l'})$.

\subsection{Transversality} \label{Transversalitysec}
The Hamiltonian of the system \eqref{Hacon} is given by
\begin{align*}
	H(x,v) = \frac{|v|^2}{2}+\Phi(x).
\end{align*}
Note that the Hamiltonian $H$ is constant along the characteristics. This implies the crucial property
\begin{align}
	\frac{|v|^2}{2}+\Phi(x) = \frac{|V(s)|^2}{2}+\Phi(X(s)) \quad \text{for }t-t_\mathbf{b}(x,v) \le s \le t.
\end{align}
From the above fact, we derive that
\begin{align*}
	&|V(s)|=\sqrt{|v|^2+2\Phi(x)-2\Phi(X(s))} \le \sqrt{|v|^2+2\Phi(x)} \le  |v|+\sqrt{2\|\Phi\|_\infty},\\
	&|v|=\sqrt{|V(s)|^2+2\Phi(X(s))-2\Phi(x)} \le \sqrt{|V(s)|^2+2\Phi(X(s))} \le |V(s)|+\sqrt{2\|\Phi\|_\infty}.
\end{align*}
This follows that 
\begin{align} \label{sese}
	\left||v|-|V(s)|\right| \le \sqrt{2\|\Phi\|_\infty}.
\end{align}

\bigskip

In Section \ref{Linftydecay} and \ref{LASEC}, we will handle the following term which comes from the double Duhamel iteration : 
\begin{align} \label{fot}
	\int_0^{t}\int_0^{s}\exp\left\{-e^{-\|\Phi\|_\infty}\nu_0(t-s')\right\}\int_{|v'| \le 2R} \int_{|v''| \le 3R}|h(s',X(s';s,X(s;t,x,v),v'),v'')|dv''dv'ds'ds.
\end{align}
In particular, non-degeneracy of mapping $v' \mapsto X(s';s,X(s;t,x,v),v')$ in \eqref{fot} is very important in deriving $L^{2}_{x,v}$ of perturbation $f$. Therefore, we focus on the degeneracy condition
\begin{align*}
	\det\left(\frac{dX(s';s,X(s;t,x,v),v')}{dv'}\right)=0.
\end{align*}
 Now we introduce some lemmas to specify degeneracy in phase space.
\begin{lemma} \cite{AsanoAlmost, ChanBELP} 
	Assume that $\Phi(x) \in C^3(\bar{\Omega})$. Suppose that $\det\left(\frac{dX}{dv}(s_0;T_0,x_0,v_0)\right)=0$ for some $(s_0;T_0,x_0,v_0)$ in $(0,\infty) \cross (0,\infty) \cross \Omega \cross \mathbb{R}^3$. Then there exist $\delta_0 >0$, an open neighborhood $U_{x_0,v_0}$ of $(x_0,v_0) \in \Omega \cross \mathbb{R}^3$, and a family of Lipschitz continuous functions $\{\psi_j\}_{j=1}^3$ on $U_{x_0,v_0}$ with $\psi_j(x_0,v_0) = 0$ such that
	\begin{align*}
		\det\left(\frac{dX}{dv}(s;T_0,x,v)\right)=0
	\end{align*}
	if and only if
	\begin{align*}
		s=s_0+\psi_j(x,v) \quad  \text{for some} \quad j=1,2,3
	\end{align*}
	for $(s,x,v)$ in $(s_0-\delta_0, s_0+\delta_0)\cross U_{x_0,v_0}$.
\end{lemma}

\bigskip

The following lemma gives that
\begin{align*}
	\det\left(\frac{dX}{dv}(s;T_0,x,v)\right) > \delta_*>0
\end{align*}
except for small time intervals. Thus we can overcome the degeneracy case by separating small time intervals. We note that in the proof of the lemma, the compactness of domain $\bar{\Omega}$ is  a crucial point to partition the space $\Omega$ into some open neighborhoods. The poof of this lemma is similar to \cite[Lemma\ 2]{ChanBELP}.
\begin{lemma}\cite{ChanBELP} \label{cpl}
	Assume that $\Phi(x) \in C^3(\bar{\Omega})$. Let $T_0 >0$, $R>0$, and $\epsilon >0$. There are open interval partitions of the time interval $[0,T_0]$, $\mathcal{P}_{i_1}^{T_0}$ for $i_1 = 1,2, ..., M_1$, open partitions of the space $\Omega$, $\mathcal{P}_{I_2}^{\Omega}$ for multi-index $I_2=(i_1,i_2, i_3) \in \{1,2, ..., M_2\}^3$, and open partitions of $[-4R,4R]^3 \subset \mathbb{R}^3$, $\mathcal{P}_{I_3}^v$ for multi-index $I_3 = 	(i_1, i_2, i_3) \in \{1,2, ..., M_3\}^3$ satisfying as follows: For each $i_1$, $I_2$, and  $I_3$, we have $t_{j,i_1,I_2,I_3}$ in $\mathcal{P}_{i_1}^{T_0}$ for $j=1,2,3$ such that
	\begin{align*}
		\left\{s\in \mathcal{P}_{i_1}^{T_0} : \det\left(\frac{dX}{dv}(s;T_0,x,v)\right)=0\right\} \subset \bigcup_{j=1}^3 \left\{s\in \left(t_{j,i_1,I_2,I_3}-\frac{\epsilon}{4M_1},t_{j,i_1,I_2,I_3}+\frac{\epsilon}{4M_1}\right)\right\}.
	\end{align*}
	for all $(x,v)$ in $\mathcal{P}_{I_2}^{\Omega}\cross \mathcal{P}_{I_3}^v$ and
	\begin{align*}
		\det\left(\frac{dX}{dv}(s;T_0,x,v)\right) > \delta_* \quad  \text{for} \quad s\notin \bigcup_{j=1}^3 \left(t_{j,i_1,I_2,I_3}-\frac{\epsilon}{4M_1},t_{j,i_1,I_2,I_3}+\frac{\epsilon}{4M_1}\right)
	\end{align*}
	if $(s,x,v)$ in $\mathcal{P}_{i_1}^{T_0} \cross \mathcal{P}_{I_2}^{\Omega}\cross \mathcal{P}_{I_3}^v$ for all $i_1 , I_2,$ and $I_3$.
\end{lemma}

\bigskip

In Lemma \ref{cpl}, $R>0$ is a sufficiently large constant and $\epsilon>0$ is a small enough constant, which are choosing later. Note that $t_{j,i_1,I_2,I_3}$ and $M_1$ are independent of $\epsilon>0$, but the spatial partitions $\mathcal{P}_{I_2}^{\Omega}$ and the velocity partitions $\mathcal{P}_{I_3}^v$ depend on $\epsilon>0$. \\

\subsection{Main result} \label{mainresult}
We introduce the weight function
\begin{align} \label{weightfnt}
	w(x,v) = \left\{1+\frac{|v|^2}{2}+\Phi(x)\right\}^{\frac{\beta}{2}},\quad \beta >5.
\end{align}
We define $h(t,x,v)= w(x,v) f(t,x,v)$. From the definition of $h$ and the equation \eqref{PBE}, we can derive the full perturbed Boltzmann equation with the external force:
\begin{align} \label{WPBE2}
	\partial_t h + v\cdot \nabla_x h - \nabla_x \Phi(x) \cdot \nabla_v h + e^{-\Phi(x)}\nu(v) h- e^{-\Phi(x)}K_w h = e^{-\frac{\Phi(x)}{2}}w\Gamma\left(\frac{h}{w},\frac{h}{w}\right),
\end{align}
with the diffuse reflection boundary condition for $h$ and where the weighted operator $K_w$ is defined by
\begin{align} \label{woperK}
	K_w h \coloneqq wK\left(\frac{h}{w}\right).
\end{align}
Applying Duhamel principle to the equation \eqref{WPBE2}, we obtain the mild form for $h$
\begin{equation} \label{mildsolh}
	\begin{aligned}
	h(t,x,v) &= S_{G_\nu}(t)h_0(x,v)+\int_0^t S_{G_\nu}(t-s)(e^{-\Phi}K_wh(s))(x,v)ds\\
	& \quad + \int_0^t S_{G_\nu}(t-s)\left(e^{-\frac{\Phi}{2}}w\Gamma\left(\frac{h}{w}, \frac{h}{w}\right)(s)\right)(x,v)ds,
\end{aligned}
	\end{equation}
where $S_{G_\nu}(t)$ is the semigroup of a solution to the equation
\begin{align*}
	\partial_t h +v\cdot \nabla_x h -\nabla_x \Phi(x) \cdot \nabla_v h + e^{-\Phi(x)} \nu(v)h=0
\end{align*}
with the diffuse reflection boundary condition for $h$.\\
Before achieving our main goal, we need to demonstrate the global existence of a solution to the Boltzmann equation \eqref{BE} with the smallness of $\|wf_0\|_{L^\infty_{x,v}}$.

\begin{theorem} [Small perturbation problem] \label{mainresult1}
	Let $w(x,v) = \left\{1+\frac{|v|^2}{2}+\Phi(x)\right\}^\frac{\beta}{2}$ with $\beta >5$. Assume that $F_0(x,v)= \mu_E(x,v)+\mu_E^{\frac{1}{2}}(x,v)f_0(x,v) \ge 0$ satisfying the mass conservation \eqref{initialmass}. Then there exists $\delta_0 >0$ such that if $F_0(x,v) = \mu_E(x,v) + \mu^{\frac{1}{2}}_E(x,v) f_0(x,v) \ge 0$ and $\|wf_0\|_{L^\infty_{x,v}} \le \delta_0$, then there exists a unique (mild) solution $F(t,x,v) = \mu_E(x,v) + \mu^{\frac{1}{2}}_E(x,v) f(t,x,v) \ge 0$ for the Boltzmann equation \eqref{BE} with initial datum $F_0$ and the diffuse reflection boundary condition \eqref{DRBC} such that
	\begin{align*}
		\sup_{0 \le t <\infty}\left\{e^{\lambda_0 t}\|wf(t)\|_{L^\infty_{x,v}}\right\} \le C_0\|wf_0\|_{L^\infty_{x,v}}
	\end{align*}
	for some $\lambda_0 >0$ and $C_0>0$.
\end{theorem}

The main goal of the paper is to prove the global existence of the solution to the Boltzmann equation \eqref{BE} with a large external force and a large oscillation initial datum near the local maxwellian $\mu_E(x,v)$. Instead of overcoming the smallness of $\|wf_0\|_{L^\infty_{x,v}}$, we pay the price that the initial relative entropy $\mathcal{E}(F_0)$ is suitably small,
where a relative entropy $\mathcal{E}(F)$ is given by
\begin{align*}
	\mathcal{E}(F)= \int_{\Omega \cross \mathbb{R}^3} \left(\frac{F}{\mu_E}\log\frac{F}{\mu_E}-\frac{F}{\mu_E}+1\right) \mu_E dxdv.
\end{align*}

\begin{theorem} [Large amplitude problem] \label{mainresult2}
		Let $w(x,v) = \left\{1+\frac{|v|^2}{2}+\Phi(x)\right\}^\frac{\beta}{2}$ with $\beta >5$. Assume that $F_0 (x,v) = \mu_E(x,v) + \mu_E^{\frac{1}{2}}(x,v)f_0(x,v) \ge 0$ satisfying the mass conservation \eqref{initialmass}. For given $M_0 \ge 1$, there exists $\bar{\epsilon}_0 >0$, depending on $\delta_0$ and $M_0$, such that if $\|wf_0\|_{L^\infty_{x,v}} \le M_0$ and $\mathcal{E}(F_0) \le \bar{\epsilon}_0$, then there is a unique (mild) solution $F(t,x,v) = \mu_E(x,v) + \mu_E^\frac{1}{2}(x,v) f(t,x,v) \ge 0$ to the Boltzmann equation \eqref{BE} with initial datum $F_0$ and the diffuse reflection boundary condition \eqref{DRBC} satisfying
		\begin{align*}
			\|wf(t)\|_{L^\infty_{x,v}} \le \tilde{C}_L M_0^5 \exp{\frac{\tilde{C}_L M_0^5}{\nu_0 e^{-\|\Phi\|_\infty}}}e^{-\lambda_L t}
		\end{align*}
		for all $t \ge 0$, where $\tilde{C}_L \ge 1$ is a generic constant and $\lambda_L := \min\left\{\lambda_0, e^{-\|\Phi\|_\infty}\frac{\nu_0}{16}\right\} >0$ with $\nu_0=\inf_{v \in \mathbb{R}^3}\nu(v)$.
\end{theorem}

\subsection{Strategy of the proof}
We will first demonstrate the small perturbation problem (Theorem \ref{mainresult1}) through Section \ref{L2decay} to \ref{Nonlineardecay}. Next, the large amplitude problem (Theorem \ref{mainresult2}), which is our main goal, will be proven in Section \ref{LASEC}. As in \cite{GDCB2010}, we will use a $L^2-L^\infty$ bootstrap argument to derive the exponential $L^\infty$ decay to the linearized Boltzmann equation. In Section \ref{L2decay}, we will use the a priori estimate to smallness for $\|wf\|_{L^\infty_{x,v}}$ to derive
\begin{align*}
	\int_0^t\|P_Lf(s)\|_{L^2_{x,v}}^2ds \lesssim \|f(t)\|_{L^2_{x,v}}^2+\|f(0)\|_{L^2_{x,v}}^2 +\int_0^t\|(I-P_L)f(s)\|_{L^2_{x,v}}^2ds+\text{(boundary effect)},
\end{align*}
where $f$ is a solution to the linearized Boltzmann equation and $P_L$ is defined in \eqref{LPro}. See \eqref{AAL2}.
Using this fact and the $L^2$ coercivity for the linear operator $L$, we can show the linear $L^2$ decay $\|f(t)\|_{L^2_{x,v}} \lesssim e^{-\lambda t}\|f_0\|_{L^2_{x,v}}$ for some $\lambda >0$. We set $h(t,x,v) =w(x,v)f(t,x,v)$. In Section \ref{Linftydecay}, we will apply the double Duhamel principle to the linearized Boltzmann equation in order to get roughly the following form:
\begin{align}
	S_G(t) h_0 &\sim \text{(initial datum's contribution)} \nonumber\\
	&+\int_0^{t}\int_0^{s}\exp\left\{-e^{-\|\Phi\|_\infty}\nu_0(t-s')\right\}\int_{|v'| \le 2R} \int_{|v''| \le 3R}|h(s',X(s';s,X(s;t,x,v),v'),v'')|dv''dv'ds'ds \label{nondp}\\ 
	&+\text{(remainder part)}\nonumber ,
\end{align}
where $S_G(t)$ is the solution operator to the linearized Boltzmann equation. Here, to address a non-degeneracy problem mentioned in subsection \ref{Transversalitysec}, we will use Lemma \ref{cpl} to the term \eqref{nondp}. Additionally, we obtain the linear $L^\infty$ decay $\|S_G(t)h_0\|_{L^\infty_{x,v}} \lesssim e^{-\lambda_
\infty t}\|h_0\|_{L^\infty_{x,v}}$ for some $\lambda_\infty>0$ by using the linear $L^2$ decay. In Section \ref{Nonlineardecay}, we can overcome the difficulty of dealing with the nonlinear term $\Gamma$ thanks to the smallness for the initial data and the Gamma estimate (Lemma \ref{Gamest}). Therefore, we can conclude the small perturbation problem, which is one of our main results.\\
\indent In Section \ref{LASEC}, we will deal with the large amplitude problem. To overcome the velocity growth in the loss term $\Gamma_-(f,f)$ like a factor $\nu(v)$, we decompose the nonlinear term $\Gamma$ into $\Gamma_+$ and $\Gamma_-$. We will then combine $\Gamma_-(f,f)$ and $\nu(v)f$, denoted by $R(f):= \int_{\mathbb{R}^3 \times \mathbb{S}^2}B(u-v,\omega)F(u)d\omega du$. Unlike the small perturbation problem, the problem in this section may involve an initial data condition with large amplitude, and we need a different approach to handle the nonlinear gain term. We will overcome this problem by using the smallness for the relative entropy $\mathcal{E}(F_0)$ and introducing an estimate to the gain term $\Gamma_+$. (See Lemma \ref{Gam+est}.) As in \cite{DW2019}, to derive the lower bound for $R(f)$, under the a priori assumption $\sup_{0\le t\le T}\|h(t)\|_{L^\infty_{x,v}} \le \bar{M}$, we will prove
\begin{align} \label{zz11}
	\int_{\mathbb{R}^3}e^{-\frac{|u|^2}{8}}|f(t,x,u)|du \le C \quad \text{for all }t\ge \tilde{t}\ \text{and }x\in \Omega,
\end{align}
for some generic small constant $C$ and and time $\tilde{t}$. In the process of proving it, we use the mild formulation:
\begin{equation} \label{mildsolh11}
	\begin{aligned}
	h(t,x,v) &= S_{G_\nu}(t)h_0(x,v)+\int_0^t S_{G_\nu}(t-s)(e^{-\Phi}K_wh(s))(x,v)ds\\
	& \quad + \int_0^t S_{G_\nu}(t-s)\left(e^{-\frac{\Phi}{2}}w\Gamma\left(\frac{h}{w}, \frac{h}{w}\right)(s)\right)(x,v)ds,
\end{aligned}
	\end{equation}
where $S_{G_\nu}(t)$ is the semigroup of a solution to the equation
\begin{align*}
	\partial_t h +v\cdot \nabla_x h -\nabla_x \Phi(x) \cdot \nabla_v h + e^{-\Phi(x)} \nu(v)h=0
\end{align*}
with the diffuse reflection boundary condition. Similar to the proof of Theorem \ref{LT}, the second and third terms of the right-hand side of \eqref{mildsolh11} can be controlled by $\mathcal{S}(\tilde{\epsilon},\lambda,R)\bar{M}^3+C_{\tilde{\epsilon},\lambda,R}[\mathcal{E}(F_0)^{\frac{1}{2}}+\mathcal{E}(F_0)]$, where $\mathcal{S}(\tilde{\epsilon},\lambda,R)$ can be small enough. Thus, thanks to the smallness for $\mathcal{E}(F_0)$, we can derive the inequality \eqref{zz11}. From this process, we obtain the exponential time-decay to the equation $\partial_t f + v\cdot \nabla_x f-\nabla_x\Phi \cdot \nabla_v f +R(f)f=0$. Next, we will obtain the $L^\infty$ estimate to the full perturbed Boltzmann equation in the large amplitude problem. Here, we will use the following mild formulation:
\begin{equation} \label{mildsolh11}
	\begin{aligned}
	h(t,x,v) &= S_{G_{f}}(t)h_0(x,v)+\int_0^t S_{G_f}(t-s)(e^{-\Phi}K_wh(s))(x,v)ds\\
	& \quad + \int_0^t S_{G_f}(t-s)\left(e^{-\frac{\Phi}{2}}w\Gamma_+\left(\frac{h}{w}, \frac{h}{w}\right)(s)\right)(x,v)ds,
\end{aligned}
	\end{equation}
where $S_{G_f}(t)$ is the semigroup of a solution to the equation
\begin{align*}
	\partial_t h +v\cdot \nabla_x h -\nabla_x \Phi(x) \cdot \nabla_v h + R(f)h=0.
\end{align*}
Through a similar approach to the proof of Theorem \ref{LT} and Lemma \ref{Rfest}, and thanks to the exponential time-decay of the solution operator $S_{G_f}(t)$, we can derive the $L^\infty$ estimate to the full perturbed Boltzmann equation (Theorem \ref{Linfty2}). Under the smallness assumption on $\mathcal{E}(F_0)$, this gives the following Gr\"{o}nwall type:
\begin{align*}
	\|h(t)\|_{L^\infty_{x,v}} \le C_{M_0} \left(1+\int_0^t \|h(s)\|_{L^\infty_{x,v}}ds\right)\exp{-e^{-\|\Phi\|_\infty} \frac{\nu_0}{8}t}+E \quad \text{for all } 0\le t \le T,
\end{align*} 
where $E$ can be small enough. From this inequality, when sufficient time $T_1$ has passed, the amplitude of the solution $h=wf$ becomes smaller than the small amplitude $\delta_0$ in Theorem \ref{mainresult1}. By the local existence theorem (Theorem \ref{LocExist}), the existence of the solution is guaranteed up to time $T_1$. On the other hand, based on the small perturbation problem (Theorem \ref{mainresult2}), we obtain the solution existence and its asymptotic stability after time $T_1$. 

\subsection{Organizaiton of the paper}
The subsequent sections are organized as follows. 
In Section \ref{L2decay}, we present the exponential decay in $L^2_{x,v}$ for solutions to the linearized Boltzmann equation. We will use the a priori assumption to derive the exponential decay. 
In Section \ref{Linftydecay}, we use the $L^2-L^\infty$ bootstrap argument to derive the exponential decay in $L^\infty_{x,v}$ of the solution to the Linearized Boltzmann equation from the result in the previous section. We also close the a priori estimate for the small amplitude problem in the previous section. 
In Section \ref{Nonlineardecay}, we will handle one of the main results. We show the global existence of the solution to the full perturbed Boltzmann equation with a small amplitude initial datum.
In Section \ref{LASEC} as a main part, given a large oscillation initial datum, we solve the global existence of the solution to the Boltzmann equation. To achieve our main aim, we will introduce the estimates to derive a main goal under the a priori assumption and then from a Gr\"onwall type we apply the global-in time existence of a solution given by the small amplitude initial datum to prove our main goal. In Section \ref{Appendix}, we provide an appendix for  the proof of Lemma \ref{coer}.

\section{Linear $L^2$ decay and A priori estimate} 
\label{L2decay}

In this section, we consider the linearized Boltzmann equation of \eqref{PBE}:
\begin{align} \label{LBE}
	\partial_t f + v\cdot \nabla_x f - \nabla_x \Phi(x) \cdot \nabla_vf + e^{-\Phi(x)}Lf = 0,
\end{align}
and our aim is to prove the exponential $L^2$ decay to the linearized Boltzmann equation under the a priori assumption.
We define the $L^2_{v}$ projection $P_L$ of $f$ corresponding to operator $L$ as
\begin{align} \label{LPro}
	P_L(f)(t,x,v) = a(t,x) \mu^{\frac{1}{2}}(v) + b(t,x)\cdot v\mu^{\frac{1}{2}}(v) + c(t,x)\frac{|v|^2-3}{\sqrt{6}}\mu^{\frac{1}{2}}(v),
\end{align}
where
\begin{align*}
	& a(t,x) = \int_{\mathbb{R}^3}f(t,x,v) \mu^{\frac{1}{2}}(v)dv,\\
	& b(t,x) = \int_{\mathbb{R}^3}vf(t,x,v) \mu^{\frac{1}{2}}(v)dv,\\
	& c(t,x) = \int_{\mathbb{R}^3}\frac{|v|^2-3}{\sqrt{6}}f(t,x,v) \mu^{\frac{1}{2}}(v)dv.
\end{align*}
It is well-known the operator $L$ satisfies the $L^2$ coercivity $(Lf,f)_{L_v^2} \ge C_L\|(I-P_L)f\|_{L_v^2}^2$ for all $f$ in $L_v^2$, where $C_L$ is a generic constant.\\
We also define the $L^2_v$ projection $P_\gamma$ of $f$ on the boundary $\gamma$ as
\begin{align} \label{Pgamm}
	P_\gamma f = c_\mu \mu^{\frac{1}{2}}(v)\int_{n(x)\cdot v' >0} f(t,x,v')\mu^{\frac{1}{2}}(v')\{n(x)\cdot v'\}dv'.
\end{align}
\newline
\indent The following lemma states the $L^2_{x,v}$ bound for $P_Lf$ by $(I-P_L)f$ and the effects of the boundary. The lemma gives the key estimate to derive the exponential decay in $L^2_{x,v}$. The proof of this lemma is left in Section \ref{Appendix}.
\begin{lemma} \label{coer}
	Let $f_0(x,v)$ and $g(t,x,v)$ be in $L^2_{x,v}$ such that $f_0$ and $g$ satisfy the mass conservation $\int_{\Omega \cross \mathbb{R}^3} f_0 \mu_E^{\frac{1}{2}}dxdv=0$ and $\int_{\Omega \cross \mathbb{R}^3} g \mu_E^{\frac{1}{2}}dxdv=0$. Suppose that $f(t,x,v) \in L^2_{x,v}$ is a solution to 
	\begin{align} \label{SLBE}
		\partial_t f + v\cdot \nabla_x f - \nabla_x \Phi(x) \cdot \nabla_vf + e^{-\Phi(x)}Lf =g
	\end{align} 
	with initial datum $f_0$, diffuse reflection boundary condition \eqref{PDRBC}, and satisfying the mass conservation \eqref{massconserv}. Assume that $f|_\gamma$ belongs to $L^2_\gamma$. Then there exist a constant $C_\perp >0$, depending on $\Phi$, and a function $G_f(t)$ such that for all $t\ge 0$,
	\begin{align*}
		&(\romannumeral 1)\quad \left|G_f(t)\right| \le C\|f(t)\|^2_{L^2_{x,v}},\\
		&(\romannumeral 2)\quad \int_0^t\left\|P_{L}(f)(s)\right\|^2_{L^2_{x,v}}ds \le G_f(t) - G_f(0) +C_\perp\int_0^t\left[\left\|\left(I-P_{L}\right)(f)(s)\right\|^2_{L^2_{x,v}}+\left\|\left(I-P_\gamma\right)(f)(s)\right\|^2_{L_{\gamma_+}^2} \right] ds\\
		& \qquad \qquad \qquad \qquad \qquad \qquad \quad +C_\perp\int_0^t\left\|g(s)\right\|^2_{L^2_{x,v}}ds + C_\perp\int_0^t \|wf(s)\|_{L^\infty_{x,v}}\|P_L(f)(s)\|_{L^2_{x,v}}^2ds,
	\end{align*} 
	where $P_L$ and $P_\gamma$ are defined in \eqref{LPro} and \eqref{Pgamm}, respectively. Recall the norm $\|\cdot\|_{L^2_{\gamma_+}}$ is defined in subsection \ref{notation}.
\end{lemma}

\bigskip

From Section \ref{L2decay} to Section \ref{Linftydecay}, we make the following a priori assumption :
\begin{align} \label{AAL2}
	\sup_{0\le t \le \tilde{T}_0} \left\{ e^{\tilde{\lambda} t}\|wf(t)\|_{L^\infty_{x,v}}\right\} \le \eta,
\end{align}
where $\tilde{T}_0>0$, $\eta>0$, and $\tilde{\lambda} >0$ is choosing later. These constants will be determined in subection \ref{Apeisdp}. We use the a priori assumption to guarantee the smallness of $\|wf(s)\|_{L^\infty_{x,v}}$ over the time interval $[0,\tilde{T}_0]$. Through this assumption, we can deduce the exponential decay in $L^2_{x,v}$ by using Lemma \ref{coer}.
\begin{theorem} \label{T31}
	Let $f_0(x,v)$ be in $L^2_{x,v}$ such that $f_0$ satisfies the mass conservation  $\int_{\Omega\times\mathbb{R}^3} f_0 \mu_{E}^{\frac{1}{2}}dxdv=0$. Suppose that $f(t,x,v) \in L^2_{x,v}$ be a solution to \eqref{LBE} with initial datum $f_0$, diffuse reflection boundary condition \eqref{PDRBC}, and satisfying the mass conservation $\int_{\Omega\times\mathbb{R}^3} f \mu_{E}^{\frac{1}{2}}dxdv=0$. Assume that $f|_\gamma$ belongs to $L^2_\gamma$. Then,  under the a priori assumption \eqref{AAL2}, there exist $C_G,\lambda_G>0$, independent of $f_0$ and $f$, such that for all $0<\lambda < \lambda_G$,
	\begin{align*}
		 \|f(t)\|_{L^2_{x,v}} \le C_Ge^{-\lambda t} \|f_0\|_{L^2_{x,v}}\quad \text{for all } t\ge 0.
	\end{align*}
\end{theorem}
\begin{proof}
	Let $0\le t \le \tilde{T}_0$ and set $g(t,x,v) = e^{\tilde{\lambda} t}f(t,x,v)$, where $\tilde{T}_0$ and $\tilde{\lambda}$ are  constants in the a priori assumption \eqref{AAL2}\\
	Then $g$ satisfies the mass conservation and $g$ is a solution of
	\begin{align*}
		\partial_t g + v\cdot \nabla_x g-\nabla_x \Phi(x) \cdot \nabla_v g +e^{-\Phi(x)} L (g) = \tilde{\lambda} g.
	\end{align*}
	Using the Green identity, we get 
	\begin{align*}
		\frac{1}{2} \frac{d}{dt}\|g(t)\|^2_{L^2_{x,v}}
		& = -\int_{\Omega\cross \mathbb{R}^3}(v\cdot \nabla_x g)gdxdv +\int_{\Omega\cross \mathbb{R}^3}(\nabla_x\Phi(x)\cdot \nabla_v g)gdxdv - \int_{\Omega\cross \mathbb{R}^3} e^{-\Phi(x)}L (g)g dxdv  \\
		&\quad + \int_{\Omega\cross \mathbb{R}^3} \tilde{\lambda} g^2 dxdv \\
		& = -\frac{1}{2}\int_{\gamma} \left(g^2\right)\{n(x) \cdot v\}dS(x)dv- \int_{\Omega}e^{-\Phi(x)} \left(L (g),g\right)_{L_v^2} dx+ \tilde{\lambda} \|g(t)\|^2_{L^2_{x,v}}.
	\end{align*}
	From the coercivity $\left(L (g),g\right)_{L_v^2} \ge C_L \left\|\left(I-P_{L}\right)g\right\|_{L^2_v}^2$, we deduce
	\begin{align*}
		\frac{1}{2} \frac{d}{dt}\|g(t)\|^2_{L^2_{x,v}} \le -\frac{1}{2}\int_{\gamma} \left(g^2\right)\{n(x) \cdot v\}dS(x)dv -C_L e^{-\|\Phi\|_\infty}\left\|\left(I-P_{L}\right)g(t)\right\|_{L^2_{x,v}}^2 + \tilde{\lambda} \|g(t)\|^2_{L^2_{x,v}}.
	\end{align*}
	Here, we can compute
	\begin{align*}
		\int_{\gamma} \left(g^2\right)\{n(x) \cdot v\}dS(x)dv &= \int_{\gamma_+}\left(g^2\right)\{n(x) \cdot v\}dS(x)dv+\int_{\gamma_-}\left(g^2\right)\{n(x) \cdot v\}dS(x)dv\\
		& = \int_{\gamma_+} \left[g^2-\left(P_\gamma g\right)^2\right]\{n(x) \cdot v\}dS(x)dv\\
		&= \left\|(I-P_\gamma)g(t)\right\|_{L^2_{\gamma_+}}^2 -2\int_{\gamma_+}\left(P_\gamma g\right)\left(\left(I-P_\gamma\right)g\right)\{n(x) \cdot v\}dS(x)dv\\
		& = \left\|(I-P_\gamma)g(t)\right\|_{L^2_{\gamma_+}}^2,
	\end{align*}
	where we have used the change of variables $\gamma_{-} \mapsto \gamma_+$ and $P_\gamma $ is defined in \eqref{Pgamm}.\\
	Taking the integral from $0$ to $t$, we deduce
	\begin{equation} \label{ddd1}
	\begin{aligned}
		\|g(t)\|^2_{L^2_{x,v}}-\|g(0)\|^2_{L^2_{x,v}} &\le -\int_0^t\left\|(I-P_\gamma)g(s)\right\|_{L^2_{\gamma_+}}^2ds -2C_Le^{-\|\Phi\|_\infty}\int_0^t \left\|\left(I-P_{L}\right)g(s)\right\|_{L^2_{x,v}}^2 ds\\
		& \quad + 2\tilde{\lambda}\int_0^t \|g(s)\|^2_{L^2_{x,v}}ds.
	\end{aligned}
	\end{equation}
	Applying Lemma \ref{coer} to $g$, we have
	\begin{equation} \label{ddd2}
	\begin{aligned}
		\int_0^t\left\|P_{L}(g)(s)\right\|^2_{L^2_{x,v}}ds &\le C_\perp\|g(t)\|^2_{L^2_{x,v}} + C_\perp\|g(0)\|^2_{L^2_{x,v}}\\
		& \quad +C_\perp\int_0^t\left\{\left\|\left(I-P_{L}\right)g(s)\right\|^2_{L^2_{x,v}}+\left\|(I-P_\gamma)g(s)\right\|_{L^2_{\gamma_+}}^2+\tilde{\lambda}^2\left\|g(s)\right\|^2_{L^2_{x,v}}\right\}ds\\
		&\quad + C_\perp\int_0^t \|wg(s)\|_{L^\infty_{x,v}}\|P_L(g)(s)\|_{L^2_{x,v}}^2ds.
	\end{aligned}
	\end{equation}
	For $\delta >0$, \eqref{ddd1} + $\delta \times$\eqref{ddd2} yields
	\begin{align*}
		&(1-\delta C_\perp) \|g(t)\|^2_{L^2_{x,v}}+ (2C_Le^{-\|\Phi\|_\infty}-\delta C_\perp) \int_0^t \left\|\left(I-P_{L}\right)g(s)\right\|_{L^2_{x,v}}^2 ds+ \delta \int_0^t\left\|P_{L}(g)(s)\right\|^2_{L^2_{x,v}}ds\\ &+ (1-\delta C_\perp)\int_0^t\left\|(I-P_\gamma)g(s)\right\|_{L^2_{\gamma_+}}^2ds\\ 
		&\le (1+\delta C_\perp) \|g(0)\|^2_{L^2_{x,v}} + (2\tilde{\lambda} + \delta C_\perp\tilde{\lambda}^2)\int_0^t \|g(s)\|^2_{L^2_{x,v}}ds + \delta C_\perp \sup_{0\le s \le \tilde{T}_0}\|wg(s)\|_{L^\infty_{x,v}} \int_0^t  \left\|P_{L}(g)(s)\right\|_{L^2_{x,v}}^2ds\\
		&\le (1+\delta C_\perp) \|g(0)\|^2_{L^2_{x,v}} + (2\tilde{\lambda} + \delta C_\perp\tilde{\lambda}^2)\int_0^t \|g(s)\|^2_{L^2_{x,v}}ds + \delta C_\perp \eta \int_0^t \left\|P_{L}(g)(s)\right\|_{L^2_{x,v}}^2ds,
	\end{align*}
	where $\eta>0$ is a constant in the a priori assumption \eqref{AAL2}.\\
	Choosing $\eta >0$ such that
	\begin{align} \label{cc3}
		\eta C_\perp < \frac{1}{2},
	\end{align} 
	 the above inequality becomes
	\begin{align*}
		&(1-\delta C_\perp) \|g(t)\|^2_{L^2_{x,v}}+ (2C_Le^{-\|\Phi\|_\infty}-\delta C_\perp) \int_0^t \left\|\left(I-P_{L}\right)g(s)\right\|_{L^2_{x,v}}^2 ds+ \frac{\delta}{2} \int_0^t\left\|P_{L}(g)(s)\right\|^2_{L^2_{x,v}}ds\\ &+ (1-\delta C_\perp)\int_0^t\left\|(I-P_\gamma)g(s)\right\|_{L^2_{\gamma_+}}^2ds\\ 
		&\le (1+\delta C_\perp) \|g(0)\|^2_{L^2_{x,v}} + (2\tilde{\lambda} + \delta C_\perp\tilde{\lambda}^2)\int_0^t \|g(s)\|^2_{L^2_{x,v}}ds.
	\end{align*}
	Note that $\left\|\left(I-P_{L}\right)g(s)\right\|_{L^2_{x,v}}^2 + \left\|P_{L}(g)(s)\right\|^2_{L^2_{x,v}} = \|g(s)\|^2_{L^2_{x,v}} $.\\ 
	Firstly, choosing sufficiently small $\delta >0$ such that $2C_Le^{-\|\Phi\|_\infty}-\delta C_\perp >\frac{\delta}{2}$, we get
	\begin{align*}
		&(1-\delta C_\perp) \|g(t)\|^2_{L^2_{x,v}}+C\delta \int_0^t\left\|g(s)\right\|^2_{L^2_{x,v}}ds + (1-\delta C_\perp)\int_0^t\left\|(I-P_\gamma)g(s)\right\|_{L^2_{\gamma_+}}^2ds \\
		&\le (1+\delta C_\perp) \|g(0)\|^2_{L^2_{x,v}} + (2\tilde{\lambda} + \delta C_\perp\tilde{\lambda}^2)\int_0^t \|g(s)\|^2_{L^2_{x,v}}ds.
	\end{align*}
	Next, choosing small enough $\tilde{\lambda}>0$ such that 
	\begin{align} \label{cc4}
		\frac{\delta}{2} > 2\tilde{\lambda} + \delta C_\perp\tilde{\lambda}^2,
	\end{align}
	we obtain
	\begin{align*}
		(1-\delta C_\perp) \|g(t)\|^2_{L^2_{x,v}} \le (1+\delta C_\perp) \|g(0)\|^2_{L^2_{x,v}}.
	\end{align*}
	Hence we conclude that
	\begin{align*}
		\|f(t)\|_{L^2_{x,v}} \le Ce^{-\tilde{\lambda} t} \|f_0\|_{L^2_{x,v}} \ \text{for all } 0\le t \le \tilde{T}_0.
	\end{align*}
	Put $\lambda_G := \tilde{\lambda}$. Choosing sufficiently large $\tilde{T}_0>0$, for all $0<\lambda < \lambda_G$,
	\begin{align} \label{cc5}
		\|f(\tilde{T}_0)\|_{L^2_{x,v}} \le Ce^{-\lambda_G \tilde{T}_0} \|f_0\|_{L^2_{x,v}} \le e^{-\lambda \tilde{T}_0}\|f_0\|_{L^2_{x,v}}
	\end{align}
	and applying repeatedly the process \eqref{cc5}, we get
	\begin{align*}
		\|f(l\tilde{T}_0)\|_{L^2_{x,v}} \le e^{-\lambda \tilde{T}_0} \|f((l-1)\tilde{T}_0)\|_{L^2_{x,v}} \le e^{-l\lambda \tilde{T}_0} \|f_0\|_{L^2_{x,v}}.
	\end{align*}
	Thus for $l\tilde{T}_0 \le t \le (l+1)\tilde{T}_0$ with $l \ge 1$, we obtain
	\begin{align*}
		\|f(t)\|_{L^2_{x,v}} \le C_{\tilde{T}_0} e^{-l\lambda \tilde{T}_0}\|f_0\|_{L^2_{x,v}} \le C_{\tilde{T}_0} e^{-\lambda t}e^{\lambda \tilde{T}_0} \|f_0\|_{L^2_{x,v}} \le C_{\tilde{T}_0} e^{-\lambda t}\|f_0\|_{L^2_{x,v}}
	\end{align*}
	since $0\le t-l\tilde{T}_0 \le \tilde{T}_0$. We complete the proof of this theorem.
\end{proof}

\bigskip

\section{Linear $L^\infty$ estimate}
\label{Linftydecay} 
Setting $h(t,x,v)=w(x,v)f(t,x,v)$, it follows from \eqref{PBE} and \eqref{PDRBC} that
\begin{align} \label{WPBE}
	\partial_t h + v\cdot \nabla_x h - \nabla_x \Phi(x) \cdot \nabla_v h + e^{-\Phi(x)}\nu(v) h- e^{-\Phi(x)}K_w h = e^{-\frac{\Phi(x)}{2}}w\Gamma\left(\frac{h}{w},\frac{h}{w}\right)
\end{align}
with the diffuse reflection boundary condition
\begin{align} \label{WWPBEBC}
	h(t,x,v)|_{\gamma_-} = w(x,v) \mu^{\frac{1}{2}}(v) \int_{n(x) \cdot v' >0} h(t,x,v') \frac{1}{w(x,v')\mu^{\frac{1}{2}}(v')}d\sigma(x).
\end{align}
Here, the probability measure $d\sigma = d\sigma(x)$ is given by
\begin{align*}
	d\sigma(x) = c_\mu \mu(v')\{n(x) \cdot v'\}dv'.
\end{align*}
\indent We denote the iterated integral by
\begin{align} \label{iterint}
	\int_{\prod_{l=1}^{k-1}\mathcal{V}_l}\prod_{l=1}^{k-1}d\sigma_l :=\int_{\mathcal{V}_1} \cdots \left\{\int_{\mathcal{V}_{k-1}}d\sigma_{k-1} \right\}\cdots d\sigma_1,
\end{align}
where $\mathcal{V}_{j}= \{v\in \mathbb{R}^3 : n(x_j) \cdot v>0\}$ is defined in subsection \ref{SDBC} and $d\sigma_j = c_\mu \mu(v_j) \{n(x_j) \cdot v_j\}dv_j$ for $v_j \in \mathcal{V}_j$.\\
\newline
\indent Motivated by \cite{GDCB2010}, the next lemma states the phase space $\prod_{l=1}^{k-1}\mathcal{V}_l$ not reaching $t=0$ is sufficiently small when $k$ is large enough.  
\begin{lemma} \label{Lsmall}
	For any $\epsilon > 0$, there exists $k_0(\epsilon,T_0)>0$ such that for $k \ge k_0$, for all $(t,x,v)$, $0 \le t \le T_0$, $x \in \bar{\Omega}$ and $v \in \mathbb{R}^3$,
	\begin{align} \label{Lsmall1}
		\int_{\prod_{l=1}^{k-1}\mathcal{V}_l} \mathbf{1}_{\{t_k(t,x,v,v_1,v_2,...,v_{k-1})>0\}} \prod_{l=1}^{k-1}d\sigma_l \le \epsilon.
	\end{align}
Furthermore, for $T_0$ sufficiently large, there exist constant $C_1,C_2>0$, independent of $T_0$, such that for $k=C_1T_0^{\frac{5}{4}}$,
\begin{align} \label{Lsmall2}
	\int_{\prod_{l=1}^{k-1}\mathcal{V}_l} \mathbf{1}_{\{t_k(t,x,v,v_1,v_2,...,v_{k-1})>0\}} \prod_{l=1}^{k-1}d\sigma_l \le \left\{\frac{1}{2}\right\}^{C_2T_0^{\frac{5}{4}}}.
\end{align}
\end{lemma}
\begin{proof}
	Take $0<\delta<1$ sufficiently small. We define non-grazing sets for $1\le l \le k-1$ as
	\begin{align*}
		\mathcal{V}_l^\delta = \left\{v_l \in \mathcal{V}_l: n(x_l) \cdot v_l \ge \delta\right\} \cap \left\{v_l \in \mathcal{V}_l : |v_l| \le \frac{1}{\delta} \right\}.
	\end{align*}
	We can easily compute that
	\begin{align*}
		\int_{\mathcal{V}_l \backslash\mathcal{V}_l^\delta} d\sigma_l \le \int_{v_l \cdot n(x_l) \le \delta} d\sigma_l + \int_{|v_l| \ge \frac{1}{\delta}} d\sigma_l \le C\delta,
	\end{align*}
	where C is a constant independent of $l$.\\
	Now, we claim that
	\begin{align*}
		|t_l-t_{l+1}| \ge \frac{\delta^3}{C_\xi} \quad \text{if} \quad v_l \in \mathcal{V}_l^\delta \quad \text{and} \quad 0\le t \le T_0. 
	\end{align*}
	From the fact that
	\begin{align*}
		\lim_{y \rightarrow x_1, y\in \partial \Omega} \frac{|\{x_1-y\}\cdot n(x_1)|}{|x_1-y|} = 0 \quad \text{for} \quad x_1 \in \partial\Omega,
	\end{align*}
	we derive that
	\begin{align*}
		\left|\int_{t_{l+1}}^{t_l}V(s;t_l,x_l,v_l)ds\right|^2 &= |x_l-x_{l+1}|^2\\
		& \ge C_{\xi}|(x_l-x_{l+1})\cdot n(x_l)|\\
		& = C_{\xi}\left|\int_{t_{l+1}}^{t_l}V(s:t_l,x_l,v_l) \cdot n(x_l) ds \right|\\
		& = C_{\xi}\left|\int_{t_{l+1}}^{t_l}\left(v_l -\int_{t_l}^s \nabla_x \Phi(X(\tau; t_l,x_l,v_l))d\tau\right)\cdot n(x_l)ds\right|\\
		& \ge C_{\xi} \left( |v_l \cdot n(x_l)| |t_l-t_{l+1}| - \left|\int_{t_{l+1}} ^{t_l}\int_{t_l}^s \nabla_x \Phi (X(\tau;t_l,x_l,v_l))\cdot n(x_l) d\tau ds \right|\right).
	\end{align*}
	This implies that
	\begin{align*}
		|v_l \cdot n(x_l)| &\le \frac{C_\xi}{|t_l-t_{l+1}|}\left|\int_{t_{l+1}}^{t_l}V(s;t_l,x_l,v_l)ds\right|^2\\
		& \quad + \frac{C_\xi}{|t_l-t_{l+1}|} \left|\int_{t_{l+1}}^{t_l}\int_{t_l}^s \nabla_x \Phi (X(\tau;t_l,x_l,v_l))\cdot n(x_l) d\tau ds \right|.
	\end{align*}
	Here, we have
	\begin{align*}
		\left|\int_{t_{l+1}}^{t_l}V(s;t_l,x_l,v_l)ds\right|^2 &= \left|v_l(t_l-t_{l+1})-\int_{t_{l+1}}^{t_l}\int_{t_l}^s \nabla_x \Phi (X(\tau;t_l,x_l,v_l))d\tau ds \right|^2\\
		&\lesssim |v_l|^2|t_l-t_{l+1}|^2 + |t_l-t_{l+1}|^4\|\nabla_x \Phi\|_{L^{\infty}_x}^2
	\end{align*}
	and
	\begin{align*}
		\left|\int_{t_{l+1}} ^{t_l}\int_{t_l}^s \nabla_x \Phi (X(\tau;t_l,x_l,v_l))\cdot n(x_l) d\tau ds \right| \le |t_l-t_{l+1}|^2\|\nabla_x \Phi\|_{L_x^\infty}.
	\end{align*}
	By the above computation, for $v_l \in \mathcal{V}_l^\delta$, we obtain that
	\begin{align*}
		|v_l \cdot n(x_l)| & \le C_\xi |t_l-t_{l+1}|\left\{|v_l|^2 + |t_l- t_{l+1}|^2\|\nabla_x \Phi\|^2_{L^\infty_x} +\|\nabla_x \Phi\|_{L^\infty_x} \right\}\\
		& \le C_\xi |t_l-t_{l+1}| \left\{\frac{1}{\delta^2}+T_0^2\|\nabla_x \Phi\|^2_{L^\infty_x}+\|\nabla_x \Phi\|_{L^\infty_x} \right\}.
	\end{align*}
	Choosing $\delta>0$ so that $\delta \le  \frac{1}{T_0\left(1+\|\nabla_x\Phi\|_{L^\infty_x}\right)}$, we get
	\begin{align*}
		|t_l-t_{l+1}| \ge \frac{\delta^3}{C_\xi} \quad \text{if} \quad v_l \in \mathcal{V}_l^\delta \quad \text{and} \quad 0\le t \le T_0. 
	\end{align*}
	Thus, if $t_k(t,x,v,v_1, v_2, ..., v_{k-1})>0,$ then there are at most $\left[\frac{C_\xi T_0}{\delta^3}\right]+1$ number of $v_l \in \mathcal{V}_l^{\delta}$ for $1\le l \le k-1$, where $[x]$ is the largest integer less than or equal to $x$. Therefore, we have
	\begin{align}
		& \int_{\prod_{l=1}^{k-1}\mathcal{V}_l} \mathbf{1}_{\{t_k(t,x,v,v_1,v_2,...,v_{k-1})>0\}} \prod_{l=1}^{k-1}d\sigma_l \nonumber \\
		& \le \sum^{\left[\frac{C_\xi T_0}{\delta^3}\right]+1}_{j=1} \int_{\{\text{There are exactly} \ j \ of  \ v_{l_i} \in \mathcal{V}_{l_i}^\delta \ and \ k-1-j \ of \ v_{l_i} \notin \mathcal{V}_{l_i}^\delta\}} \prod_{l=1}^{k-1}d\sigma_l \nonumber\\
		& \le \sum^{\left[\frac{C_\xi T_0}{\delta^3}\right]+1}_{j=1} \binom{k-1}{j} \left|\sup_l \int_{\mathcal{V}_l^\delta} d\sigma_l\right|^{j} \left|\sup_l \int_{\mathcal{V}_l \backslash \mathcal{V}_l^\delta} d\sigma_l\right|^{k-1-j}\nonumber\\ \label{LSM1}
		& \le \left(\left[\frac{C_\xi T_0}{\delta^3}\right]+1\right)(k-1)^{\left[\frac{C_\xi T_0}{\delta^3}\right]+1} (C\delta)^{k-2-\left[\frac{C_\xi T_0}{\delta^3}\right]}.
	\end{align}
	For $\epsilon >0$, taking $k \gg \left[\frac{C_\xi T_0}{\delta^3}\right] +1$ and $C\delta <1$, it follows that
	\begin{align*}
		\int_{\prod_{l=1}^{k-1}\mathcal{V}_l} \mathbf{1}_{\{t_k(t,x,v,v_1,v_2,...,v_{k-1})>0\}} \prod_{l=1}^{k-1}d\sigma_l \le \epsilon.
	\end{align*}
	For \eqref{Lsmall2}, we take $k-2 = 15\left\{\left[\frac{C_\xi T_0}{\delta^3}\right]+1\right\}$. Then \eqref{LSM1} becomes
	\begin{align*}
		\left\{15\left(\frac{C_\xi T_0}{\delta^3} +1 \right) (C\delta)^{15}\right\}^{\left[\frac{C_\xi T_0}{\delta^3}\right]+1} \le \left\{\frac{30 C_\xi T_0}{\delta^3} (C\delta)^{15}\right\}^{\left[\frac{C_\xi T_0}{\delta^3}\right]+1}\le \left\{\tilde{C}_\xi T_0 \delta^{12}\right\}^{\left[\frac{C_\xi T_0}{\delta^3}\right]+1}.
	\end{align*}
	Choosing $\delta \le \min\left\{\left(\frac{1}{2T_0\tilde{C}_\xi}\right)^{\frac{1}{12}},\frac{1}{T_0\left(1+\|\nabla_x\Phi\|_{L^\infty_x}\right)}\right\}$, we obtain
	\begin{align*}
		\left[\frac{\tilde{C}_\xi T_0}{\delta^3}\right]+1 \sim C_\xi T_0^{\frac{5}{4}} \quad \text{and} \quad k \sim CT_0^{\frac{5}{4}}
	\end{align*}
	for sufficiently large $T_0$.
	Therefore, we conclude that
	\begin{align*}
		\int_{\prod_{l=1}^{k-1}\mathcal{V}_l} \mathbf{1}_{\{t_k(t,x,v,v_1,v_2,...,v_{k-1})>0\}} \prod_{l=1}^{k-1}d\sigma_l \le \left\{\frac{1}{2}\right\}^{C_2T_0^{\frac{5}{4}}}.
	\end{align*}
\end{proof}

\subsection{Exponential decay for Damped Transport equation}

In this subsecton, we consider the following equation:
\begin{align} \label{eeBE} 
	\partial_t f +v\cdot \nabla_x f -\nabla_x \Phi(x) \cdot\nabla_v f +e^{-\Phi(x)}\nu(v)f = 0. 
\end{align}
with the boundary condition
\begin{align} \label{eeBEBC}
	f(t,x,v)|_{\gamma_{-}} = c_\mu \mu^{\frac{1}{2}}(v)\int_{v'\cdot n(x)>0}f(t,x,v')\mu^{\frac{1}{2}}(v')\{n(x)\cdot v'\}dv',
\end{align}
where
\begin{align*}
	c_\mu\int_{v'\cdot n(x)>0}\mu(v')\{n(x) \cdot v'\}dv'=1.
\end{align*}
Set $h(t,x,v) = w(x,v)f(t,x,v)$. Then we derive the following equation:
\begin{align}\label{eBE}
	\partial_t h +v\cdot \nabla_x h -\nabla_x \Phi(x) \cdot\nabla_v h +e^{-\Phi(x)}\nu(v)h = 0. 
\end{align}
with diffuse reflection boundary condition
\begin{align}\label{eBEBC}
	h(t,x,v)|_{\gamma_{-}} = \frac{1}{\tilde{w}(x,v)} \int_{v'\cdot n(x)>0} h(t,x,v')\tilde{w}(x,v')d\sigma(x),
\end{align}
where
\begin{align} \label{tildeweight}
	\tilde{w}(x,v) = \frac{1}{w(x,v)\mu_E^{1/2}(x,v)}.
\end{align}
We now present some lemmas related to the equation \eqref{eBE}. These lemmas will be frequently used in proving our main steps.
\newline
\indent The following lemma gives the representation of a solution to \eqref{eBE} with the diffuse reflection boundary condition along the back-time cycle. Recall the definition \ref{Backwardexit}, especially \eqref{timecycle}, \eqref{Xtrajcycle}, and \eqref{Vtrajcycle}, as well as the definition of the iterated integral \eqref{iterint}.
\begin{lemma} \label{L41}
	Assume that $h,\frac{q}{\nu} \in L_{x,v}^\infty$ satisfy $\{\partial_t+v\cdot \nabla_x -\nabla_x \Phi(x) \cdot \nabla_v+ e^{-\Phi(x)}\nu\}h=q$, with the diffuse reflection boundary condition \eqref{eBEBC}. Then for any $0 \le s \le t$, for almost every $x,v$, if $t_1(t,x,v) \le s$,
	\begin{align*}
		h(t,x,v) & = \exp{-\int_s^t e^{-\Phi(X(\tau))}\nu(V(\tau))d\tau}h(s,X(s),V(s)) \\
		& \quad +\int_s^t \exp{-\int_\tau^t e^{-\Phi(X(\tau'))}\nu(V(\tau')) d\tau'} q(\tau,X(\tau),V(\tau))d\tau.
	\end{align*}
	If $t_1(t,x,v)>s$, then for $k \ge 2$,
	\begin{align*}
		h(t,x,v) &= \int_{t_1}^t \exp{-\int_\tau^t e^{-\Phi(X(\tau'))}\nu(V(\tau'))d\tau'}q(\tau,X(\tau),V(\tau))d\tau \\
		&\quad + \frac{\exp{-\int_{t_1}^t e^{-\Phi(X(\tau))}\nu(V(\tau))d\tau}}{\tilde{w}(x_1,V(t_1))}\sum_{l=1}^{k-1}\int_{\prod_{j=1}^{k-1}\mathcal{V}_j}\mathbf{1}_{\{t_{l+1}\le s < t_l\}}h(s,X_l(s),V_l(s))d\Sigma_l(s)\\
		&\quad +\frac{\exp{-\int_{t_1}^t e^{-\Phi(X(\tau))}\nu(V(\tau))d\tau}}{\tilde{w}(x_1,V(t_1))}\sum_{l=1}^{k-1} \int_s^{t_l}\int_{\prod_{j=1}^{k-1}\mathcal{V}_j}\mathbf{1}_{\{t_{l+1}\le s < t_l\}} q(\tau,X_l(\tau),V_l(\tau))d\Sigma_l(\tau)d\tau\\
		&\quad +\frac{\exp{-\int_{t_1}^t e^{-\Phi(X(\tau))}\nu(V(\tau))d\tau}}{\tilde{w}(x_1,V(t_1))}\sum_{l=1}^{k-1}\int_{t_{l+1}}^{t_l}\int_{\prod_{j=1}^{k-1}\mathcal{V}_j}\mathbf{1}_{\{t_{l+1}>s\}}q(\tau,X_l(\tau),V_l(\tau))d\Sigma_l(\tau)d\tau\\
		&\quad +\frac{\exp{-\int_{t_1}^t e^{-\Phi(X(\tau))}\nu(V(\tau))d\tau}}{\tilde{w}(x_1,V(t_1))}\int_{\prod_{j=1}^{k-1}\mathcal{V}_j}\mathbf{1}_{\{t_{k}>s\}} h(t_k,x_k, V_{k-1}(t_k))d\Sigma_{k-1}(t_k),
	\end{align*}
	where
	\begin{align*}
		d\Sigma_l(s) &= \left\{\prod_{j=l+1}^{k-1}d\sigma_j\right\}\left\{ \exp{-\int_s^{t_l} e^{-\Phi(X_l(\tau))}\nu(V_l(\tau))d\tau} \tilde{w}(x_l,v_l)d\sigma_l\right\}\\
		&\quad \cross \prod_{j=1}^{l-1}\left\{\exp{-\int_{t_{j+1}}^{t_j} e^{-\Phi(X_j(\tau))}\nu(V_j(\tau))d\tau}d\sigma_j\right\}.
	\end{align*}
	and the weight function $\tilde{w}$ is defined in \eqref{tildeweight}.
\end{lemma}
\bigskip

\begin{lemma} \label{L42}
	Let $\mathcal{M}$ be an operator on $L^\infty(\gamma^+)\rightarrow L^\infty(\gamma^-)$ such that $\|\mathcal{M}\|_{\mathcal{L}(L^\infty,L^\infty)}=1$. Then for any $\epsilon >0$, there exists $h(t) \in L_{x,v}^\infty$ and $h|_\gamma \in L^\infty_t L^\infty(\gamma)$ solving
	\begin{align*}
		\{\partial_t + v\cdot \nabla_x -\nabla_x \Phi(x) \cdot \nabla_v+  e^{-\Phi(x)}\nu\}h = 0, \quad h|_{\gamma_-}=(1-\epsilon)\mathcal{M}h|_{\gamma_+}, \quad h(0,x,v) = h_0 \in L^\infty_{x,v}.
	\end{align*}
\end{lemma}
\begin{proof}
	Fix $\epsilon >0$. Set $h^{(0)}|_{\gamma_+} \equiv 0 $. We use the following iterative scheme to construct a solution :
	\begin{align*}
		\{\partial_t + v\cdot \nabla_x -\nabla_x \Phi(x) \cdot \nabla_v+  e^{-\Phi(x)}\nu\}h^{(k+1)}=0, \quad h^{(k+1)}|_{\gamma_-}=(1-\epsilon)\mathcal{M}h^{(k)}|_{\gamma_+}, \quad h^{(k+1)}(0,x,v) =h_0.
	\end{align*}
	Now we claim $h^{(k)}$ and $h^{(k)}|_\gamma$ are Cauchy sequences. Taking differences, we get
	\begin{align*}
		\{\partial_t + v\cdot \nabla_x -\nabla_x \Phi(x) \cdot \nabla_v+  e^{-\Phi(x)}\nu\}&\left(h^{(k+1)}-h^{(k)}\right)=0,\quad h^{(k+1)}|_{\gamma_-}-h^{(k)}|_{\gamma_-}=(1-\epsilon)\mathcal{M}\left(h^{(k)}|_{\gamma_+}-h^{(k-1)}|_{\gamma_+}\right),\\
		&\left(h^{(k+1)}-h^{(k)}\right)_{t=0}=0.
	\end{align*}
	Note that
	\begin{align*}
		\sup_s\left\|h^{(k+1)}|_{\gamma_+}(s) - h^{(k)}|_{\gamma_+}(s)\right\|_{L^\infty_{x,v}} \le (1-\epsilon)\sup_s \left\|h^{(k)}|_{\gamma_+}(s)-h^{(k-1)}|_{\gamma_+}(s)\right\|_{L^\infty_{x,v}}.
	\end{align*}
	Repeatedly using such inequality for $k=1,2,...$, we obtain
	\begin{align*}
		\sup_s\left\|h^{(k+1)}|_{\gamma_+}(s) - h^{(k)}|_{\gamma_+}(s)\right\|_{L^\infty_{x,v}} \le (1-\epsilon)^k\sup_s \left\|h^{(1)}|_{\gamma_+}(s)-h^{(0)}|_{\gamma_+}(s)\right\|_{L^\infty_{x,v}}.
	\end{align*}
	Thus $\left\{h^{(k)}|_{\gamma_+}\right\}$ is Cauchy in $L^\infty_{t}L^\infty(\gamma_-)$, and then both $\left\{h^{(k)}|_{\gamma_-}\right\}$ and $\left\{h^{(k)}\right\}$ are Cauchy, respectively. Hence we conclude our aim.
\end{proof}

\bigskip
We denote by $S_{G_\nu}(t)h_0$ the semigroup of a solution to the equation $\eqref{eBE}$ with the initial datum $h_0$ and the diffuse reflection boundary condition \eqref{eBEBC}. We now introduce two useful exponential time-decay to $S_{G_\nu}h_0$ in $L^\infty$.
\begin{lemma} \label{L43}
	Let $h_0 \in L_{x,v}^\infty$. There exists a unique solution $h(t) = S_{G_{\nu}}(t)h_0  \in L_{x,v}^\infty$ to  \eqref{eBE} with the initial datum $h_0$ and the diffuse reflection boundary condition \eqref{eBEBC}. Moreover, for all $0< \tilde{\nu}_0 < \nu_0$, there exists $C_\Phi>0$, depending on $\Phi$ and $\beta$, such that 
	\begin{align*}
		\sup_{t\ge 0}\left\{\exp\left\{e^{-\|\Phi\|_\infty}\tilde{\nu}_0 t\right\}\|S_{G_{\nu}}(t)h_0\|_{L_{x,v}^\infty}\right\} \le C_\Phi\|h_0\|_{L_{x,v}^\infty},
	\end{align*}
	where $\nu_0:=\inf_{v\in \mathbb{R}^3} \nu(v)$.
\end{lemma}
\begin{proof}
	We first show the uniqueness of solution. Assume that there exists two solutions $h,\tilde{h}$ in $L^\infty_{x,v}$.\\
	Since $\|f\|_{L^1_{x,v}}\le \|wf\|_{L^\infty_{x,v}}\int_{\mathbb{R}^3} \frac{1}{w(x,v)} dv < \infty$ and $\|f\|_{L^1_{\gamma}}\le \|wf\|_{L^\infty_{x,v}}\int_{\mathbb{R}^3} \frac{|v|}{w(x,v)} dv < \infty$,\\ $f$, $\tilde{f}$ are in $L^1_{x,v}$ and $f|_{\gamma},\tilde{f}|_{\gamma}$ are in $L^1_{\gamma}$.
	So, by the divergence theorem,
	\begin{align*}
		\frac{d}{dt}\left\|\left(f-\tilde{f}\right)(t)\right\|_{L^1_{x,v}} &= \int_{\Omega \cross \mathbb{R}^3} \text{sgn}(f-\tilde{f}) [-v\cdot\nabla_x+\nabla_x\Phi(x)\cdot\nabla_v-e^{-\Phi(x)}\nu(v)]\left(f-\tilde{f}\right)(t,x,v)dxdv\\
		&=-\int_{\Omega \cross \mathbb{R}^3} v\cdot\nabla_x\left(|f-\tilde{f}|\right)dxdv +\int_{\Omega \cross \mathbb{R}^3} \nabla_x\Phi\cdot\nabla_v\left(|f-\tilde{f}|\right)dxdv\\
		& \quad- \left\|e^{-\Phi(x)}\nu(v)\left(f-\tilde{f}\right)\right\|_{L^1_{x,v}}\\
		&\le -\int_\gamma \left|\left(f-\tilde{f}\right)(t,x,v)\right| \{n(x)\cdot v\}dS(x)dv - e^{-\|\Phi\|_\infty}\nu_0\|f-\tilde{f}\|_{L^1_{x,v}}.
	\end{align*}
	Also, we have
	\begin{align*}
		\int_{\gamma_-}\left|\left(f-\tilde{f}\right)(t,x,v)\right||n(x)\cdot v|dS(x)dv \le \int_{\gamma_+}\left|\left(f-\tilde{f}\right)(t,x,v)\right||n(x)\cdot v|dS(x)dv,
	\end{align*}
	so that $\int_\gamma \left|\left(f-\tilde{f}\right)(t,x,v)\right| \{n(x)\cdot v\}dS(x)dv$ is positive.\\
	By the Gr\"onwall's inequality, we obtain
	\begin{align*}
		\left\|\left(f-\tilde{f}\right)(t)\right\|_{L^1_{x,v}} \le \|(f-\tilde{f})(0)\|_{L^1_{x,v}}\exp\left\{-\nu_0e^{-\|\Phi\|_\infty}t\right\}=0,
	\end{align*}
	and we conclude that a solution is unique.\\
	Let $h(t,x,v) =w(x,v)f(t,x,v)$.  
	Given any $m \ge 1$, we construct a solution to
	\begin{align} \label{qq1}
		\{\partial_t + v\cdot \nabla_x -\nabla_x \Phi(x) \cdot \nabla_v+  e^{-\Phi(x)}\nu\}h^{(m)}=0,
	\end{align}
	with the boundary and initial condition
	\begin{equation} \label{qq2}
	\begin{aligned}
		&h^{(m)}(t,x,v) = \left\{1-\frac{1}{m}\right\}\frac{1}{\tilde{w}(x,v)}\int_{n(x)\cdot v'>0} \left[h^{(m)}(t,x,v')\right]\tilde{w}(x,v')d\sigma(x),\\
		&h^{(m)}(0,x,v) = h_0 \mathbf{1}_{\{|v|\le m\}}.
	\end{aligned}
	\end{equation}
	Setting $\tilde{h}^{(m)}(t,x,v) = \tilde{w}(x,v) h^{(m)}(t,x,v)$, the equation \eqref{qq1} and the condition \eqref{qq2} become
	\begin{align*}
		&\{\partial_t + v\cdot \nabla_x -\nabla_x \Phi(x) \cdot \nabla_v+  e^{-\Phi(x)}\nu\}\tilde{h}^{(m)}=0,\\
		&\tilde{h}^{(m)}(t,x,v) = \left\{1-\frac{1}{m}\right\}\int_{n(x)\cdot v'>0} \tilde{h}^{(m)}(t,x,v') d\sigma(x),\\
		& \tilde{h}^{(m)}(0,x,v) = \tilde{h}_0 \mathbf{1}_{\{|v|\le m\}}
	\end{align*}
	Since $\int_{n(x) \cdot v'>0} d\sigma(x) =1$, the boundary operator maps $L^\infty_{x,v}$ to $L^\infty_{x,v}$ with a norm bounded by $1-\frac{1}{m}$, and
	\begin{align*}
		\|\tilde{h}^{(m)}(0)\|_{L^\infty_{x,v}} \le C_{m}\|h_0\|_{L^\infty_{x,v}} < \infty.
	\end{align*}
	By Lemma \ref{L42}, there exists a solution $\tilde{h}^{(m)}(t,x,v) \in L^\infty_{x,v}$ to the above equation, and $h^{(m)}$ is bounded because $h^{(m)}=\frac{1}{\tilde{w}(x,v)} \tilde{h}^{(m)}$.\\
	Now, we show the uniform $L^\infty_{x,v}$ bound for $h$. We consider the case $0\le t \le T_0$.\\
	If $t_1(t,x,v) \le 0$, we know  
	\begin{align*}
		\left(S_{G_\nu}(t)h_0\right)(x,v) = \exp\left\{-\int_0^te^{-\Phi(X(s))}\nu(V(s)) ds\right\}h_0(X(0),V(0)),
	\end{align*}
	and we deduce
	\begin{align} \label{qq3}
		\|S_{G_\nu}(t)h_0\|_{L^\infty_{x,v}} \le \exp{e^{-\|\Phi\|_\infty}\nu_0t}\|h_0\|_{L^\infty_{x,v}} \quad \text{for all} \quad 0\le t \le T_0.
	\end{align}
	We consider the case $t_1(t,x,v) > 0$. Recall the definition of the iterated integral in \eqref{iterint}. By Lemma \ref{L41}, we get
	\begin{align*}
		\left|h^{(m)}(t,x,v)\right| &\le \frac{\exp{-\int_{t_1}^t e^{-\Phi(X(\tau))}\nu(V(\tau))d\tau}}{\tilde{w}(x_1,V(t_1))}\sum_{l=1}^{k-1}\int_{\prod_{j=1}^{k-1}\mathcal{V}_j} \mathbf{1}_{\{t_{l+1}\le 0 < t_l\}} \left|h^{(m)}(0,X_l(0),V_l(0))\right|d\Sigma_l(0)\\
		&\quad +\frac{\exp{-\int_{t_1}^t e^{-\Phi(X(\tau))}\nu(V(\tau))d\tau}}{\tilde{w}(x_1,V(t_1))}\int_{\prod_{j=1}^{k-1}\mathcal{V}_j} \mathbf{1}_{\{t_k>0\}} \left|h^{(m)}(t_k,x_k,V_{k-1}(t_k))\right|d\Sigma_{k-1}(t_k)\\
		&=: I_1+I_2
	\end{align*}
	First of all, let us consider $I_2$. Using the boundary condition
	\begin{align*}
		h^{(m)}(t_k,x_k,V_{k-1}(t_k)) = \left\{1-\frac{1}{m}\right\}\frac{1}{\tilde{w}(x_k,V_{k-1}(t_k))}\int_{\mathcal{V}_k} h^{(m)}(t_k,x_k,v_k)\tilde{w}(x_k,v_k)d\sigma_k
	\end{align*}
and the fact $h^{(m)}(t_k,x_k,v_k) = \mathbf{1}_{\{t_{k+1} \le 0 <t_k\}}\exp\left\{-\int_0^{t_k}e^{-\Phi(X_k(s))}\nu(V_k(s))ds\right\}h^{(m)}(0,X_k(0),V_k(0))\\ +\mathbf{1}_{\{t_{k+1}>0\}}h^{(m)}(t_k,x_k,v_k)$, where $X_k(s)=x_k-\int_{s}^{t_k} V_k(\tau)d\tau$ and $V_k(s) = v_k-\int_s^{t_k}\nabla_x \Phi(X_k(\tau))d\tau$,
\begin{align*}
	I_2 &\le \frac{\exp{-\int_{t_1}^t e^{-\Phi(X(\tau))}\nu(V(\tau))d\tau}}{\tilde{w}(x_1,V(t_1))}\int_{\prod_{j=1}^{k}\mathcal{V}_j} \mathbf{1}_{\{t_{k+1} \le 0 <t_k\}}\left|h^{(m)}(0,X_k(0),V_k(0))\right|d\Sigma_{k}(0)\\
	&\quad +\frac{\exp{-\int_{t_1}^t e^{-\Phi(X(\tau))}\nu(V(\tau))d\tau}}{\tilde{w}(x_1,V(t_1))}\int_{\prod_{j=1}^{k}\mathcal{V}_j} \mathbf{1}_{\{t_{k+1}>0\}} \left|h^{(m)}(t_k,x_k,v_k)\right|d\Sigma_{k}(t_k)\\
	&=:J_1+J_2.
\end{align*}
Since $t_1(t_k,x_k,v_k)>0$ over $\{t_{k+1}>0\}$, we deduce
\begin{align} \label{qq5}
	\mathbf{1}_{\{t_{k+1}>0\}}\left|h^{(m)}(t_k,x_k,v_k)\right| \le \sup_{x,v}\left|h^{(m)}(t_k,x,v)\mathbf{1}_{\{t_1>0\}}\right|.
\end{align}
We know that the exponential in $d\Sigma_l(s)$ is bounded by $\exp\{-e^{-\|\Phi\|_\infty} \nu_0(t_1-s) \}$. By Lemma \ref{Lsmall}, we can choose $C_1,C_2>0$ such that for $k=C_1T_0^{\frac{5}{4}}$
\begin{align} \label{qq6}
	\int_{\prod_{j=1}^{k-1}\mathcal{V}_j} \mathbf{1}_{\{t_k(t,x,v,v_1,v_2,...,v_{k-1})>0\}}\prod_{j=1}^{k-1}d\sigma_j \le \left\{\frac{1}{2}\right\}^{C_2T_0^{\frac{5}{4}}}
\end{align}
Using \eqref{qq5} and \eqref{qq6}, we obtain
\begin{align*}
	J_2 &\le \frac{\exp\left\{-e^{-\|\Phi\|_\infty} \nu_0 (t-t_1)\right\}}{\tilde{w}(x_1,V(t_1))}\left\|h^{(m)}(t_k)\mathbf{1}_{\{t_1>0\}}\right\|_{L^\infty_{x,v}} \int_{\prod_{j=1}^k\mathcal{V}_j} \mathbf{1}_{\{t_k>0\}}\tilde{w}(x_k,v_k)\exp\{-e^{-\|\Phi\|_\infty} \nu_0(t_1-t_k) \}\prod_{j=1}^kd\sigma_j\\
	&\le C_\Phi \sup_{0\le s \le t \le T_0} \left\{\exp\left\{-e^{-\|\Phi\|_\infty} \nu_0(t-s) \right\}\left\|h^{(m)}(s) \mathbf{1}_{\{t_1>0\}}\right\|_{L^\infty_{x,v}}\right\}\left(\int_{\prod_{j=1}^{k-1}\mathcal{V}_j}\mathbf{1}_{\{t_{k}>0\}}\prod_{j=1}^{k-1}d\sigma_j\right)\\
	& \quad  \times \left(\int_{\mathcal{V}_k}\tilde{w}(x_k,v_k) d\sigma_k\right)\\
	&\le C_\Phi \left(\frac{1}{2}\right)^{C_2T_0^{\frac{5}{4}}}\sup_{0\le s \le t \le T_0} \left\{\exp\left\{-e^{-\|\Phi\|_\infty} \nu_0(t-s) \right\}\left\|h^{(m)}(s) \mathbf{1}_{\{t_1>0\}}\right\|_{L^\infty_{x,v}}\right\},
\end{align*}
where $\int_{\mathcal{V}_k}\tilde{w}(x_k,v_k) d\sigma_k$ is finite.\\
Let us consider $I_1$ and $J_1$. By inserting $\int_{\mathcal{V}_k}d\sigma_k=1$ into $I_1$, we get
\begin{align*}
	I_1+J_1 &= \frac{\exp{-\int_{t_1}^t e^{-\Phi(X(\tau))}\nu(V(\tau))d\tau}}{\tilde{w}(x_1,V(t_1))}\sum_{l=1}^{k}\int_{\prod_{j=1}^{k}\mathcal{V}_j} \mathbf{1}_{\{t_{l+1}\le 0 < t_l\}} \left|h^{(m)}(0,X_l(0),V_l(0))\right|d\Sigma_l(0)\\
	&\le \frac{\exp\left\{-e^{-\|\Phi\|_\infty} \nu_0 t\right\}}{\tilde{w}(x_1,V(t_1))}\left\|h^{(m)}(0)\right\|_{L^\infty_{x,v}}\sum_{l=1}^{k}\int_{\prod_{j=1}^{k}\mathcal{V}_j} \mathbf{1}_{\{t_{l+1}\le 0 < t_l\}} \left\{\prod_{j=l+1}^k d\sigma_j\right\}\left\{\tilde{w}(x_l,v_l) d\sigma_l\right\}\left\{\prod_{j=1}^{l-1} d\sigma_j\right\}.
\end{align*}
Now, we fix $l$ and consider the $l$-th term
\begin{align*}
	&\int_{\prod_{j=1}^{k}\mathcal{V}_j} \mathbf{1}_{\{t_{l+1}\le 0 < t_l\}} \left\{\prod_{j=l+1}^k d\sigma_j\right\}\left\{\tilde{w}(x_l,v_l) d\sigma_l\right\}\left\{\prod_{j=1}^{l-1} d\sigma_j\right\}\\
	&\le \int_{\prod_{j=1}^{l-1}\mathcal{V}_j} \left(\int_{\mathcal{V}_l}\tilde{w}(x_l,v_l) d\sigma_l\right)\left\{\prod_{j=1}^{l-1} d\sigma_j\right\}\\
	&\le C_\Phi,
\end{align*}
where $\int_{\mathcal{V}_l}\tilde{w}(x_l,v_l) d\sigma_l$ is finite.
Summing $1\le l \le k$, it follows that
\begin{align*}
	I_1+J_1 \le C_1T_0^{\frac{5}{4}}C_\Phi \frac{\exp\left\{-e^{-\|\Phi\|_\infty} \nu_0 t\right\}}{\tilde{w}(x_1,V(t_1))}\left\|h^{(m)}(0)\right\|_{L^\infty_{x,v}} \le C_\Phi T_0^{\frac{5}{4}} \exp\left\{-e^{-\|\Phi\|_\infty} \nu_0 t \right\} \left\|h^{(m)}(0)\right\|_{L^\infty_{x,v}}.
\end{align*}
Gathering $I_1$, $J_1$, and $J_2$, we deduce that for $0 \le t \le T_0$,
\begin{align*}
	\exp\left\{e^{-\|\Phi\|_\infty} \nu_0 t \right\}\left|h^{(m)}(t,x,v)\mathbf{1}_{\{t_1>0\}}\right| &\le C_\Phi \left(\frac{1}{2}\right)^{C_2T_0^{\frac{5}{4}}}\sup_{0\le s \le T_0} \left\{\exp\left\{-e^{-\|\Phi\|_\infty} \nu_0 s \right\}\left\|h^{(m)}(s) \mathbf{1}_{\{t_1>0\}}\right\|_{L^\infty_{x,v}}\right\}\\
	& \quad +C_\Phi T_0^{\frac{5}{4}}\left\|h^{(m)}(0)\right\|_{L^\infty_{x,v}}.
\end{align*}
Choosing sufficiently large $T_0 > 0$ such that $C_\Phi \left(\frac{1}{2}\right)^{C_2T_0^{\frac{5}{4}}} \le \frac{1}{2}$,
\begin{align} \label{qq4}
	\sup_{0\le t \le T_0}\left\{\exp\left\{e^{-\|\Phi\|_\infty} \nu_0 t \right\}\left\|h^{(m)}(t)\mathbf{1}_{\{t_1>0\}}\right\|_{L^\infty_{x,v}}\right\} \le C_\Phi T_0^{\frac{5}{4}}\left\|h^{(m)}(0)\right\|_{L^\infty_{x,v}} = C_\Phi T_0^{\frac{5}{4}}\left\|h_0\right\|_{L^\infty_{x,v}}.
\end{align}
From now on, we extend the exponential decay to all time $t>0$. Letting $t=T_0$ in \eqref{qq4} and choosing sufficiently large $T_0>0$, for all $\tilde{\nu}_0 < \nu_0$,
\begin{align} \label{qq7}
	\left\|h^{(m)}\left(T_0\right)\right\|_{L^\infty_{x,v}}\le C_\Phi T_0^{\frac{5}{4}}\exp\left\{-e^{-\|\Phi\|_\infty} \nu_0 T_0 \right\}\left\|h_0\right\|_{L^\infty_{x,v}} \le \exp\left\{-e^{-\|\Phi\|_\infty} \tilde{\nu}_0 T_0 \right\}\left\|h_0\right\|_{L^\infty_{x,v}},
\end{align}
and applying repeatedly the process \eqref{qq7}, we can derive for $l\ge 1$
\begin{align*}
	\left\|h^{(m)}(lT_0)\right\|_{L^\infty_{x,v}} \le\exp\left\{-e^{-\|\Phi\|_\infty} \tilde{\nu}_0 T_0 \right\}\left\|h^{(m)}\left((l-1)T_0\right)\right\|_{L^\infty_{x,v}} \le \exp\left\{-le^{-\|\Phi\|_\infty} \tilde{\nu}_0 T_0 \right\}\left\|h_0\right\|_{L^\infty_{x,v}}.
\end{align*}
Thus, for $lT_0 \le t \le (l+1)T_0$ with $l\ge 1$, we deduce that
\begin{align*}
	\left\|h^{(m)}\left(t\right)\right\|_{L^\infty_{x,v}} 
	&\le C_{T_0}\exp\left\{-le^{-\|\Phi\|_\infty} \tilde{\nu}_0 T_0 \right\}\left\|h_0\right\|_{L^\infty_{x,v}}\\ 
	&\le C_{T_0}\exp\left\{-e^{-\|\Phi\|_\infty} \tilde{\nu}_0 t\right\}\exp\left\{e^{-\|\Phi\|_\infty} \tilde{\nu}_0 T_0 \right\}\left\|h_0\right\|_{L^\infty_{x,v}}\\ 
	&\le C_{T_0} \exp\left\{-e^{-\|\Phi\|_\infty} \tilde{\nu}_0 t\right\}\left\|h_0\right\|_{L^\infty_{x,v}}.
\end{align*}
since $0 \le t-lT_0 \le T_0$. Hence $\left(h^{(m)}\right)$ is uniformly bounded, and the sequence has weak* limit in $L^\infty_{x,v}$. Letting $m \rightarrow \infty$, we conclude the existence of a solution and the exponential decay for the solution.
\end{proof}

\bigskip

\begin{corollary} \label{L43C}
	Assume that $h(t) = S_{G_{\nu}}(t)h_0  \in L_{x,v}^\infty$ is a solution to  \eqref{eBE} with the diffuse reflection boundary condition \eqref{eBEBC}. For all $0< \tilde{\nu}_0 < \nu_0$, there exists $C_\Phi>0$, depending on $\Phi$ and $\beta$, such that 
	\begin{align*}
		\left\|\int_0^t S_{G_\nu}(t-s) h(s) ds\right\|_{L^\infty_{x,v}} \le C_\Phi \sup_{0 \le s \le t} \left\{\exp\left\{-e^{-\|\Phi\|_\infty}\tilde{\nu}_0(t-s)\right\}\|h(s)\|_{L^\infty_{x,v}(\nu^{-1})}\right\}
	\end{align*}
	for all $t\ge 0$.
\end{corollary}
\begin{proof}
	Let  $0<\tilde{\nu}_0 < \nu_0$ and $0\le  t \le T_0$ . By Lemma \ref{L41}, we have
	\begin{align*}
		\int_0^t S_{G_\nu}(t-s)h(s)ds  &= \int_0^t\mathbf{1}_{\{t_1 \le s\}} \exp{-\int_s^t e^{-\Phi(X(\tau))}\nu(V(\tau))d\tau} h(s,X(s),V(s))ds\\
		&\quad  + \frac{\exp{-\int_{t_1}^t e^{-\Phi(X(\tau))}\nu(V(\tau))d\tau}}{\tilde{w}(x_1,V(t_1))}\sum_{l=1}^{k-1}\int_{t_{l+1}}^{t_l}\int_{\prod_{j=1}^{k-1}\mathcal{V}_j}  \left|h(s,X_l(s),V_l(s))\right|d\Sigma_l(s)ds\\
		&\quad +\frac{\exp{-\int_{t_1}^t e^{-\Phi(X(\tau))}\nu(V(\tau))d\tau}}{\tilde{w}(x_1,V(t_1))}\int_0^t\int_{\prod_{j=1}^{k-1}\mathcal{V}_j} \mathbf{1}_{\{t_k>s\}} \left|S_{G_\nu}(t_k-s)h(s)\right|d\Sigma_{k-1}(t_k)ds\\
		& =: I_1+I_2+I_3.
	\end{align*}
	Firstly, we compute
	\begin{align*}
		I_1 &\le \int_0^t\mathbf{1}_{\{t_1 \le s\}} \exp{-\int_s^t e^{-\Phi(X(\tau))}\nu(V(\tau))d\tau} \nu(V(s))\|h(s)\|_{L^\infty_{x,v}(\nu^{-1})}ds\\
		& \le C_\Phi \sup_{0 \le s \le t} \left\{\exp\left\{-e^{-\|\Phi\|_\infty}\tilde{\nu}_0(t-s)\right\}\|h(s)\|_{L^\infty_{x,v}(\nu^{-1})}\right\}.
	\end{align*}
	Next, let us estimate $I_2$. From \eqref{sese},
	\begin{align*}
		I_2 &\le C_\Phi \sup_{0 \le s \le t} \left\{\exp\left\{-e^{-\|\Phi\|_\infty}\tilde{\nu}_0(t-s)\right\}\|h(s)\|_{L^\infty_{x,v}(\nu^{-1})}\right\}\sum_{l=1}^{k-1} \int_{t_{l+1}}^{t_l} \int_{\prod_{j=1}^{l-1}\mathcal{V}_j} \exp{-e^{-\|\Phi\|_\infty}(\nu_0 -\tilde{\nu}_0)(t_l-s)}\\
		& \quad \times \left(\int_{\mathbb{R}^3} \nu(V_l(s))\tilde{w}(x_l,v_l)\mu(v_l)|v_l|dv_l\right) \prod_{j=1}^{l-1} \left\{ \exp{-e^{-\|\Phi\|_\infty}\nu_0(t_j-t_{j+1})}d\sigma_j\right\}ds\\
		& \le  C_\Phi \sup_{0 \le s \le t} \left\{\exp\left\{-e^{-\|\Phi\|_\infty}\tilde{\nu}_0(t-s)\right\}\|h(s)\|_{L^\infty_{x,v}(\nu^{-1})}\right\} \sum_{l=1}^{k-1} \int_{t_{l+1}}^{t_l} \exp{-e^{-\|\Phi\|_\infty}(\nu_0 -\tilde{\nu}_0)(t_1-s)}\\ 
		& \quad \times \int_{\prod_{j=1}^{l-1}\mathcal{V}_j}\left(\int_{\mathbb{R}^3} \left(1+|v_l|+\sqrt{2\|\Phi\|_\infty}\right)^\gamma\tilde{w}(x_l,v_l)\mu(v_l)|v_l|dv_l\right) \left\{\prod_{j=1}^{l-1}  d\sigma_j\right\}ds\\
		& \le C_\Phi \sup_{0 \le s \le t} \left\{\exp\left\{-e^{-\|\Phi\|_\infty}\tilde{\nu}_0(t-s)\right\}\|h(s)\|_{L^\infty_{x,v}(\nu^{-1})}\right\}\int_0^{t_1}\exp{-e^{-\|\Phi\|_\infty}(\nu_0 -\tilde{\nu}_0)(t_1-s)}ds\\
		& \le C_\Phi \sup_{0 \le s \le t} \left\{\exp\left\{-e^{-\|\Phi\|_\infty}\tilde{\nu}_0(t-s)\right\}\|h(s)\|_{L^\infty_{x,v}(\nu^{-1})}\right\},
	\end{align*}
	where $\int_0^{t_1}\exp{-e^{-\|\Phi\|_\infty}(\nu_0 -\tilde{\nu}_0)(t_1-s)}ds$ is finite.\\
	For $I_3$, from Lemma \ref{Lsmall}, we derive
	\begin{align*}
		I_3 \le  C_\Phi \left(\frac{1}{2}\right)^{C_2T_0^{\frac{5}{4}}}\sup_{t \ge 0}\left\|\int_0^t S_{G_\nu}(t-s) h(s) ds\right\|_{L^\infty_{x,v}} \le \frac{1}{2} \sup_{t \ge 0}\left\|\int_0^t S_{G_\nu}(t-s) h(s) ds\right\|_{L^\infty_{x,v}}.
	\end{align*}
	for large enough $T_0$.\\
	Combining $I_1$, $I_2$, and $I_3$, we conclude that
	\begin{align*}
		\left\|\int_0^t S_{G_\nu}(t-s) h(s) ds\right\|_{L^\infty_{x,v}} \le C_\Phi \sup_{0 \le s \le t} \left\{\exp\left\{-e^{-\|\Phi\|_\infty}\tilde{\nu}_0(t-s)\right\}\|h(s)\|_{L^\infty_{x,v}(\nu^{-1})}\right\}
	\end{align*}
	for all $t\ge 0$ because $T_0$ is arbitrarily large and $C_\Phi$ is independent of $T_0$.
\end{proof}

\bigskip
\subsection{Exponential decay for the Linearized Boltzmann equation}

From the equation \eqref{LBE}, we derive the following equation in terms of $h(t,x,v) = w(x,v)f(t,x,v)$ : 
\begin{align}\label{WLBE}
	\partial_t h +v\cdot \nabla_x h -\nabla_x \Phi(x) \cdot\nabla_v h +e^{-\Phi(x)}\nu(v)h -e^{-\Phi(x)}K_wh = 0. 
\end{align}
with diffuse reflection boundary condition
\begin{align}\label{wLBEBC}
	h(t,x,v)|_{\gamma_{-}} =  \frac{1}{\tilde{w}(x,v)} \int_{v'\cdot n(x)>0} h(t,x,v')\tilde{w}(x,v')d\sigma,
\end{align}
where $\tilde{w}(x,v)$ is defined in \eqref{tildeweight}.\\
\indent Our purpose in this subsection is to prove the linear (weighted) $L^\infty$ decay to the linearized Boltzmann equation. The following lemma gives a crucial estimate for the operator $K$ with weight $w$.
\begin{lemma} \label{Kest}\cite{ChanBELP}
	There exists $k(v,u) \ge 0$ such that for all $v$ in $\mathbb{R}^3$,
	\begin{align*}
		K(f)(v) = \int_{\mathbb{R}^3}k(v,u)f(u)du.
	\end{align*}
	Moreover, for $\alpha>0$, there exists $C_\alpha >0$, depending on $\alpha$ and $\beta$, such that
	\begin{align*}
		\int_{\mathbb{R}^3}\left|k(v,u)\right|e^{\frac{1}{16}|v-u|^2+\frac{1}{16}\frac{\left||v|^2-|u|^2\right|^2}{|v-u|^2}}\frac{w(x,v)}{w(x,u)}(1+|u|)^{-\alpha}du \le C_\alpha(1+|v|)^{-1-\alpha}.
	\end{align*}
\end{lemma}
From this lemma, we can express the weighted operator $K_w$ in \eqref{woperK} as
\begin{align} \label{wKexpress}
	K_w h(v) = \int_{\mathbb{R}^3} k_w(v,u)h(u)du.
\end{align}

\bigskip

The next theorem provides that the global-in-time exponential decay of a solution to linearized Boltzmann equation which satisfies an assumption in a finite time. 
\begin{theorem} \label{Fest}
	Assume that there exist $T_0>0$ and $C_{T_0},\lambda>0$ such that for all solution $f(t,x,v)$ in $L^\infty_{x,v}(w)$ to
	\begin{align} \label{lb}
		\partial_tf+v \cdot \nabla_x f-\nabla_x\Phi(x) \cdot \nabla_v f +e^{-\Phi(x)}L(f)=0 
	\end{align}
	with initial datum $f_0$ in $L^\infty_{x,v}(w)$, the following holds
	\begin{align} \label{Fas1}
		\|wf(T_0)\|_{L_{x,v}^\infty} \le e^{-\lambda T_0}\|wf_0\|_{L_{x,v}^\infty} + C_{T_0}\int_0^{T_0}\|f(s)\|_{L_{x,v}^2} ds.
	\end{align}
	Then,  under the a priori assumption \eqref{AAL2},  for all $0<\tilde{\lambda}<\min\{\lambda,\lambda_G\}$, there exist $C>0$, independent of $f_0$, but depending on $T_0$, $\lambda$, $\lambda_G$, and $\tilde{\lambda}$, such that for all solution $f$ to \eqref{lb} in $L^\infty_{x,v}(w)$, 
	\begin{align*}
		\|wf(t)\|_{L_{x,v}^\infty} \le Ce^{-\tilde{\lambda}t}\|wf_0\|_{L_{x,v}^\infty} \quad \text{for all t}\ge 0.
	\end{align*}
\end{theorem}
\begin{proof}
	It suffices to prove only for $t \ge 1$.\\
	For any $m \ge 1$, we apply the assumption \eqref{Fas1} repeatedly to functions $f(lT_0+s)$:
	\begin{align*}
		\|wf\left(mT_0\right)\|_{L^\infty_{x,v}} 
		&\le e^{-\lambda T_0}\|wf(\{m-1\}T_0)\|_{L^\infty_{x,v}} +C_{T_0}\int_0^{T_0}\|f(\{m-1\}T_0+s)\|_{L^2_{x,v}}ds\\
		&= e^{-\lambda T_0}\|wf(\{m-1\}T_0)\|_{L^\infty_{x,v}} + C_{T_0}\int_{\{m-1\}T_0}^{mT_0} \|f(s)\|_{L^2_{x,v}}ds\\
		& \le e^{-2\lambda T_0} \|wf(\{m-2\}T_0)\|_{L^\infty_{x,v}} + e^{-\lambda T_0}C_{T_0}\int_{\{m-2\}T_0}^{\{m-1\}T_0} \|f(s)\|_{L^2_{x,v}}ds\\
		& \quad + C_{T_0}\int_{\{m-1\}T_0}^{mT_0} \|f(s)\|_{L^2_{x,v}}ds\\
		& \le e^{-m\lambda T_0}\|wf_0\|_{L^\infty_{x,v}} + C_{T_0}\sum_{k=0}^{m-1}e^{-k\lambda T_0}\int_{\{m-k-1\}T_0}^{\{m-k\}T_0} \|f(s)\|_{L^2_{x,v}}ds.
	\end{align*}
	By Theorem \ref{T31}, we have in $\{m-k-1\}T_0 \le s \le \{m-k\}T_0$,
	\begin{align*}
		\|f(s)\|_{L^2_{x,v}} \le C_G e^{-\lambda_G s}\|f_0\|_{L^2_{x,v}} \le C_G e^{-\lambda_G \{m-k-1\}T_0}\|f_0\|_{L^2_{x,v}}.
	\end{align*}
	Let $\lambda_{\min}:= \min\{\lambda,\lambda_G\}$ and we take $\tilde{\lambda}$ such that $0<\tilde{\lambda}<\lambda_{\min}$.\\
	Then we can bound
	\begin{align*}
		\|wf\left(mT_0\right)\|_{L^\infty_{x,v}} &\le e^{-m\lambda T_0}\|wf_0\|_{L^\infty_{x,v}}\\
		&\quad + C_{T_0}\sum_{k=0}^{m-1}e^{-k\lambda T_0}\int_{\{m-k-1\}T_0}^{\{m-k\}T_0} e^{-\lambda_G \{m-k-1\}T_0}\|f_0\|_{L^2_{x,v}}ds\\
		&= e^{-m\lambda T_0}\|wf_0\|_{L^\infty_{x,v}} + C_{T_0}\sum_{k=0}^{m-1}T_0 e^{-k\lambda T_0}e^{-\lambda_G \{m-k-1\}T_0}\|f_0\|_{L^2_{x,v}}\\
		& \le e^{-m\tilde{\lambda} T_0}\|wf_0\|_{L^\infty_{x,v}} + C_{T_0} mT_0 e^{-\lambda_{\min}\{m-1\}T_0}\|f_0\|_{L^2_{x,v}}.
	\end{align*}
	We note that $\|f_0\|_{L^2_{x,v}} \le C\|wf_0\|_{L^\infty_{x,v}}$.\\
	We split the exponent as  $e^{-\lambda_{\min}\{m-1\}T_0} =  e^{-\lambda_{\min}mT_0} e^{\lambda_{\min}T_0} = e^{-\tilde{\lambda}mT_0}e^{-(\lambda_{\min}-\tilde{\lambda})mT_0} e^{\lambda_{\min}T_0}$, and we can absorb $mT_0$ by $e^{-(\lambda_{\min}-\tilde{\lambda})mT_0}$, so that we obtain
	\begin{align*}
		\|wf\left(mT_0\right)\|_{L^\infty_{x,v}} \le C_{T_0,\tilde{\lambda},\lambda_{\min}}e^{-\tilde{\lambda}mT_0}\|wf_0\|_{L^\infty_{x,v}}.
	\end{align*}
	For any $t$, we can find $m$ such that $mT_0 \le t \le \{m+1\}T_0$, and we deduce
	\begin{align*}
		\|wf(t)\|_{L_{x,v}^\infty} &\le C_{T_0}\|wf\left(mT_0\right)\|_{L^\infty_{x,v}}\\
		& \le C_{T_0,\tilde{\lambda},\lambda_{\min}}e^{-\tilde{\lambda}mT_0}\|wf_0\|_{L^\infty_{x,v}}\\
		& \le C_{T_0,\tilde{\lambda},\lambda_{\min}}e^{-\tilde{\lambda}t}\|wf_0\|_{L^\infty_{x,v}}
	\end{align*}
	since $e^{-\tilde{\lambda}mT_0} \le e^{-\tilde{\lambda}t}e^{\tilde{\lambda}T_0}$.
\end{proof}

\bigskip
We denote by $S_G(t)h_0$ the semigroup of a solution of the equation \eqref{WLBE} with initial datum $h_0$ and the diffuse reflection boundary condition \eqref{wLBEBC}. The next theorem represents the exponential decay of the linearized Boltzmann equation \eqref{WLBE}. Recall the definition \ref{Backwardexit}, especially \eqref{timecycle}, \eqref{Xtrajcycle}, and \eqref{Vtrajcycle}, as well as the definition of the iterated integral \eqref{iterint}.

\begin{theorem} \label{LT}
	Let $h_0$ in $L^\infty_{x,v}$ and $\beta > 5$. Then there exists a unique global solution to \eqref{WLBE} with initial datum $h_0$ and the diffuse reflection boundary condition \eqref{wLBEBC}. Moreover, there exist $\lambda_\infty>0$ and $C_\infty>0$, depending on $\Phi$ and $\beta$, such that
	\begin{align} \label{expdec}
		\|S_G(t)h_0\|_{L^\infty_{x,v}} \le C_\infty e^{-\lambda_\infty t} \|h_0\|_{L^\infty_{x,v}} \quad \text{for all} \quad t\ge 0.	
	\end{align}
\end{theorem}
\begin{proof}
	We show the existence and uniqueness of solution to \eqref{WLBE} and the exponential decay.\\
	\newline
	$\mathbf{(Existence)}$\\
	\indent We consider an iterative system to \eqref{WLBE} as followings:
	\begin{align*}
		\left(\partial_t +v\cdot \nabla_x-\nabla_x \Phi(x) \cdot \nabla_v+e^{-\Phi(x)}\nu(v)\right)h^{(m+1)} = e^{-\Phi(x)}K_w(h^{(m)}).
	\end{align*}
	By Duhamel's principle, we get
	\begin{align*}
		h^{(m+1)}(t,x,v) = \left(S_{G_\nu}(t)h_0\right)(x,v) + \int_0^t S_{G_\nu}(t-s)\left(e^{-\Phi}K_wh^{(m)}(s)\right)(x,v)ds,
	\end{align*}
	By Lemma \ref{L43} and Lemma \ref{Kest}, we deduce that for $0\le t \le T$,
	\begin{align*}
		\left|\left(h^{(m+1)}-h^{(m)}\right)(t,x,v)\right|&\le \int_0^{T} \left|S_{G_\nu}(t-s)\left(e^{-\Phi}K_w(h^{(m)}-h^{(m-1)})(s)\right)(x,v)\right|ds\\
		&\le T\sup_{0 \le s \le T}\left\|S_{G_\nu}(t-s)\left(e^{-\Phi}K_w(h^{(m)}-h^{(m-1)})(s)\right)\right\|_{L^\infty_{x,v}}\\
		&\le TC_\Phi \sup_{0 \le s \le T}\left\|\left(h^{(m)}-h^{(m-1)}\right)(s)\right\|_{L^\infty_{x,v}}.
	\end{align*}
	This yields
	\begin{align*}
		 \sup_{0 \le s \le T}\left\|\left(h^{(m+1)}-h^{(m)}\right)(s)\right\|_{L^\infty_{x,v}} \le TC_\Phi \sup_{0 \le s \le T}\left\|\left(h^{(m)}-h^{(m-1)}\right)(s)\right\|_{L^\infty_{x,v}}.
	\end{align*}
	Choosing small $T>0$ such that $TC_\Phi<1$, the above inequality is a contraction. Thus we complete the existence for small times. Lately, thanks to the exponential decay, we can conclude the global existence of solution.\\
	\newline
	$\mathbf{(Uniqueness)}$\\
	\indent It is similar to the proof of existence part. We can deduce the uniqueness of solution for small times. Thanks to the exponential decay, we can conclude the global uniqueness of solution.\\
	\newline
	$\mathbf{(Exponential\ decay )}$\\
 	\indent Thanks to Theorem \ref{Fest}, it suffices to prove \eqref{Fas1} for a finite time in order to derive the exponential decay \eqref{expdec}.
	Let $T_0>0$ be fixed and $0\le t \le T_0$.
	By Duhamel's principle, we get
	\begin{align*}
		h(t,x,v) &= \left(S_{G_\nu}(t)h_0\right)(x,v) + \int_0^t S_{G_\nu}(t-s)\left(e^{-\Phi}K_w h(s)\right)(x,v)ds.
	\end{align*}
	We use the Duhamel's principle one more time to obtain
	\begin{equation} \label{dd1}
	\begin{aligned}
		h(t) &= S_{G_\nu}(t)h_0+ \int_0^t S_{G_\nu}(t-s) e^{-\Phi}K_w S_{G_\nu}(s)h_0ds\\ 
		&\quad + \int_0^t \int_0^s S_{G_\nu}(t-s)e^{-\Phi}K_wS_{G_\nu}(s-s')e^{-\Phi}K_wh(s')ds'ds\\
		&=:I_1 + I_2 +I_3.
	\end{aligned}
	\end{equation}
	First of all, by Lemma \ref{L43}, we can easily get
	\begin{align} \label{dd2}
		|I_1| \le \left\|S_{G_\nu}(t)(h_0)\right\|_{L^\infty_{x,v}} \le C \exp\left\{-e^{-\|\Phi\|_\infty} \frac{\nu_0}{2} t \right\}\|h_0\|_{L^\infty_{x,v}}.
	\end{align}
	Next, by Lemma \ref{L43} and Lemma \ref{Kest}, we deduce
	\begin{equation} \label{dd3}
	\begin{aligned} 
		|I_2| 
		&\le C_K \int^t_0\exp\left\{-e^{-\|\Phi\|_\infty} \frac{\nu_0}{2} (t-s) \right\} \|S_{G_\nu}(s)h_0\|_{L^\infty_{x,v}}ds\\
		&\le C_K t \exp\left\{-e^{-\|\Phi\|_\infty} \frac{\nu_0}{2} t\right\}\|h_0\|_{L^\infty_{x,v}}.
	\end{aligned}
	\end{equation}
	Now, we consider the term $I_3$. We can divide $I_3$ into three integrations as follows:
	\begin{align*}
		&\int_0^t \int_0^s S_{G_\nu}(t-s)e^{-\Phi}K_wS_{G_\nu}(s-s')e^{-\Phi}K_wh(s')ds'ds\\
		&=\int_{t-\epsilon}^t\int_0^s S_{G_\nu}(t-s)e^{-\Phi}K_wS_{G_\nu}(s-s')e^{-\Phi}K_wh(s')ds'ds\\
		& \quad + \int_0^{t-\epsilon}\int_{s-\epsilon}^s  S_{G_\nu}(t-s)e^{-\Phi}K_wS_{G_\nu}(s-s')e^{-\Phi}K_wh(s')ds'ds\\
		& \quad + \int_0^{t-\epsilon}\int_0^{s-\epsilon} S_{G_\nu}(t-s)e^{-\Phi}K_wS_{G_\nu}(s-s')e^{-\Phi}K_wh(s')ds'ds\\
		&=: J_1+J_2+J_3.
	\end{align*}
	\indent From now on, we will derive bounds for $J_1$, $J_2$, and $J_3$ in sequence.
	First, we can compute the bound for $J_1$:
	\begin{align*}
		|J_1| 
		& \le C_K \int_{t-\epsilon}^t\int_0^s \exp\left\{-e^{-\|\Phi\|_\infty} \frac{\nu_0}{2} (t-s) \right\}\|S_{G_\nu}(s-s')e^{-\Phi}K_wh(s')\|_{L^\infty_{x,v}}ds'ds\\
		& \le C_K \int_{t-\epsilon}^t\int_0^s \exp\left\{-e^{-\|\Phi\|_\infty} \frac{\nu_0}{2} (t-s') \right\}\|h(s')\|_{L^\infty_{x,v}}ds'ds\\
		& \le C_K \sup_{0 \le s' \le T_0 }\left\{\exp\left\{-e^{-\|\Phi\|_\infty} \frac{\nu_0}{4} (t-s') \right\}\|h(s')\|_{L^\infty_{x,v}}\right\} \int_{t-\epsilon}^{t}\int_0^s \exp\left\{-e^{-\|\Phi\|_\infty} \frac{\nu_0}{4} (t-s') \right\}ds'ds,
	\end{align*}
	where we have used Lemma \ref{L43} and Lemma \ref{Kest}.
	Here, we get
	\begin{align*}
		\int_{t-\epsilon}^{t}\int_0^s \exp\left\{-e^{-\|\Phi\|_\infty} \frac{\nu_0}{4} (t-s') \right\}ds'ds \le C_\Phi \int_{t-\epsilon}^{t} \exp\left\{-e^{-\|\Phi\|_\infty} \frac{\nu_0}{4} (t-s)\right\}ds \le C_\Phi\epsilon.
	\end{align*}
	Thus we deduce 
	\begin{align} \label{dd4}
		|J_1| \le \epsilon C_\Phi \sup_{0 \le s' \le T_0 }\left\{\exp\left\{-e^{-\|\Phi\|_\infty} \frac{\nu_0}{4} (t-s') \right\}\|h(s')\|_{L^\infty_{x,v}}\right\}.
	\end{align}
	Next, we can compute the bound for $J_2$:
	\begin{align*}
		|J_2|
		& \le C_K \int_{0}^t\int_{s-\epsilon}^s \exp\left\{-e^{-\|\Phi\|_\infty} \frac{\nu_0}{2} (t-s) \right\}\|S_{G_\nu}(s-s')e^{-\Phi}K_wh(s')\|_{L^\infty_{x,v}}ds'ds\\
		& \le C_K \int_{0}^t\int_{s-\epsilon}^s \exp\left\{-e^{-\|\Phi\|_\infty} \frac{\nu_0}{2} (t-s') \right\}\|h(s')\|_{L^\infty_{x,v}}ds'ds\\
		& \le C_K  \sup_{0 \le s' \le T_0 }\left\{\exp\left\{-e^{-\|\Phi\|_\infty} \frac{\nu_0}{4} (t-s') \right\}\|h(s')\|_{L^\infty_{x,v}}\right\} \int_{0}^{t}\int_{s-\epsilon}^s \exp\left\{-e^{-\|\Phi\|_\infty} \frac{\nu_0}{4} (t-s') \right\}ds'ds,
	\end{align*}
	where we have used Lemma \ref{L43} and Lemma \ref{Kest}.
	Here, we get
	\begin{align*}
		\int_0^{t}\int_{s-\epsilon}^s \exp\left\{-e^{-\|\Phi\|_\infty} \frac{\nu_0}{4} (t-s') \right\}ds'ds \le \epsilon \int_0^{t} \exp\left\{-e^{-\|\Phi\|_\infty} \frac{\nu_0}{4} (t-s) \right\}ds \le C\epsilon.
	\end{align*}
	Thus we deduce 
	\begin{align} \label{dd5}
		|J_2| \le \epsilon C_\Phi  \sup_{0 \le s' \le T_0 }\left\{\exp\left\{-e^{-\|\Phi\|_\infty} \frac{\nu_0}{4} (t-s') \right\}\|h(s')\|_{L^\infty_{x,v}}\right\}.
	\end{align}
	Now, controlling the term $J_3$ remains. Fix $(t,x,v)$ so that $(x,v) \notin \gamma_0$. Recall the definition of the iterated integral in \eqref{iterint}. From Lemma \ref{L41}, we have
	\begin{align*}
		&S_{G_\nu}(t-s)e^{-\Phi}K_wS_{G_\nu}(s-s')e^{-\Phi}K_wh(s')\\
		& = \exp\left\{-\int_s^te^{-\Phi(X(\tau))}\nu(V(\tau))d\tau\right\} \mathbf{1}_{\{t_{1}\le s \}}\left\{e^{-\Phi}K_wS_{G_\nu}(s-s')e^{-\Phi}K_wh(s')\right\}(s,X(s),V(s))\\
		&\quad + \frac{\exp{-\int_{t_1}^t e^{-\Phi(X(\tau))}\nu(V(\tau))d\tau}}{\tilde{w}(x_1,V(t_1))} \sum_{l=1}^{k-1}\int_{\prod_{j=1}^{k-1}\mathcal{V}_j} \mathbf{1}_{\{t_{l+1}\le s < t_l\}} \left\{e^{-\Phi}K_wS_{G_\nu}(s-s')e^{-\Phi}K_wh(s')\right\}(s,X_l(s),V_l(s))\\
		&\qquad \times d\Sigma_l(s)\\
		&\quad +\frac{\exp{-\int_{t_1}^t e^{-\Phi(X(\tau))}\nu(V(\tau))d\tau}}{\tilde{w}(x_1,V(t_1))}\int_{\prod_{j=1}^{k-1}\mathcal{V}_j} \mathbf{1}_{\{t_k>s\}} \left\{S_{G_\nu}(t-s)e^{-\Phi}K_wS_{G_\nu}(s-s')e^{-\Phi}K_wh(s')\right\}(t_k,x_k,V_{k-1}(t_k))\\
		&\qquad \times d\Sigma_{k-1}(t_k)\\
		&=: M_1+M_2+M_3,
	\end{align*}
	where
	\begin{align*}
		d\Sigma_l(s) &= \left\{\prod_{j=l+1}^{k-1}d\sigma_j\right\}\left\{\exp{-\int_s^{t_l} e^{-\Phi(X_l(\tau))}\nu(V_l(\tau))d\tau} \tilde{w}(x_l,v_l)d\sigma_l\right\}\\
		&\quad \cross \prod_{j=1}^{l-1}\left\{\exp{-\int_{t_{j+1}}^{t_j} e^{-\Phi(X_j(\tau))}\nu(V_j(\tau))d\tau}d\sigma_j\right\},
	\end{align*}
	and the exponential factor in $d \Sigma_l(s)$ is bounded by $\exp\left\{-e^{-\|\Phi\|_\infty}\nu_0(t_1-s)\right\}$.\\
	\indent From now on, we will sequentially adress $M_3$, $M_1$, and $M_2$.
	Firstly, we consider the term $M_3$. By Lemma \ref{L41} and Lemma \ref{Kest}, we compute
	\begin{equation} \label{ppp0}
	\begin{aligned}
		&\left\|\left\{S_{G_\nu}(t-s)e^{-\Phi}K_wS_{G_\nu}(s-s')e^{-\Phi}K_wh(s')\right\}(t_k)\right\|_{L^\infty_{x,v}}\\ 
		&\le C_\Phi \exp\left\{-e^{-\|\Phi\|_\infty} \frac{\nu_0}{2} (t_k-s) \right\}\left\|\left\{S_{G_\nu}(s-s')e^{-\Phi}K_wh(s')\right\}\right\|_{L^\infty_{x,v}}\\
		&\le C_\Phi \exp\left\{-e^{-\|\Phi\|_\infty} \frac{\nu_0}{2} (t_k-s') \right\}\left\|h(s')\right\|_{L^\infty_{x,v}}.
	\end{aligned}
	\end{equation}
	Since $t-s \ge \epsilon >0$, from Lemma \ref{Lsmall}, we can choose large $k=k_0(\epsilon,T_0)$ such that for $k \ge k_0$, for all $(t,x,v)$, $0 \le t \le T_0$, $x \in \bar{\Omega}$, $v \in \mathbb{R}^3$,
	\begin{align*}
		\int_{\prod_{j=1}^{k-2}\mathcal{V}_j} \mathbf{1}_{\{t_{k-1}(t,x,v,v_1,v_2,...,v_{k-2})>s\}} \prod_{j=1}^{k-2}d\sigma_j \le \epsilon.
	\end{align*}
	Thus from \eqref{ppp0}, we get
	\begin{align*}
		|M_3| &\le C_\Phi \exp\left\{-e^{-\|\Phi\|_\infty}\nu_0(t-t_1)\right\}\int_{\prod_{j=1}^{k-1}\mathcal{V}_j}  \mathbf{1}_{\{t_k>s\}} \exp\left\{-e^{-\|\Phi\|_\infty} \frac{\nu_0}{2} (t_1-s') \right\}\left\|h(s')\right\|_{L^\infty_{x,v}}\\
		&\quad \times \tilde{w}(x_{k-1},v_{k-1}) \prod_{j=1}^{k-1}d\sigma_j\\
		& \le C_\Phi \exp\left\{-e^{-\|\Phi\|_\infty} \frac{\nu_0}{2} (t-s') \right\}\left\|h(s')\right\|_{L^\infty_{x,v}} \int_{\mathcal{V}_{k-1}}\tilde{w}(x_{k-1},v_{k-1})\left(\int_{\prod_{j=1}^{k-2}\mathcal{V}_j} \mathbf{1}_{\{t_{k-1}>s\}} \prod_{j=1}^{k-2}d\sigma_j\right)d\sigma_{k-1}\\
		& \le \epsilon C_\Phi \exp\left\{-e^{-\|\Phi\|_\infty} \frac{\nu_0}{2} (t-s')\right\}\left\|h(s')\right\|_{L^\infty_{x,v}}.
	\end{align*}
	This yields
	\begin{equation} \label{dd6}
	\begin{aligned}
		\int_0^{t-\epsilon}\int_0^{s-\epsilon} M_3ds'ds		
		\le\epsilon C_\Phi \sup_{0 \le s' \le T_0 }\left\{\exp\left\{-e^{-\|\Phi\|_\infty} \frac{\nu_0}{4} (t-s') \right\}\|h(s')\|_{L^\infty_{x,v}}\right\},
	\end{aligned}
	\end{equation}
	where $\int_{0}^t\int_0^s \exp\left\{-e^{-\|\Phi\|_\infty} \frac{\nu_0}{4} (t-s')\right\}ds'ds$ is finite.\\
	Now, we consider the terms $M_1$ and $M_2$. We derive from Lemma \ref{L41} the formula for $e^{-\Phi}K_wS_{G_\nu}(s-s')e^{-\Phi}K_wh(s')$. Recall the definition \eqref{Backwardexit}. We denote $X'(s') = X_l(s) -\int_{s'}^s V'(\tau)d\tau$.
	\begin{align*}
		&\left\{e^{-\Phi}K_wS_{G_\nu}(s-s')e^{-\Phi}K_wh(s')\right\}(s,X_{l}(s),V_l(s))\\
		& = \int_{\mathbb{R}^3} e^{-\Phi(X_l(s))}k_w(V_l(s),v')\left\{S_{G_\nu}(s-s')e^{-\Phi}K_wh(s')\right\}(s,X_l(s),v')dv'\\
		& = \int_{\mathbb{R}^3}e^{-\Phi(X_l(s))} k_w(V_l(s),v')\exp{-\int_{s'}^s e^{-\Phi(X(\tau))}\nu(V(\tau))d\tau} \mathbf{1}_{\{t_1' \le s'\}}\left\{e^{-\Phi}K_wh(s')\right\}(s',X'(s'),V'(s'))dv'\\
		& \quad + \int_{\mathbb{R}^3} e^{-\Phi(X_l(s))}k_w(V_l(s),v')\frac{\exp{-\int_{t_1'}^s e^{-\Phi(X'(\tau))}\nu(V'(\tau))d\tau}}{\tilde{w}(x_1',V'(t_1'))} \sum_{l'=1}^{k-1}\int_{\prod_{j=1}^{k-1}\mathcal{V}_j'} \mathbf{1}_{\{t_{l'+1}'\le s' < t_{l'}'\}} \\
		&\qquad \times \left\{\int_{\mathbb{R}^3}e^{-\Phi(X_{l'}'(s'))}k_w(V_{l'}'(s'),v'')h(s',X_{l'}'(s'),v'')dv''\right\}d\Sigma_{l'}'(s')dv'\\
		& \quad +  \int_{\mathbb{R}^3}e^{-\Phi(X_l(s))} k_w(V_l(s),v')\frac{\exp{-\int_{t_1'}^s e^{-\Phi(X'(\tau))}\nu(V'(\tau))d\tau}}{\tilde{w}(x_1',V'(t_1'))}  \int_{\prod_{j=1}^{k-1}\mathcal{V}_j'}\mathbf{1}_{\{t_k' > s'\}} \\
		&\qquad \times \left\{S_{G_\nu}(s-s')e^{-\Phi}K_wh(s')\right\}(t_k',x_k',V_{k-1}'(t_k'))d\Sigma_{k-1}'(t_k')dv'\\
		& =: L_1+L_2+L_3.
	\end{align*}
	\indent We first consider the term $L_3$. Since $s-s' \ge \epsilon >0$, by Lemma \ref{Lsmall}, we can choose large $k=k_0(\epsilon,T_0)$ such that for $k \ge k_0$, for all $(s,,X_{l}(s),V_l(s))$, $0 \le s \le T_0$, $X_{l}(s) \in \bar{\Omega}$, $V_l(s) \in \mathbb{R}^3$,
	\begin{align*}
		\int_{\prod_{j=1}^{k-2}\mathcal{V}_j'} \mathbf{1}_{\{t_{k-1}'(s,X_{l}(s),v',v'_1,v'_2,...,v'_{k-2})>s'\}} \prod_{j=1}^{k-2}d\sigma'_j \le \epsilon.
	\end{align*}
	From Lemma \ref{L43} and Lemma \ref{Kest}, we have
	\begin{align} \label{ppp10}
		\left\|\left\{S_{G_\nu}(s-s')e^{-\Phi}K_wh(s')\right\}(t_k')\right\|_{L^\infty_{x,v}} \le C_\Phi\exp\left\{-e^{-\|\Phi\|_\infty} \frac{\nu_0}{2} (t_k'-s') \right\} \|h(s')\|_{L^\infty_{x,v}}.
	\end{align}
	Using the estimate \eqref{ppp10}, we obtain
	\begin{align*}
		L_3 
		&\le C_\Phi \exp\left\{-e^{-\|\Phi\|_\infty} \frac{\nu_0}{2} (s-s') \right\}\|h(s')\|_{L^\infty_{x,v}}\int_{\mathbb{R}^3}k_w(V_l(s),v')\left(\int_{\prod_{j=1}^{k-2}\mathcal{V}_j'} \mathbf{1}_{\{t_{k-1}'>s'\}} \prod_{j=1}^{k-2}d\sigma_j'\right) \\
		&\quad \times  \left(\int_{\mathcal{V}_{k-1}'}\tilde{w}(x_{k-1}',v_{k-1}')d\sigma_{k-1}'\right) dv'\\
		&\le \epsilon C_\Phi \exp\left\{-e^{-\|\Phi\|_\infty} \frac{\nu_0}{2} (s-s') \right\}\|h(s')\|_{L^\infty_{x,v}}.
	\end{align*}
	Hence the first contribution of $L_3$ (contribution of $M_1$ and $L_3$) is 
	\begin{align} \label{dd7}
		&\int_0^{t-\epsilon}\int_0^{s-\epsilon} \exp{-\int_{s}^t e^{-\Phi(X(\tau))}\nu(V(\tau))d\tau}\mathbf{1}_{\{t_{1}\le s \}} L_3ds'ds\nonumber\\
		&\le \epsilon C_\Phi \sup_{0 \le s' \le T_0 }\left\{\exp\left\{-e^{-\|\Phi\|_\infty} \frac{\nu_0}{4} (t-s') \right\}\|h(s')\|_{L^\infty_{x,v}}\right\}.
	\end{align}
	The second contribution of $L_3$ (contribution of $M_2$ and $L_3$) also is
	\begin{align}\label{dd8}
		&\int_0^{t-\epsilon}\int_0^{s-\epsilon}\frac{\exp{-\int_{t_1}^t e^{-\Phi(X(\tau))}\nu(V(\tau))d\tau}}{\tilde{w}(x_1,V(t_1))} \sum_{l=1}^{k-1}\int_{\prod_{j=1}^{k-1}\mathcal{V}_j} \mathbf{1}_{\{t_{l+1}\le s < t_l\}} L_3(V_l(s))d\Sigma_l(s)ds'ds\nonumber\\
		& \le \epsilon C_\Phi  \sum_{l=1}^{k-1}\int_{t_{l+1}}^{t_l}\int_0^{s} \frac{\exp\left\{-e^{-\|\Phi\|_\infty}\nu_0(t-t_1)\right\}}{\tilde{w}(x_1,V(t_1))} \int_{\prod_{j=1}^{l}\mathcal{V}_j}\exp\left\{-e^{-\|\Phi\|_\infty} \frac{\nu_0}{2} (t_1-s') \right\}\|h(s')\|_{L^\infty_{x,v}}\tilde{w}(x_l,v_l)\prod_{j=1}^{l}d\sigma_jds'ds\nonumber\\
		& \le \epsilon C_\Phi \sum_{l=1}^{k-1}\int_{t_{l+1}}^{t_l}\int_0^{s} \exp\left\{-e^{-\|\Phi\|_\infty} \frac{\nu_0}{2} (t-s') \right\}\|h(s')\|_{L^\infty_{x,v}}ds'ds\nonumber\\
		& \le \epsilon C_\Phi \sup_{0 \le s' \le T_0 }\left\{\exp\left\{-e^{-\|\Phi\|_\infty} \frac{\nu_0}{4} (t-s') \right\}\|h(s')\|_{L^\infty_{x,v}}\right\}.
	\end{align}
	The remaining terms is
	\begin{equation} \label{RMC}
	\begin{aligned}
		&\int_0^{t-\epsilon}\int_0^{s-\epsilon} \exp{-\int_{s}^t e^{-\Phi(X(\tau))}\nu(V(\tau))d\tau} \mathbf{1}_{\{t_{1}\le s \}}\int_{\mathbb{R}^3} e^{-\Phi(X(s))}k_w(V(s),v') \exp{-\int_{s'}^s e^{-\Phi(X'(\tau))}\nu(V'(\tau))d\tau}\\ 
		&\quad \times \mathbf{1}_{\{t_1' \le s'\}} \left\{e^{-\Phi}K_wh(s')\right\}(s',X'(s'),V'(s'))dv'ds'ds\\
		&+\int_0^{t-\epsilon}\int_0^{s-\epsilon}  \exp{-\int_{s}^t e^{-\Phi(X(\tau))}\nu(V(\tau))d\tau}  \mathbf{1}_{\{t_{1}\le s \}}\int_{\mathbb{R}^3} e^{-\Phi(X(s))}k_w(V(s),v')\frac{\exp{-\int_{t_1'}^s e^{-\Phi(X'(\tau))}\nu(V'(\tau))d\tau}}{\tilde{w}(x_1',V'(t_1'))}\\  
		&\quad \times \sum_{l'=1}^{k-1}\int_{\prod_{j=1}^{k-1}\mathcal{V}_j'} \mathbf{1}_{\{t_{l'+1}'\le s' < t_{l'}'\}} \left\{\int_{\mathbb{R}^3}e^{-\Phi(X_{l'}'(s'))}k_w(V_{l'}'(s'),v'')h(s',X_{l'}'(s'),v'')dv''\right\}d\Sigma_{l'}'(s')dv'ds'ds\\
		&+\int_0^{t-\epsilon}\int_0^{s-\epsilon} \frac{\exp{-\int_{t_1}^t e^{-\Phi(X(\tau))}\nu(V(\tau))d\tau}}{\tilde{w}(x_1,V(t_1))} \sum_{l=1}^{k-1}\int_{\prod_{j=1}^{k-1}\mathcal{V}_j} \mathbf{1}_{\{t_{l+1}\le s < t_l\}} \int_{\mathbb{R}^3} e^{-\Phi(X_l(s))} k_w(V_l(s),v')\\
		&\quad \times \exp{-\int_{s'}^s e^{-\Phi(X'(\tau))}\nu(V'(\tau))d\tau} \mathbf{1}_{\{t_1' \le s'\}}\left\{e^{-\Phi}K_wh(s')\right\}(s',X'(s'),V'(s'))dv'd\Sigma_l(s)ds'ds\\
		&+\int_0^{t-\epsilon}\int_0^{s-\epsilon} \frac{\exp{-\int_{t_1}^t e^{-\Phi(X(\tau))}\nu(V(\tau))d\tau}}{\tilde{w}(x_1,V(t_1))} \sum_{l=1}^{k-1}\int_{\prod_{j=1}^{k-1}\mathcal{V}_j} \mathbf{1}_{\{t_{l+1}\le s < t_l\}} \int_{\mathbb{R}^3} e^{-\Phi(X_l(s))}k_w(V_l(s),v')\\
		&\quad \times\frac{\exp{-\int_{t_1'}^s e^{-\Phi(X'(\tau))}\nu(V'(\tau))d\tau}}{\tilde{w}(x_1',V'(t_1'))} \sum_{l'=1}^{k-1}\int_{\prod_{j=1}^{k-1}\mathcal{V}_j'} \mathbf{1}_{\{t_{l'+1}'\le s' < t_{l'}'\}}\biggl\{\int_{\mathbb{R}^3}e^{-\Phi(X_{l'}'(s'))}k_w(V_{l'}'(s'),v'')\\
		&\quad \times h(s',X'_{l'}(s'),v'')dv''\biggr\}d\Sigma_{l'}'(s')dv'd\Sigma_l(s)ds'ds\\
		&=: R_1+R_2+R_3+R_4.
	\end{aligned}
	\end{equation}
	\newline
	First of all, let us estimate the term $R_1$ in \eqref{RMC} :
	\begin{align}
		&\int_{t_1}^{t-\epsilon}\int_{t_1'}^{s-\epsilon} \exp{-\int_{s}^t e^{-\Phi(X(\tau))}\nu(V(\tau))d\tau} \int_{\mathbb{R}^3} e^{-\Phi(X(s))}k_w(V(s),v') \exp{-\int_{s'}^s e^{-\Phi(X(\tau))}\nu(V(\tau))d\tau}\nonumber \\
		&\times \int_{\mathbb{R}^3} e^{-\Phi(X'(s'))}k_w(V'(s')	,v'')h(s',X'(s'),v'')dv''dv'ds'ds \nonumber \\ \label{FTMC}
		&\le \int_0^{t-\epsilon}\int_0^{s-\epsilon}\exp\left\{-e^{-\|\Phi\|_\infty}\nu_0(t-s')\right\}\int_{\mathbb{R}^3} \int_{\mathbb{R}^3}|k_w(V(s),v')| |k_w(V'(s'),v'')||h(s',X'(s'),v'')|\nonumber \\
		&\quad \times dv''dv'ds'ds.
	\end{align}
	We will divide this term into 3 cases.\\
	\newline
	$\mathbf{Case\ 1:}$ $|v| \ge R$ with $R \gg 2\sqrt{2\|\Phi\|_\infty}$.\\
	By \eqref{sese}, we get
	\begin{align*}
		|V(s)|\ge |v|-\sqrt{2\|\Phi\|_\infty}\ge \frac{R}{2}.
	\end{align*}
	From Lemma \ref{Kest}, we have
	\begin{align*}
		\int_{\mathbb{R}^3} \int_{\mathbb{R}^3} |k_w(V(s),v')| |k_w(V'(s'),v'')|dv''dv'\le \frac{C_\Phi}{1+R}.
	\end{align*}
	Then $R_1$ in this case is bounded by
	\begin{align}\label{class1}
		&\frac{C}{1+R}\int_0^t \int_0^s \exp\left\{-e^{-\|\Phi\|_\infty}\nu_0(t-s')\right\} \|h(s')\|_{L^\infty_{x,v}}ds'ds\nonumber\\
		& \le \frac{C}{1+R} \sup_{0\le s' \le T_0} \left\{\exp\left\{-e^{-\|\Phi\|_\infty} \frac{\nu_0}{2} (t-s') \right\}\|h(s')\|_{L^\infty_{x,v}}\right\} \int_0^t \int_0^s \exp\left\{-e^{-\|\Phi\|_\infty} \frac{\nu_0}{2} (t-s') \right\} ds'ds\nonumber\\
		& \le \frac{C_\Phi}{1+R} \sup_{0\le s' \le T_0} \left\{\exp\left\{-e^{-\|\Phi\|_\infty} \frac{\nu_0}{2} (t-s') \right\}\|h(s')\|_{L^\infty_{x,v}}\right\},
	\end{align}
	where we have used the fact $\int_0^t \int_0^s \exp\left\{-e^{-\|\Phi\|_\infty} \frac{\nu_0}{2} (t-s') \right\} ds'ds$ is finite.\\
	\newline
	$\mathbf{Case\ 2:}$ $|v| \le R$, $|v'| \ge 2R$, or $|v'| \le 2R$, $|v''| \ge 3R$.\\
	Note that either $|v-v'| \ge R$ or $|v'-v''| \ge R$. From \eqref{sese}, either one of the followings holds:
	\begin{align*}
		&|V(s)-v'| \ge |v-v'|-|V(s)-v| \ge R-\frac{R}{2}=\frac{R}{2},\\
		&|V'(s')-v''| \ge |v'-v''|-|V'(s')-v'| \ge R-\frac{R}{2}=\frac{R}{2}.
	\end{align*}
	Then we have either one of the followings:
	\begin{equation} \label{dd12}
	\begin{aligned}
		&|k_w(V(s),v')| \le e^{-\frac{R^2}{64}}|k_w(V(s),v')|e^{\frac{1}{16}|V(s)-v'|^2},\\
		&|k_w(V'(s'),v'')| \le e^{-\frac{R^2}{64}}|k_w(V'(s'),v'')|e^{\frac{1}{16}|V'(s')-v''|^2}.
	\end{aligned}
	\end{equation}
	This yields from Lemma \ref{Kest}
	\begin{equation} \label{dd13}
	\begin{aligned}
		&\int_{|v'| \ge 2R} |k_w(V(s),v')|e^{\frac{1}{16}|V(s)-v'|^2}dv' < C,\\
		&\int_{|v''| \ge 3R} |k_w(V'(s'),v'')|e^{\frac{1}{16}|V'(s')-v''|^2}dv'' < C,
	\end{aligned}
	\end{equation}
	for some constant $C$.
	Thus we use \eqref{dd12} and \eqref{dd13} to bound $R_1$ in this case by
	\begin{align} 
		&\int_0^t\int_0^s \exp\left\{-e^{-\|\Phi\|_\infty}\nu_0(t-s')\right\}\int_{|v'|\ge 2R} \int_{\mathbb{R}^3}|k_w(V(s),v')| |k_w(V'(s'),v'')|\|h(s')\|_{L^\infty_{x,v}}dv''dv'ds'ds\nonumber\\
		&+ \int_0^t\int_0^s \exp\left\{-e^{-\|\Phi\|_\infty}\nu_0(t-s')\right\}\int_{|v'|\le 2R} \int_{|v''| \ge 3R}|k_w(V(s),v')| |k_w(V'(s'),v'')|\|h(s')\|_{L^\infty_{x,v}}dv''dv'ds'ds\nonumber\\
		&\le C_\Phi \int_0^t\int_0^s \exp\left\{-e^{-\|\Phi\|_\infty}\nu_0(t-s')\right\}\int_{|v'|\ge 2R} |k_w(V(s),v')| \|h(s')\|_{L^\infty_{x,v}}dv'ds'ds\nonumber\\
		&\quad + C_\Phi e^{-\frac{R^2}{64}}\int_0^t\int_0^s \exp\left\{-e^{-\|\Phi\|_\infty}\nu_0(t-s')\right\}\int_{|v'|\le 2R} |k_w(V(s),v')| \|h(s')\|_{L^\infty_{x,v}}dv'ds'ds\nonumber\\	
		&\le C_\Phi e^{-\frac{R^2}{64}}\int_0^t\int_0^s \exp\left\{-e^{-\|\Phi\|_\infty}\nu_0(t-s')\right\}\|h(s')\|_{L^\infty_{x,v}}ds'ds\nonumber\\
		&\le C_\Phi e^{-\frac{R^2}{64}} \sup_{0\le s' \le T_0} \left\{\exp\left\{-e^{-\|\Phi\|_\infty} \frac{\nu_0}{2} (t-s') \right\}\|h(s')\|_{L^\infty_{x,v}}\right\}\int_0^t\int_0^s \exp\left\{-e^{-\|\Phi\|_\infty} \frac{\nu_0}{2} (t-s') \right\}ds'ds\nonumber\\\label{class2}
		&\le C_\Phi e^{-\frac{R^2}{64}}\sup_{0\le s' \le T_0} \left\{\exp\left\{-e^{-\|\Phi\|_\infty} \frac{\nu_0}{2} (t-s') \right\}\|h(s')\|_{L^\infty_{x,v}}\right\},
	\end{align}
	where we have used the fact $\int_0^t\int_0^s \exp\left\{-e^{-\|\Phi\|_\infty} \frac{\nu_0}{2} (t-s') \right\}ds'ds$ is finite.\\
	\newline
	$\mathbf{Case\ 3:}$ $|v|\le R$, $|v'| \le 2R$, $|v''| \le 3R$.\\
	Since $k_w(v,v')$ has possible integrable singularity of $\frac{1}{|v-v'|}$, we can choose smooth function $k_R(v,v')$ with compact support such that
	\begin{align} \label{kk1}
	 	\sup_{|v|\le 3R}\int_{|v'| \le 3R} \left|k_R(v,v')-k_w(v,v')\right|dv' \le \frac{1}{R}.
	\end{align}
	We split 
	\begin{equation} \label{kk2}
	\begin{aligned}
		k_w(V(s),v')k_w(V'(s'),v'') &= \left\{k_w(V(s),v')-k_R(V(s),v')\right\}k_w(V'(s'),v'')\\
		&\quad + \{k_w(V'(s'),v'')-k_R(V'(s'),v'')\}k_R(V(s),v')\\
		&\quad +k_R(V(s),v')k_R(V'(s'),v'').
	\end{aligned}
	\end{equation}
	From \eqref{kk1} and \eqref{kk2}, $R_1$ in this case is bounded by 
	\begin{align}
		&\frac{C_\Phi}{R} \int_0^t \int_0^{s-\epsilon} \exp\left\{-e^{-\|\Phi\|_\infty}\nu_0(t-s')\right\} \|h(s')\|_{L^\infty_{x,v}}ds'ds \nonumber\\ 
		& + \int_0^t \int_0^{s-\epsilon} \exp\left\{-e^{-\|\Phi\|_\infty}\nu_0(t-s')\right\} \int_{|v'|\le 2R, |v''|\le 3R}|k_R(V(s),v')||k_R(V'(s'),v'')||h(s',X'(s'),v'')|dv''dv'ds'ds \nonumber\\ \label{class3}
		&\le \frac{C_\Phi}{R}\sup_{0 \le s' \le T_0}\left\{\exp{-e^{-\|\Phi\|_\infty} \frac{\nu_0}{2} (t-s')}\|h(s')\|_{L^\infty_{x,v}}\right\} \\ \nonumber 
		&\quad + C_{R,\Phi}\int_0^t \int_0^{s-\epsilon} \exp\left\{-e^{-\|\Phi\|_\infty}\nu_0(t-s')\right\} \int_{|v'|\le 2R, |v''|\le 3R}|h(s',X'(s'),v'')|dv''dv'ds'ds\\ \nonumber
		&=: R_{11}+R_{12}, 
	\end{align}
	where we have used the fact $|k_R(V(s),v')||k_R(V'(s'),v'')| \le C_R$.\\
	In the term $R_1$, we recall that $X'(s') = X(s';s,X(s;t,x,v),v')$.
	Since the potential is time dependent, we have
	\begin{align*}
		X(s';s,X(s;t,x,v),v') = X(s'-s+T_0;T_0,X(T_0;t-s+T_0,x,v),v')
	\end{align*}
	for all $0\le s' \le s \le t$.\\
	By Lemma \ref{cpl}, the term $R_{12}$ becomes
	\begin{align} 
		&C_{R,\Phi}\sum_{i_1}^{M_1}\sum_{I_2}^{(M_2)^3}\sum_{I_3}^{(M_3)^3} \int_0^t \mathbf{1}_{\{X(T_0;t-s+T_0,x,v)\in \mathcal{P}_{I_2}^{\Omega}\}}(s)\int_0^s\mathbf{1}_{\mathcal{P}_{i_1}^{T_0}}(s'-s+T_0) \exp\left\{-e^{-\|\Phi\|_\infty}\nu_0(t-s')\right\} \nonumber \\ \label{ltlt1}
		& \times \int_{|v'|\le 2R, |v''|\le 3R}\mathbf{1}_{\mathcal{P}_{I_3}^v}(v')|h(s',X(s'-s+T_0;T_0,X(T_0;t-s+T_0,x,v),v'),v'')|dv''dv'ds'ds.
	\end{align}
	Let $\tilde{\epsilon}>0$.
	From Lemma \ref{cpl}, we have the following partitions:
	\begin{align*}
		&\Big\{(s'-s+T_0,X(T_0;t-s+T_0,x,v),v')\in \mathcal{P}_{i_1}^{T_0}\cross \mathcal{P}_{I_2}^{\Omega}\cross \mathcal{P}_{I_3}^{v} \\
		&\quad : \det\left(\frac{dX}{dv'}(s'-s+T_0;T_0,X(T_0;t-s+T_0,x,v),v')\right)=0\Big\}\\
		& \subset \bigcup_{j=1}^3 \bigg\{(s'-s+T_0,X(T_0;t-s+T_0,x,v),v')\in \mathcal{P}_{i_1}^{T_0}\cross \mathcal{P}_{I_2}^{\Omega}\cross \mathcal{P}_{I_3}^{v}\\
		&\quad :s'-s+T_0\in \left(t_{j,i_1,I_2,I_3}-\frac{\tilde{\epsilon}}{4M_1},t_{j,i_1,I_2,I_3}+\frac{\tilde{\epsilon}}{4M_1}\right)\bigg\}.
	\end{align*}
	Thus for each $i_1$,$I_2$, and $I_3$, we split $\mathbf{1}_{\mathcal{P}_{i_1}^{T_0}}(s'-s+T_0)$ as
	\begin{align} \label{fst1}
		&\mathbf{1}_{\mathcal{P}_{i_1}^{T_0}}(s'-s+T_0)\mathbf{1}_{\cup_{j=1}^3(t_{j,i_1,I_2,I_3}-\frac{\tilde{\epsilon}}{4M_1},t_{j,i_1,I_2,I_3}+\frac{\tilde{\epsilon}}{4M_1})}(s'-s+T_0)\\ \label{sst1}
		& +\mathbf{1}_{\mathcal{P}_{i_1}^{T_0}}(s'-s+T_0)\left\{1-\mathbf{1}_{\cup_{j=1}^3(t_{j,i_1,I_2,I_3}-\frac{\tilde{\epsilon}}{4M_1},t_{j,i_1,I_2,I_3}+\frac{\tilde{\epsilon}}{4M_1})}(s'-s+T_0)\right\}.
	\end{align}
	$\mathbf{Case\ 3 \ (\romannumeral 1):}$ The integration \eqref{ltlt1} corresponding to \eqref{fst1} is bounded by
	\begin{align}
		&C_{R,\Phi}\sum_{i_1}^{M_1}\sum_{I_2}^{(M_2)^3}\sum_{I_3}^{(M_3)^3} \sum_{j=1}^3 \int_0^t \mathbf{1}_{\{X(T_0;t-s+T_0,x,v)\in \mathcal{P}_{I_2}^{\Omega}\}}(s)\int_0^s\mathbf{1}_{\mathcal{P}_{i_1}^{T_0}}(s'-s+T_0) \nonumber \\
		& \times \mathbf{1}_{(t_{j,i_1,I_2,I_3}-\frac{\tilde{\epsilon}}{4M_1},t_{j,i_1,I_2,I_3}+\frac{\tilde{\epsilon}}{4M_1})}(s'-s+T_0) \exp\left\{-e^{-\|\Phi\|_\infty} \nu_0 (t-s') \right\} \nonumber \\ \label{mlt1}
		&\times \int_{|v'|\le 2R}\mathbf{1}_{\mathcal{P}_{I_3}^v}(v')\int_{|v''|\le 3R}|h(s',X(s'-s+T_0;T_0,X(T_0;t-s+T_0,x,v),v'),v'')|dv''dv'ds'ds.
	\end{align}
	We split
	\begin{align*}
		\exp\left\{-e^{-\|\Phi\|_\infty} \nu_0 (t-s') \right\} &= \exp\left\{-e^{-\|\Phi\|_\infty} \frac{\nu_0}{2} (t-s) \right\}\exp\left\{-e^{-\|\Phi\|_\infty} \frac{\nu_0}{2} (s-s') \right\}\\
	& \quad \times  \exp\left\{-e^{-\|\Phi\|_\infty} \frac{\nu_0}{2} (t-s') \right\}.
	\end{align*}
	and we can bound the integration \eqref{mlt1} by
	\begin{align*}
		&C_{R,\Phi}\sum_{i_1}^{M_1}\sum_{I_2}^{(M_2)^3}\sum_{I_3}^{(M_3)^3} \sum_{j=1}^3 \int_0^t \mathbf{1}_{\{X(T_0;t-s+T_0,x,v)\in \mathcal{P}_{I_2}^{\Omega}\}}(s)\exp\left\{-e^{-\|\Phi\|_\infty} \frac{\nu_0}{2} (t-s) \right\}\\
		& \times \underbrace{\int_0^s\mathbf{1}_{\mathcal{P}_{i_1}^{T_0}}(s'-s+T_0)\mathbf{1}_{(t_{j,i_1,I_2,I_3}-\frac{\tilde{\epsilon}}{4M_1},t_{j,i_1,I_2,I_3}+\frac{\tilde{\epsilon}}{4M_1})}(s'-s+T_0) \exp\left\{-e^{-\|\Phi\|_\infty} \frac{\nu_0}{2} (s-s') \right\}}_{(*1)}\\
		&\times \int_{|v'|\le 2R}\mathbf{1}_{\mathcal{P}_{I_3}^v}(v')\int_{|v''|\le 3R}\exp\left\{-e^{-\|\Phi\|_\infty} \frac{\nu_0}{2} (t-s') \right\}\|h(s')\|_{L^\infty_{x,v}}dv''dv'ds'ds
	\end{align*}
	Here, $(*1)$ is bounded by
	\begin{align}
		&\int_0^s \mathbf{1}_{\mathcal{P}_{i_1}^{T_0}}(s'-s+T_0)\mathbf{1}_{(t_{j,i_1,I_2,I_3}-\frac{\tilde{\epsilon}}{4M_1},t_{j,i_1,I_2,I_3}+\frac{\tilde{\epsilon}}{4M_1})}(s'-s+T_0) \exp\left\{-e^{-\|\Phi\|_\infty} \frac{\nu_0}{2} (s-s') \right\}ds'\nonumber\\
		& \le \int_{s-T_0+t_{j,i_1,I_2,I_3}-\frac{\tilde{\epsilon}}{4M_1}}^{s-T_0+t_{j,i_1,I_2,I_3}+\frac{\tilde{\epsilon}}{4M_1}} \exp\left\{-e^{-\|\Phi\|_\infty} \frac{\nu_0}{2} (s-s') \right\}ds'\nonumber\\ \label{ppp1}
		& \le \frac{\tilde{\epsilon}}{2M_1}.
	\end{align}
	From the partition of the time interval $[0,T_0]$ and velocity domain $[-4R,4R]^3$ in Lemma \ref{cpl}, we have
	\begin{equation} \label{ppp2}
	\begin{aligned}
		&\sum_{I_2}^{(M_2)^3}\mathbf{1}_{\{X(t-s+T_0,x,v)\in \mathcal{P}_{I_2}^{\Omega}\}}(s) \le \mathbf{1}_{\{0\le s \le T_0\}}(s),\\
		&\sum_{I_3}^{(M_3)^3} \mathbf{1}_{\mathcal{P}_{I_3}^v}(v')\mathbf{1}_{\{|v'|\le 2R\}}(v') = \mathbf{1}_{\{|v'|\le 2R\}}(v').
	\end{aligned}
	\end{equation}
	Using \eqref{ppp1} and \eqref{ppp2}, \eqref{mlt1} is bounded by
	\begin{align}
		& C_{R,\Phi} \sup_{0\le s' \le T_0} \left\{\exp\left\{-e^{-\|\Phi\|_\infty} \frac{\nu_0}{2} (t-s') \right\}\|h(s')\|_{L^\infty_{x,v}}\right\}\ \sum_{i_1}^{M_1} \sum_{I_2}^{(M_2)^3} \int_0^t \mathbf{1}_{\{X(T_0;t-s+T_0,x,v) \in \mathcal{P}_{I_2}^\Omega\}}(s) \exp\left\{-e^{-\|\Phi\|_\infty} \frac{\nu_0}{2} (t-s) \right\}\nonumber\\
		& \times \int_0^s \mathbf{1}_{\mathcal{P}_{i_1}^{T_0}}(s'-s+T_0)\mathbf{1}_{(t_{j,i_1,I_2,I_3}-\frac{\tilde{\epsilon}}{4M_1},t_{j,i_1,I_2,I_3}+\frac{\tilde{\epsilon}}{4M_1})}(s'-s+T_0) \exp\left\{-e^{-\|\Phi\|_\infty} \frac{\nu_0}{2} (s-s') \right\}ds'ds\nonumber\\
		&\le C_{R,\Phi} \frac{\tilde{\epsilon}}{2}\sup_{0\le s' \le T_0} \left\{\exp\left\{-e^{-\|\Phi\|_\infty} \frac{\nu_0}{2} (t-s') \right\}\|h(s')\|_{L^\infty_{x,v}}\right\}\nonumber\\
		&\quad \times \sum_{I_2}^{(M_2)^3} \int_0^t \mathbf{1}_{\{X(T_0;t-s+T_0,x,v) \in \mathcal{P}_{I_2}^\Omega\}}(s)\exp\left\{-e^{-\|\Phi\|_\infty} \frac{\nu_0}{2} (t-s) \right\} ds\nonumber\\ \label{class4}
		&\le \tilde{\epsilon} \ C_{R,\Phi} \sup_{0\le s' \le T_0} \left\{\exp\left\{-e^{-\|\Phi\|_\infty} \frac{\nu_0}{2} (t-s') \right\}\|h(s')\|_{L^\infty_{x,v}}\right\}.
	\end{align}
	\newline
	$\mathbf{Case\ 3 \ (\romannumeral 2):}$ The integration \eqref{ltlt1} corresponding to \eqref{sst1} is bounded by
	\begin{align}
		&C_{R,\Phi}\sum_{i_1}^{M_1}\sum_{I_2}^{(M_2)^3}\sum_{I_3}^{(M_3)^3} \int_0^t \mathbf{1}_{\{X(T_0;t-s+T_0,x,v)\in \mathcal{P}_{I_2}^{\Omega}\}}(s)\int_0^s\mathbf{1}_{\mathcal{P}_{i_1}^{T_0}}(s'-s+T_0) \nonumber \\
		& \times \left\{1-\mathbf{1}_{\cup_{j=1}^3(t_{j,i_1,I_2,I_3}-\frac{\tilde{\epsilon}}{4M_1},t_{j,i_1,I_2,I_3}+\frac{\tilde{\epsilon}}{4M_1})}(s'-s+T_0)\right\}\exp\left\{-e^{-\|\Phi\|_\infty} \nu_0 (t-s') \right\} \nonumber \\ \label{mmlt1}
		&\times \underbrace{\int_{|v'|\le 2R}\mathbf{1}_{\mathcal{P}_{I_3}^v}(v')\int_{|v''|\le 3R}|h(s',X(s'-s+T_0;T_0,X(T_0;t-s+T_0,x,v),v'),v'')|dv''dv'}_{(\#1)}ds'ds.
	\end{align}
	By Lemma \ref{cpl}, we have made a change of variables $v' \rightarrow y:=X(s'-s+T_0;T_0,X(T_0;t-s+T_0,x,v),v')$ so that 
	\begin{align*}
		\det\left(\frac{dX}{dv'}(s'-s+T_0;T_0,X(T_0;t-s+T_0,x,v),v')\right)> \delta_*
	\end{align*}
	and the term $(\# 1)$ is bounded by
	\begin{align*}
		&\int_{|v'|\le 2R}\int_{|v''|\le 3R}|h(s',X(s'-s+T_0;T_0,X(T_0;t-s+T_0,x,v),v'),v'')|dv''dv'\\
		&\le \frac{1}{\delta_*}\int_{\Omega} \int_{|v''|\le 3R}|h(s',y,v'')|dv''dy\\
		&\le \frac{1}{\delta_*}\left(\int_{\Omega}\int_{|v''|\le 3R}w(y,v'')^2dv''dy\right)^{\frac{1}{2}}\|f(s')\|_{L^2_{x,v}}\\
		&\le \frac{C_{R,\Phi}}{\delta_*} \|f(s')\|_{L^2_{x,v}},
	\end{align*}
	where we have used the Cauchy-Schwarz inequality.
	Hence \eqref{mmlt1} is bounded by
	\begin{align}
		&\frac{C_{R,\Phi,M_1,M_2,M_3}}{\delta_*} \int_0^t\int_0^s\exp\left\{-e^{-\|\Phi\|_\infty} \nu_0 (t-s') \right\} \|f(s')\|_{L^2_{x,v}}ds'ds\nonumber\\ \label{class5}
		& \le \frac{C_{R,\Phi,M_1,M_2,M_3}}{\delta_*} \int_0^{T_0}\|f(s')\|_{L^2_{x,v}}ds'.
	\end{align}
	Combining the bounds \eqref{class1}, \eqref{class2}, \eqref{class3}, \eqref{class4}, and \eqref{class5}, we can bound $R_1$ in \eqref{RMC} by
	\begin{equation} \label{dd19}
	\begin{aligned}
		&\left(\frac{C_\Phi}{R}+C_\Phi e^{-\frac{R^2}{64}}+ \tilde{\epsilon} \ C_{R,\Phi} \right) \exp\left\{-e^{-\|\Phi\|_\infty} \frac{\nu_0}{2} t \right\} \sup_{0 \le s' \le T_0}\left\{\exp\left\{e^{-\|\Phi\|_\infty} \frac{\nu_0}{2} s' \right\}\|h(s')\|_{L^\infty_{x,v}}\right\}\\
		& +\frac{C_{R,\Phi,M_1,M_2,M_3}}{\delta_*} \int_0^{T_0}\|f(s')\|_{L^2_{x,v}}ds'.
	\end{aligned}
	\end{equation}
	\newline
	Next, let us estimate the term $R_2$ in \eqref{RMC} :
	\begin{align} 
		&\int_{t_1}^{t-\epsilon} \exp{-\int_{s}^t e^{-\Phi(X(\tau))}\nu(V(\tau))d\tau} \int_{\mathbb{R}^3} e^{-\Phi(X(s))}k_w(V(s),v')\frac{\exp{-\int_{t_1'}^s e^{-\Phi(X'(\tau))}\nu(V'(\tau))d\tau}}{\tilde{w}(x_1',V'(t_1'))}\nonumber\\
		&\quad \times \sum_{l'=1}^{k-1}\int_{t_{l'+1}'}^{t_{l'}'}\int_{\prod_{j=1}^{k-1}\mathcal{V}_j'}\left\{\int_{\mathbb{R}^3}e^{-\Phi(X'_{l'}(s'))}k_w(V_{l'}'(s'),v'')h(s',X_{l'}'(s'),v'')dv''\right\}d\Sigma_{l'}'(s')ds'dv'ds \nonumber \\
		& \le \int_{t_1}^{t-\epsilon} \exp\left\{-e^{-\|\Phi\|_\infty}\nu_0(t-t_1)\right\}\int_{\mathbb{R}^3} |k_w(V(s),v')|\frac{1}{\tilde{w}(x_1',V'(t_1'))}\sum_{l'=1}^{k-1}\int_{t_{l'+1}'}^{t_{l'}'}\exp\left\{-e^{-\|\Phi\|_\infty}\nu_0(t_1-s)\right\} \nonumber \\ \label{STMC}
		&\quad \times  \int_{\prod_{j=1}^{k-1}\mathcal{V}_j'}\Biggl\{\int_{\mathbb{R}^3}|k(V_{l'}'(s'),v'')| |h(s',X_{l'}'(s'),v'')|dv''\Biggr\}\left\{\prod_{j=l'+1}^{k-1}d\sigma_j'\right\}\left\{\tilde{w}(x'_{l'},v'_{l'})d\sigma_{l'}'\right\}\left\{\prod_{j=1}^{l'-1}d\sigma_j'\right\}ds'dv'ds.
	\end{align}
	Fix $l'$. Note that $\tilde{w}(x'_{l'},v_{l'}')\mu(v_{l'}')|v_{l'}'|\le C_\Phi$ for some constant $C_\Phi >0$.
	We will divide this term into 3 cases.\\
	\newline
	$\mathbf{Case\ 1:}$ $|v| \ge R$ or $|v_{l'}'| \ge R$ with $R \gg 2\sqrt{2\|\Phi\|_\infty}$.\\
	By \eqref{sese}, we get
	\begin{align*}
		|V(s)|\ge |v|-\sqrt{2\|\Phi\|_\infty}\ge \frac{R}{2} \quad \text{or} \quad |V'_{l'}(s')|\ge |v_{l'}'|-\sqrt{2\|\Phi\|_\infty}\ge \frac{R}{2}
	\end{align*}
	From Lemma \ref{Kest}, we have
	\begin{align*}
		\int_{|v_{l'}'|\ge R} \left(\int_{\mathbb{R}^3} |k_w(V_{l'}'(s'),v'')||h(s',X_{l'}'(s'),v'')|dv''\right)\tilde{w}(x'_{l'}v_{l'}')d\sigma_{l'}' \le \frac{C_\Phi}{1+R} \|h(s')\|_{L^\infty_{x,v}}.
	\end{align*}
	Then $R_2$ in the case $|v_{l'}'| \ge R$ is bounded by
	\begin{align}
		&\frac{C_\Phi}{1+R} \int_0^{t-\epsilon} \int_{t_{l'+1}'}^{t_{l'}'} \exp\left\{-e^{-\|\Phi\|_\infty}\nu_0(t-s')\right\}\|h(s')\|_{L^\infty_{x,v}} \int_{\mathbb{R}^3}|k_w(V(s),v')|dv'ds'ds\nonumber\\ \label{class11}
		& \le \frac{C_\Phi}{1+R} \sup_{0\le s' \le T_0} \left\{\exp\left\{-e^{-\|\Phi\|_\infty} \frac{\nu_0}{2} (t-s') \right\}\|h(s')\|_{L^\infty_{x,v}}\right\} ,
	\end{align}
	where we have used the fact $\int_0^t \int_0^s \exp\left\{-e^{-\|\Phi\|_\infty} \frac{\nu_0}{2} (t-s') \right\} ds'ds$ is finite.\\
	By Lemma \ref{Kest}, $R_2$ in the case $|v| \ge R$ is bounded by 
	\begin{align}
		& C_\Phi \int_0^{t-\epsilon}\int_{t_{l'+1}'}^{t_{l'}'}\exp\left\{-e^{-\|\Phi\|_\infty}\nu_0(t-s')\right\}\int_{\mathbb{R}^3}|k_w(V(s),v')|\int_{\prod_{j=1}^{k-1}\mathcal{V}_j'}\Biggl\{\int_{\mathbb{R}^3}|k_w(V_{l'}'(s'),v'')| \|h(s')\|_{L^\infty_{x,v}}dv''\Biggr\} \nonumber\\
		&\quad  \times  \left\{\prod_{j=l'+1}^{k-1}d\sigma_j'\right\}\left\{\tilde{w}(x_{l'}',v_{l'}')d\sigma_{l'}'\right\}\left\{\prod_{j=1}^{l'-1}d\sigma_j'\right\}dv'ds'ds\nonumber\\
		& \le C_\Phi \int_0^{t}\int_{0}^{s} \exp\left\{-e^{-\|\Phi\|_\infty}\nu_0(t-s')\right\} \|h(s')\|_{L^\infty_{x,v}} \left\{\int_{\mathbb{R}^3}|k_w(V(s),v')|dv'\right\}ds'ds\nonumber\\  \label{class12}
		& \le \frac{C_\Phi}{1+R} \sup_{0\le s' \le T_0} \left\{\exp\left\{-e^{-\|\Phi\|_\infty} \frac{\nu_0}{2} (t-s') \right\}\|h(s')\|_{L^\infty_{x,v}}\right\}.
	\end{align}
	\newline
	$\mathbf{Case\ 2:}$ $|v| \le R$, $|v'| \ge 2R$, or $|v_{l'}'| \le R$, $|v''| \ge 2R$.\\
	Note that either $|v-v'| \ge R$ or $|v_{l'}'-v''| \ge R$. From \eqref{sese}, either one of the followings holds:
	\begin{align*}
		&|V(s)-v'| \ge |v-v'|-|V(s)-v| \ge R-\frac{R}{2}=\frac{R}{2},\\
		&|V_{l'}'(s')-v''| \ge |v'-v''|-|V_{l'}'(s')-v'| \ge R-\frac{R}{2}=\frac{R}{2}.
	\end{align*}
	Then we have either one of the followings:
	\begin{equation} \label{dd21}
	\begin{aligned}
		&|k_w(V(s),v')| \le e^{-\frac{R^2}{64}}|k_w(V(s),v')|e^{\frac{1}{16}|V(s)-v'|^2},\\
		&|k_w(V_{l'}'(s'),v'')| \le e^{-\frac{R^2}{64}}|k_w(V_{l'}'(s'),v'')|e^{\frac{1}{16}|V_{l'}'(s')-v''|^2}.
	\end{aligned}
	\end{equation}
	This yields from Lemma \ref{Kest},
	\begin{equation} \label{dd22}
	\begin{aligned}
		&\int_{|v'| \ge 2R} |k_w(V(s),v')|e^{\frac{1}{16}|V(s)-v'|^2}dv' < C,\\
		&\int_{|v''| \ge 2R} |k_w(V_{l'}'(s'),v'')|e^{\frac{1}{16}|V_{l'}'(s')-v''|^2}dv'' < C,
	\end{aligned}
	\end{equation}
	for some constant $C$.
	Thus we use \eqref{dd21} and \eqref{dd22} to bound $R_2$ in the case $|v| \le R, |v'|\ge 2R$ by
	\begin{align} 
		&C_\Phi \int_0^{t-\epsilon}\int_{t_{l'+1}'}^{t_{l'}'}\exp\left\{-e^{-\|\Phi\|_\infty}\nu_0(t-s')\right\}\int_{|v'| \ge 2R}|k_w(V(s),v')|\int_{\prod_{j=1}^{k-1}\mathcal{V}_j'}\Biggl\{\int_{\mathbb{R}^3}|k_w(V_{l'}'(s'),v'')|\|h(s')\|_{L^\infty_{x,v}}dv''\Biggr\} \nonumber\\
		&\quad  \times  \left\{\prod_{j=l'+1}^{k-1}d\sigma_j'\right\}\left\{\tilde{w}(x_{l'}',v_{l'}')d\sigma_{l'}'\right\}\left\{\prod_{j=1}^{l'-1}d\sigma_j'\right\}dv'ds'ds\nonumber\\
		&\le C_\Phi \int_0^t\int_0^s \exp\left\{-e^{-\|\Phi\|_\infty}\nu_0(t-s')\right\}\|h(s')\|_{L^\infty_{x,v}} \int_{|v'| \ge 2R}|k_w(V(s),v')| dv'ds'ds \nonumber\\  \label{class13}
		&\le  C_\Phi e^{-\frac{R^2}{64}}\sup_{0\le s' \le T_0} \left\{\exp\left\{-e^{-\|\Phi\|_\infty} \frac{\nu_0}{2} (t-s') \right\}\|h(s')\|_{L^\infty_{x,v}}\right\},
	\end{align}
	where we have used the fact $\int_0^t\int_0^s \exp\left\{-e^{-\|\Phi\|_\infty} \frac{\nu_0}{2} (t-s') \right\}ds'ds$ is finite.\\
	Similarly, we use \eqref{dd21} and \eqref{dd22} to bound $R_2$ in the case $|v'| \le 2R, |v''| \ge 3R$ by 
	\begin{align}
		&C_\Phi \int_0^{t-\epsilon}\int_{t_{l'+1}'}^{t_{l'}'}\exp\left\{-e^{-\|\Phi\|_\infty}\nu_0(t-s')\right\}\int_{\mathbb{R}^3}|k_w(V(s),v')| \int_{\prod_{j=1}^{k-1}\mathcal{V}_j' ,\ |v_{l'}'|\le R}\Biggl\{\int_{|v''| \ge 2R}|k_w(V_{l'}'(s'),v'')|\nonumber\\
		&\quad  \times  \|h(s')\|_{L^\infty_{x,v}}dv''\Biggr\}\left\{\prod_{j=l'+1}^{k-1}d\sigma_j'\right\}\left\{\tilde{w}(x_{l'}',v_{l'}')d\sigma_{l'}'\right\}\left\{\prod_{j=1}^{l'-1}d\sigma_j'\right\}dv'ds'ds\nonumber\\
		&\le C_\Phi e^{-\frac{R^2}{64}} \int_0^t\int_0^s \exp\left\{-e^{-\|\Phi\|_\infty}\nu_0(t-s')\right\}\|h(s')\|_{L^\infty_{x,v}}ds'ds\nonumber\\ \label{class14}
		&\le C_\Phi e^{-\frac{R^2}{64}} \sup_{0\le s' \le T_0} \left\{\exp\left\{-e^{-\|\Phi\|_\infty} \frac{\nu_0}{2} (t-s') \right\}\|h(s')\|_{L^\infty_{x,v}}\right\}.
	\end{align}
	\newline
	$\mathbf{Case\ 3:}$ $|v|\le R$, $|v'| \le 2R$, $|v_{l'}'| \le R$, $|v''| \le 2R$.\\
	Since $k_w(v,v')$ has possible integrable singularity of $\frac{1}{|v-v'|}$, we can choose smooth function $k_R(v,v')$ with compact support such that
	\begin{align} \label{kk3}
		\sup_{|v|\le 2R}\int_{|v'| \le 2R} \left|k_R(v,v')-k_w(v,v')\right|dv' \le \frac{1}{R}.
	\end{align}
	We split 
	\begin{equation} \label{kk4}
	\begin{aligned}
		k_w(V(s),v')k_w(V_{l'}'(s'),v'') &= \left\{k_w(V(s),v')-k_R(V(s),v')\right\}k_w(V_{l'}'(s'),v'')\\
		& \quad + \{k_w(V_{l'}'(s'),v'')-k_R(V_{l'}'(s'),v'')\}k_R(V(s),v')\\
		&\quad +k_R(V(s),v')k_R(V_{l'}'(s'),v'').
	\end{aligned}
	\end{equation}
	From \eqref{kk3} and \eqref{kk4}, $R_2$ in this case is bounded by 
	\begin{align}
		&\frac{C_\Phi}{R} \int_0^t \int_0^{s} \exp\left\{-e^{-\|\Phi\|_\infty}\nu_0(t-s')\right\} \|h(s')\|_{L^\infty_{x,v}}ds'ds \nonumber \\
		& +C_\Phi \int_{t_1}^t \int_{t_{l'+1}'}^{t_{l'}'} \exp\left\{-e^{-\|\Phi\|_\infty}\nu_0(t-s')\right\} \int_{|v'|\le 2R} |k_R(V(s),v')| \int_{\prod_{j=1}^{k-1}\mathcal{V}_j' ,\ |v_{l'}'|\le R} \Biggl\{\int_{|v''| \le 2R}|k_R(V_{l'}'(s'),v'')|\nonumber \\ 
		&\quad  \times  |h(s',X'_{l'}(s'),v'')|dv''\Biggr\}\left\{\prod_{j=l'+1}^{k-1}d\sigma_j'\right\}\left\{\tilde{w}(x_{l'}',v_{l'}')d\sigma_{l'}'\right\}\left\{\prod_{j=1}^{l'-1}d\sigma_j'\right\}dv'ds'ds \nonumber \\ \label{class15}
		&\le \frac{C_\Phi}{R} \exp\left\{-e^{-\|\Phi\|_\infty} \frac{\nu_0}{2} t \right\}\sup_{0 \le s' \le T_0}\left\{\exp\left\{e^{-\|\Phi\|_\infty} \frac{\nu_0}{2} s' \right\}\|h(s')\|_{L^\infty_{x,v}}\right\} \\  
		&\quad + C_{R,\Phi}\int_{t_1}^t \int_{t_{l'+1}'}^{t_{l'}'} \exp\left\{-e^{-\|\Phi\|_\infty}\nu_0(t-s')\right\} \int_{|v'|\le 2R} \int_{\prod_{j=1}^{k-1}\mathcal{V}_j' ,\ |v_{l'}'|\le R} \Biggl\{\int_{|v''| \le 2R}\nonumber |h(s',X'_{l'}(s'),v'')|dv''\Biggr\}\nonumber\\ 
		&\quad \quad \times  \left\{\prod_{j=l'+1}^{k-1}d\sigma_j'\right\}\left\{\tilde{w}(x_{l'}',v_{l'}')d\sigma_{l'}'\right\}\left\{\prod_{j=1}^{l'-1}d\sigma_j'\right\}dv'ds'ds\nonumber\\
		&=: R_{21}+R_{22}, \nonumber
	\end{align}
	where we have used the fact $|k_R(V(s),v')||k_R(V'_{l'}(s'),v'')| \le C_R$.\\
	In the term $R_2$, we recall that $X'_{l'}(s') = X(s';t_{l'}',x_{l'}',v_{l'}')$.
	Since the potential is time dependent, we have
	\begin{align*}
		X(s';t_{l'}',x_{l'}',v_{l'}') = X(s'-t_{l'}'+T_0;T_0,x_{l'}',v_{l'}')
	\end{align*}
	for all $0\le s' \le t_{l'}' \le T_0$.\\
	By Lemma \ref{cpl}, the term $R_{22}$ becomes
	\begin{align} 
		&C_{R,\Phi}\sum_{i_1}^{M_1}\sum_{I_2}^{(M_2)^3}\sum_{I_3}^{(M_3)^3} \int_0^t \mathbf{1}_{\{x_{l'}'\in \mathcal{P}_{I_2}^{\Omega}\}}(s)\int_{t_{l'+1}'}^{t_{l'}'}\mathbf{1}_{\mathcal{P}_{i_1}^{T_0}}(s'-t_{l'}'+T_0) \exp\left\{-e^{-\|\Phi\|_\infty}\nu_0(t-s')\right\} \nonumber \\ 
		& \times  \int_{|v'|\le 2R} \int_{\prod_{j=1}^{k-1}\mathcal{V}_j' ,\ |v_{l'}'|\le R}\mathbf{1}_{\mathcal{P}_{I_3}^v}(v'_{l'})\int_{|v''| \le 2R}|h(s',X(s'-t_{l'}'+T_0;T_0,x_{l'}',v_{l'}'),v'')|dv''\nonumber\\\label{ltlt2}
		&\times  \left\{\prod_{j=l'+1}^{k-1}d\sigma_j'\right\}\left\{\tilde{w}(x_{l'}',v_{l'}')d\sigma_{l'}'\right\}\left\{\prod_{j=1}^{l'-1}d\sigma_j'\right\}dv'ds'ds.
	\end{align}
	From Lemma \ref{cpl}, we have the following partitions:
	\begin{align*}
		&\left\{(s'-t_{l'}'+T_0,x_{l'}',v_{l'}')\in \mathcal{P}_{i_1}^{T_0}\cross \mathcal{P}_{I_2}^{\Omega}\cross \mathcal{P}_{I_3}^{v} : \det\left(\frac{dX}{dv_{l'}'}(s'-t_{l'}'+T_0;T_0,x_{l'}',v_{l'}')\right)=0\right\}\\
		& \subset \bigcup_{j=1}^3 \left\{(s'-t_{l'}'+T_0,x_{l'}',v_{l'}')\in \mathcal{P}_{i_1}^{T_0}\cross \mathcal{P}_{I_2}^{\Omega}\cross \mathcal{P}_{I_3}^{v} :s'-t_{l'}'+T_0\in \left(t_{j,i_1,I_2,I_3}-\frac{\tilde{\epsilon}}{4M_1},t_{j,i_1,I_2,I_3}+\frac{\tilde{\epsilon}}{4M_1}\right)\right\}.
	\end{align*}
	Thus for each $i_1$,$I_2$, and $I_3$, we split $\mathbf{1}_{\mathcal{P}_{i_1}^{T_0}}(s'-t_{l'}'+T_0)$ as
	\begin{align} \label{fst2}
		&\mathbf{1}_{\mathcal{P}_{i_1}^{T_0}}(s'-t_{l'}'+T_0)\mathbf{1}_{\cup_{j=1}^3(t_{j,i_1,I_2,I_3}-\frac{\tilde{\epsilon}}{4M_1},t_{j,i_1,I_2,I_3}+\frac{\tilde{\epsilon}}{4M_1})}(s'-t_{l'}'+T_0)\\ \label{sst2}
		& +\mathbf{1}_{\mathcal{P}_{i_1}^{T_0}}(s'-t_{l'}'+T_0)\left\{1-\mathbf{1}_{\cup_{j=1}^3(t_{j,i_1,I_2,I_3}-\frac{\tilde{\epsilon}}{4M_1},t_{j,i_1,I_2,I_3}+\frac{\tilde{\epsilon}}{4M_1})}(s'-t_{l'}'+T_0)\right\}.
	\end{align}
	$\mathbf{Case\ 3 \ (\romannumeral 1):}$ The integration \eqref{ltlt2} corresponding to \eqref{fst2} is bounded by
	\begin{align}
		&C_{R,\Phi}\sum_{i_1}^{M_1}\sum_{I_2}^{(M_2)^3}\sum_{I_3}^{(M_3)^3} \sum_{j=1}^3 \int_0^t \mathbf{1}_{\{x_{l'}'\in \mathcal{P}_{I_2}^{\Omega}\}}(s)\int_{t_{l'+1}'}^{t_{l'}'}\mathbf{1}_{\mathcal{P}_{i_1}^{T_0}}(s'-t_{l'}'+T_0) \nonumber \\
		& \times \mathbf{1}_{(t_{j,i_1,I_2,I_3}-\frac{\tilde{\epsilon}}{4M_1},t_{j,i_1,I_2,I_3}+\frac{\tilde{\epsilon}}{4M_1})}(s'-t_{l'}'+T_0) \exp\left\{-e^{-\|\Phi\|_\infty} \nu_0 (t-s') \right\} \nonumber \\ 
		&\times \int_{|v'|\le 2R} \int_{\prod_{j=1}^{k-1}\mathcal{V}_j' ,\ |v_{l'}'|\le R}\mathbf{1}_{\mathcal{P}_{I_3}^v}(v_{l'}')\int_{|v''|\le 2R}|h(s',X(s'-t_{l'}'+T_0;T_0,x_{l'}',v_{l'}'),v'')|dv'' \nonumber\\\label{mlt2}
		&\times  \left\{\prod_{j=l'+1}^{k-1}d\sigma_j'\right\}\left\{\tilde{w}(x_{l'}',v_{l'}')d\sigma_{l'}'\right\}\left\{\prod_{j=1}^{l'-1}d\sigma_j'\right\}dv'ds'ds.
	\end{align}
	We split
	\begin{align*}
		\exp\left\{-e^{-\|\Phi\|_\infty} \nu_0 (t-s') \right\} &= \exp\left\{-e^{-\|\Phi\|_\infty} \frac{\nu_0}{2} (t-s) \right\}\exp\left\{-e^{-\|\Phi\|_\infty} \frac{\nu_0}{2} (s-s') \right\}\\
	& \quad \times  \exp\left\{-e^{-\|\Phi\|_\infty} \frac{\nu_0}{2} (t-s') \right\}.
	\end{align*}
	and we can bound the integration \eqref{mlt2} by
	\begin{align*}
		&C_{R,\Phi}\sum_{i_1}^{M_1}\sum_{I_2}^{(M_2)^3}\sum_{I_3}^{(M_3)^3} \sum_{j=1}^3 \int_0^t \mathbf{1}_{\{x_{l'}'\in \mathcal{P}_{I_2}^{\Omega}\}}(s)\exp\left\{-e^{-\|\Phi\|_\infty} \frac{\nu_0}{2} (t-s) \right\}\\
		& \times \underbrace{\int_{t_{l'+1}'}^{t_{l'}'}\mathbf{1}_{\mathcal{P}_{i_1}^{T_0}}(s'-t_{l'}'+T_0)\mathbf{1}_{(t_{j,i_1,I_2,I_3}-\frac{\tilde{\epsilon}}{4M_1},t_{j,i_1,I_2,I_3}+\frac{\tilde{\epsilon}}{4M_1})}(s'-t_{l'}'+T_0) \exp\left\{-e^{-\|\Phi\|_\infty} \frac{\nu_0}{2} (t_{l'}'-s') \right\}}_{(*2)}\\
		&\times \int_{|v'|\le 2R} \int_{\prod_{j=1}^{k-1}\mathcal{V}_j' ,\ |v_{l'}'|\le R}\mathbf{1}_{\mathcal{P}_{I_3}^v}(v_{l'}')\int_{|v''|\le 2R}\exp\left\{-e^{-\|\Phi\|_\infty} \frac{\nu_0}{2} (t-s') \right\}\|h(s')\|_{L^\infty_{x,v}}dv''\\
		&\times  \left\{\prod_{j=l'+1}^{k-1}d\sigma_j'\right\}\left\{\tilde{w}(x_{l'}',v_{l'}')d\sigma_{l'}'\right\}\left\{\prod_{j=1}^{l'-1}d\sigma_j'\right\}dv'ds'ds.
	\end{align*}
	Here, $(*2)$ is bounded by
	\begin{align}
		&\int_{t_{l'+1}'}^{t_{l'}'} \mathbf{1}_{\mathcal{P}_{i_1}^{T_0}}(s'-t_{l'}'+T_0)\mathbf{1}_{(t_{j,i_1,I_2,I_3}-\frac{\tilde{\epsilon}}{4M_1},t_{j,i_1,I_2,I_3}+\frac{\tilde{\epsilon}}{4M_1})}(s'-t_{l'}'+T_0) \exp\left\{-e^{-\|\Phi\|_\infty} \frac{\nu_0}{2} (t_{l'}'-s') \right\}ds'\nonumber\\
		& \le \int_{t_{l'}'-T_0+t_{j,i_1,I_2,I_3}-\frac{\tilde{\epsilon}}{4M_1}}^{t_{l'}'-T_0+t_{j,i_1,I_2,I_3}+\frac{\tilde{\epsilon}}{4M_1}} \exp\left\{-e^{-\|\Phi\|_\infty} \frac{\nu_0}{2} (t_{l'}'-s') \right\}ds'\nonumber\\\label{ppp3}
		& \le \frac{\tilde{\epsilon}}{2M_1}.
	\end{align}
	From the partition of the time interval $[0,T_0]$ and velocity domain $[-4R,4R]^3$ in Lemma \ref{cpl}, we have
	\begin{equation} \label{ppp4}
	\begin{aligned}
		&\sum_{I_2}^{(M_2)^3}\mathbf{1}_{\{x_{l'}'\in \mathcal{P}_{I_2}^{\Omega}\}}(s) \le \mathbf{1}_{\{0\le s \le T_0\}}(s),\\
		&\sum_{I_3}^{(M_3)^3} \mathbf{1}_{\mathcal{P}_{I_3}^v}(v_{l'}')\mathbf{1}_{\{|v_{l'}'|\le R\}}(v_{l'}') = \mathbf{1}_{\{|v_{l'}'|\le R\}}(v_{l'}').
	\end{aligned}
	\end{equation}
	Using \eqref{ppp3} and \eqref{ppp4}, \eqref{mlt2} is bounded by
	\begin{align}
		& C_{R,\Phi} \sup_{0\le s' \le T_0} \left\{\exp\left\{-e^{-\|\Phi\|_\infty} \frac{\nu_0}{2} (t-s') \right\}\|h(s')\|_{L^\infty_{x,v}}\right\} \sum_{i_1}^{M_1} \sum_{I_2}^{(M_2)^3} \int_0^t \mathbf{1}_{\{x_{l'}' \in \mathcal{P}_{I_2}^\Omega\}}(s) \exp\left\{-e^{-\|\Phi\|_\infty} \frac{\nu_0}{2} (t-s) \right\}\nonumber\\
		&\quad \times \int_{t_{l'+1}'}^{t_{l'}'} \mathbf{1}_{\mathcal{P}_{i_1}^{T_0}}(s'-t_{l'}'+T_0)\mathbf{1}_{(t_{j,i_1,I_2,I_3}-\frac{\tilde{\epsilon}}{4M_1},t_{j,i_1,I_2,I_3}+\frac{\tilde{\epsilon}}{4M_1})}(s'-t_{l'}'+T_0) \exp\left\{-e^{-\|\Phi\|_\infty} \frac{\nu_0}{2} (t_{l'}'-s') \right\}ds'ds\nonumber\\ \label{class16}
		&\le \tilde{\epsilon} \ C_{R,\Phi} \sup_{0\le s' \le T_0} \left\{\exp\left\{-e^{-\|\Phi\|_\infty} \frac{\nu_0}{2} (t-s') \right\}\|h(s')\|_{L^\infty_{x,v}}\right\}.
	\end{align}
	\newline
	$\mathbf{Case\ 3 \ (\romannumeral 2):}$ The integration \eqref{ltlt2} corresponding to \eqref{sst2} is bounded by
	\begin{align}
		&C_{R,\Phi}\sum_{i_1}^{M_1}\sum_{I_2}^{(M_2)^3}\sum_{I_3}^{(M_3)^3} \int_0^t \mathbf{1}_{\{x_{l'}'\in \mathcal{P}_{I_2}^{\Omega}\}}(s)\int_{t_{l'+1}'}^{t_{l'}'}\mathbf{1}_{\mathcal{P}_{i_1}^{T_0}}(s'-t_{l'}'+T_0) \nonumber \\
		& \times \left\{1-\mathbf{1}_{\cup_{j=1}^3(t_{j,i_1,I_2,I_3}-\frac{\tilde{\epsilon}}{4M_1},t_{j,i_1,I_2,I_3}+\frac{\tilde{\epsilon}}{4M_1})}(s'-t_{l'}'+T_0)\right\}\exp\left\{-e^{-\|\Phi\|_\infty} \nu_0 (t-s') \right\} \int_{|v'|\le 2R} \nonumber \\ 
		&\times \underbrace{\int_{\prod_{j=1}^{l'}\mathcal{V}_j' ,\ |v_{l'}'|\le R}\mathbf{1}_{\mathcal{P}_{I_3}^v}(v_{l'}')\int_{|v''|\le 2R}|h(s',X(s'-t_{l'}'+T_0;T_0,x_{l'}',v_{l'}'),v'')|dv''\left\{\tilde{w}(x_{l'}',v_{l'}')d\sigma_{l'}'\right\}\left\{\prod_{j=1}^{l'-1}d\sigma_j'\right\}}_{(\#2)}\nonumber\\ \label{mmlt2}
		&\times dv'ds'ds.
	\end{align}
	By Lemma \ref{cpl}, we have made a change of variables $v' \rightarrow y:=X(s'-t_{l'}'+T_0;T_0,x_{l'}',v_{l'}')$ satisfying 
	\begin{align*}
		\det\left(\frac{dX}{dv_{l'}'}(s'-t_{l'}'+T_0;T_0,x_{l'}',v_{l'}')\right)> \delta_*
	\end{align*}
	and the term $(\#2)$ is bounded by 
	\begin{align*}
		&\int_{\prod_{j=1}^{l'-1}\mathcal{V}_j' ,\ |v_{l'}'|\le R}\mathbf{1}_{\mathcal{P}_{I_3}^v}(v_{l'}')\int_{|v''|\le 2R}|h(s',X(s'-t_{l'}'+T_0;T_0,x_{l'}',v_{l'}'),v'')|dv''\left\{\tilde{w}(x_{l'}',v_{l'}')d\sigma_{l'}'\right\}\left\{\prod_{j=1}^{l'-1}d\sigma_j'\right\}\\
		&\le \frac{C_{R,\Phi}}{\delta_*}\int_{\Omega} \int_{|v''|\le 2R}|h(s',y,v'')|dv''dy\\
		&\le \frac{C_{R,\Phi}}{\delta_*}\left(\int_{\Omega}\int_{|v''|\le 2R}w(y,v'')^2dv''dy\right)^{\frac{1}{2}}\|f(s')\|_{L^2_{x,v}}\\
		&\le \frac{C_{R,\Phi}}{\delta_*} \|f(s')\|_{L^2_{x,v}},
	\end{align*}
	where we have used the Cauchy-Schwarz inequality.
	Hence \eqref{mmlt2} is bounded by
	\begin{align}
		&\frac{C_{R,\Phi,M_1,M_2,M_3}}{\delta_*} \int_0^t\int_0^s\exp\left\{-e^{-\|\Phi\|_\infty} \nu_0 (t-s') \right\} \|f(s')\|_{L^2_{x,v}}ds'ds\nonumber\\ \label{class17}
		& \le \frac{C_{R,\Phi,M_1,M_2,M_3}}{\delta_*} \int_0^{T_0}\|f(s')\|_{L^2_{x,v}}ds'.
	\end{align}
	Combining the bounds \eqref{class11}, \eqref{class12}, \eqref{class13}, \eqref{class14}, \eqref{class15}, \eqref{class16}, \eqref{class17} and summing over $1\le l' \le k(\epsilon)-1$, we can bound $R_2$ in \eqref{RMC} by
	\begin{equation} \label{dd29}
	\begin{aligned}
		&\left(\frac{C_{\epsilon,\Phi}}{R}+C_{\epsilon,\Phi} e^{-\frac{R^2}{64}}+ \tilde{\epsilon} \ C_{\epsilon,R,\Phi} \right) \sup_{0\le s' \le T_0} \left\{\exp\left\{-e^{-\|\Phi\|_\infty} \frac{\nu_0}{2} (t-s') \right\}\|h(s')\|_{L^\infty_{x,v}}\right\}\\
		& +\frac{C_{\epsilon,R,\Phi,M_1,M_2,M_3}}{\delta_*} \int_0^{T_0}\|f(s')\|_{L^2_{x,v}}ds'.
	\end{aligned}
	\end{equation}
	\newline
	Third, let us estimate the term $R_3$ in \eqref{RMC} : 
	\begin{align} 
		& \sum_{l=1}^{k-1} \int_{t_{l+1}}^{t_l}\int_{t_1'}^{s-\epsilon} \frac{\exp{-\int_{t_1}^t e^{-\Phi(X(\tau))}\nu(V(\tau))d\tau}}{\tilde{w}(x_1,V(t_1))}\int_{\prod_{j=1}^{k-1}\mathcal{V}_j}  \int_{\mathbb{R}^3}e^{-\Phi(X_l(s))} k_w(V_l(s),v') \exp{-\int_{s'}^s e^{-\Phi(X'(\tau'))}\nu(V'(\tau'))d\tau'}\nonumber\\
		&\quad \times \int_{\mathbb{R}^3}e^{-\Phi(X'(s')))}k_w(V'(s'),v'')h(s',X'(s'),v'') dv''dv'd\Sigma_l(s)ds'ds \nonumber \\
		&\le \sum_{l=1}^{k-1} \int_{t_{l+1}}^{t_l}\int_{t_1'}^{s-\epsilon}\frac{\exp\left\{-e^{-\|\Phi\|_\infty}\nu_0(t-s')\right\}}{\tilde{w}(x_1,V(t_1))}\int_{\prod_{j=1}^{k-1}\mathcal{V}_j}  \int_{\mathbb{R}^3}|k_w(V_l(s),v')| \int_{\mathbb{R}^3}|k_w(V'(s'),v'')|\nonumber \\ \label{TTMC} 
		&\quad \times|h(s',X'(s'),v'')| dv''dv'\left\{\prod_{j=l+1}^{k-1}d\sigma_j\right\}\left\{\tilde{w}(x_{l},v_{l})d\sigma_{l}\right\}\left\{\prod_{j=1}^{l-1}d\sigma_j\right\}ds'ds.
	\end{align}
	Fix $l$. We will divide this term into 3 cases.\\
	\newline
	$\mathbf{Case\ 1:}$ $|v_l| \ge R$ with $R \gg 2\sqrt{2\|\Phi\|_\infty}$.\\
	By \eqref{sese}, we get
	\begin{align*}
		|V_l(s)|\ge |v_l|-\sqrt{2\|\Phi\|_\infty}\ge \frac{R}{2}.
	\end{align*}
	From Lemma \ref{Kest}, we have
	\begin{align*}
		\int_{\mathbb{R}^3} \int_{\mathbb{R}^3} |k_w(V_l(s),v')| |k_w(V'(s'),v'')|dv''dv' \le \frac{C_\Phi}{1+R}.
	\end{align*}
	Then $R_3$ in this case is bounded by
	\begin{align}
		&\frac{C_\Phi}{1+R}\int_{t_{l+1}}^{t_l} \int_{t_1'}^{s-\epsilon} \exp\left\{-e^{-\|\Phi\|_\infty}\nu_0(t-s')\right\} \|h(s')\|_{L^\infty_{x,v}}ds'ds\nonumber\\ \label{class21}
		& \le \frac{C_\Phi}{1+R}\sup_{0\le s' \le T_0} \left\{\exp\left\{-e^{-\|\Phi\|_\infty} \frac{\nu_0}{2} (t-s') \right\}\|h(s')\|_{L^\infty_{x,v}}\right\},
	\end{align}
	where we have used the fact $\int_0^s \exp\left\{-e^{-\|\Phi\|_\infty} \frac{\nu_0}{2} (t-s') \right\} ds'ds$ is finite.\\
	\newline
	$\mathbf{Case\ 2:}$ $|v_l| \le R$, $|v'| \ge 2R$, or $|v'| \le 2R$, $|v''| \ge 3R$.\\
	Note that either $|v_l-v'| \ge R$ or $|v'-v''| \ge R$. From \eqref{sese}, either one of the followings holds:
	\begin{align*}
		&|V_l(s)-v'| \ge |v_l-v'|-|V_l(s)-v_l| \ge R-\frac{R}{2}=\frac{R}{2},\\
		&|V'(s')-v''| \ge |v'-v''|-|V'(s')-v'| \ge R-\frac{R}{2}=\frac{R}{2}.
	\end{align*}
	Then we have either one of the followings:
	\begin{equation} \label{dd32}
	\begin{aligned}
		&|k_w(V_l(s),v')| \le e^{-\frac{R^2}{64}}|k_w(V_l(s),v')|e^{\frac{1}{16}|V_l(s)-v'|^2},\\
		&|k_w(V'(s'),v'')| \le e^{-\frac{R^2}{64}}|k_w(V'(s'),v'')|e^{\frac{1}{16}|V'(s')-v''|^2}.
	\end{aligned}
	\end{equation}
	This yields from Lemma \ref{Kest},
	\begin{equation} \label{dd33}
	\begin{aligned}
		&\int_{|v'| \ge 2R} |k_w(V_l(s),v')|e^{\frac{1}{16}|V_l(s)-v'|^2}dv' < C,\\
		&\int_{|v''| \ge 3R} |k_w(V'(s'),v'')|e^{\frac{1}{16}|V'(s')-v''|^2}dv'' < C
	\end{aligned}
	\end{equation}
	for some constant $C$.
	Thus we use \eqref{dd32} and \eqref{dd33} to bound $R_3$ in this case by
	\begin{align}
		&C_\Phi \int_0^t\int_0^s \exp\left\{-e^{-\|\Phi\|_\infty}\nu_0(t-s')\right\} \int_{\prod_{j=1}^{k-1}\mathcal{V}_j,\ |v_l|\le R} \int_{|v'|\ge 2R} \int_{\mathbb{R}^3}|k_w(V_l(s),v')| |k_w(V'(s'),v'')|\|h(s')\|_{L^\infty_{x,v}}\nonumber\\
		& \quad \times dv''dv'\left\{\prod_{j=l+1}^{k-1}d\sigma_j\right\}\left\{\tilde{w}(x_{l},v_{l})d\sigma_{l}\right\}\left\{\prod_{j=1}^{l-1}d\sigma_j\right\}ds'ds\nonumber\\
		&+ C_\Phi \int_0^t\int_0^s \exp\left\{-e^{-\|\Phi\|_\infty}\nu_0(t-s')\right\}\int_{\prod_{j=1}^{k-1}\mathcal{V}_j}  \int_{|v'|\le 2R} \int_{|v''| \ge 3R}|k_w(V_l(s),v')| |k_w(V'(s'),v'')|\|h(s')\|_{L^\infty_{x,v}}\nonumber\\
		&\quad \times dv''dv'\left\{\prod_{j=l+1}^{k-1}d\sigma_j\right\}\left\{\tilde{w}(x_l,v_{l})d\sigma_{l}\right\}\left\{\prod_{j=1}^{l-1}d\sigma_j\right\}ds'ds\nonumber\\
		&\le C_\Phi e^{-\frac{R^2}{64}}\int_0^t\int_0^s \exp\left\{-e^{-\|\Phi\|_\infty}\nu_0(t-s')\right\}\|h(s')\|_{L^\infty_{x,v}}ds'ds\nonumber\\ \label{class22}
		&\le C_\Phi e^{-\frac{R^2}{64}}\sup_{0\le s' \le T_0} \left\{\exp\left\{-e^{-\|\Phi\|_\infty} \frac{\nu_0}{2} (t-s') \right\}\|h(s')\|_{L^\infty_{x,v}}\right\}.
	\end{align}
	\newline
	$\mathbf{Case\ 3:}$ $|v_l|\le R$, $|v'| \le 2R$, $|v''| \le 3R$.\\
	Since $k_w(v,v')$ has possible integrable singularity of $\frac{1}{|v-v'|}$, we can choose smooth function $k_R(v,v')$ with compact support such that
	\begin{align} \label{kk31}
		\sup_{|v|\le 3R}\int_{|v'| \le 3R} \left|k_R(v,v')-k_w(v,v')\right|dv' \le \frac{1}{R}.
	\end{align}
	We split 
	\begin{equation} \label{kk32}
	\begin{aligned}
		k_w(V_l(s),v')k_w(V'(s'),v'') &= \left\{k_w(V_l(s),v')-k_R(V_l(s),v')\right\}k_w(V'(s'),v'') \\
		&\quad + \{k_w(V'(s'),v'')-k_R(V'(s'),v'')\}k_R(V_l(s),v')\\
		&\quad +k_R(V_l(s),v')k_R(V'(s'),v'').
	\end{aligned}
	\end{equation}
	From \eqref{kk31} and \eqref{kk32}, $R_3$ in this case is bounded by 
	\begin{align}
		&\frac{C_\Phi}{R} \int_0^t \int_0^{s-\epsilon} \exp\left\{-e^{-\|\Phi\|_\infty}\nu_0(t-s')\right\}\|h(s')\|_{L^\infty_{x,v}}ds'ds \nonumber \\ 
		& + \int_0^t \int_0^{s-\epsilon} \frac{\exp\left\{-e^{-\|\Phi\|_\infty}\nu_0(t-s')\right\}}{\tilde{w}(x_1,V(t_1))}\int_{\prod_{j=1}^{k-1}\mathcal{V}_j, |v_l| \le R} \int_{|v'|\le 2R, |v''|\le 3R}|k_R(V_l(s),v')||k_R(V'(s'),v'')| \nonumber \\
		& \quad \times |h(s',X'(s'),v'')|dv''dv'\left\{\prod_{j=l+1}^{k-1}d\sigma_j\right\}\left\{\tilde{w}(x_l,v_{l})d\sigma_{l}\right\}\left\{\prod_{j=1}^{l-1}d\sigma_j\right\}ds'ds \nonumber \\ \label{class23}
		&\le \frac{C_\Phi}{R} \sup_{0\le s' \le T_0} \left\{\exp\left\{-e^{-\|\Phi\|_\infty} \frac{\nu_0}{2} (t-s') \right\}\|h(s')\|_{L^\infty_{x,v}}\right\} \\ 
		&\quad + C_{R,\Phi}\int_0^t \int_0^{s-\epsilon} \exp\left\{-e^{-\|\Phi\|_\infty}\nu_0(t-s')\right\} \int_{\prod_{j=1}^{k-1}\mathcal{V}_j, |v_l| \le R}\int_{|v'|\le 2R, |v''|\le 3R}|h(s',X'(s'),v'')|dv''dv' \nonumber\\ 
		&\qquad \times \left\{\prod_{j=l+1}^{k-1}d\sigma_j\right\}\left\{\tilde{w}(x_l,v_{l})d\sigma_{l}\right\}\left\{\prod_{j=1}^{l-1}d\sigma_j\right\}ds'ds \nonumber\\
		&=: R_{31}+R_{32}, \nonumber 
	\end{align}
	where we have used the fact $|k_R(V_l(s),v')||k_R(V'(s'),v'')| \le C_R$.\\
	In the term $R_3$, we recall that $X'(s') = X(s';s,X(s;t_l,x_l,v_l),v')$.
	Since the potential is time dependent, we have
	\begin{align*}
		X(s';s,X(s;t_l,x_l,v_l),v') = X(s'-s+T_0;T_0,X(T_0;T_0+t_l-s,x_l,v_l),v')
	\end{align*}
	for all $0\le s' \le s \le t_l \le T_0$.\\
	By Lemma \ref{cpl}, the term $R_{32}$ becomes
	\begin{align} 
		&C_{R,\Phi}\sum_{i_1}^{M_1}\sum_{I_2}^{(M_2)^3}\sum_{I_3}^{(M_3)^3} \int_0^t \mathbf{1}_{\{X(T_0;T_0+t_l-s,x_l,v_l)\in \mathcal{P}_{I_2}^{\Omega}\}}(s)\int_0^s\mathbf{1}_{\mathcal{P}_{i_1}^{T_0}}(s'-s+T_0) \exp\left\{-e^{-\|\Phi\|_\infty}\nu_0(t-s')\right\} \nonumber \\ 
		& \times \int_{\prod_{j=1}^{k-1}\mathcal{V}_j,|v_l| \le R}\int_{|v'|\le 2R}\mathbf{1}_{\mathcal{P}_{I_3}^v}(v')\int_{|v''|\le3R}|h(s',X(s'-s+T_0;T_0,X(T_0;T_0+t_l-s,x_l,v_l),v'),v'')|dv''dv' \nonumber\\ \label{ltlt3}
		& \times \left\{\prod_{j=l+1}^{k-1}d\sigma_j\right\}\left\{\tilde{w}(x_l,v_{l})d\sigma_{l}\right\}\left\{\prod_{j=1}^{l-1}d\sigma_j\right\}ds'ds.
	\end{align}
	From Lemma \ref{cpl}, we have the following partitions:
	\begin{align*}
		&\biggl\{(s'-s+T_0,X(T_0;T_0+t_l-s,x_l,v_l),v')\in \mathcal{P}_{i_1}^{T_0}\cross \mathcal{P}_{I_2}^{\Omega}\cross \mathcal{P}_{I_3}^{v} \\
		& \quad: \det\left(\frac{dX}{dv'}(s'-s+T_0;T_0,X(T_0;T_0+t_l-s,x_l,v_l),v')\right)=0\biggr\}\\
		& \subset \bigcup_{j=1}^3 \biggl\{(s'-s+T_0,X(T_0;T_0+t_l-s,x_l,v_l),v')\in \mathcal{P}_{i_1}^{T_0}\cross \mathcal{P}_{I_2}^{\Omega}\cross \mathcal{P}_{I_3}^{v} \\
		&\qquad :s'-s+T_0\in \left(t_{j,i_1,I_2,I_3}-\frac{\tilde{\epsilon}}{4M_1},t_{j,i_1,I_2,I_3}+\frac{\tilde{\epsilon}}{4M_1}\right)\biggr\}.
	\end{align*}
	Thus for each $i_1$,$I_2$, and $I_3$, we split $\mathbf{1}_{\mathcal{P}_{i_1}^{T_0}}(s'-s+T_0)$ as
	\begin{align} \label{fst3}
		&\mathbf{1}_{\mathcal{P}_{i_1}^{T_0}}(s'-s+T_0)\mathbf{1}_{\cup_{j=1}^3(t_{j,i_1,I_2,I_3}-\frac{\tilde{\epsilon}}{4M_1},t_{j,i_1,I_2,I_3}+\frac{\tilde{\epsilon}}{4M_1})}(s'-s+T_0)\\ \label{sst3}
		& +\mathbf{1}_{\mathcal{P}_{i_1}^{T_0}}(s'-s+T_0)\left\{1-\mathbf{1}_{\cup_{j=1}^3(t_{j,i_1,I_2,I_3}-\frac{\tilde{\epsilon}}{4M_1},t_{j,i_1,I_2,I_3}+\frac{\tilde{\epsilon}}{4M_1})}(s'-s+T_0)\right\}.
	\end{align}
	$\mathbf{Case\ 3 \ (\romannumeral 1):}$ The integration \eqref{ltlt3} corresponding to \eqref{fst3} is bounded by
	\begin{align}
		&C_{R,\Phi}\sum_{i_1}^{M_1}\sum_{I_2}^{(M_2)^3}\sum_{I_3}^{(M_3)^3} \sum_{j=1}^3\int_0^t \mathbf{1}_{\{X(T_0;T_0+t_l-s,x_l,v_l)\in \mathcal{P}_{I_2}^{\Omega}\}}(s)\int_0^s\mathbf{1}_{\mathcal{P}_{i_1}^{T_0}}(s'-s+T_0)  \nonumber \\ 
		& \times \mathbf{1}_{(t_{j,i_1,I_2,I_3}-\frac{\tilde{\epsilon}}{4M_1},t_{j,i_1,I_2,I_3}+\frac{\tilde{\epsilon}}{4M_1})}(s'-s+T_0)\exp\left\{-e^{-\|\Phi\|_\infty}\nu_0(t-s')\right\}\int_{\prod_{j=1}^{k-1}\mathcal{V}_j,|v_l| \le R}\int_{|v'|\le 2R}\mathbf{1}_{\mathcal{P}_{I_3}^v}(v')\nonumber\\ 
		& \times \int_{|v''|\le3R}|h(s',X(s'-s+T_0;T_0,X(T_0;T_0+t_l-s,x_l,v_l),v'),v'')|dv''dv' \left\{\prod_{j=l+1}^{k-1}d\sigma_j\right\}\left\{\tilde{w}(x_l,v_{l})d\sigma_{l}\right\}\nonumber \\ \label{mlt3}
		& \times \left\{\prod_{j=1}^{l-1}d\sigma_j\right\}ds'ds.
	\end{align}
	We split
	\begin{align*}
		\exp\left\{-e^{-\|\Phi\|_\infty} \nu_0 (t-s') \right\} &= \exp\left\{-e^{-\|\Phi\|_\infty} \frac{\nu_0}{2} (t-s) \right\}\exp\left\{-e^{-\|\Phi\|_\infty} \frac{\nu_0}{2} (s-s') \right\}\\
	& \quad \times  \exp\left\{-e^{-\|\Phi\|_\infty} \frac{\nu_0}{2} (t-s') \right\}.
	\end{align*}
	and we can bound the integration \eqref{mlt3} by
	\begin{align*}
		&C_{R,\Phi}\sum_{i_1}^{M_1}\sum_{I_2}^{(M_2)^3}\sum_{I_3}^{(M_3)^3} \sum_{j=1}^3 \int_0^t \mathbf{1}_{\{X(T_0;T_0+t_l-s,x_l,v_l)\in \mathcal{P}_{I_2}^{\Omega}\}}(s)\exp\left\{-e^{-\|\Phi\|_\infty} \frac{\nu_0}{2} (t-s) \right\}\\
		& \times \underbrace{\int_0^s\mathbf{1}_{\mathcal{P}_{i_1}^{T_0}}(s'-s+T_0)\mathbf{1}_{(t_{j,i_1,I_2,I_3}-\frac{\tilde{\epsilon}}{4M_1},t_{j,i_1,I_2,I_3}+\frac{\tilde{\epsilon}}{4M_1})}(s'-s+T_0) \exp\left\{-e^{-\|\Phi\|_\infty} \frac{\nu_0}{2} (s-s') \right\}}_{(*3)}\\
		&\times \int_{\prod_{j=1}^{k-1}\mathcal{V}_j,|v_l| \le R}\int_{|v'|\le 2R}\mathbf{1}_{\mathcal{P}_{I_3}^v}(v')\int_{|v''|\le 3R}\exp\left\{-e^{-\|\Phi\|_\infty} \frac{\nu_0}{2} (t-s') \right\}\|h(s')\|_{L^\infty_{x,v}}dv''dv'\left\{\prod_{j=l+1}^{k-1}d\sigma_j\right\}\\
		&\times\left\{\tilde{w}(x_l,v_{l})d\sigma_{l}\right\}\left\{\prod_{j=1}^{l-1}d\sigma_j\right\}ds'ds.
	\end{align*}
	Here, $(*3)$ is bounded by
	\begin{align}
		&\int_0^s \mathbf{1}_{\mathcal{P}_{i_1}^{T_0}}(s'-s+T_0)\mathbf{1}_{(t_{j,i_1,I_2,I_3}-\frac{\tilde{\epsilon}}{4M_1},t_{j,i_1,I_2,I_3}+\frac{\tilde{\epsilon}}{4M_1})}(s'-s+T_0) \exp\left\{-e^{-\|\Phi\|_\infty} \frac{\nu_0}{2} (s-s') \right\}ds'\nonumber\\ \label{ppp5}
		& \le \frac{\tilde{\epsilon}}{2M_1}.
	\end{align}
	From the partition of the time interval $[0,T_0]$ and velocity domain $[-4R,4R]^3$ in Lemma \ref{cpl}, we have
	\begin{equation} \label{ppp6}
	\begin{aligned}
		&\sum_{I_2}^{(M_2)^3}\mathbf{1}_{\{X(T_0;T_0+t_l-s,x_l,v_l)\in \mathcal{P}_{I_2}^{\Omega}\}}(s) \le \mathbf{1}_{\{0\le s \le T_0\}}(s),\\
		&\sum_{I_3}^{(M_3)^3} \mathbf{1}_{\mathcal{P}_{I_3}^v}(v')\mathbf{1}_{\{|v'|\le 2R\}}(v') = \mathbf{1}_{\{|v'|\le 2R\}}(v').
	\end{aligned}
	\end{equation}
	Using \eqref{ppp5} and \eqref{ppp6}, \eqref{mlt3} is bounded by
	\begin{align}
		& C_{R,\Phi} \sup_{0\le s' \le t} \left\{\exp\left\{-e^{-\|\Phi\|_\infty} \frac{\nu_0}{2} (t-s') \right\}\|h(s')\|_{L^\infty_{x,v}}\right\}\ \sum_{i_1}^{M_1} \sum_{I_2}^{(M_2)^3} \int_0^t \mathbf{1}_{\{X(T_0;T_0+t_l-s,x_l,v_l)\in \mathcal{P}_{I_2}^{\Omega}\}}(s) \nonumber\\
		&\quad \times \exp\left\{-e^{-\|\Phi\|_\infty} \frac{\nu_0}{2} (t-s) \right\}\int_0^s \mathbf{1}_{\mathcal{P}_{i_1}^{T_0}}(s'-s+T_0)\mathbf{1}_{(t_{j,i_1,I_2,I_3}-\frac{\tilde{\epsilon}}{4M_1},t_{j,i_1,I_2,I_3}+\frac{\tilde{\epsilon}}{4M_1})}(s'-s+T_0) \nonumber\\
		&\quad \times \exp\left\{-e^{-\|\Phi\|_\infty} \frac{\nu_0}{2} (s-s') \right\}ds'ds\nonumber\\ \label{class24}
		&\le \tilde{\epsilon} \ C_{R,\Phi}\sup_{0\le s' \le T_0} \left\{\exp\left\{-e^{-\|\Phi\|_\infty} \frac{\nu_0}{2} (t-s') \right\}\|h(s')\|_{L^\infty_{x,v}}\right\}.
	\end{align}
	\newline
	$\mathbf{Case\ 3 \ (\romannumeral 2):}$ The integration \eqref{ltlt3} corresponding to \eqref{sst3} is bounded by
	\begin{align}
		&C_{R,\Phi}\sum_{i_1}^{M_1}\sum_{I_2}^{(M_2)^3}\sum_{I_3}^{(M_3)^3}\int_0^t \mathbf{1}_{\{X(T_0;T_0+t_l-s,x_l,v_l)\in \mathcal{P}_{I_2}^{\Omega}\}}(s)\int_0^s\mathbf{1}_{\mathcal{P}_{i_1}^{T_0}}(s'-s+T_0) \nonumber \\ 
		& \times \left\{1-\mathbf{1}_{\cup_{j=1}^3(t_{j,i_1,I_2,I_3}-\frac{\tilde{\epsilon}}{4M_1},t_{j,i_1,I_2,I_3}+\frac{\tilde{\epsilon}}{4M_1})}(s'-s+T_0) \right\}\exp\left\{-e^{-\|\Phi\|_\infty}\nu_0(t-s')\right\}\int_{\prod_{j=1}^{k-1}\mathcal{V}_j,|v_l| \le R}\nonumber\\ 
		& \times \underbrace{\int_{|v'|\le 2R}\mathbf{1}_{\mathcal{P}_{I_3}^v}(v')\int_{|v''|\le 3R}|h(s',X(s'-s+T_0;T_0,X(T_0;T_0+t_l-s,x_l,v_l),v'),v'')|dv''dv'}_{(\#3)}\nonumber\\ \label{mmlt3}
		&\times \left\{\prod_{j=l+1}^{k-1}d\sigma_j\right\}\left\{\tilde{w}(x_l,v_{l})d\sigma_{l}\right\}\left\{\prod_{j=1}^{l-1}d\sigma_j\right\}ds'ds.
	\end{align}
	By Lemma \ref{cpl}, we have made a change of variables $v' \rightarrow y:=X(s'-s+T_0;T_0,X(T_0;T_0+t_l-s,x_l,v_l),v')$ so that 
	\begin{align*}
		\det\left(\frac{dX}{dv'}(s'-s+T_0;T_0,X(T_0;T_0+t_l-s,x_l,v_l),v')\right)> \delta_*
	\end{align*}
	and the term $(\#3)$ is bounded by
	\begin{align*}
		\int_{|v'|\le 2R}\int_{|v''|\le 3R}|h(s',X(s'-s+T_0;T_0,X(T_0;T_0+t_l-s,x_l,v_l),v'),v'')|dv''dv'
		\le \frac{C_{R,\Phi}}{\delta_*} \|f(s')\|_{L^2_{x,v}},
	\end{align*}
	where we have used the Cauchy-Schwarz inequality.
	Hence \eqref{mmlt3} is bounded by
	\begin{align}
		&\frac{C_{R,\Phi,M_1,M_2,M_3}}{\delta_*} \int_0^t\int_0^s\exp\left\{-e^{-\|\Phi\|_\infty} \nu_0 (t-s') \right\} \|f(s')\|_{L^2_{x,v}}ds'ds\nonumber\\ \label{class25}
		& \le \frac{C_{R,\Phi,M_1,M_2,M_3}}{\delta_*} \int_0^{T_0}\|f(s')\|_{L^2_{x,v}}ds'.
	\end{align}
	Combining the bounds \eqref{class21}, \eqref{class22}, \eqref{class23}, \eqref{class24}, \eqref{class25} and summing over $1\le l \le k(\epsilon)-1$, we can bound $R_3$ in \eqref{RMC} by
	\begin{equation} \label{dd39}
	\begin{aligned}
		&\left(\frac{C_{\epsilon,\Phi}}{R}+C_{\epsilon,\Phi} e^{-\frac{R^2}{64}}+ \tilde{\epsilon} \ C_{\epsilon,R,\Phi} \right) \sup_{0\le s' \le T_0} \left\{\exp\left\{-e^{-\|\Phi\|_\infty} \frac{\nu_0}{2} (t-s') \right\}\|h(s')\|_{L^\infty_{x,v}}\right\}\\
		& +\frac{C_{\epsilon,R,\Phi,M_1,M_2,M_3}}{\delta_*} \int_0^{T_0}\|f(s')\|_{L^2_{x,v}}ds'.
	\end{aligned}
	\end{equation}
	\newline
	Finally, let us estimate  the term $R_4$ in \eqref{RMC} :
	\begin{align} 
		&\sum_{l=1}^{k-1}\int_{t_{l+1}}^{t_l}\sum_{l'=1}^{k-1}\int_{t_{l'+1}'}^{t_{l'}'}\frac{\exp{-\int_{t_1}^t e^{-\Phi(X(\tau))}\nu(V(\tau))d\tau}}{\tilde{w}(x_1,V(t_1))}\int_{\prod_{j=1}^{k-1}\mathcal{V}_j} \int_{\mathbb{R}^3} e^{-\Phi(X_l(s))}k_w(V_l(s),v')   \nonumber \\
		&\quad \times \frac{\exp{-\int_{t_1'}^s e^{-\Phi(X'(\tau))}\nu(V'(\tau))d\tau}}{\tilde{w}(x_1',V'(t'_1))}\int_{\prod_{j=1}^{k-1}\mathcal{V}_j'}\left\{\int_{\mathbb{R}^3}e^{-\Phi(X'_{l'}(s'))}k_w(V_{l'}'(s'),v'')h(s',X_{l'}'(s'),v'')dv''\right\} \nonumber \\
		& \quad \times d\Sigma_{l'}'(s')dv'd\Sigma_l(s)ds'ds \nonumber \\
		&\le \sum_{l=1}^{k-1}\int_{t_{l+1}}^{t_l}\sum_{l'=1}^{k-1}\int_{t_{l'+1}'}^{t_{l'}'} \frac{\exp\left\{-e^{-\|\Phi\|_\infty}\nu_0(t-s')\right\}}{\tilde{w}(x_1,V(t_1))} \int_{\prod_{j=1}^{k-1}\mathcal{V}_j} \int_{\mathbb{R}^3} |k_w(V_l(s),v')|\frac{1}{\tilde{w}(x_1',V'(t'_1))}  \nonumber \\
		& \quad \times  \int_{\prod_{j=1}^{k-1}\mathcal{V}_j'}\left\{\int_{\mathbb{R}^3}|k_w(V_{l'}'(s'),v'')||h(s',X_{l'}'(s'),v'')|dv''\right\} \left\{\prod_{j=l'+1}^{k-1}d\sigma_j'\right\}\left\{\tilde{w}(x_{l'}',v_{l'}')d\sigma_{l'}'\right\}\left\{\prod_{j=1}^{l'-1}d\sigma_j'\right\} \nonumber\\ \label{FFTMC}
		& \quad \times dv'\left\{\prod_{j=l+1}^{k-1}d\sigma_j\right\}\left\{\tilde{w}(x_l,v_{l})d\sigma_{l}\right\}\left\{\prod_{j=1}^{l-1}d\sigma_j\right\}ds'ds.
	\end{align}
	Fix $l,l'$. Note that $\tilde{w}(x'_{l'},v_{l'}')\mu(v_{l'}')|v_{l'}'|\le C_\Phi$ for some constant $C_\Phi >0$.
	We will divide this term into 3 cases.\\
	\newline
	$\mathbf{Case\ 1:}$ $|v_l| \ge R$ or $|v_{l'}'| \ge R$ with $R \gg 2\sqrt{2\|\Phi\|_\infty}$.\\
	By \eqref{sese}, we get
	\begin{align*}
		|V_l(s)|\ge |v_l|-\sqrt{2\|\Phi\|_\infty}\ge \frac{R}{2} \quad \text{or} \quad |V'_{l'}(s')|\ge |v_{l'}'|-\sqrt{2\|\Phi\|_\infty}\ge \frac{R}{2}.
	\end{align*}
	From Lemma \ref{Kest}, we have
	\begin{align*}
		\int_{|v_{l'}'|\ge R} \left(\int_{\mathbb{R}^3} |k_w(V_{l'}'(s'),v'')||h(s',X_{l'}'(s'),v'')|dv''\right)\tilde{w}(x_{l'}',v_{l'}')d\sigma_{l'}'\le \frac{C_\Phi}{1+R} \|h(s')\|_{L^\infty_{x,v}}.
	\end{align*}
	Then $R_4$ in the case $|v_{l'}'| \ge R$ is bounded by
	\begin{align}
		&\frac{C_\Phi}{1+R} \int_{t_{l+1}}^{t_l} \int_{t_{l'+1}'}^{t_{l'}'} \exp\left\{-e^{-\|\Phi\|_\infty}\nu_0(t-s')\right\}\|h(s')\|_{L^\infty_{x,v}}\int_{\prod_{j=1}^{k-1}\mathcal{V}_j} \int_{\mathbb{R}^3}|k_w(V_l(s),v')|dv'\nonumber\\
		& \quad \times \left\{\prod_{j=l+1}^{k-1}d\sigma_j\right\}\left\{\tilde{w}(x_l,v_{l})d\sigma_{l}\right\}\left\{\prod_{j=1}^{l-1}d\sigma_j\right\}ds'ds\nonumber\\
		& \le \frac{C_\Phi}{1+R} \int_0^t \int_0^s \exp\left\{-e^{-\|\Phi\|_\infty}\nu_0(t-s')\right\}\|h(s')\|_{L^\infty_{x,v}} ds'ds\nonumber\\ \label{class31}
		& \le \frac{C_\Phi}{1+R} \sup_{0\le s' \le T_0} \left\{\exp\left\{-e^{-\|\Phi\|_\infty} \frac{\nu_0}{2} (t-s') \right\}\|h(s')\|_{L^\infty_{x,v}}\right\},
	\end{align}
	where we have used the fact $\int_0^t \int_0^s \exp\left\{-e^{-\|\Phi\|_\infty} \frac{\nu_0}{2} (t-s') \right\} ds'ds$ is finite.\\
	By Lemma \ref{Kest}, $R_4$ in the case $|v_l| \ge R$ is bounded by 
	\begin{align}
		& C_\Phi \int_{t_{l+1}}^{t_l}\int_{t_{l'+1}'}^{t_{l'}'}\exp\left\{-e^{-\|\Phi\|_\infty}\nu_0(t-s')\right\}\int_{\prod_{j=1}^{k-1}\mathcal{V}_j,\ |v_l|\ge R}\int_{\mathbb{R}^3}|k_w(V_l(s),v')|\nonumber\\
		&\quad \times \int_{\prod_{j=1}^{k-1}\mathcal{V}_j'}\Biggl\{\int_{\mathbb{R}^3}|k_w(V_{l'}'(s'),v'')| \|h(s')\|_{L^\infty_{x,v}}dv''\Biggr\} \left\{\prod_{j=l'+1}^{k-1}d\sigma_j'\right\}\left\{\tilde{w}(x_{l'}',v_{l'}')d\sigma_{l'}'\right\}\left\{\prod_{j=1}^{l'-1}d\sigma_j'\right\}\nonumber\\
		&\quad \times dv'\left\{\prod_{j=l+1}^{k-1}d\sigma_j\right\}\left\{\tilde{w}(x_l,v_{l})d\sigma_{l}\right\}\left\{\prod_{j=1}^{l-1}d\sigma_j\right\}ds'ds\nonumber\\
		& \le \frac{C_\Phi}{1+R} \int_0^{t}\int_{0}^{s} \exp\left\{-e^{-\|\Phi\|_\infty}\nu_0(t-s')\right\} \|h(s')\|_{L^\infty_{x,v}} ds'ds\nonumber\\ \label{class32}
		& \le \frac{C_\Phi}{1+R} \sup_{0\le s' \le T_0} \left\{\exp\left\{-e^{-\|\Phi\|_\infty} \frac{\nu_0}{2} (t-s') \right\}\|h(s')\|_{L^\infty_{x,v}}\right\},
	\end{align}
	where we have used the fact $\int_0^t \int_0^s \exp\left\{-e^{-\|\Phi\|_\infty} \frac{\nu_0}{2} (t-s') \right\} ds'ds$ is finite.\\
	\newline
	$\mathbf{Case\ 2:}$ $|v_l| \le R$, $|v'| \ge 2R$, or $|v_{l'}'| \le R$, $|v''| \ge 2R$.\\
	Note that either $|v_l-v'| \ge R$ or $|v_{l'}'-v''| \ge R$.
	From \eqref{sese}, either one of the followings holds:
	\begin{align*}
		&|V_l(s)-v'| \ge |v_l-v'|-|V_l(s)-v_l| \ge R-\frac{R}{2}=\frac{R}{2},\\
		&|V_{l'}'(s')-v''| \ge |v_{l'}'-v''|-|V_{l'}'(s')-v_{l'}'| \ge R-\frac{R}{2}=\frac{R}{2}.
	\end{align*}
	Then we have either one of the followings:
	\begin{equation} \label{dd41}
	\begin{aligned}
		&|k(V_l(s),v')| \le e^{-\frac{R^2}{64}}|k(V_l(s),v')|e^{\frac{1}{16}|V_l(s)-v'|^2},\\
		&|k(V_{l'}'(s'),v'')| \le e^{-\frac{R^2}{64}}|k(V_{l'}'(s'),v'')|e^{\frac{1}{16}|V_{l'}'(s')-v''|^2}.
	\end{aligned}
	\end{equation}
	This yields from Lemma \ref{Kest},
	\begin{equation} \label{dd42}
	\begin{aligned}
		&\int_{|v'| \ge 2R} |k_w(V_l(s),v')|e^{\frac{1}{16}|V_l(s)-v'|^2}dv' < C,\\
		&\int_{|v''| \ge 2R} |k_w(V_{l'}'(s'),v'')|e^{\frac{1}{16}|V_{l'}'(s')-v''|^2}dv'' < C.
	\end{aligned}
	\end{equation}
	for some constant $C$. Thus we use \eqref{dd41} and \eqref{dd42} to bound $R_4$ in the case $|v_l| \le R, |v'| \ge 2R$ by
	\begin{align}
		&C_\Phi \int_{t_{l+1}}^{t_1}\int_{t_{l'+1}'}^{t_{l'}'}\exp\left\{-e^{-\|\Phi\|_\infty}\nu_0(t-s')\right\}\int_{\prod_{j=1}^{k-1}\mathcal{V}_j,\ |v_l|\le R}\int_{|v'| \ge 2R}|k_w(V_l(s),v')|\int_{\prod_{j=1}^{k-1}\mathcal{V}_j'}\nonumber\\
		& \quad  \times \Biggl\{\int_{\mathbb{R}^3}|k_w(V_{l'}'(s'),v'')|\|h(s')\|_{L^\infty_{x,v}}dv''\Biggr\} \left\{\prod_{j=l'+1}^{k-1}d\sigma_j'\right\}\left\{\tilde{w}(x_{l'}'v_{l'}')d\sigma_{l'}'\right\}\left\{\prod_{j=1}^{l'-1}d\sigma_j'\right\}dv'\nonumber\\
		& \quad \times \left\{\prod_{j=l+1}^{k-1}d\sigma_j\right\}\left\{\tilde{w}(x_l,v_{l})d\sigma_{l}\right\}\left\{\prod_{j=1}^{l-1}d\sigma_j\right\}ds'ds\nonumber\\
		&\le C_\Phi e^{-\frac{R^2}{64}} \int_0^t\int_0^s \exp\left\{-e^{-\|\Phi\|_\infty}\nu_0(t-s')\right\}\|h(s')\|_{L^\infty_{x,v}}ds'ds\nonumber\\ \label{class33}
		&\le C_\Phi e^{-\frac{R^2}{64}} \sup_{0\le s' \le T_0} \left\{\exp\left\{-e^{-\|\Phi\|_\infty} \frac{\nu_0}{2} (t-s') \right\}\|h(s')\|_{L^\infty_{x,v}}\right\},
	\end{align}
	where we have used the fact $\int_0^t\int_0^s \exp\left\{-e^{-\|\Phi\|_\infty} \frac{\nu_0}{2} (t-s) \right\}ds'ds$ is finite.\\
	Similarly, we use \eqref{dd41} and \eqref{dd42} to bound $R_4$ in the case $|v_{l'}'| \le R, |v''|\ge 2R$ by
	\begin{align}
		&C_\Phi \int_{t_{l+1}}^{t_1}\int_{t_{l'+1}'}^{t_{l'}'}\exp\left\{-e^{-\|\Phi\|_\infty}\nu_0(t-s')\right\}\int_{\prod_{j=1}^{k-1}\mathcal{V}_j}\int_{\mathbb{R}^3}|k_w(V_l(s),v')|\int_{\prod_{j=1}^{k-1}\mathcal{V}_j',\ |v_{l'}'|\le R}\nonumber\\
		& \quad  \times \Biggl\{\int_{|v''|\ge 2R}|k_w(V_{l'}'(s'),v'')|\|h(s')\|_{L^\infty_{x,v}}dv''\Biggr\} \left\{\prod_{j=l'+1}^{k-1}d\sigma_j'\right\}\left\{\tilde{w}(x_{l'}',v_{l'}')d\sigma_{l'}'\right\}\left\{\prod_{j=1}^{l'-1}d\sigma_j'\right\}dv'\nonumber\\
		& \quad \times \left\{\prod_{j=l+1}^{k-1}d\sigma_j\right\}\left\{\tilde{w}(x_l,v_{l})d\sigma_{l}\right\}\left\{\prod_{j=1}^{l-1}d\sigma_j\right\}ds'ds\nonumber\\
		&\le C_\Phi e^{-\frac{R^2}{64}} \int_0^t\int_0^s \exp\left\{-e^{-\|\Phi\|_\infty}\nu_0(t-s')\right\}\|h(s')\|_{L^\infty_{x,v}}ds'ds\nonumber\\ \label{class34}
		&\le C_\Phi e^{-\frac{R^2}{64}}\sup_{0\le s' \le T_0} \left\{\exp\left\{-e^{-\|\Phi\|_\infty} \frac{\nu_0}{2} (t-s') \right\}\|h(s')\|_{L^\infty_{x,v}}\right\},
	\end{align}
	where we have used the fact $\int_0^t\int_0^s \exp\left\{-e^{-\|\Phi\|_\infty} \frac{\nu_0}{2} (t-s') \right\}ds'ds$ is finite.\\
	\newline
	$\mathbf{Case\ 3:}$ $|v_l|\le R$, $|v'| \le 2R$, $|v_{l'}'| \le R$, $|v''| \le 2R$.\\
	Since $k_w(v,v')$ has possible integrable singularity of $\frac{1}{|v-v'|}$, we can choose smooth function $k_R(v,v')$ with compact support such that
	\begin{align} \label{kk7}
		\sup_{|v|\le 2R}\int_{|v'| \le 2R} \left|k_R(v,v')-k_w(v,v')\right|dv' \le \frac{1}{R}.
	\end{align}
	We split 
	\begin{equation} \label{kk8}
	\begin{aligned}
		k_w(V_l(s),v')k_w(V_{l'}'(s'),v'') &= \left\{k_w(V_l(s),v')-k_R(V_l(s),v')\right\}k(V_{l'}'(s'),v'')\\ 
		&\quad + \{k_w(V_{l'}'(s'),v'')-k_R(V_{l'}'(s'),v'')\}k_R(V_l(s),v')\\
		&\quad +k_R(V_l(s),v')k_R(V_{l'}'(s'),v'').
	\end{aligned}
	\end{equation}
	From \eqref{kk7} and \eqref{kk8}, $R_4$ in this case is bounded by 
	\begin{align}
		&\frac{C_\Phi}{R} \int_0^t \int_0^{s}\exp\left\{-e^{-\|\Phi\|_\infty}\nu_0(t-s')\right\} \|h(s')\|_{L^\infty_{x,v}}ds'ds \nonumber \\
		& +C_\Phi \int_{t_{l+1}}^{t_l} \int_{t_{l'+1}'}^{t_{l'}'} \exp\left\{-e^{-\|\Phi\|_\infty}\nu_0(t-s')\right\}  \int_{\prod_{j=1}^{k-1}\mathcal{V}_j,\ |v_l|\le R}\int_{|v'|\le 2R} |k_R(V_l(s),v')|  \nonumber \\ 
		&\quad  \times \int_{\prod_{j=1}^{k-1}\mathcal{V}_j' ,\ |v_{l'}'|\le R} \Biggl\{\int_{|v''| \le 2R}|k_R(V_{l'}'(s'),v'')|  |h(s',X'_{l'}(s'),v'')|dv''\Biggr\}\left\{\prod_{j=l'+1}^{k-1}d\sigma_j'\right\}\left\{\tilde{w}(x_{l'}',v_{l'}')d\sigma_{l'}'\right\}\nonumber\\
		&\quad \times \left\{\prod_{j=1}^{l'-1}d\sigma_j'\right\}dv'\left\{\prod_{j=l+1}^{k-1}d\sigma_j\right\}\left\{\tilde{w}(x_l,v_{l})d\sigma_{l}\right\}\left\{\prod_{j=1}^{l-1}d\sigma_j\right\}ds'ds \nonumber  \\ \label{class35}
		&\le \frac{C_\Phi}{R} \sup_{0\le s' \le T_0} \left\{\exp\left\{-e^{-\|\Phi\|_\infty} \frac{\nu_0}{2} (t-s') \right\}\|h(s')\|_{L^\infty_{x,v}}\right\}  \\ 
		&\quad + C_{R,\Phi}\int_{t_{l+1}}^{t_l} \int_{t_{l'+1}'}^{t_{l'}'} \exp\left\{-e^{-\|\Phi\|_\infty}\nu_0(t-s')\right\} \int_{\prod_{j=1}^{k-1}\mathcal{V}_j,\ |v_l|\le R}\int_{|v'|\le 2R} \int_{\prod_{j=1}^{k-1}\mathcal{V}_j' ,\ |v_{l'}'|\le R} \nonumber \\ 
		&\qquad \times \Biggl\{\int_{|v''| \le 2R}|h(s',X'_{l'}(s'),v'')|dv''\Biggr\} \left\{\prod_{j=l'+1}^{k-1}d\sigma_j'\right\}\left\{\tilde{w}(x_{l'}',v_{l'}')d\sigma_{l'}'\right\}\left\{\prod_{j=1}^{l'-1}d\sigma_j'\right\}dv'\nonumber\\ 
		&\qquad \times \left\{\prod_{j=l+1}^{k-1}d\sigma_j\right\}\left\{\tilde{w}(x_l,v_{l})d\sigma_{l}\right\}\left\{\prod_{j=1}^{l-1}d\sigma_j\right\}ds'ds \nonumber\\
		&=: R_{41}+R_{42}, \nonumber
	\end{align}
	where we have used the fact $|k_R(V_l(s),v')||k_R(V_{l'}'(s'),v'')|\le C_R$.\\
	In the term $R_4$, we recall that $X_{l'}'(s') = X(s';t_{l'}',x_{l'}',v_{l'}')$.
	Since the potential is time dependent, we have
	\begin{align*}
		X(s';t_{l'}',x_{l'}',v_{l'}') = X(s'-t_{l'}'+T_0;T_0,x_{l'}',v_{l'}')
	\end{align*}
	for all $0\le s' \le t_{l'}' \le T_0$.\\
	By Lemma \ref{cpl}, the term $R_{42}$ becomes
	\begin{align}
		&C_{R,\Phi}\sum_{i_1}^{M_1}\sum_{I_2}^{(M_2)^3}\sum_{I_3}^{(M_3)^3}\int_{t_{l+1}}^{t_l} \mathbf{1}_{\{x_{l'}'\in \mathcal{P}_{I_2}^{\Omega}\}}(s)\int_{t_{l'+1}'}^{t_{l'}'}\mathbf{1}_{\mathcal{P}_{i_1}^{T_0}}(s'-t_{l'}'+T_0) \exp\left\{-e^{-\|\Phi\|_\infty}\nu_0(t-s')\right\} \int_{\prod_{j=1}^{k-1}\mathcal{V}_j,\ |v_l|\le R} \nonumber \\ 
		&\quad \times \int_{|v'|\le 2R} \int_{\prod_{j=1}^{k-1}\mathcal{V}_j' ,\ |v_{l'}'|\le R}\mathbf{1}_{\mathcal{P}_{I_3}^v}(v_{l'}')\Biggl\{\int_{|v''| \le 2R}|h(s',X'_{l'}(s'),v'')|dv''\Biggr\} \left\{\prod_{j=l'+1}^{k-1}d\sigma_j'\right\}\left\{\tilde{w}(x_{l'}',v_{l'}')d\sigma_{l'}'\right\}\nonumber\\ \label{ltlt4}
		&\quad \times \left\{\prod_{j=1}^{l'-1}d\sigma_j'\right\}dv'\left\{\prod_{j=l+1}^{k-1}d\sigma_j\right\}\left\{\tilde{w}(x_l,v_{l})d\sigma_{l}\right\}\left\{\prod_{j=1}^{l-1}d\sigma_j\right\}ds'ds.
	\end{align}
	From Lemma \ref{cpl}, we have the following partitions:
	\begin{align*}
		&\biggl\{(s'-t_{l'}'+T_0,x_{l'}',v_{l'}')\in \mathcal{P}_{i_1}^{T_0}\cross \mathcal{P}_{I_2}^{\Omega}\cross \mathcal{P}_{I_3}^{v} : \det\left(\frac{dX}{dv'}(s'-t_{l'}'+T_0;T_0,x_{l'}',v_{l'}')\right)=0\biggr\}\\
		& \subset \bigcup_{j=1}^3 \biggl\{(s'-t_{l'}'+T_0,x_{l'}',v_{l'}')\in \mathcal{P}_{i_1}^{T_0}\cross \mathcal{P}_{I_2}^{\Omega}\cross \mathcal{P}_{I_3}^{v}:s'-t_{l'}'+T_0\in \left(t_{j,i_1,I_2,I_3}-\frac{\tilde{\epsilon}}{4M_1},t_{j,i_1,I_2,I_3}+\frac{\tilde{\epsilon}}{4M_1}\right)\biggr\}.
	\end{align*}
	Thus for each $i_1$,$I_2$, and $I_3$, we split $\mathbf{1}_{\mathcal{P}_{i_1}^{T_0}}(s'-t_{l'}'+T_0)$ as
	\begin{align} \label{fst4}
		&\mathbf{1}_{\mathcal{P}_{i_1}^{T_0}}(s'-t_{l'}'+T_0)\mathbf{1}_{\cup_{j=1}^3(t_{j,i_1,I_2,I_3}-\frac{\tilde{\epsilon}}{4M_1},t_{j,i_1,I_2,I_3}+\frac{\tilde{\epsilon}}{4M_1})}(s'-t_{l'}'+T_0)\\ \label{sst4}
		& +\mathbf{1}_{\mathcal{P}_{i_1}^{T_0}}(s'-t_{l'}'+T_0)\left\{1-\mathbf{1}_{\cup_{j=1}^3(t_{j,i_1,I_2,I_3}-\frac{\tilde{\epsilon}}{4M_1},t_{j,i_1,I_2,I_3}+\frac{\tilde{\epsilon}}{4M_1})}(s'-t_{l'}'+T_0)\right\}.
	\end{align}
	$\mathbf{Case\ 3 \ (\romannumeral 1):}$ The integration \eqref{ltlt4} corresponding to \eqref{fst4} is bounded by
	\begin{align}
		&C_{R,\Phi}\sum_{i_1}^{M_1}\sum_{I_2}^{(M_2)^3}\sum_{I_3}^{(M_3)^3}\sum_{j=1}^3\int_{t_{l+1}}^{t_l} \mathbf{1}_{\{x_{l'}'\in \mathcal{P}_{I_2}^{\Omega}\}}(s)\int_{t_{l'+1}'}^{t_{l'}'}\mathbf{1}_{\mathcal{P}_{i_1}^{T_0}}(s'-t_{l'}'+T_0)\mathbf{1}_{(t_{j,i_1,I_2,I_3}-\frac{\tilde{\epsilon}}{4M_1},t_{j,i_1,I_2,I_3}+\frac{\tilde{\epsilon}}{4M_1})}(s'-t_{l'}'+T_0)  \nonumber \\ 
		&\quad \times \exp\left\{-e^{-\|\Phi\|_\infty}\nu_0(t-s')\right\} \int_{\prod_{j=1}^{k-1}\mathcal{V}_j,\ |v_l|\le R}\int_{|v'|\le 2R} \int_{\prod_{j=1}^{k-1}\mathcal{V}_j' ,\ |v_{l'}'|\le R}\mathbf{1}_{\mathcal{P}_{I_3}^v}(v_{l'}') \nonumber\\ 
		&\quad \times \Biggl\{\int_{|v''| \le 2R}|h(s',X(s'-t_{l'}'+T_0;T_0,x_{l'}',v_{l'}'),v'')|dv''\Biggr\}\left\{\prod_{j=l'+1}^{k-1}d\sigma_j'\right\}\left\{\tilde{w}(x_{l'}',v_{l'}')d\sigma_{l'}'\right\}\left\{\prod_{j=1}^{l'-1}d\sigma_j'\right\}\nonumber\\ \label{mlt4}
		&\quad \times dv'\left\{\prod_{j=l+1}^{k-1}d\sigma_j\right\} \left\{\tilde{w}(x_l,v_{l})d\sigma_{l}\right\}\left\{\prod_{j=1}^{l-1}d\sigma_j\right\}ds'ds.
	\end{align}
	We split
	\begin{align*}
		\exp\left\{-e^{-\|\Phi\|_\infty} \nu_0 (t-s') \right\} &= \exp\left\{-e^{-\|\Phi\|_\infty} \frac{\nu_0}{2} (t-s) \right\}\exp\left\{-e^{-\|\Phi\|_\infty} \frac{\nu_0}{2} (s-s') \right\}\\
	& \quad \times  \exp\left\{-e^{-\|\Phi\|_\infty} \frac{\nu_0}{2} (t-s') \right\}.
	\end{align*}
	and we can bound the integration \eqref{mlt4} by
	\begin{align*}
		&C_{R,\Phi}\sum_{i_1}^{M_1}\sum_{I_2}^{(M_2)^3}\sum_{I_3}^{(M_3)^3}\sum_{j=1}^3\int_{t_{l+1}}^{t_l} \mathbf{1}_{\{x_{l'}'\in \mathcal{P}_{I_2}^{\Omega}\}}(s)\exp\left\{-e^{-\|\Phi\|_\infty} \frac{\nu_0}{2} (t-s) \right\}\\  
		& \times  \underbrace{\int_{t_{l'+1}'}^{t_{l'}'}\mathbf{1}_{\mathcal{P}_{i_1}^{T_0}}(s'-t_{l'}'+T_0)\mathbf{1}_{(t_{j,i_1,I_2,I_3}-\frac{\tilde{\epsilon}}{4M_1},t_{j,i_1,I_2,I_3}+\frac{\tilde{\epsilon}}{4M_1})}(s'-t_{l'}'+T_0)\exp\left\{-e^{-\|\Phi\|_\infty} \frac{\nu_0}{2} (s-s')\right\} }_{(*4)} \\ 
		& \times  \int_{\prod_{j=1}^{k-1}\mathcal{V}_j,\ |v_l|\le R}\int_{|v'|\le 2R}\int_{\prod_{j=1}^{k-1}\mathcal{V}_j' ,\ |v_{l'}'|\le R}\mathbf{1}_{\mathcal{P}_{I_3}^v}(v_{l'}')\Biggl\{\int_{|v''| \le 2R}\exp\left\{-e^{-\|\Phi\|_\infty} \frac{\nu_0}{2} (t-s') \right\}\|h(s')\|_{L^\infty_{x,v}}dv''\Biggr\}\\ 
		&\times \left\{\prod_{j=l'+1}^{k-1}d\sigma_j'\right\}\left\{\tilde{w}(x_{l'}',v_{l'}')d\sigma_{l'}'\right\}\left\{\prod_{j=1}^{l'-1}d\sigma_j'\right\}dv'\left\{\prod_{j=l+1}^{k-1}d\sigma_j\right\}\left\{\tilde{w}(x_l,v_{l})d\sigma_{l}\right\}\left\{\prod_{j=1}^{l-1}d\sigma_j\right\}ds'ds.
	\end{align*}
	Here, $(*4)$ is bounded by
	\begin{align}
		&\int_0^s \mathbf{1}_{\mathcal{P}_{i_1}^{T_0}}(s'-t_{l'}'+T_0)\mathbf{1}_{(t_{j,i_1,I_2,I_3}-\frac{\tilde{\epsilon}}{4M_1},t_{j,i_1,I_2,I_3}+\frac{\tilde{\epsilon}}{4M_1})}(s'-t_{l'}'+T_0) \exp\left\{-e^{-\|\Phi\|_\infty} \frac{\nu_0}{2} (s-s') \right\}ds'\nonumber\\ \label{ppp7}
		& \le \frac{\tilde{\epsilon}}{2M_1}.
	\end{align}
	From the partition of the time interval $[0,T_0]$ and velocity domain $[-4R,4R]^3$ in Lemma \ref{cpl}, we have
	\begin{equation} \label{ppp8}
	\begin{aligned}
		&\sum_{I^2}^{(M_2)^3}\mathbf{1}_{\{x_{l'}'\in \mathcal{P}_{I_2}^{\Omega}\}}(s) \le \mathbf{1}_{\{0\le s \le T_0\}}(s),\\
		&\sum_{I^3}^{(M_3)^3} \mathbf{1}_{\mathcal{P}_{I_3}^v}(v_{l'}')\mathbf{1}_{\{|v_{l'}'|\le R\}}(v_{l'}') = \mathbf{1}_{\{|v_{l'}'|\le R\}}(v_{l'}').
	\end{aligned}
	\end{equation}
	Using \eqref{ppp7} and \eqref{ppp8}, \eqref{mlt4} is bounded by
	\begin{align}
		& C_{R,\Phi} \sup_{0\le s' \le T_0} \left\{\exp\left\{-e^{-\|\Phi\|_\infty} \frac{\nu_0}{2} (t-s') \right\}\|h(s')\|_{L^\infty_{x,v}}\right\}\ \sum_{i_1}^{M_1} \sum_{I_2}^{(M_2)^3} \int_0^t \mathbf{1}_{\{x_{l'}'\in \mathcal{P}_{I_2}^{\Omega}\}}(s)\exp\left\{-e^{-\|\Phi\|_\infty} \frac{\nu_0}{2} (t-s) \right\} \nonumber\\
		&\quad  \times \int_0^s \mathbf{1}_{\mathcal{P}_{i_1}^{T_0}}(s'-t_{l'}'+T_0)\mathbf{1}_{(t_{j,i_1,I_2,I_3}-\frac{\tilde{\epsilon}}{4M_1},t_{j,i_1,I_2,I_3}+\frac{\tilde{\epsilon}}{4M_1})}(s'-t_{l'}'+T_0)  \exp\left\{-e^{-\|\Phi\|_\infty} \frac{\nu_0}{2} (s-s') \right\}ds'ds\nonumber\\
		&\le \tilde{\epsilon}\ C_{R,\Phi} \sup_{0\le s' \le T_0} \left\{\exp\left\{-e^{-\|\Phi\|_\infty} \frac{\nu_0}{2} (t-s') \right\}\|h(s')\|_{L^\infty_{x,v}}\right\} \sum_{I_2}^{(M_2)^3} \int_0^t \mathbf{1}_{\{x_{l'}'\in \mathcal{P}_{I_2}^{\Omega}\}}(s)\exp\left\{-e^{-\|\Phi\|_\infty} \frac{\nu_0}{2} (t-s) \right\} ds\nonumber\\ \label{class36}
		&\le \tilde{\epsilon} \ C_{R,\Phi} \sup_{0\le s' \le T_0} \left\{\exp\left\{-e^{-\|\Phi\|_\infty} \frac{\nu_0}{2} (t-s') \right\}\|h(s')\|_{L^\infty_{x,v}}\right\}.
	\end{align}
	\newline
	$\mathbf{Case\ 3 \ (\romannumeral 2):}$ The integration \eqref{ltlt4} corresponding to \eqref{sst4} is bounded by
	\begin{align}
		&C_{R,\Phi}\sum_{i_1}^{M_1}\sum_{I_2}^{(M_2)^3}\sum_{I_3}^{(M_3)^3}\int_{t_{l+1}}^{t_l} \mathbf{1}_{\{x_{l'}'\in \mathcal{P}_{I_2}^{\Omega}\}}(s)\int_{t_{l'+1}'}^{t_{l'}'}\mathbf{1}_{\mathcal{P}_{i_1}^{T_0}}(s'-t_{l'}'+T_0)\left\{1-\mathbf{1}_{\cup_{j=1}^3(t_{j,i_1,I_2,I_3}-\frac{\tilde{\epsilon}}{4M_1},t_{j,i_1,I_2,I_3}+\frac{\tilde{\epsilon}}{4M_1})}(s'-t_{l'}'+T_0)\right\}  \nonumber \\ 
		& \times \exp\left\{-e^{-\|\Phi\|_\infty}\nu_0(t-s')\right\} \int_{\prod_{j=1}^{k-1}\mathcal{V}_j,\ |v_l|\le R}\int_{|v'|\le 2R}  \nonumber\\ 
		&\times \underbrace{\int_{\prod_{j=1}^{l'}\mathcal{V}_j' ,\ |v_{l'}'|\le R}\mathbf{1}_{\mathcal{P}_{I_3}^v}(v_{l'}')\Biggl\{\int_{|v''| \le 2R}|h(s',X(s'-t_{l'}'+T_0;T_0,x_{l'}',v_{l'}'),v'')|dv''\Biggr\}\left\{\tilde{w}(x_{l'}',v_{l'}')d\sigma_{l'}'\right\}\left\{\prod_{j=1}^{l'-1}d\sigma_j'\right\}dv'}_{(\#4)}\nonumber\\ \label{mmlt4}
		& \times \left\{\prod_{j=l+1}^{k-1}d\sigma_j\right\}\left\{\tilde{w}(x_l,v_{l})d\sigma_{l}\right\}\left\{\prod_{j=1}^{l-1}d\sigma_j\right\}ds'ds.
	\end{align}
	By Lemma \ref{cpl}, we have made a change of variables $v' \rightarrow y:=X(s'-t_{l'}'+T_0;T_0,x_{l'}',v_{l'}')$ so that 
	\begin{align*}
		\det\left(\frac{dX}{dv_{l'}'}(s'-t_{l'}'+T_0;T_0,x_{l'}',v_{l'}')\right)> \delta_*
	\end{align*}
	and the term $(\#4)$ is bounded by
	\begin{align*}
		&\int_{\prod_{j=1}^{l'}\mathcal{V}_j' ,\ |v_{l'}'|\le R}\int_{|v''|\le 2R}|h(s',X(s'-t_{l'}'+T_0;T_0,x_{l'}',v_{l'}'),v'')|dv''\left\{\tilde{w}(x_{l'}',v_{l'}')d\sigma_{l'}'\right\}\left\{\prod_{j=1}^{l'-1}d\sigma_j'\right\} \le \frac{C_{R,\Phi}}{\delta_*} \|f(s')\|_{L^2_{x,v}},
	\end{align*}
	where we have used the Cauchy-Schwarz inequality. 
	Hence \eqref{mmlt4} is bounded by
	\begin{align}
		&\frac{C_{R,\Phi,M_1,M_2,M_3}}{\delta_*} \int_0^t\int_0^s\exp\left\{-e^{-\|\Phi\|_\infty} \nu_0 (t-s') \right\} \|f(s')\|_{L^2_{x,v}}ds'ds\nonumber\\ \label{class37}
		& \le \frac{C_{R,\Phi,M_1,M_2,M_3}}{\delta_*} \int_0^{T_0}\|f(s')\|_{L^2_{x,v}}ds'.
	\end{align}
	Combining the bounds \eqref{class31}, \eqref{class32}, \eqref{class33}, \eqref{class34}, \eqref{class35}, \eqref{class36}, \eqref{class37} and summing over $1\le l, l' \le k(\epsilon)-1$, we can bound $R_4$ in \eqref{RMC} by
	\begin{equation} \label{dd49}
	\begin{aligned}
		&\left(\frac{C_{\epsilon,\Phi}}{R}+C_{\epsilon,\Phi} e^{-\frac{R^2}{64}}+ \tilde{\epsilon} \ C_{\epsilon,R,\Phi} \right) \sup_{0\le s' \le T_0} \left\{\exp\left\{-e^{-\|\Phi\|_\infty} \frac{\nu_0}{2} (t-s') \right\}\|h(s')\|_{L^\infty_{x,v}}\right\}\\
		& +\frac{C_{\epsilon,R,\Phi,M_1,M_2,M_3}}{\delta_*} \int_0^{T_0}\|f(s')\|_{L^2_{x,v}}ds'.
	\end{aligned}
	\end{equation}
	\newline
	Gathering \eqref{dd2}, \eqref{dd3}, \eqref{dd4}, \eqref{dd5}, \eqref{dd6}, \eqref{dd7}, \eqref{dd8}, \eqref{dd19}, \eqref{dd29}, \eqref{dd39}, and \eqref{dd49}, we deduce for $0\le t \le T_0$
	\begin{align*}
		\|h(t)\|_{L^\infty_{x,v}} 
		& \le  C_\Phi(1+t)\exp\left\{-e^{-\|\Phi\|_\infty} \frac{\nu_0}{4} t\right\}\|h_0\|_{L^\infty_{x,v}}\\
		& \quad + \left( \epsilon C_\Phi^{(1)}+\frac{C^{(2)}_{\epsilon,\Phi}}{R}+C^{(3)}_{\epsilon,\Phi} e^{-\frac{R^2}{64}} + \tilde{\epsilon}\ C^{(4)}_{\epsilon, R,\Phi}\right) \sup_{0 \le s' \le T_0}\left\{\exp\left\{-e^{-\|\Phi\|_\infty} \frac{\nu_0}{4} (t-s') \right\}\|h(s')\|_{L^\infty_{x,v}}\right\}\\
		& \quad +\frac{C_{\epsilon, R,\Phi,M_1,M_2,M_3}}{\delta_*} \int_0^{T_0}\|f(s')\|_{L^2_{x,v}}ds'.
	\end{align*}
	This implies that
	\begin{align*}
		&\sup_{0 \le t \le T_0 } \left\{\exp\left\{e^{-\|\Phi\|_\infty} \frac{\nu_0}{4} t \right\} \|h(t)\|_{L^\infty_{x,v}}\right\}\\
		&\le  C_\Phi(1+T_0)\|h_0\|_{L^\infty_{x,v}} \\
		& \quad + \left(\epsilon C_\Phi^{(1)}+\frac{C^{(2)}_{\epsilon,\Phi}}{R}+C^{(3)}_{\epsilon,\Phi} e^{-\frac{R^2}{64}} + \tilde{\epsilon} C^{(4)}_{\epsilon, R,\Phi}\right) \sup_{0 \le s' \le T_0}\left\{\exp\left\{e^{-\|\Phi\|_\infty} \frac{\nu_0}{4} s' \right\}\|h(s')\|_{L^\infty_{x,v}}\right\}\\
		& \quad +\frac{C_{\epsilon, R,\Phi,M_1,M_2,M_3}}{\delta_*} \int_0^{T_0}\|f(s')\|_{L^2_{x,v}}ds'.
	\end{align*}
	First, we choose $\epsilon>0$ small enough such that $\epsilon C_\Phi^{(1)} < \frac{1}{3}$, then choose $R$ sufficiently large so that $\frac{C_{\epsilon,\Phi}^{(2)}}{R}+C_{\epsilon,\Phi}^{(3)} e^{-\frac{R^2}{64}}  < \frac{1}{3}$, and last choose $\tilde{\epsilon}$ large enough such that $\tilde{\epsilon} C^{(4)}_{\epsilon, R,\Phi} < \frac{1}{3}$.
	Hence, we obtain
	\begin{align*}
		\sup_{0 \le t \le T_0 } \left\{\exp\left\{e^{-\|\Phi\|_\infty} \frac{\nu_0}{4} t \right\} \|h(t)\|_{L^\infty_{x,v}}\right\} &\le  C_\Phi(1+T_0)\|h_0\|_{L^\infty_{x,v}} +C_{T_0,\Phi} \int_0^{T_0}\|f(s')\|_{L^2_{x,v}}ds'.
	\end{align*}
	This yields that
	\begin{align*}
		\|h(T_0)\|_{L^\infty_{x,v}} \le C_\Phi(1+T_0)\exp\left\{-e^{-\|\Phi\|_\infty} \frac{\nu_0}{4} T_0 \right\}\|h_0\|_{L^\infty_{x,v}} +C_{T_0, \Phi} \int_0^{T_0}\|f(s')\|_{L^2_{x,v}}ds'.
	\end{align*}
	Choosing large $\tilde{T}_0 >0$ such that
	\begin{align} \label{cc33}
		C_\Phi(1+\tilde{T}_0)\exp\left\{-e^{-\|\Phi\|_\infty} \frac{\nu_0}{4} \tilde{T}_0 \right\} \le e^{-\lambda \tilde{T}_0} \quad \text{for} \quad \lambda < e^{-\|\Phi\|_\infty} \frac{\nu_0}{4},
	\end{align}
	we obtain
	\begin{align*}
		\|h(\tilde{T}_0)\|_{L^\infty_{x,v}} \le e^{-\lambda \tilde{T}_0}\|h_0\|_{L^\infty_{x,v}} +C_{\tilde{T}_0, \Phi} \int_0^{\tilde{T}_0}\|f(s')\|_{L^2_{x,v}}ds'.
	\end{align*}
	By Theorem \ref{Fest}, we conclude the exponential decay.
\end{proof}

\subsection{A priori estimate in a small data problem} \label{Apeisdp}
In this subsection, we need to take constants $\tilde{\lambda}$, $\tilde{T}_0$, and $\eta$ in the a priori assumption \eqref{AAL2}. Firstly, we take $\tilde{\lambda}>0$ satisfying conditions \eqref{cc4} and 
\begin{align} \label{cc6}
	\tilde{\lambda}< e^{-\|\Phi\|_\infty}\frac{\nu_0}{4}.
\end{align}
 Next, we choose $\tilde{T}_0>0$ which satisfies the conditions \eqref{cc5} and \eqref{cc33}. In the proof of Theorem \ref{Fest}, it holds that
\begin{align}\label{cc2}
	\|wf(t)\|_{L^\infty_{x,v}} \le C_Ae^{-\tilde{\lambda}t}\|wf_0\|_{L^\infty_{x,v}} \quad \text{for all } 0\le t\le \tilde{T}_0,
\end{align}
for some constant $C_A>0$. Here we take $\eta := C_A\|wf_0\|_{L^\infty_{x,v}}$. In Theorem \ref{mainresult1}, $\|wf_0\|_{L^\infty_{x,v}}$ is chosen sufficiently small. Thus we take $\eta>0$ satisfying the condition \eqref{cc3} and a condition of smallness in Theorem \ref{mainresult1}. Hence we have closed the a priori assumption \eqref{AAL2}.

\bigskip

\section{Nonlinear Asymptotic Stability near $\mu_{E}$} \label{Nonlineardecay}
In this section we prove Theorem \ref{mainresult1}, asymptotic stability of small perturbation problem with the diffuse reflection boundary condition \eqref{PDRBC}. Theorem \ref{mainresult1} provides a key foundation to prove the large amplitude problem, i.e. Theorem \ref{mainresult2}. Before proving Theorem \ref{mainresult1}, we first introduce the next lemma, called {\it the Gamma estimate}. Using the following lemma, we will handle the nonlinear term $\Gamma(g_1,g_2)$ in Theorem \ref{mainresult1}. Because the lemma is essentially identical to \cite[Lemma\ 5]{GDCB2010}, we only mention the statement of the lemma.

\begin{lemma} \cite{GDCB2010} \label{Gamest}
	Let $g_1, g_2$ be in $L^\infty_{x,v}(w)$. Then there exists $C>0$ such that
	\begin{align*}
		\left|w(x,v)\Gamma\left(g_1,g_2\right)(x,v)\right| \le C \nu(v)\|wg_1\|_{L^\infty_{x,v}}\|wg_2\|_{L^\infty_{x,v}}.
	\end{align*}
\end{lemma}

\bigskip
Theorem \ref{mainresult1} states the global existence, uniqueness, and exponential decay of a solution to the full perturbed Boltzmann equation \eqref{PBE} with small amplitude data and the diffuse reflection boundary condition \eqref{PDRBC}. This theorem may be essentially used to demonstrate our main goal.

\begin{proof} [\textbf{Proof of Theorem \ref{mainresult1}}]
	Set $h^{(0)} \equiv 0$. We consider the following iterative system:
	\begin{align} \label{NIS}
		\left\{\partial_{t}+ v\cdot \nabla_x  -\nabla_x \Phi(x) \cdot \nabla_v+e^{-\Phi(x)}\nu -e^{-\Phi(x)}K_w \right\} h^{(m+1)} = e^{-\frac{\Phi(x)}{2}} w\Gamma \left(\frac{h^{(m)}}{w},\frac{h^{(m)}}{w}\right)
	\end{align}
	with $h^{(m+1)}|_{t=0} = h_0$ and the diffuse reflection boundary condition
	\begin{align*}
		h^{(m+1)}|_{\gamma_-} = \frac{1}{\tilde{w}(x,v)}\int_{n(x) \cdot v' >0} h^{(m+1)}(t,x,v')\tilde{w}(x,v') d\sigma,
	\end{align*}
	 where $\tilde{w}(x,v)$ is defined in \eqref{tildeweight}.\\
	By the Duhamel principle, we have
	\begin{align*}
		h^{(m+1)}(t,x,v) = S_G(t)h_0 + \int_0^t S_G(t-s)\left[e^{-\frac{\Phi}{2}}w\Gamma \left(\frac{h^{(m)}}{w},\frac{h^{(m)}}{w}\right)\right] (s) ds.
	\end{align*}
	We use the Duhamel principle once again to get
	\begin{align*}
		h^{(m+1)}(t,x,v) &= S_G(t)h_0\\
		& \quad + \int_0^t S_{G_\nu}(t-s)\left[e^{-\frac{\Phi}{2}}w\Gamma \left(\frac{h^{(m)}}{w},\frac{h^{(m)}}{w}\right)\right] (s) ds\\
		&\quad  + \int_0^t\int_s^t S_{G_\nu}(t-s')e^{-\Phi}K_w S_G(s'-s) \left[e^{-\frac{\Phi}{2}}w\Gamma\left(\frac{h^{(m)}}{w},\frac{h^{(m)}}{w}\right)\right](s) ds'ds\\
		&=: I_1+I_2+I_3.
	\end{align*}
	From Theorem \ref{LT}, for some $0 < \lambda \le \lambda_\infty$,
	\begin{align*}
		I_1 \le Ce^{-\lambda t}\|h_0\|_{L^\infty_{x,v}}
	\end{align*}
	for some constant $C>0$.\\
	First, we consider the term $I_2$. From Corollary \ref{L43C} and Lemma \ref{Gamest}, we deduce that
	\begin{equation} \label{ddd51}
	\begin{aligned}		
		&\left|\int_0^t S_{G_\nu}(t-s)\left[e^{-\frac{\Phi}{2}}w\Gamma\left(\frac{h^{(m)}}{w},\frac{h^{(m)}}{w}\right)\right](s)ds\right|\\
		&\le C_\Phi \exp\left\{-e^{-\|\Phi\|_\infty} \frac{\nu_0}{2} t \right\} \sup_{0 \le s \le t}\left\{ \exp\left\{e^{-\|\Phi\|_\infty} \frac{\nu_0}{4} s \right\}\|h(s)\|_{L^\infty_{x,v}}\right\}^2.
	\end{aligned}
	\end{equation}		
	On the other hand, for the term $I_3$, for any given initial datum $\tilde{h}_0$, we consider the semigroup $S_{\tilde{G}}(t)\tilde{h}_0$ which solves
	\begin{align*}
		\left\{\partial_{t}+ v\cdot \nabla_x  -\nabla_x \Phi(x) \cdot \nabla_v+e^{-\Phi(x)}\nu -e^{-\Phi(x)}K_{w/(1+|v|)} \right\} S_{\tilde{G}}(t)\tilde{h}_0 = 0
	\end{align*}
	with $S_{\tilde{G}}(0)\tilde{h}_0 = \tilde{h}_0$ and the diffuse reflection boundary condition
	\begin{align*}
		\left(S_{\tilde{G}}(t) \tilde{h}_0\right)(t,x,v)\Big|_{\gamma_-} = \frac{1}{\tilde{w}_1(x,v)} \int_{n(x) \cdot v' >0} \left(S_{\tilde{G}}(t) \tilde{h}_0\right)(t,x,v) \tilde{w}_1(x,v') d\sigma,
	\end{align*}
	where $\tilde{w}_1(x,v) = \frac{1+|v|}{w(x,v)\mu_E^{1/2}(x,v)}$.\\
	Then $(1+|v|)S_{\tilde{G}}(t)$ solves the linear Boltzmann equation \eqref{WLBE}. By the uniqueness in Theorem \ref{LT} with the initial datum $\tilde{h}_0$, we have
	\begin{align*}
		S_G(t)h_0 \equiv (1+|v|)S_{\tilde{G}}(t)\left(\frac{h_0}{1+|v|}\right).
	\end{align*}
	Thus we can rewrite $I_3$ as following:
	\begin{align*}
		&\int_0^t \int_s^t S_{G_\nu}(t-s')\int_{\mathbb{R}^3}e^{-\Phi(X(s'))}k_w(V(s'),v')\left\{(1+|v'|)S_{\tilde{G}}(s'-s)\left[e^{-\frac{\Phi}{2}}\frac{w}{1+|v'|}\Gamma\left(\frac{h^{(m)}}{w},\frac{h^{(m)}}{w}\right)\right]\right\}(s)\\
		&\quad \times dv'ds'ds
	\end{align*}
	From the proof of Lemma \ref{Gamest}, we obtain
	\begin{align*}
		\int_{\mathbb{R}^3}k_w(V(s'),v')(1+|v'|)dv' \le C\int_{\mathbb{R}^3}k_w(V(s'),v')\left\{|V(s')-v'|+|V(s')|\right\}dv' < +\infty.
	\end{align*} 
	Using Lemma \ref{L43}, Theorem \ref{LT}, and Lemma \ref{Gamest}, we can bound $I_3$ by
	\begin{equation} \label{ddd52}
	\begin{aligned}			
		&C_\Phi \int_0^t \int_s^t \exp\left\{-e^{-\|\Phi\|_\infty} \frac{\nu_0}{2} (t-s') \right\} \left\| \left\{S_{\tilde{G}}(s'-s)\left[e^{-\frac{\Phi}{2}}\frac{w}{1+|v'|}\Gamma\left(\frac{h^{(m)}}{w},\frac{h^{(m)}}{w}\right)\right]\right\}(s) \right\|_{L^\infty_{x,v}}ds'ds\\ 
		& \le C_\Phi \int_0^t \int_s^t \exp\left\{-e^{-\|\Phi\|_\infty} \frac{\nu_0}{2} (t-s') \right\} e^{-\lambda(s'-s)}\|h^{(m)}(s)\|_{L^\infty_{x,v}}^2ds'ds\\
		& \le C_\Phi e^{-\frac{\lambda}{2}t}\sup_{0\le s \le t} \left\{e^{\frac{\lambda}{2}s}\|h^{(m)}(s)\|_{L^\infty_{x,v}} \right\}^2 \int_0^t\int_s^t e^{-\frac{\lambda}{2}(t-s)}ds'ds\\
		& \le C_\Phi e^{-\frac{\lambda}{2}t}\sup_{0\le s \le t} \left\{e^{\frac{\lambda}{2}s}\|h^{(m)}(s)\|_{L^\infty_{x,v}} \right\}^2,
	\end{aligned}
	\end{equation}	
	where $\frac{\nu(v')}{1+|v'|} \le C$ for some constant $C$.\\
	We therefore deduce for $0\le \lambda \le \lambda_\infty$
	\begin{align*}
		e^{\lambda t}\|h^{(m+1)}(t)\|_{L^\infty_{x,v}} \le C\|h_0\|_{L^\infty_{x,v}}+Ce^{\frac{\lambda}{2}t}\sup_{0\le s \le \infty}\left\{e^{\frac{\lambda}{2}s}\|h^{(m)}(s)\|_{L^\infty_{x,v}}\right\}^2 \quad \text{for all} \ t \ge 0.
	\end{align*}
	We use an induction to get
	\begin{align*}
		\sup_m\sup_{0\le t < \infty}\left\{e^{\frac{\lambda}{2}t}\|h^{(m)}(t)\|_{L^\infty_{x,v}}\right\} \le C\|h_0\|_{L^\infty_{x,v}},
	\end{align*}
	where $\|h_0\|_{L^\infty_{x,v}}$ is sufficiently small.\\
	From \eqref{NIS}, we can derive the following system:
	\begin{align*}
		&\left\{\partial_{t}+ v\cdot \nabla_x  -\nabla_x \Phi(x) \cdot \nabla_v+e^{-\Phi(x)}\nu -e^{-\Phi(x)}K_w \right\} \left(h^{(m+1)}-h^{(m)}\right)\\
		& \quad = e^{-\frac{\Phi(x)}{2}} w\left[\Gamma \left(\frac{h^{(m)}}{w},\frac{h^{(m)}}{w}\right)-\Gamma \left(\frac{h^{(m-1)}}{w},\frac{h^{(m-1)}}{w}\right)\right]
	\end{align*}
	with zero initial value.\\
	We split 
	\begin{align*}
		\Gamma \left(\frac{h^{(m)}}{w},\frac{h^{(m)}}{w}\right)-\Gamma \left(\frac{h^{(m-1)}}{w},\frac{h^{(m-1)}}{w}\right)
		& = \Gamma \left(\frac{h^{(m)}-h^{(m-1)}}{w},\frac{h^{(m)}}{w}\right) -\Gamma \left(\frac{h^{(m-1)}}{w},\frac{h^{(m-1)}-h^{(m)}}{w}\right)
	\end{align*}
	and by the similar way in \eqref{ddd51} and \eqref{ddd52}, we obtain
	\begin{align*}
		e^{\frac{\lambda}{2}t}\|h^{(m+1)}(t)-h^{(m)}(t)\|_{L^\infty_{x,v}} &\le C\sup_{0\le s < \infty}\left\{e^{\frac{\lambda}{2}s}\|h^{(m)}(s)-h^{(m-1)}(s)\|_{L^\infty_{x,v}}\right\}\\
		& \quad \times \left[\sup_{0\le s < \infty}\left\{e^{\frac{\lambda}{2}s}\|h^{(m)}(s)\|_{L^\infty_{x,v}}\right\}+\sup_{0\le s < \infty}\left\{e^{\frac{\lambda}{2}s}\|h^{(m-1)}(s)\|_{L^\infty_{x,v}}\right\}\right]\\
		& \le  C\|h_0\|_{L^\infty_{x,v}}\sup_{0\le s < \infty}\left\{e^{\frac{\lambda}{2}s}\|h^{(m)}(s)-h^{(m-1)}(s)\|_{L^\infty_{x,v}}\right\},
	\end{align*}
	where $\|h_0\|_{L^\infty_{x,v}}$ is sufficiently small.\\
	Thus $h^{(m)}$ is a Cauchy sequence and the limit $h$ is a desired unique solution of \eqref{WPBE}.\\
	\indent Finally, we show the positivity of $F = \mu_E + \mu_E^{\frac{1}{2}}f$. Let $F^{(0)} \equiv \mu_{E}$. We consider the following iterative system:
	\begin{align} \label{NIS2}
		\left\{\partial_t  + v\cdot \nabla_x -\nabla_x \Phi(x) \cdot \nabla_v \right\}F^{(m+1)} + \nu(F^{(m)})F^{(m+1)} = Q_+(F^{(m)},F^{(m)})
	\end{align}
	with the initial value $F^{(m+1)}(0,x,v)=\mu_E(x,v)+\mu_E^{\frac{1}{2}}(x,v)f_0(x,v)$ and the diffuse reflection boundary condition
	\begin{align*}
		F^{(m+1)} = c_\mu \mu(v) \int_{n(x) \cdot v' >0} F^{(m+1)}(t,x,v') \{n(x) \cdot v'\} dv',
	\end{align*}
	where $\nu(F^{(m)}) = \int_{\mathbb{R}^3 \cross \mathbb{S}^2} B(u-v,\omega) F^{(m)}(u)d\omega du$.\\
	Set $f^{(m)} = \frac{F^{(m)}-\mu_E}{\sqrt{\mu_E}}$. From \eqref{NIS2}, we derive that
	\begin{align*}
		\left\{\partial_t + v\cdot \nabla_x -\nabla_x \Phi(x) \cdot \nabla_v +e^{-\Phi(x)}\nu(v)\right\}f^{(m+1)} &= e^{-\Phi(x)}Kf^{(m)} + e^{-\frac{\Phi(x)}{2}}\Gamma_+(f^{(m)},f^{(m)})\\
		&\quad  - e^{-\frac{\Phi(x)}{2}}\Gamma_-(f^{(m)},f^{(m+1)}).
	\end{align*}
	We can show that $h^{(m)} = wf^{(m)}$ is a Cauchy sequence in $L^\infty_{x,v}$, locally in time $[0, T_0]$, where $T_0$ depends on $\|h_0\|_{L^\infty_{x,v}}$ by the similar way in the proof of a solution existence.\\
	Assume that $F^{(m)} \ge 0$. Then $Q_+(F^{(m)},F^{(m)})\ge 0 $. By the deviation of Lemma \ref{L41}, if $t_1(t,x,v) \le 0$,
	\begin{align*}
		F^{(m+1)}(t,x,v) &= \exp{-\int_{0}^t e^{-\Phi(X(\tau))}\nu(F^{(m)})(\tau,X(\tau),V(\tau))d\tau}F_0(X(0),V(0))\\
		& \quad +\int_0^t \exp{-\int_{\tau}^t e^{-\Phi(X(\tau'))}\nu(F^{(m)})(\tau',X(\tau'),V(\tau'))d\tau'}Q_+(F^{(m)},F^{(m)})(\tau,X(\tau),V(\tau) d\tau\\
		& \ge 0.
	\end{align*}
	On the other hand, if $t_1(t,x,v) \ge 0$, then for $k \ge 2$,
	\begin{align*}
		F^{(m+1)}(t,x,v) &= \int_{t_1}^t \exp{-\int_{\tau}^t e^{-\Phi(X(\tau'))}\nu(F^{(m)})(\tau', X(\tau'),V(\tau')) d\tau'}Q_+(F^{(m)},F^{(m)})(\tau,X(\tau),V(\tau)) d\tau\\
		&\quad + \exp{-\int_{t_1}^t e^{-\Phi(X(\tau))}\nu(F^{(m)})(\tau, X(\tau),V(\tau)) d\tau} \mu_E(x_1, V(t_1))\\
		&\qquad \times \sum_{l=1}^{k-1}\int_{\prod_{j=1}^{k-1}\mathcal{V}_j}\mathbf{1}_{\{t_{l+1}\le 0 < t_l\}}
			F_0(X_l(0),V_l(0))d\Sigma_l^{(m)}(0)\\
		&\quad +\exp{-\int_{t_1}^t e^{-\Phi(X(\tau))}\nu(F^{(m)})(\tau, X(\tau),V(\tau)) d\tau} \mu_E(x_1, V(t_1))\\
		&\qquad \times \sum_{l=1}^{k-1} \int_0^{t_l}\int_{\prod_{j=1}^{k-1}\mathcal{V}_j}\mathbf{1}_{\{t_{l+1}\le 0 < t_l\}} Q_+(F^{(m)},F^{(m)})(\tau,X_l(\tau),V_l(\tau))d\Sigma_l^{(m)}(\tau)d\tau\\
		&\quad +\exp{-\int_{t_1}^t e^{-\Phi(X(\tau))}\nu(F^{(m)})(\tau, X(\tau),V(\tau)) d\tau} \mu_E(x_1, V(t_1))\\
		&\qquad \times \sum_{l=1}^{k-1}\int_{t_{l+1}}^{t_l}\int_{\prod_{j=1}^{k-1}\mathcal{V}_j}\mathbf{1}_{\{t_{l+1}>0\}}Q_+(F^{(m)},F^{(m)})(\tau,X_l(\tau),V_l(\tau))d\Sigma_l^{(m)}(\tau)d\tau\\
		&\quad +\exp{-\int_{t_1}^t e^{-\Phi(X(\tau))}\nu(F^{(m)})(\tau, X(\tau),V(\tau)) d\tau} \mu_E(x_1, V(t_1))\\
		&\qquad \times \int_{\prod_{j=1}^{k-1}\mathcal{V}_j}\mathbf{1}_{\{t_{k}>0\}} F^{(m+1)}(t_k,x_k, V_{k-1}(t_k))d\Sigma_{k-1}^{(m)}(t_k),
	\end{align*}
	where
	\begin{align*}
		d\Sigma_l^{(m)}(s) &= \left\{\prod_{j=l+1}^{k-1}d\sigma_j\right\}\left\{\exp{-\int_s^{t_l} e^{-\Phi(X_l(\tau))}\nu(F^{(m)})(\tau,X_l(\tau),V_l(\tau))d\tau} \frac{1}{\mu_E(x_l,v_l)} d\sigma_l\right\}\\
		&\quad \times \prod_{j=1}^{l-1}\left\{\exp{-\int_{t_{j+1}}^{t_j} e^{-\Phi(X_j(\tau))}\nu(F^{(m)})(\tau,X_j(\tau),V_j(\tau))d\tau}d\sigma_j\right\}.
	\end{align*}
	For any $\epsilon>0$, by Lemma \ref{Lsmall}, there exists $k$ large such that
	\begin{align*}
		\int_{\prod_{l=1}^{k-2}\mathcal{V}_l} \mathbf{1}_{\{t_{k-1}(t,x,v,v_1,v_2,...,v_{k-2})>0\}} \prod_{l=1}^{k-2}d\sigma_l \le \epsilon.
	\end{align*}
	It follows that
	\begin{align*}
		F^{(m+1)}(t,x,v) &\ge -C_\Phi\int_{\prod_{j=1}^{k-1}\mathcal{V}_j}\mathbf{1}_{\{t_{k}>0\}} \left|F^{(m+1)}(t_k,x_k, V_{k-1}(t_k))\right|dv_{k-1}\prod_{j=1}^{k-2}d\sigma_j	\\
		&\ge -C_\Phi\int_{\prod_{j=1}^{k-2}\mathcal{V}_j}\mathbf{1}_{\{t_{k-1}>0\}} \int_{\mathcal{V}_{k-1}}\{\mu_E+\mu_E^{\frac{1}{2}}|f^{(m+1)}|\}(t_k,x_k, V_{k-1}(t_k)) dv_{k-1}\prod_{j=1}^{k-2}d\sigma_j \\
		&\ge -C_\Phi\left(\sup_{m, 0\le s \le T_0}\|wf^{(m)}(s)\|_{L^{\infty}_{x,v}}\right) \epsilon.
	\end{align*}
	Since $\epsilon >0$ is arbitrary, $F^{(m+1)} \ge 0$ over $[0,T_0]$, and thus $F \ge 0$ over $[0, T_0]$. By uniqueness, we conclude that $F\ge 0$.
\end{proof}

\bigskip

\section{Large Amplitude Solution}
\label{LASEC}
\subsection{Preliminaries}
We define the relative entropy by
\begin{align} \label{relent}
	\mathcal{E}(F):= \int_{\Omega \cross \mathbb{R}^3} \left(\frac{F}{\mu_E}\log\frac{F}{\mu_E}-\frac{F}{\mu_E}+1\right) \mu_E dxdv.
\end{align}

The following lemma provides the global-in-time a priori estimate of the relative entropy. 
\begin{lemma} \label{entd}
	Assume $F$ satisfies the Boltzmann equation \eqref{BE} and the diffuse reflection boundary condition \eqref{DRBC}. Then
	\begin{align*}
		\mathcal{E}(F(t)) \le \mathcal{E}(F_0) \quad \text{for all } t\ge 0.
	\end{align*}
\end{lemma}
\begin{proof}
	We define a function $\Psi(s) = s\log s-s+1$ for $s>0$. Then $\Psi$ is nonnegative and convex on $(0,\infty)$ with $\Psi'(s)=\log s$. From \eqref{BE}, we obtain
	\begin{align*}
		\partial_t\left[\mu_E \Psi\left(\frac{F}{\mu_E}\right)\right]+\nabla_x \cdot \left[v \mu_E \Psi\left(\frac{F}{\mu_E}\right)\right]-\nabla_v \cdot \left[\nabla_x\Phi(x) \mu_E\Psi\left(\frac{F}{\mu_E}\right)\right] = Q(F,F) \log \frac{F}{\mu_E}.
	\end{align*}
	Taking an integration for $x \in \Omega$ and $v \in \mathbb{R}^3$,
	\begin{align*}
		\frac{d}{dt}\int_{\Omega \cross \mathbb{R}^3} \mu_E \Psi\left(\frac{F}{\mu_E}\right)dxdv + \int_{\partial\Omega \cross \mathbb{R}^3} \mu_E \Psi\left(\frac{F}{\mu_E}\right) \{n(x) \cdot v\}dS(x)dv = \int_{\Omega \cross \mathbb{R}^3} Q(F,F) \log F dxdv.
	\end{align*}
	For $x\in \partial \Omega$, we consider
	\begin{align*}
		I_x := \int_{\mathbb{R}^3} \mu_E \Psi \left(\frac{F}{\mu_E}\right) \{n(x) \cdot v\}dv.
	\end{align*}
	We can split
	\begin{align*}
		I_x = \int_{n(x) \cdot v>0}\mu_E \Psi \left(\frac{F}{\mu_E}\right) \{n(x) \cdot v\}dv +\int_{n(x)\cdot v <0} \mu_E \Psi \left(\frac{F}{\mu_E}\right) \{n(x) \cdot v\}dv =: I_1+I_2.
	\end{align*}
	Let $z(t,x) = e^{\Phi(x)}\int_{n(x) \cdot v' >0} F(t,x,v') \{n(x) \cdot v'\}dv'.$
	Taking a change of variables on $v$ and applying the boundary condition \eqref{DRBC}, we deduce that
	\begin{align*}
		I_2 &= -\int_{n(x)\cdot v>0}\left[c_\mu z(t,x) \log\left(c_\mu z(t,x)\right)-c_\mu z(t,x)+1\right]\mu_E \{n(x) \cdot v\} dv\\
		& = -\frac{1}{c_\mu}e^{\Phi(x)}\left[c_\mu z(t,x) \log\left(c_\mu z(t,x)\right)-c_\mu z(t,x)+1\right],
	\end{align*}
	where we have used the fact $\int_{n(x) \cdot v >0}\mu(v) \{n(x) \cdot v\} dv=1$.\\
	We use the Jensen's inequality to get
	\begin{align*}
		I_1 &= \frac{1}{c_\mu}e^{\Phi(x)} \int_{n(x) \cdot v>0} \Psi \left(\frac{F}{\mu_E}\right) c_\mu \mu(v) \{n(x) \cdot v\}dv\\
		& \ge \frac{1}{c_\mu}e^{\Phi(x)}\left[c_\mu z(t,x) \log\left(c_\mu z(t,x)\right)-c_\mu z(t,x)+1\right].
	\end{align*}
	Thus $I_x \ge 0$, which implies that
	\begin{align*}
		\int_{\partial\Omega \cross \mathbb{R}^3} \mu_E \Psi\left(\frac{F}{\mu_E}\right) \{n(x) \cdot v\}dS(x)dv \ge 0 \quad \text{for all } t\ge 0.
	\end{align*}
	From the fact $\int_{\Omega \cross \mathbb{R}^3} Q(F,F) \log F dxdv \le 0$, we conclude that
	\begin{align*}
		\mathcal{E}(F(t)) \le \mathcal{E}(F_0) \quad \text{for all } t\ge 0.
	\end{align*}

\end{proof}

\bigskip
We now introduce the following lemma which is a similar result in \cite{GuoQAM}. The following lemma means the relative entropy can control the $L^1$ norm and $L^2$ norm of $F-\mu_E$ over the different domains.
\begin{lemma} \cite{GuoQAM} \label{L1L2cont}
	Assume $F$ satisfies the Boltzmann equation \eqref{BE} and the diffuse reflection boundary condition \eqref{DRBC}. We have
	\begin{align*}
		\int_{\Omega \cross \mathbb{R}^3} \frac{1}{4\mu_E}|F-\mu_E|^2 \mathbf{1}_{|F-\mu_E|\le \mu_E} dxdv + \int_{\Omega \cross \mathbb{R}^3} \frac{1}{4}|F-\mu_E| \mathbf{1}_{|F-\mu_E| > \mu_E} dxdv \le \mathcal{E}(F_0)
	\end{align*}
	for all $t \ge 0$.
	Moreover, if we write $F = \mu_E + \mu_E^{\frac{1}{2}} f$, then
	\begin{align*}
		\int_{\Omega \cross \mathbb{R}^3} \frac{1}{4}|f|^2 \mathbf{1}_{|f|\le \sqrt{\mu_E}} dxdv + \int_{\Omega \cross \mathbb{R}^3} \frac{\sqrt{\mu_E}}{4}|f| \mathbf{1}_{|f| > \sqrt{\mu_E}} dxdv \le \mathcal{E}(F_0)
	\end{align*}
	for all $t \ge 0$.
\end{lemma}

\bigskip
The below lemma implies we can bound the weighted Gamma gain term $w\Gamma_+$ by the product of the weighted $L^\infty$ norm and the $L^2$ norm with a good kernel. Because the proof of the lemma is nearly identical to \cite[Lemma\ 2.2]{DW2019}, we omit this proof.
\begin{lemma} \cite{DW2019} \label{Gam+est}
	There is a generic contant $C_{\beta} >0$ such that
	\begin{align*}
		|w(x,v)\Gamma_+(f,f)(x,v)| \le \frac{C_{\beta}\|wf\|_{L^{\infty}_{x,v}}}{1+|v|} \left(\int_{\mathbb{R}^3} (1+|\eta|)^4 |f(\eta)|^2 d\eta\right)^{\frac{1}{2}}
	\end{align*}
	for all $x \in \Omega$, $v \in \mathbb{R}^3$. In particular,
	\begin{align*}
		|w(x,v)\Gamma_+(f,f)(x,v)| \le \frac{C_{\beta}\|wf\|_{L^{\infty}_{x,v}}^2}{1+|v|}
	\end{align*}
	for all $x \in \Omega$, $v \in \mathbb{R}^3$, where $C_{\beta}>0$ is a constant.
\end{lemma}

\bigskip

We define $R(f)$ by
\begin{align} \label{Rfdef}
	R(f)(t,x,v) = \int_{\mathbb{R}^3 \cross \mathbb{S}^2} B(v-u,\omega) \left[\mu_E(x,u)+\mu_E^{\frac{1}{2}}(x,u)f(x,u)\right]d\omega du.
\end{align}
The full perturbed Boltzmann equation \eqref{WPBE} with weight $w$ becomes
\begin{align} \label{WPBBE}
	\partial_t h+v\cdot \nabla_x h -\nabla_x \Phi(x) \cdot \nabla_v h + R(f)h = e^{-\Phi(x)}K_wh+ e^{-\frac{\Phi(x)}{2}}w\Gamma_+\left(\frac{h}{w},\frac{h}{w}\right).
\end{align}

\indent Now, let us consider the following equation:
\begin{align} \label{WLBBE}
	\partial_t h+v\cdot \nabla_x h -\nabla_x \Phi(x) \cdot \nabla_v h + R(\varphi)h = 0,
\end{align}
where $\varphi = \varphi(t,x,v)$ is a given function satisfying
\begin{align} \label{phc}
	\mu_E(x,v) + \mu_E^{\frac{1}{2}}(x,v) \varphi(t,x,v) \ge 0, \quad \|\varphi(t)\|_{L^\infty_{x,v}}< \infty.
\end{align} 
We denote by $S_{G_\varphi}(t)h_0$ the semigroup of a solution to the equation \eqref{WLBBE} with initial datum $h_0$ and the diffuse boundary condition
\begin{align}\label{wwLBEBC}
	h(t,x,v)|_{\gamma_{-}} =  \frac{1}{\tilde{w}(x,v)} \int_{v'\cdot n(x)>0} h(t,x,v')\tilde{w}(x,v')d\sigma,
\end{align}
where $\tilde{w}(x,v)$ is defined in \eqref{tildeweight}.
The following lemma provides an useful $L^\infty$ estimate of $S_{G_\varphi}(t)h_0$ in a finite time.

\begin{lemma} \label{Phiest}
	Assume that $\rho > 1$ is sufficiently large and $\varphi$ satisfies the condition \eqref{phc}. Let $h_0 \in L^\infty_{x.v}$. Then there exists a unique solution $h(t)=S_{G_\varphi}(t)h_0$ to the equation \eqref{WLBBE} with initial datum $h_0$ and the boundary condition \eqref{wwLBEBC}. Moreover, there is a constant $C_3$, depending on $\beta$ and $\Phi$, such that
	\begin{align} \label{Phire}
		\|S_{G_\varphi}(t)h_0\|_{L^\infty_{x,v}} \le C_3 \rho^{\frac{5}{4}}\|h_0\|_{L^\infty_{x,v}} \quad \text{for all } 0\le t \le \rho.
	\end{align}
\end{lemma}
\begin{proof} 
	Given any $m \ge 1$, we construct a solution to
	\begin{align} \label{qq11}
		\{\partial_t + v\cdot \nabla_x -\nabla_x \Phi(x) \cdot \nabla_v+  R(\varphi)\}h^{(m)}=0,
	\end{align}
	with the boundary and initial condition
	\begin{equation} \label{qq12}	
	\begin{aligned} 
		&h^{(m)}(t,x,v) = \left\{1-\frac{1}{m}\right\}\frac{1}{\tilde{w}(x,v)}\int_{n(x)\cdot v'>0} \left[h^{(m)}(t,x,v')\right]\tilde{w}(x,v')d\sigma(x),\\
		&h^{(m)}(0,x,v) = h_0 \mathbf{1}_{\{|v|\le m\}}.
	\end{aligned}
	\end{equation}
	Setting $\tilde{h}^{(m)}(t,x,v) = \tilde{w}(x,v) h^{(m)}(t,x,v)$, the equation \eqref{qq11} and the condition \eqref{qq12} become
	\begin{align*}
		&\{\partial_t + v\cdot \nabla_x -\nabla_x \Phi(x) \cdot \nabla_v+ R(\varphi)\}\tilde{h}^{(m)}=0,\\
		&\tilde{h}^{(m)}(t,x,v) = \left\{1-\frac{1}{m}\right\}\int_{n(x)\cdot v'>0} \tilde{h}^{(m)}(t,x,v') d\sigma(x),\\
		& \tilde{h}^{(m)}(0,x,v) = \tilde{h}_0 \mathbf{1}_{\{|v|\le m\}}.
	\end{align*}
	Since $\int d\sigma =1$, the boundary operator maps $L^\infty_{x,v}$ to $L^\infty_{x,v}$ with a norm bounded by $1-\frac{1}{m}$, and
	\begin{align*}
		\|\tilde{h}^{(m)}(0)\|_{L^\infty_{x,v}} \le C_{m}\|h_0\|_{L^\infty_{x,v}} < \infty.
	\end{align*}
	By the deviation of Lemma \ref{L42}, there exists a solution $\tilde{h}^{(m)}(t,x,v) \in L^\infty_{x,v}$ to the above equation, and $h^{(m)}$ is bounded because $h^{(m)}=\frac{1}{\tilde{w}(x,v)} \tilde{h}^{(m)}$.\\
	From now on, we show the uniform $L^\infty_{x,v}$ bound for $h$. Let $0\le t \le \rho$.
	If $t_1(t,x,v) \le  0$, we know
	\begin{align*}
		\left(S_{G_\varphi}(t)h_0\right)(x,v) = \exp\left\{-\int_0^t R(\varphi)(s,X(s),V(s)) ds\right\}h_0(X(0),V(0)),
	\end{align*}
	and \eqref{Phire} is valid.\\
	We consider the case $t_1(t,x,v) > 0$. Recall the definition of the iterated integral in \eqref{iterint}. By the deviation of Lemma \ref{L41}, we deduce 
	\begin{align*}
		\left|h^{(m)}(t,x,v)\right| &\le \frac{\exp{-\int_{t_1}^t R(\varphi)(\tau,X(\tau),V(\tau))d\tau}}{\tilde{w}(x_1,V(t_1))}\sum_{l=1}^{k-1}\int_{\prod_{j=1}^{k-1}\mathcal{V}_j} \mathbf{1}_{\{t_{l+1}\le 0 < t_l\}} \left|h^{(m)}(0,X_l(0),V_l(0))\right|d\Sigma_l^\varphi(0)\\
		&\quad +\frac{\exp{-\int_{t_1}^t R(\varphi)(\tau,X(\tau),V(\tau))d\tau}}{\tilde{w}(x_1,V(t_1))}\int_{\prod_{j=1}^{k-1}\mathcal{V}_j} \mathbf{1}_{\{t_k>0\}} \left|h^{(m)}(t_k,x_k,V_{k-1}(t_k))\right|d\Sigma_{k-1}^\varphi(t_k)\\
		&=: I_1+I_2,
	\end{align*}
	where
	\begin{align*}
		d\Sigma_l^\varphi(s) &= \left\{\prod_{j=l+1}^{k-1}d\sigma_j\right\}\left\{ \exp{-\int_s^{t_l} R(\varphi)(\tau, X_l(\tau),V_l(\tau))d\tau} \tilde{w}(x_l,v_l)d\sigma_l\right\}\\
		&\quad \cross \prod_{j=1}^{l-1}\left\{\exp{-\int_{t_{j+1}}^{t_j} R(\varphi)(\tau, X_j(\tau),V_j(\tau))d\tau}d\sigma_j\right\}.
	\end{align*}
	First of all, we consider $I_2$. Using the boundary condition
	\begin{align*}
		h^{(m)}(t_k,x_k,V_{k-1}(t_k)) = \left\{1-\frac{1}{m}\right\}\frac{1}{\tilde{w}(x_k,V_{k-1}(t_k))}\int_{\mathcal{V}_k} h^{(m)}(t_k,x_k,v_k)\tilde{w}(x_k,v_k)d\sigma_k
	\end{align*}
and the fact $h^{(m)}(t_k,x_k,v_k) = \mathbf{1}_{\{t_{k+1} \le 0 <t_k\}}\exp\left\{-\int_0^{t_k}R(\varphi)(s,X(s),V(s))ds\right\}h^{(m)}(0,X_k(0),V_k(0))$\\ $+\mathbf{1}_{\{t_{k+1}>0\}}h^{(m)}(t_k,x_k,v_k)$,
\begin{align*}
	I_2 &\le \frac{\exp{-\int_{t_1}^t R(\varphi)(\tau,X(\tau),V(\tau))d\tau}}{\tilde{w}(x_1,V(t_1))}\int_{\prod_{j=1}^{k}\mathcal{V}_j} \mathbf{1}_{\{t_{k+1} \le 0 <t_k\}}\left|h^{(m)}(0,X_k(0),V_k(0))\right|d\Sigma_{k}^\varphi(0)\\
	&\quad +\frac{\exp{-\int_{t_1}^t R(\varphi)(\tau,X(\tau),V(\tau))d\tau}}{\tilde{w}(x_1,V(t_1))}\int_{\prod_{j=1}^{k}\mathcal{V}_j} \mathbf{1}_{\{t_{k+1}>0\}} \left|h^{(m)}(t_k,x_k,v_k)\right|d\Sigma_{k}^\varphi(t_k)\\
	&=:J_1+J_2.
\end{align*}
Since $t_1(t_k,x_k,v_k)>0$ over $\{t_{k+1}>0\}$, we deduce that
\begin{align} \label{qq13}
	\mathbf{1}_{\{t_{k+1}>0\}}\left|h^{(m)}(t_k,x_k,v_k)\right| \le \sup_{x,v}\left|h^{(m)}(t_k,x,v)\mathbf{1}_{\{t_1>0\}}\right|.
\end{align}
We know that the exponential in $d\Sigma_l^\varphi(s)$ is bounded by $1$. By Lemma \ref{Lsmall}, we can choose $C_1,C_2>0$ such that for $k=C_1\rho^{\frac{5}{4}}$
\begin{align} \label{qq14}
	\int_{\prod_{j=1}^{k-1}\mathcal{V}_j} \mathbf{1}_{\{t_k(t,x,v,v_1,v_2,...,v_{k-1})>0\}}\prod_{j=1}^{k-1}d\sigma_j \le \left(\frac{1}{2}\right)^{C_2 \rho^{\frac{5}{4}}}
\end{align}
Using \eqref{qq13} and \eqref{qq14}, we obtain
\begin{align*}
	J_2 &\le \frac{1}{\tilde{w}(x_1,V(t_1))}\left\|h^{(m)}(t_k)\mathbf{1}_{\{t_1>0\}}\right\|_{L^\infty_{x,v}} \int_{\prod_{j=1}^k\mathcal{V}_j} \mathbf{1}_{\{t_k>0\}}\tilde{w}(x_k,v_k)\prod_{j=1}^kd\sigma_j\\
	&\le C_\Phi \sup_{0\le s \le t \le \rho} \left\{\left\|h^{(m)}(s) \mathbf{1}_{\{t_1>0\}}\right\|_{L^\infty_{x,v}}\right\}\left(\int_{\prod_{j=1}^{k-1}\mathcal{V}_j}\mathbf{1}_{\{t_{k}>0\}}\prod_{j=1}^{k-1}d\sigma_j\right)\left(\int_{\mathcal{V}_k}\tilde{w}(x_k,v_k) d\sigma_k\right)\\
	&\le C_\Phi \left(\frac{1}{2}\right)^{C_2 \rho^{\frac{5}{4}}}\sup_{0\le s \le t \le \rho} \left\{\left\|h^{(m)}(s) \mathbf{1}_{\{t_1>0\}}\right\|_{L^\infty_{x,v}}\right\}.
\end{align*}
On the other hand, we consider $I_1$ and $J_1$. By inserting $\int_{\mathcal{V}_k}d\sigma_k=1$ into $I_1$, we get
\begin{align*}
	I_1+J_1 &= \frac{\exp{-\int_{t_1}^t R(\varphi)(\tau,X(\tau),V(\tau))d\tau}}{\tilde{w}(x_1,V(t_1))}\sum_{l=1}^{k}\int_{\prod_{j=1}^{k}\mathcal{V}_j} \mathbf{1}_{\{t_{l+1}\le 0 < t_l\}} \left|h^{(m)}(0,X_l(0),V_l(0))\right|d\Sigma_l^\varphi(0)\\
	&\le \frac{1}{\tilde{w}(x_1,V(t_1))}\left\|h^{(m)}(0)\right\|_{L^\infty_{x,v}}\sum_{l=1}^{k}\int_{\prod_{j=1}^{k}\mathcal{V}_j} \mathbf{1}_{\{t_{l+1}\le 0 < t_l\}} \left\{\prod_{j=l+1}^k d\sigma_j\right\}\left\{\tilde{w}(x_l,v_l) d\sigma_l\right\}\left\{\prod_{j=1}^{l-1} d\sigma_j\right\}.
\end{align*}
Now, we fix $l$ and consider the $l$-th term
\begin{align*}
	\int_{\prod_{j=1}^{k}\mathcal{V}_j} \mathbf{1}_{\{t_{l+1}\le 0 < t_l\}} \left\{\prod_{j=l+1}^k d\sigma_j\right\}\left\{\tilde{w}(x_l,v_l) d\sigma_l\right\}\left\{\prod_{j=1}^{l-1} d\sigma_j\right\} &\le \int_{\prod_{j=1}^{l-1}\mathcal{V}_j} \left(\int_{\mathcal{V}_l}\tilde{w}(x_l,v_l) d\sigma_l\right)\left\{\prod_{j=1}^{l-1} d\sigma_j\right\}\\
	&\le C_\Phi.
\end{align*}
Summing $1\le l \le k$, it follows that
\begin{align*}
	I_1+J_1 \le C_1 \rho^{\frac{5}{4}}C_\Phi \frac{1}{\tilde{w}(x_1,V(t_1))}\left\|h^{(m)}(0)\right\|_{L^\infty_{x,v}} \le C_\Phi \rho^{\frac{5}{4}}  \left\|h^{(m)}(0)\right\|_{L^\infty_{x,v}}.
\end{align*}
Thus, we deduce that for $0 \le t \le \rho$,
\begin{align*}
	\left|h^{(m)}(t,x,v)\mathbf{1}_{\{t_1>0\}}\right| &\le C_\Phi \left(\frac{1}{2}\right)^{C_2\rho^{\frac{5}{4}}}\sup_{0\le s \le t \le \rho} \left\{\left\|h^{(m)}(s) \mathbf{1}_{\{t_1>0\}}\right\|_{L^\infty_{x,v}}\right\}+C_\Phi \rho^{\frac{5}{4}}\left\|h^{(m)}(0)\right\|_{L^\infty_{x,v}}.
\end{align*}
Choosing sufficiently large $\rho > 0$ such that $C_\Phi \left(\frac{1}{2}\right)^{C_2\rho^{\frac{5}{4}}} \le \frac{1}{2}$,
\begin{align*} 
	\sup_{0\le t \le \rho}\left\{\left\|h^{(m)}(t)\mathbf{1}_{\{t_1>0\}}\right\|_{L^\infty_{x,v}}\right\} \le C_\Phi \rho^{\frac{5}{4}}\left\|h^{(m)}(0)\right\|_{L^\infty_{x,v}} = C_\Phi \rho^{\frac{5}{4}}\left\|h_0\right\|_{L^\infty_{x,v}}.
\end{align*}
Therefore $\left(h^{(m)}\right)$ is uniformly bounded, and the sequence has weak* limit in $L^\infty_{x,v}$. Letting $m \rightarrow \infty$, we conclude the existence of a solution and the uniform bound for the solution.
\end{proof}

\bigskip
In \eqref{WLBBE} and \eqref{phc}, we can take $\varphi$ as $f$, and we consider the semigroup $S_{G_f}(t)$. The below lemma gives a $L^\infty$ estimate to $S_{G_f}(t)h_0$ when we have some assumption related to $R(f)$.
\begin{lemma} \label{LL62} 
	Assume that $\rho > 1$ is sufficiently large and $h_0 \in L^\infty_{x.v}$. Let $h(t,x,v)=S_{G_f}(t)h_0$ be the solution to the equation \eqref{WLBBE} with $\varphi = f$, initial datum $h_0$, and the boundary condition \eqref{wwLBEBC}. Suppose that
	\begin{align*}
		R(f)(t,x,v) \ge \frac{1}{2}e^{-\Phi(x)}\nu(v) \quad \text{for all }(t,x,v)\in [0,\infty)\cross \Omega \cross \mathbb{R}^3.
	\end{align*}    
	It holds that
	\begin{align*}
		\|S_{G_f}(t)h_0\|_{L^\infty_{x,v}} \le C_3 \rho^{\frac{5}{4}}\exp\left\{-e^{-\|\Phi\|_{\infty}}\frac{\nu_0}{2}t\right\}\|h_0\|_{L^\infty_{x,v}} \quad \text{for all } 0\le t \le \rho,
	\end{align*}
	where $C_3$ is a constant in Lemma \ref{Phiest}.
	Furthermore, there is $C_{\rho, \Phi, \beta} \ge 1$ such that
	\begin{align*}
		\|S_{G_f}(t)h_0\|_{L^\infty_{x,v}} \le C_{\rho, \Phi, \beta}\exp\left\{-e^{-\|\Phi\|_{\infty}}\frac{\nu_0}{4}t\right\}\|h_0\|_{L^\infty_{x,v}} \quad \text{for all } t \ge 0.
	\end{align*}
\end{lemma}
\begin{proof}
	The proof of this Lemma is similar to the proof of Lemma \ref{L43} and Lemma \ref{Phiest}.
\end{proof}

\bigskip
Theorem \ref{LocExist} implies the local-in-time existence of the full perturbed Boltzmann equation with given initial data and the diffuse reflection boundary condition.
\begin{theorem} \label{LocExist}
	Assume that $\rho >1$ is sufficiently large. Suppose that $F_0 = \mu_E(x,v) + \mu_E^{\frac{1}{2}}(x,v)f_0(x,v) \ge 0$ and $\|wf_0\|_{L^\infty_{x,v}} < \infty$. Then there is a time $\hat{t}_0 := (\hat{C}_{\rho,\Phi}[1+\|wf_0\|_{L^\infty_{x,v}}])^{-1} >0$ such that there exists a unique solution $F(t,x,v) = \mu_E(x,v) + \mu_E^{\frac{1}{2}}(x,v)f(t,x,v) \ge 0$  on time interval $[0,\hat{t}_0]$ to the Boltzmann equation \eqref{BE} with initial datum $F_0$ and the diffuse reflection boundary condition \eqref{DRBC} satisfying
	\begin{align*}
		\sup_{0\le t \le \hat{t}_0}\|wf(t)\|_{L^\infty_{x,v}} \le 2C_3 \rho^{\frac{5}{4}} \|wf_0\|_{L^\infty_{x,v}},
	\end{align*}
	where $\hat{C}_{\rho,\Phi} >0$ is a constant depending on $\rho>0$ and $C_3$ is a constant in Lemma \ref{Phiest}. 
\end{theorem}

\bigskip

\subsection{A priori estimate}
Fix $\rho>0$  and $\beta>5$ so that
\begin{align} \label{rhoest}
	\left(C_3 \rho^{\frac{5}{4}}\right)^\frac{1}{\rho} \le \exp{e^{-\|\Phi\|_\infty} \frac{\nu_0}{2}},
\end{align}
where $C_3 >0$ is a constant in Lemma \ref{Phiest}. The inequality \eqref{rhoest} will be used to drive an estimate in $L^\infty_{x,v}$ to a semigroup $S_{G_f}$ in Theorem \ref{L61}. \\
We make the a priori assumption as following : 
\begin{align} \label{Ape}
	\sup_{0\le t \le T} \|h(t)\|_{L^\infty_{x,v}} \le \bar{M},
\end{align}
where $\bar{M} \ge 1$ and $T>0$. The a priori assumption \eqref{Ape} is crucial to achieve our main aim. The constants $\bar{M}$ and $T$ are determined in Section \ref{NASLAS}. See \eqref{aa30} and \eqref{aa31}. We note that $\bar{M}$ depends only on an initial amplitude $M_0 \ge 1$, which can be large, but does not depend on $T$ and the solution.

\bigskip
 Through the next lemma, we can estimate the term $R(f)$ to deduce the exponential decay in $L^\infty_{x,v}$. The estimate \eqref{Rfestimate} in this lemma is called the $R(f)$ {\it estimate}. Recall the definition \ref{Backwardexit}, especially \eqref{timecycle}, \eqref{Xtrajcycle}, and \eqref{Vtrajcycle}, as well as the definition of the iterated integral \eqref{iterint}.
\begin{lemma} \label{Rfest}
	Assume the a priori assumption \eqref{Ape}. Let $f$ be a solution to \eqref{PBE} with initial datum $f_0$ and the boundary condition \eqref{PDRBC}. Let $M_0$ be an initial amplitude. There exists a generic constant $C_4 \ge 1$ such that for given $T_1 > \tilde{t}$ with
	\begin{align*}
		\tilde{t} = \frac{2}{\nu_0 e^{-\|\Phi\|_\infty}} \log (C_4M_0)>0,
	\end{align*}
	there is a small constant $\epsilon_0 = \epsilon_0(\bar{M}, T_1) >0$ so that if $\mathcal{E}(F_0) \le \epsilon_0$, then
	\begin{align} \label{Rfestimate}
		R(f)(t,x,v) \ge \frac{1}{2}e^{-\Phi(x)}\nu(v) \quad \text{for all }(t,x,v)\in [\tilde{t},T_1)\cross \Omega \cross \mathbb{R}^3.
	\end{align}
\end{lemma}
\begin{proof}
	We recall
	\begin{align*}
		R(f)(t,x,v) &= \int_{\mathbb{R}^3 \cross \mathbb{S}^2} B(v-u,\omega) [\mu_E(x,u) + \mu_E^{\frac{1}{2}}(x,u)f(t,x,u)]d\omega du\\
		& = e^{-\Phi(x)}\left\{\nu(v)+\int_{\mathbb{R}^3 \cross \mathbb{S}^2} e^{\frac{\Phi(x)}{2}}B(v-u,\omega)\mu^{\frac{1}{2}}(u)f(t,x,u) d\omega du\right\}.
	\end{align*}
	Here,
	\begin{align*}
		\left|\int_{\mathbb{R}^3 \cross \mathbb{S}^2} e^{\frac{\Phi(x)}{2}}B(v-u,\omega)\mu^{\frac{1}{2}}(u)f(t,x,u) d\omega du\right| \le C_5 \nu(v) \int_{\mathbb{R}^3} e^{-\frac{|u|^2}{8}}|f(t,x,u)|du
	\end{align*}
	for some constant $C_5$.\\
	If it holds that
	\begin{align} \label{aa61}
		\int_{\mathbb{R}^3} e^{-\frac{|u|^2}{8}}|h(t,x,u)|du\le  \frac{1}{2C_5} \quad \text{for all } t\ge \tilde{t},\ x\in \Omega,
	\end{align}
	where $\tilde{t}>0$ is a constant to be suitably chosen, then we can complete the proof of this lemma. Thus it suffices to show \eqref{aa61}. We set $h(t,x,v) =w(x,v)f(t,x,v)$.
	Then by Duhamel principle, we derive
	\begin{align*}
		\int_{\mathbb{R}^3}e^{-\frac{|v|^2}{8}}|h(t,x,v)|dv
		& \le \int_{\mathbb{R}^3} e^{-\frac{|v|^2}{8}} |\left(S_{G_\nu}(t)h_0\right)(t,x,v)|dv\\
		& \quad + \int_0^{t}  \int_{\mathbb{R}^3}e^{-\frac{|v|^2}{8}}\left|(S_{G_\nu}(t-s)e^{-\Phi}K_wh(s))(t,x,v)\right|dvds\\ 
		& \quad + \int_0^{t} \int_{\mathbb{R}^3}e^{-\frac{|v|^2}{8}}\left|(S_{G_\nu}(t-s)e^{-\frac{\Phi}{2}}w\Gamma(f,f)(s))(t,x,v)\right|dvds\\
		& =: I_1+I_2+I_3.
	\end{align*}
	From Lemma \ref{L43}, we can easily get
	\begin{align} \label{qq31}
		I_1 \le C\|S_{G_\nu}(t)h_0\|_{L^\infty_{x,v}} \le C_\Phi \exp\left\{-e^{-\|\Phi\|_\infty} \frac{\nu_0}{2} t \right\}\|h_0\|_{L^\infty_{x,v}}.
	\end{align}
	From now on, let us estimate $I_2 +I_3$.\\
	$\mathbf{Case\ 1:}$ $|v| \ge R$.\\
	Using Corollary \ref{L43C}, Lemma \ref{Kest}, and Lemma \ref{Gamest}, $I_2+I_3$ in this case is bounded by
	\begin{align*}
		&\int_{|v| \ge R} e^{-\frac{|v|^2}{8}} \int_0^t \left\{|(S_{G_\nu}(t-s)e^{-\Phi}K_wh(s))(t,x,v)|+|(S_{G_\nu}(t-s)e^{-\frac{\Phi}{2}}w\Gamma(f,f)(s))(t,x,v)|\right\}dsdv\\
		& \le \frac{C_\Phi}{R}\sup_{0\le s \le t}\left\{\|h(s)\|_{L^\infty_{x,v}}+\|h(s)\|_{L^\infty_{x,v}}^2\right\}.
	\end{align*}
	It remains to estimate $I_2+I_3$ in the case $|v| \le R$. By Lemma \ref{L41}, we deduce 
	\begin{align*}
		&(S_{G_\nu}(t-s)e^{-\Phi}K_wh(s))(t,x,v)\\
		&=\mathbf{1}_{\{t_1 \le s\}} \exp{-\int_{s}^t e^{-\Phi(X(\tau))}\nu(V(\tau))d\tau} e^{-\Phi(X(s))}K_wh(s,X(s),V(s))\\
		& \quad + \frac{\exp{-\int_{t_1}^t e^{-\Phi(X(\tau))}\nu(V(\tau))d\tau}}{\tilde{w}(x_1,V(t_1))} \sum_{l=1}^{k-1} \int_{\prod_{j=1}^{k-1} \mathcal{V}_j} \mathbf{1}_{\{t_{l+1} \le s <t_l\}} e^{-\Phi(X_l(s))}K_wh(s,X_l(s),V_l(s))d\Sigma_l(s)\\
		& \quad + \frac{\exp{-\int_{t_1}^t e^{-\Phi(X(\tau))}\nu(V(\tau))d\tau}}{\tilde{w}(x_1,V(t_1))}\int_{\prod_{j=1}^{k-1} \mathcal{V}_j}\mathbf{1}_{\{t_k>s\}} (S_{G_\nu}(t-s)e^{-\Phi}K_wh(s))(t_k,x_k,V_{k-1}(t_k))d\Sigma_{k-1}(t_k)\\
		& =: J_{11}+J_{12}+J_{13},
	\end{align*}
	and
	\begin{align*}
		&(S_{G_\nu}(t-s)e^{-\frac{\Phi}{2}}w\Gamma(f,f)(s))(t,x,v)\\
		&=\mathbf{1}_{\{t_1 \le s\}} \exp{-\int_{s}^t e^{-\Phi(X(\tau))}\nu(V(\tau))d\tau} e^{-\frac{\Phi(X(s))}{2}}w\Gamma(f,f)(s,X(s),V(s))\\
		& \quad + \frac{\exp{-\int_{t_1}^t e^{-\Phi(X(\tau))}\nu(V(\tau))d\tau}}{\tilde{w}(x_1,V(t_1))} \sum_{l=1}^{k-1} \int_{\prod_{j=1}^{k-1} \mathcal{V}_j} \mathbf{1}_{\{t_{l+1} \le s <t_l\}} e^{-\frac{\Phi(X_l(s))}{2}}w\Gamma(f,f)(s,X_l(s),V_l(s))d\Sigma_l(s)\\
		& \quad + \frac{\exp{-\int_{t_1}^t e^{-\Phi(X(\tau))}\nu(V(\tau))d\tau}}{\tilde{w}(x_1,V(t_1))}\int_{\prod_{j=1}^{k-1} \mathcal{V}_j}\mathbf{1}_{\{t_k>s\}} (S_{G_\nu}(t-s)e^{-\frac{\Phi}{2}}w\Gamma(f,f)(s))(t_k,x_k,V_{k-1}(t_k))d\Sigma_{k-1}(t_k)\\
		& =: J_{21}+J_{22}+J_{23}.
	\end{align*}
	Now, let us consider $J_{13}$ and $J_{23}$. Let $\epsilon >0$. By Lemma \ref{Lsmall}, we can choose $k = k(\epsilon, T_1)$ large such that
	\begin{align} \label{qq21}
		\int_{\prod_{j=1}^{k-2}}\mathbf{1}_{\{t_{k-1} >s\}} \prod_{j=1}^{k-2}d\sigma_j < \epsilon.
	\end{align} 
	Using Corollary \ref{L43C} and \eqref{qq21}, we obtain
	\begin{align*}
		\int_0^t \int_{|v| \le R} e^{-\frac{|v|^2}{8}} \left(J_{13}+J_{23}\right)dvds
		\le \epsilon \ C_\Phi \sup_{0\le s \le t} \left\{\|h(s)\|_{L^\infty_{x,v}}+\|h(s)\|_{L^\infty_{x,v}}^2\right\}.
	\end{align*}
	Firstly, let us consider $J_{11}$.
	\begin{align*}
		&\int_0^t \int_{|v|\le R} e^{-\frac{|v|^2}{8}}|J_{11}|dvds\\
		& \le \int_0^t \int_{|v| \le R} \exp\left\{-e^{-\|\Phi\|_\infty}\nu_0(t-s)\right\} \mathbf{1}_{\{t_1 \le s\}}e^{-\frac{|v|^2}{8}}\left\{\int_{|v'|\le 2R}+\int_{|v'| \ge 2R}\right\} |k_w(V(s),v')|\\
		& \quad \times |h(s,X(s),v'|dv'dvds
	\end{align*}
	\newline
	$\mathbf{Case\ 2\  of\  J_{11}:}$ $|v| \le R$ and $|v'|\ge 2R$ with $R \gg 2\sqrt{2\|\Phi\|_{\infty}}$.\\
	Note that $|v-v'| \ge R$. From \eqref{sese}, it holds that
	\begin{align*}
		&|V(s)-v'| \ge |v-v'|-|V(s)-v| \ge R-\frac{R}{2}=\frac{R}{2}.
	\end{align*}
	Then we have
	\begin{equation} \label{aa62}
	\begin{aligned}
		&|k_w(V(s),v')| \le e^{-\frac{R^2}{64}}|k_w(V(s),v')|e^{\frac{1}{16}|V(s)-v'|^2}.
	\end{aligned}
	\end{equation}
	This yields from Lemma \ref{Kest},
	\begin{equation} \label{aa63}
	\begin{aligned}
		&\int_{|v'| \ge 2R} |k_w(V(s),v')|e^{\frac{1}{16}|V(s)-v'|^2}dv' < C,
	\end{aligned}
	\end{equation}
	for some constant $C$.
	Thus we use \eqref{aa62} and \eqref{aa63} to obtain
	\begin{equation} \label{ta61}
	\begin{aligned}
		&\int_0^t \int_{|v|\le R} e^{-\frac{|v|^2}{8}}|J_{11}|dvds\\
		& \le e^{-\frac{R^2}{64}} \sup_{0 \le s \le t}\|h(s)\|_{L^\infty_{x,v}}\int_0^t\int_{|v|\le R} \exp\left\{-e^{-\|\Phi\|_\infty}\nu_0(t-s)\right\}e^{-\frac{|v|^2}{8}}\int_{|v'|\ge 2R} |k_w(V(s),v')e^{\frac{|V(s)-v'|^2}{16}}|dv'dvds\\
		& \le C_\Phi e^{-\frac{R^2}{64}} \sup_{0 \le s \le t}\|h(s)\|_{L^\infty_{x,v}}.
	\end{aligned}
	\end{equation}
	$\mathbf{Case\ 3\  of\  J_{11}:}$ $|v| \le R$ and $|v'|\le 2R$.\\
	Since $k_w(v,v')$ has possible integrable singularity of $\frac{1}{|v-v'|}$, we can choose smooth function $k_R(v,v')$ with compact support such that
	\begin{align*}
		\sup_{|v|\le 2R}\int_{|v'| \le 2R} \left|k_R(v,v')-k_w(v,v')\right|dv' \le \frac{1}{R}.
	\end{align*}
	We split 
	\begin{align*}
		k_w(V(s),v') = \{k_w(V(s),v') - k_R(V(s),v')\}+k_R(V(s),v').
	\end{align*}
	Then it follows that
	\begin{align*}
		&\int_0^t \int_{|v|\le R} e^{-\frac{|v|^2}{8}}|J_{11}|dvds\\
		& \le \frac{C}{R} \sup_{0\le s \le t} \|h(s)\|_{L^\infty_{x,v}}\\
		&\quad  + C_{R, \Phi} \int_0^t \int_{|v|\le R} e^{-\frac{|v|^2}{8}} \exp\left\{-e^{-\|\Phi\|_\infty}\nu_0(t-s)\right\}\int_{|v'|\le 2R}|h(s,X(s),v')|dv'dvds\\
		& =: L_{11}+L_{12},
	\end{align*}
	where we have used the fact $|k_R(V(s),v')|\le C_R$.\\
	In this case, we recall that $X(s) = X(s;t,x,v)$.
	Since the potential is time dependent, we have
	\begin{align*}
		X(s;t,x,v) = X(s-t+T_1;T_1,x,v).
	\end{align*}
	for all $0\le s \le t \le T_1$.\\
	By Lemma \ref{cpl}, the term $L_{12}$ becomes
	\begin{align}
		&C_{R,\Phi}\sum_{i_1}^{M_1}\sum_{I_3}^{(M_3)^3} \int_0^{T_1} \mathbf{1}_{\mathcal{P}_{i_1}^{T_1}}(s-t+T_1) \exp\left\{-e^{-\|\Phi\|_\infty}\nu_0(t-s)\right\} \nonumber \\ \label{aa64}
		& \times \int_{|v|\le R, |v'|\le 2R}e^{-\frac{|v|^2}{8}}\mathbf{1}_{\mathcal{P}_{I_3}^v}(v)|h(s,X(s-t+T_1;T_1,x,v),v')|dv'dvds.
	\end{align}
	Let $\tilde{\epsilon}>0$.
	From Lemma \ref{cpl}, we have the following partitions:
	\begin{align*}
		&\left\{(s-t+T_1,x,v)\in \mathcal{P}_{i_1}^{T_1}\cross \mathcal{P}_{I_2}^{\Omega}\cross \mathcal{P}_{I_3}^{v} : \det\left(\frac{dX}{dv}(s-t+T_1;T_1,x,v)\right)=0\right\}\\
		& \subset \bigcup_{j=1}^3 \left\{(s-t+T_1,x,v)\in \mathcal{P}_{i_1}^{T_1}\cross \mathcal{P}_{I_2}^{\Omega}\cross \mathcal{P}_{I_3}^{v} :s-t+T_1\in \left(t_{j,i_1,I_2,I_3}-\frac{\tilde{\epsilon}}{4M_1},t_{j,i_1,I_2,I_3}+\frac{\tilde{\epsilon}}{4M_1}\right)\right\}.
	\end{align*}
	Thus for each $i_1$,$I_2$, and $I_3$, we split $\mathbf{1}_{\mathcal{P}_{i_1}^{T_1}}(s-t+T_1)$ as
	\begin{align} \label{aa65}
		&\mathbf{1}_{\mathcal{P}_{i_1}^{T_1}}(s-t+T_1)\mathbf{1}_{\cup_{j=1}^3(t_{j,i_1,I_2,I_3}-\frac{\tilde{\epsilon}}{4M_1},t_{j,i_1,I_2,I_3}+\frac{\tilde{\epsilon}}{4M_1})}(s-t+T_1)\\ \label{aa66}
		& +\mathbf{1}_{\mathcal{P}_{i_1}^{T_1}}(s-t+T_1)\left\{1-\mathbf{1}_{\cup_{j=1}^3(t_{j,i_1,I_2,I_3}-\frac{\tilde{\epsilon}}{4M_1},t_{j,i_1,I_2,I_3}+\frac{\tilde{\epsilon}}{4M_1})}(s-t+T_1)\right\}.
	\end{align}
	$\mathbf{Case\ 3 \ (\romannumeral 1)\ of\ J_{11} :}$ The integration \eqref{aa64} corresponding to \eqref{aa65} is bounded by
	\begin{align}
		&C_{R,\Phi}\sum_{i_1}^{M_1}\sum_{I_3}^{(M_3)^3} \sum_{j=1}^3 \underbrace{\int_0^{T_1} \mathbf{1}_{\mathcal{P}_{i_1}^{T_1}}(s-t+T_1) \mathbf{1}_{(t_{j,i_1,I_2,I_3}-\frac{\tilde{\epsilon}}{4M_1},t_{j,i_1,I_2,I_3}+\frac{\tilde{\epsilon}}{4M_1})}(s-t+T_1)}_{(*5)}\nonumber \\ \label{aa67}
		& \times   \int_{|v|\le R}\mathbf{1}_{\mathcal{P}_{I_3}^v}(v)\int_{|v'|\le 2R}|h(s,X(s-t+T_1;T_1,x,v),v')|dv'dvds.
	\end{align}
	Here, $(*5)$ is bounded by
	\begin{align} \label{qq22}
		\int_0^{T_1} \mathbf{1}_{\mathcal{P}_{i_1}^{T_1}}(s-t+T_1)\mathbf{1}_{(t_{j,i_1,I_2,I_3}-\frac{\tilde{\epsilon}}{4M_1},t_{j,i_1,I_2,I_3}+\frac{\tilde{\epsilon}}{4M_1})}(s-t+T_1) ds  \le \frac{\tilde{\epsilon}}{2M_1}.
	\end{align}
	From the partition of the velocity domain $[-4R,4R]^3$ in Lemma \ref{cpl}, we have
	\begin{align} \label{qq23}
		\sum_{I^3}^{(M_3)^3} \mathbf{1}_{\mathcal{P}_{I_3}^v}(v)\mathbf{1}_{\{|v|\le R\}}(v) = \mathbf{1}_{\{|v|\le R\}}(v).
	\end{align}
	From \eqref{qq22} and \eqref{qq23}, \eqref{aa67} is bounded by 
	\begin{align*}
		& C_{R,\Phi} \sup_{0\le s \le t} \|h(s)\|_{L^\infty_{x,v}}\ \sum_{i_1}^{M_1}  \int_0^{T_1} \mathbf{1}_{\mathcal{P}_{i_1}^{T_1}}(s-t+T_1)\mathbf{1}_{(t_{j,i_1,I_2,I_3}-\frac{\tilde{\epsilon}}{4M_1},t_{j,i_1,I_2,I_3}+\frac{\tilde{\epsilon}}{4M_1})}(s-t+T_1) ds\\
		&\le \tilde{\epsilon}\ C_{R,\Phi}\ \sup_{0\le s \le t} \|h(s)\|_{L^\infty_{x,v}}.
	\end{align*}
	\newline
	$\mathbf{Case\ 3 \ (\romannumeral 2)\ of  \ J_{11}:}$ The integration \eqref{aa64} corresponding to \eqref{aa66} is bounded by
	\begin{align}
		&C_{R,\Phi}\sum_{i_1}^{M_1}\sum_{I_3}^{(M_3)^3} \int_0^{t} \mathbf{1}_{\mathcal{P}_{i_1}^{T_1}}(s-t+T_1) \left\{1-\mathbf{1}_{\cup_{j=1}^3(t_{j,i_1,I_2,I_3}-\frac{\tilde{\epsilon}}{4M_1},t_{j,i_1,I_2,I_3}+\frac{\tilde{\epsilon}}{4M_1})}(s-t+T_1)\right\} \nonumber \\ \label{aa68}
		&\times \exp\left\{-e^{-\|\Phi\|_\infty}\nu_0(t-s)\right\}\underbrace{\int_{|v|\le R}\mathbf{1}_{\mathcal{P}_{I_3}^v}(v)\int_{|v'|\le 2R}|h(s,X(s-t+T_1;T_1,x,v),v')|dv'dv}_{(\# 5)}ds.
	\end{align}
	By Lemma \ref{cpl}, we have made a change of variables $v \rightarrow y:=X(s-t+T_1;T_1,x,v)$ satisfying 
	\begin{align*}
		\det\left(\frac{dX}{dv}(s-t+T_1;T_1,x,v)\right)> \delta_*
	\end{align*}
	and the term $(\# 5)$ is bounded by
	\begin{align*}
		\int_{|v|\le R}\int_{|v'|\le 2R}|h(s,X(s-t+T_1;T_1,x,v),v')|dv'dv
		&\le \frac{1}{\delta_*}\int_{\Omega} \int_{|v'|\le 2R}|h(s,y,v')|dv'dy\\
		&\le \frac{C_{R,\Phi}}{\delta_*} \left(\int_{\Omega} \int_{|v'|\le 2R}|h(s,y,v')|^2dv'dy\right)^{\frac{1}{2}},
	\end{align*}
	where we have used the Cauchy-Schwarz inequality.
	Then \eqref{aa68} is bounded by
	\begin{align*}
		C_{R,\Phi,M_1,M_3,\delta_*} \int_0^{t}\exp\left\{-e^{-\|\Phi\|_\infty}\nu_0(t-s)\right\}\left(\int_{\Omega} \int_{|v'|\le 2R}|h(s,y,v')|^2dv'dy\right)^{\frac{1}{2}}ds.
	\end{align*}
	From Lemma \ref{L1L2cont} and Young's inequality, we obtain 
	\begin{align}
		&C_{R,\Phi,M_1,M_3,\delta_*} \int_{\Omega} \int_{|v'|\le 2R}|h(s,y,v')|^2dv'dy \nonumber\\
		& \le C_{R,\Phi,M_1,M_3,\delta_*} \left( \int_{\Omega} \int_{|v'|\le 2R}|h(s,y,v')|^2 \mathbf{1}_{|F-\mu_E|\le\mu_E} dv'dy+\int_{\Omega} \int_{|v'|\le 2R}|h(s,y,v')|^2 \mathbf{1}_{|F-\mu_E|>\mu_E} dv'dy\right)\nonumber\\
		& \le C_{R,\Phi,M_1,M_3,\delta_*}\Bigg(\int_{\Omega} \int_{|v'|\le 2R}|f(s,y,v')|^2 \mathbf{1}_{|F-\mu_E|\le\mu_E} dv'dy+\sup_{0\le s\le t} \|h(s)\|_{L^\infty_{x,v}}\nonumber\\
		& \quad \times \int_{\Omega} \int_{|v'|\le 2R}|h(s,y,v')| \mathbf{1}_{|F-\mu_E|>\mu_E} dv'dy\Bigg)\nonumber\\
		& \le C_{R,\Phi,M_1,M_3,\delta_*}\mathcal{E}(F_0) + C_{R,\Phi,M_1,M_3,\delta_*}\sup_{0\le s\le t} \|h(s)\|_{L^\infty_{x,v}}\int_{\Omega} \int_{|v'|\le 2R}\mu_E^{\frac{1}{2}}(y,v')|f(s,y,v')| \mathbf{1}_{|F-\mu_E|>\mu_E} dv'dy\nonumber\\
		& \le  C_{R,\Phi,M_1,M_3,\delta_*}\mathcal{E}(F_0) + C_{R,\Phi,M_1,M_3,\delta_*}\sup_{0\le s\le t} \|h(s)\|_{L^\infty_{x,v}}\mathcal{E}(F_0)\nonumber\\ \label{L2cont}
		& \le \frac{C_\Phi}{R} \sup_{0\le s\le t} \|h(s)\|_{L^\infty_{x,v}}^2 + C_{R,\Phi,M_1,M_3,\delta_*}\left[\mathcal{E}(F_0)+\mathcal{E}(F_0)^2\right]
	\end{align}
	Hence \eqref{aa68} is bounded by
	\begin{align*}
		\frac{C_\Phi}{R} \sup_{0\le s\le t} \|h(s)\|_{L^\infty_{x,v}}^2 + C_{R,\Phi,M_1,M_3,\delta_*}\left[\mathcal{E}(F_0)+\mathcal{E}(F_0)^2\right].
	\end{align*}
	\newline
	Next, let us consider $J_{12}$.
	\begin{align*}
		&\int_0^t \int_{|v|\le R}e^{-\frac{|v|^2}{8}}|J_{12}|dvds\\
		& \le C_\Phi \int_0^t \int_{|v|\le R} e^{-\frac{|v|^2}{8}} \exp\left\{-e^{-\|\Phi\|_\infty}\nu_0(t-s)\right\} \sum_{l=1}^{k-1} \int_{\prod_{j=1}^{l-1}\mathcal{V}_j}\int_{\mathcal{V}_l} \mathbf{1}_{\{t_{l+1} \le s < t_l\}} \int_{\mathbb{R}^3} |k_w(V_l(s),v')|\\
		&\quad \times |h(s,X_l(s),v')|dv' e^{-\frac{|v_l|^2}{8}} dv_l \left\{\prod_{j=1}^{l-1}d\sigma_j\right\}dvds.
	\end{align*}
	Fix $l$. We divide the following term into 3 cases:
	\begin{align}
		&C_\Phi \int_0^t \int_{|v|\le R} e^{-\frac{|v|^2}{8}} \exp\left\{-e^{-\|\Phi\|_\infty}\nu_0(t-s)\right\}  \int_{\prod_{j=1}^{l-1}\mathcal{V}_j}\int_{\mathcal{V}_l} \mathbf{1}_{\{t_{l+1} \le s < t_l\}} \int_{\mathbb{R}^3} |k_w(V_l(s),v')|\nonumber\\ \label{aa611}
		&\quad \times |h(s,X_l(s),v')|dv' e^{-\frac{|v_l|^2}{8}}dv_l \left\{\prod_{j=1}^{l-1}d\sigma_j\right\}dvds.
	\end{align}
	\newline
	$\mathbf{Case\ 2\  of\  J_{12}:}$ $|v_l| \ge R$ with $R \gg 2\sqrt{2\|\Phi\|_{\infty}}$.\\
	By \eqref{sese}, we get
	\begin{align*}
		|V_l(s)| \ge |v_l| -\sqrt{2\|\Phi\|_{\infty}} \ge \frac{R}{2}.
	\end{align*}
	From Lemma \ref{Kest}, we have
	\begin{align*}
		\int_{\mathbb{R}^3}|k_w(V_l(s),v')|dv' \le \frac{C_\Phi}{1+R}.
	\end{align*}
	Then \eqref{aa611} in this case is bounded by
	\begin{equation} \label{aa612}
	\begin{aligned}
		&\frac{C_\Phi}{1+R} \sup_{0\le s \le t}\|h(s)\|_{L^\infty_{x,v}} \int_0^t \int_{|v|\le R} e^{-\frac{|v|^2}{8}}\exp\left\{-e^{-\|\Phi\|_\infty}\nu_0(t-s)\right\}dvds\\
		&\le \frac{C_\Phi}{1+R}\sup_{0\le s \le t}\|h(s)\|_{L^\infty_{x,v}}.
	\end{aligned}
	\end{equation}
	\newline
	$\mathbf{Case\ 3\  of\  J_{12}:}$ $|v_l| \le R$ and $|v'|\ge 2R$.\\
	Note that $|v_l-v'| \ge R$. From \eqref{sese}, it holds that
	\begin{align*}
		&|V_l(s)-v'| \ge |v-v'|-|V_l(s)-v| \ge R-\frac{R}{2}=\frac{R}{2}.
	\end{align*}
	Then we have 
	\begin{equation} \label{aa613}
	\begin{aligned}
		&|k_w(V_l(s),v')| \le e^{-\frac{R^2}{64}}|k_w(V_l(s),v')|e^{\frac{1}{16}|V_l(s)-v'|^2}.
	\end{aligned}
	\end{equation}
	This yields from Lemma \ref{Kest},
	\begin{equation} \label{aa614}
	\begin{aligned}
		&\int_{|v'| \ge 2R} |k_w(V_l(s),v')|e^{\frac{1}{16}|V_l(s)-v'|^2}dv' < C,
	\end{aligned}
	\end{equation}
	for some constant $C$.
	Thus we use \eqref{aa613} and \eqref{aa614} to obtain
	\begin{equation} \label{ta62}
	\begin{aligned}
		\int_0^t \int_{|v|\le R} e^{-\frac{|v|^2}{8}}|J_{12}|dvds \le C_\Phi e^{-\frac{R^2}{64}} \sup_{0 \le s \le t}\|h(s)\|_{L^\infty_{x,v}}.
	\end{aligned}
	\end{equation}
	$\mathbf{Case\ 4\  of\  J_{12}:}$ $|v_l| \le R$ and $|v'|\le 2R$.\\
	Since $k_w(v,v')$ has possible integrable singularity of $\frac{1}{|v-v'|}$, we can choose smooth function $k_R(v,v')$ with compact support such that
	\begin{align*}
		\sup_{|v|\le 2R}\int_{|v'| \le 2R} \left|k_R(v,v')-k_w(v,v')\right|dv' \le \frac{1}{R}.
	\end{align*}
	We split 
	\begin{align*}
		k_w(V_l(s),v') = \{k_w(V_l(s),v') - k_R(V_l(s),v')\}+k_R(V_l(s),v').
	\end{align*}
	Then it follows that
	\begin{align*}
		&\int_0^t \int_{|v|\le R} e^{-\frac{|v|^2}{8}}|J_{12}|dvds\\
		& \le \frac{C}{R} \sup_{0\le s \le t} \|h(s)\|_{L^\infty_{x,v}}\\
		&\quad  + C_{R, \Phi} \int_0^t \int_{|v|\le R}  \exp\left\{-e^{-\|\Phi\|_\infty}\nu_0(t-s)\right\}\int_{\prod_{j=1}^{l-1}\mathcal{V}_j}\int_{|v_l|\le R}\int_{|v'|\le 2R}|h(s,X_l(s),v')|dv'dv_l \left\{\prod_{j=1}^{l-1}d\sigma_j\right\}dvds\\
		& =: L_{21}+L_{22},
	\end{align*}
	where we have used the fact $|k_R(V_l(s),v')|\le C_R$.\\
	In this case, we recall that $X_l(s) = X(s;t_l,x_l,v_l)$.
	Since the potential is time dependent, we have
	\begin{align*}
		X(s;t_l,x_l,v_l) = X(s-t_l+T_1;T_1,x_l,v_l).
	\end{align*}
	for all $0\le s \le t_l \le T_1$.\\
	By Lemma \ref{cpl}, the term $L_{22}$ becomes
	\begin{align}
		&C_{R,\Phi}\sum_{i_1}^{M_1}\sum_{I_3}^{(M_3)^3} \int_0^{T_1} \mathbf{1}_{\mathcal{P}_{i_1}^{T_1}}(s-t_l+T_1) \exp\left\{-e^{-\|\Phi\|_\infty}\nu_0(t-s)\right\}\int_{|v|\le R}\int_{\prod_{j=1}^{l-1}\mathcal{V}_j}\int_{|v_l|\le R}\int_{|v'|\le 2R} \nonumber \\ \label{aa615}
		& \times \mathbf{1}_{\mathcal{P}_{I_3}^v}(v_l)|h(s,X(s-t_l+T_1;T_1,x_l,v_l),v')| dv'dv_l \left\{\prod_{j=1}^{l-1}d\sigma_j\right\}dvds.
	\end{align}
	From Lemma \ref{cpl}, we have the following partitions:
	\begin{align*}
		&\left\{(s-t_l+T_1,x_l,v_l)\in \mathcal{P}_{i_1}^{T_1}\cross \mathcal{P}_{I_2}^{\Omega}\cross \mathcal{P}_{I_3}^{v} : \det\left(\frac{dX}{dv_l}(s-t_l+T_1;T_1,x_l,v_l)\right)=0\right\}\\
		& \subset \bigcup_{j=1}^3 \left\{(s-t_l+T_1,x_l,v_l)\in \mathcal{P}_{i_1}^{T_1}\cross \mathcal{P}_{I_2}^{\Omega}\cross \mathcal{P}_{I_3}^{v} :s-t_l+T_1\in \left(t_{j,i_1,I_2,I_3}-\frac{\tilde{\epsilon}}{4M_1},t_{j,i_1,I_2,I_3}+\frac{\tilde{\epsilon}}{4M_1}\right)\right\}.
	\end{align*}
	Thus for each $i_1$,$I_2$, and $I_3$, we split $\mathbf{1}_{\mathcal{P}_{i_1}^{T_1}}(s-t_l+T_1)$ as
	\begin{align} \label{aa616}
		&\mathbf{1}_{\mathcal{P}_{i_1}^{T_1}}(s-t_l+T_1)\mathbf{1}_{\cup_{j=1}^3(t_{j,i_1,I_2,I_3}-\frac{\tilde{\epsilon}}{4M_1},t_{j,i_1,I_2,I_3}+\frac{\tilde{\epsilon}}{4M_1})}(s-t_l+T_1)\\ \label{aa617}
		& +\mathbf{1}_{\mathcal{P}_{i_1}^{T_1}}(s-t_l+T_1)\left\{1-\mathbf{1}_{\cup_{j=1}^3(t_{j,i_1,I_2,I_3}-\frac{\tilde{\epsilon}}{4M_1},t_{j,i_1,I_2,I_3}+\frac{\tilde{\epsilon}}{4M_1})}(s-t_l+T_1)\right\}.
	\end{align}
	\bigskip
	$\mathbf{Case\ 4 \ (\romannumeral 1)\ of\ J_{12} :}$ The integration \eqref{aa615} corresponding to \eqref{aa616} is bounded by
	\begin{align}
		&C_{R,\Phi}\sum_{i_1}^{M_1}\sum_{I_3}^{(M_3)^3} \sum_{j=1}^3 \underbrace{\int_0^{T_1} \mathbf{1}_{\mathcal{P}_{i_1}^{T_1}}(s-t_l+T_1) \mathbf{1}_{(t_{j,i_1,I_2,I_3}-\frac{\tilde{\epsilon}}{4M_1},t_{j,i_1,I_2,I_3}+\frac{\tilde{\epsilon}}{4M_1})}(s-t_l+T_1)}_{(*6)}\nonumber \\ \label{aa618}
		& \times   \int_{|v|\le R}\int_{\prod_{j=1}^{l-1}\mathcal{V}_j}\int_{|v_l|\le R}\int_{|v'|\le 2R}\mathbf{1}_{\mathcal{P}_{I_3}^v}(v_l)|h(s,X(s-t_l+T_1;T_1,x_l,v_l),v')| dv'dv_l \left\{\prod_{j=1}^{l-1}d\sigma_j\right\}dvds.
	\end{align}
	Here, the term $(*6)$ is bounded by
	\begin{align} \label{qq24}
		\int_0^{T_1} \mathbf{1}_{\mathcal{P}_{i_1}^{T_1}}(s-t_l+T_1)\mathbf{1}_{(t_{j,i_1,I_2,I_3}-\frac{\tilde{\epsilon}}{4M_1},t_{j,i_1,I_2,I_3}+\frac{\tilde{\epsilon}}{4M_1})}(s-t_l+T_1) ds
		 \le \frac{\tilde{\epsilon}}{2M_1}.
	\end{align}
	From the partition of the velocity domain $[-4R,4R]^3$ in Lemma \ref{cpl}, we have
	\begin{align} \label{qq25}
		\sum_{I^3}^{(M_3)^3} \mathbf{1}_{\mathcal{P}_{I_3}^v}(v_l)\mathbf{1}_{\{|v_l|\le R\}}(v_l) = \mathbf{1}_{\{|v_l|\le R\}}(v_l).
	\end{align}
	From \eqref{qq24} and \eqref{qq25}, \eqref{aa618} is bounded by
	\begin{align*}
		& C_{R,\Phi} \sup_{0\le s \le t} \|h(s)\|_{L^\infty_{x,v}}\ \sum_{i_1}^{M_1}  \int_0^{T_1} \mathbf{1}_{\mathcal{P}_{i_1}^{T_1}}(s-t_l+T_1)\mathbf{1}_{(t_{j,i_1,I_2,I_3}-\frac{\tilde{\epsilon}}{4M_1},t_{j,i_1,I_2,I_3}+\frac{\tilde{\epsilon}}{4M_1})}(s-t_l+T_1) ds\\
		&\le \tilde{\epsilon} \ C_{R,\Phi} \sup_{0\le s \le t} \|h(s)\|_{L^\infty_{x,v}}.
	\end{align*}
	\newline
	$\mathbf{Case\ 4 \ (\romannumeral 2)\ of  \ J_{12}:}$ The integration \eqref{aa615} corresponding to \eqref{aa617} is bounded by
	\begin{align}
		&C_{R,\Phi}\sum_{i_1}^{M_1}\sum_{I_3}^{(M_3)^3} \int_0^{t} \mathbf{1}_{\mathcal{P}_{i_1}^{T_1}}(s-t_l+T_1) \left\{1-\mathbf{1}_{\cup_{j=1}^3(t_{j,i_1,I_2,I_3}-\frac{\tilde{\epsilon}}{4M_1},t_{j,i_1,I_2,I_3}+\frac{\tilde{\epsilon}}{4M_1})}(s-t_l+T_1)\right\} \nonumber \\ 
		&\times \exp\left\{-e^{-\|\Phi\|_\infty}\nu_0(t-s)\right\}\int_{|v|\le R}\int_{\prod_{j=1}^{l-1}\mathcal{V}_j}\underbrace{\int_{|v_l|\le R}\int_{|v'|\le 2R}\mathbf{1}_{\mathcal{P}_{I_3}^v}(v_l)|h(s,X(s-t_l+T_1;T_1,x_l,v_l),v')|dv'dv_l}_{(\# 6)} \nonumber \\ \label{aa619}
		&\times  \left\{\prod_{j=1}^{l-1}d\sigma_j\right\}dvds.
	\end{align}
	By Lemma \ref{cpl}, we have made a change of variables $v_l \rightarrow y:=X(s-t_l+T_1;T_1,x_l,v_l)$ satisfying 
	\begin{align*}
		\det\left(\frac{dX}{dv_l}(s-t_l+T_1;T_1,x_l,v_l)\right)> \delta_*
	\end{align*}
	and the term $(\# 6)$ is bounded by
	\begin{align*}
		\int_{|v_l|\le R}\int_{|v'|\le 2R}|h(s,X(s-t_l+T_1;T_1,x_l,v_l),v')|dv'dv_l &\le \frac{1}{\delta_*}\int_{\Omega} \int_{|v'|\le 2R}|h(s,y,v')|dv'dy\\
		&\le \frac{C_{R,\Phi}}{\delta_*} \left(\int_{\Omega} \int_{|v'|\le 2R}|h(s,y,v')|^2dv'dy\right)^{\frac{1}{2}}.
	\end{align*}
	Then \eqref{aa619} is bounded by
	\begin{align*}
		C_{R,\Phi,M_1,M_3,\delta_*} \int_0^{t}\exp\left\{-e^{-\|\Phi\|_\infty}\nu_0(t-s)\right\}\left(\int_{\Omega} \int_{|v'|\le 2R}|h(s,y,v')|^2dv'dy\right)^{\frac{1}{2}}ds.
	\end{align*}
	Hence from \eqref{L2cont}, \eqref{aa619} is bounded by
	\begin{align*}
		\frac{C_\Phi}{R} \sup_{0\le s\le t} \|h(s)\|_{L^\infty_{x,v}}^2 + C_{R,\Phi,M_1,M_3,\delta_*}\left[\mathcal{E}(F_0)+\mathcal{E}(F_0)^2\right].
	\end{align*}
	\newline
	It remains to estimate $J_{21}$ and $J_{22}$. By Lemma \ref{Gam+est} and Cauchy-Schwartz inequality, we deduce that
	\begin{equation} \label{aa620}
	\begin{aligned}
	|w\Gamma(f,f)(s,y,v)| &\le  |w\Gamma_+(f,f)(s,y,v)|+|w\Gamma_-(f,f)(s,y,v)|\\
	& \le C_{\beta}\|h(s)\|_{L^{\infty}_{x,v}} \left(\int_{\mathbb{R}^3} (1+|\eta|)^4 |f(s,y,\eta)|^2 d\eta\right)^{\frac{1}{2}}+C_\beta \nu(v)\|h(s)\|_{L^{\infty}_{x,v}} \int_{\mathbb{R}^3} e^{-\frac{|u|^2}{8}}|f(s,y,u)|du\\
	& \le C_{\beta}\nu(v)\|h(s)\|_{L^\infty_{x,v}}\left(\int_{\mathbb{R}^3}(1+|\eta|)^{-2\beta+4}|h(s,y,\eta)|^2 d\eta \right)^{\frac{1}{2}}.
	\end{aligned}
	\end{equation}
	We use \eqref{aa620} to estimate $J_{21}$.
	\begin{align}
		&\int_0^t \int_{|v|\le R} e^{-\frac{|v|^2}{8}}|J_{21}|dvds\nonumber\\\label{aa621}
		&\le C_\Phi \int_0^t  \exp\left\{-e^{-\|\Phi\|_\infty}\nu_0(t-s)\right\} \|h(s)\|_{L^\infty_{x,v}}\left(\int_{|v|\le R}\int_{\mathbb{R}^3}e^{-\frac{|v|^2}{16}}(1+|\eta|)^{-2\beta+4}|h(s,X(s),\eta)|^2 d\eta dv \right)^{\frac{1}{2}}ds.
	\end{align}
	\newline
	$\mathbf{Case\ 2\  of\  J_{21}:}$ $|\eta| \ge R$.\\
	The term \eqref{aa621} in this case is bounded by
	\begin{align*}
		& C_\Phi \int_0^t  \exp\left\{-e^{-\|\Phi\|_\infty}\nu_0(t-s)\right\} \sup_{0\le s \le t} \|h(s)\|^2_{L^\infty_{x,v}}\left(\int_{|v|\le R}\int_{|\eta|\ge R}e^{-\frac{|v|^2}{16}}(1+|\eta|)^{-2\beta+4} d\eta dv \right)^{\frac{1}{2}}ds\\
		& \le \frac{C_\Phi}{R}\sup_{0\le s \le t} \|h(s)\|^2_{L^\infty_{x,v}}.
	\end{align*}
	\newline
	$\mathbf{Case\ 3\  of\  J_{21}:}$ $|\eta| \le R$.\\
	In this case, we recall that $X(s) = X(s;t,x,v)$.
	Since the potential is time dependent, we have
	\begin{align*}
		X(s;t,x,v) = X(s-t+T_1;T_1,x,v).
	\end{align*}
	for all $0\le s \le t \le T_1$.\\
	By Lemma \ref{cpl}, the term \eqref{aa621} becomes
	\begin{align}
		&C_{R,\Phi}\sum_{i_1}^{M_1}\sum_{I_3}^{(M_3)^3} \int_0^{T_1} \mathbf{1}_{\mathcal{P}_{i_1}^{T_1}}(s-t+T_1) \exp\left\{-e^{-\|\Phi\|_\infty}\nu_0(t-s)\right\} \|h(s)\|_{L^\infty_{x,v}} \nonumber \\ \label{aa622}
		& \times \Bigg(\int_{|v|\le R}\int_{|\eta|\le R} \mathbf{1}_{\mathcal{P}_{I_3}^v}(v) e^{-\frac{|v|^2}{16}}(1+|\eta|)^{-2\beta+4}|h(s,X(s-t+T_1;T_1,x,v),\eta)|^2 d\eta dv \Bigg)^{\frac{1}{2}}ds .
	\end{align}
	From Lemma \ref{cpl}, we have the following partitions:
	\begin{align*}
		&\left\{(s-t+T_1,x,v)\in \mathcal{P}_{i_1}^{T_1}\cross \mathcal{P}_{I_2}^{\Omega}\cross \mathcal{P}_{I_3}^{v} : \det\left(\frac{dX}{dv}(s-t+T_1;T_1,x,v)\right)=0\right\}\\
		& \subset \bigcup_{j=1}^3 \left\{(s-t+T_1,x,v)\in \mathcal{P}_{i_1}^{T_1}\cross \mathcal{P}_{I_2}^{\Omega}\cross \mathcal{P}_{I_3}^{v} :s-t+T_1\in \left(t_{j,i_1,I_2,I_3}-\frac{\tilde{\epsilon}}{4M_1},t_{j,i_1,I_2,I_3}+\frac{\tilde{\epsilon}}{4M_1}\right)\right\}.
	\end{align*}
	Thus for each $i_1$,$I_2$, and $I_3$, we split $\mathbf{1}_{\mathcal{P}_{i_1}^{T_1}}(s-t+T_1)$ as
	\begin{align} \label{aa623}
		&\mathbf{1}_{\mathcal{P}_{i_1}^{T_1}}(s-t+T_1)\mathbf{1}_{\cup_{j=1}^3(t_{j,i_1,I_2,I_3}-\frac{\tilde{\epsilon}}{4M_1},t_{j,i_1,I_2,I_3}+\frac{\tilde{\epsilon}}{4M_1})}(s-t+T_1)\\ \label{aa624}
		& +\mathbf{1}_{\mathcal{P}_{i_1}^{T_1}}(s-t+T_1)\left\{1-\mathbf{1}_{\cup_{j=1}^3(t_{j,i_1,I_2,I_3}-\frac{\tilde{\epsilon}}{4M_1},t_{j,i_1,I_2,I_3}+\frac{\tilde{\epsilon}}{4M_1})}(s-t+T_1)\right\}.
	\end{align}
	\bigskip
	$\mathbf{Case\ 3 \ (\romannumeral 1)\ of\ J_{21} :}$ The integration \eqref{aa622} corresponding to \eqref{aa623} is bounded by
	\begin{align}
		&C_{R,\Phi}\sum_{i_1}^{M_1}\sum_{I_3}^{(M_3)^3} \sum_{j=1}^3 \underbrace{\int_0^{T_1} \mathbf{1}_{\mathcal{P}_{i_1}^{T_1}}(s-t+T_1) \mathbf{1}_{(t_{j,i_1,I_2,I_3}-\frac{\tilde{\epsilon}}{4M_1},t_{j,i_1,I_2,I_3}+\frac{\tilde{\epsilon}}{4M_1})}(s-t+T_1)}_{(*7)}\|h(s)\|_{L^\infty_{x,v}}\nonumber \\ \label{aa625}
		& \times \Bigg(\int_{|v|\le R}\int_{|\eta|\le R} \mathbf{1}_{\mathcal{P}_{I_3}^v}(v) e^{-\frac{|v|^2}{16}}(1+|\eta|)^{-2\beta+4}|h(s,X(s-t+T_1;T_1,x,v),\eta)|^2 d\eta dv \Bigg)^{\frac{1}{2}}ds.
	\end{align}
	Here, the term $(*7)$ is bounded by
	\begin{align}\label{qq26}
		\int_0^{T_1} \mathbf{1}_{\mathcal{P}_{i_1}^{T_1}}(s-t+T_1)\mathbf{1}_{(t_{j,i_1,I_2,I_3}-\frac{\tilde{\epsilon}}{4M_1},t_{j,i_1,I_2,I_3}+\frac{\tilde{\epsilon}}{4M_1})}(s-t+T_1) ds\le \frac{\tilde{\epsilon}}{2M_1}.
	\end{align}
	From the partition of the velocity domain $[-4R,4R]^3$ in Lemma \ref{cpl}, we have
	\begin{align}\label{qq27}
		\sum_{I^3}^{(M_3)^3} \mathbf{1}_{\mathcal{P}_{I_3}^v}(v)\mathbf{1}_{\{|v|\le R\}}(v) = \mathbf{1}_{\{|v|\le R\}}(v).
	\end{align}
	From \eqref{qq26} and \eqref{qq27}, \eqref{aa625} is bounded by
	\begin{align*}
		& C_{R,\Phi} \sup_{0\le s \le t} \|h(s)\|_{L^\infty_{x,v}}^2\ \sum_{i_1}^{M_1}  \int_0^{T_1} \mathbf{1}_{\mathcal{P}_{i_1}^{T_1}}(s-t+T_1)\mathbf{1}_{(t_{j,i_1,I_2,I_3}-\frac{\tilde{\epsilon}}{4M_1},t_{j,i_1,I_2,I_3}+\frac{\tilde{\epsilon}}{4M_1})}(s-t+T_1) ds\\
		&\le \tilde{\epsilon} \ C_{R,\Phi} \sup_{0\le s \le t} \|h(s)\|_{L^\infty_{x,v}}^2.
	\end{align*}
	\newline
	$\mathbf{Case\ 3 \ (\romannumeral 2)\ of  \ J_{21}:}$ The integration \eqref{aa622} corresponding to \eqref{aa624} is bounded by
	\begin{align}
		&C_{R,\Phi}\sum_{i_1}^{M_1}\sum_{I_3}^{(M_3)^3} \int_0^{t} \mathbf{1}_{\mathcal{P}_{i_1}^{T_1}}(s-t+T_1) \left\{1-\mathbf{1}_{\cup_{j=1}^3(t_{j,i_1,I_2,I_3}-\frac{\tilde{\epsilon}}{4M_1},t_{j,i_1,I_2,I_3}+\frac{\tilde{\epsilon}}{4M_1})}(s-t+T_1)\right\} \nonumber \\ \label{aa626}
		&\times \exp\left\{-e^{-\|\Phi\|_\infty}\nu_0(t-s)\right\} \|h(s)\|_{L^\infty_{x,v}} \Bigg(\int_{|v|\le R}\int_{|\eta|\le R} \mathbf{1}_{\mathcal{P}_{I_3}^v}(v) e^{-\frac{|v|^2}{16}}|h(s,X(s-t+T_1;T_1,x,v),\eta)|^2 d\eta dv \Bigg)^{\frac{1}{2}}ds.
	\end{align}
	By Lemma \ref{cpl}, we have made a change of variables $v \rightarrow y:=X(s-t+T_1;T_1,x,v)$ satisfying 
	\begin{align*}
		\det\left(\frac{dX}{dv}(s-t+T_1;T_1,x,v)\right)> \delta_*
	\end{align*}
	and \eqref{aa626} is bounded by
	\begin{align*}
		C_{R,\Phi,M_1,M_3,\delta_*} \int_0^{t}\exp\left\{-e^{-\|\Phi\|_\infty}\nu_0(t-s)\right\}\|h(s)\|_{L^\infty_{x,v}}\left(\int_{\Omega} \int_{|\eta|\le R}|h(s,y,\eta)|^2d\eta dy\right)^{\frac{1}{2}}ds.
	\end{align*}
	Hence from \eqref{L2cont}, \eqref{aa626} is bounded by
	\begin{align*}
		&C_{R,\Phi,M_1,M_3,\delta_*} \int_0^{t}\exp\left\{-e^{-\|\Phi\|_\infty}\nu_0(t-s)\right\}\left[\sup_{0\le s\le t} \|h(s)\|_{L^\infty_{x,v}}\mathcal{E}(F_0)^{\frac{1}{2}} + \sup_{0\le s\le t} \|h(s)\|^{\frac{3}{2}}_{L^\infty_{x,v}}\mathcal{E}(F_0)^{\frac{1}{2}}\right]ds\\
		& \le C_{R,\Phi,M_1,M_3,\delta_*}\sup_{0\le s\le t} \|h(s)\|_{L^\infty_{x,v}}\mathcal{E}(F_0)^{\frac{1}{2}} +C_{R,\Phi,M_1,M_3,\delta_*} \sup_{0\le s\le t} \|h(s)\|^{\frac{3}{2}}_{L^\infty_{x,v}}\mathcal{E}(F_0)^{\frac{1}{2}}\\
		& \le \frac{C_\Phi}{R}\left[\sup_{0\le s\le t} \|h(s)\|_{L^\infty_{x,v}}^2+\sup_{0\le s\le t} \|h(s)\|^{3}_{L^\infty_{x,v}}\right]+C_{R,\Phi,M_1,M_3,\delta_*}\mathcal{E}(F_0).
	\end{align*}
	\newline
	We use \eqref{aa620} to estimate $J_{22}$:
	\begin{align*}
		\int_0^t \int_{|v|\le R} e^{-\frac{|v|^2}{8}}|J_{22}|dvds 
		&\le C_\Phi \int_0^t  \exp\left\{-e^{-\|\Phi\|_\infty}\nu_0(t-s)\right\} \int_{|v| \le R}e^{-\frac{|v|^2}{8}}\sum_{l=1}^{k-1}\int_{\prod_{j=1}^{l-1}\mathcal{V}_j}\int_{\mathcal{V}_l}e^{-\frac{|v_l|^2}{16}}\|h(s)\|_{L^\infty_{x,v}} \\ 
		&\quad \times \left(\int_{\mathbb{R}^3}(1+|\eta|)^{-2\beta+4}|h(s,X_l(s),\eta)|^2 d\eta \right)^{\frac{1}{2}} dv_l \left\{\prod_{j=1}^{l-1}d\sigma_j\right\}dvds.
	\end{align*}
	Fix $l$. We divide the following term into 3 cases:
	\begin{align}
		& C_\Phi \int_0^t  \exp\left\{-e^{-\|\Phi\|_\infty}\nu_0(t-s)\right\} \int_{|v| \le R}e^{-\frac{|v|^2}{8}}\int_{\prod_{j=1}^{l-1}\mathcal{V}_j}\int_{\mathcal{V}_l}e^{-\frac{|v_l|^2}{16}}\|h(s)\|_{L^\infty_{x,v}} \nonumber \\ \label{aa631}
		&\quad \times \left(\int_{\mathbb{R}^3}(1+|\eta|)^{-2\beta+4}|h(s,X_l(s),\eta)|^2 d\eta \right)^{\frac{1}{2}} dv_l \left\{\prod_{j=1}^{l-1}d\sigma_j\right\}dvds.
	\end{align}
	$\mathbf{Case\ 2\  of\  J_{22}:}$ $|v_l| \ge R$.\\
	The term \eqref{aa631} in this case is bounded by
	\begin{align*}
		\frac{C_\Phi}{R}\sup_{0 \le s \le t}\|h(s)\|^2_{L^\infty_{x,v}}.
	\end{align*}
	\newline
	$\mathbf{Case\ 3\  of\  J_{22}:}$ $|v_l| \le R$ and $|\eta| \ge R$.\\
	In a similar way in \eqref{ta62}, the term \eqref{aa631} in this case is bounded by
	\begin{align*}
		 \frac{C_\Phi}{R}\sup_{0 \le s \le t}\|h(s)\|^2_{L^\infty_{x,v}}.
	\end{align*}
	\newline
	$\mathbf{Case\ 4\  of\  J_{22}:}$ $|v_l| \le R$ and $|\eta| \le R$.\\
	In this case, we recall that $X_l(s) = X(s;t_l,x_l,v_l)$.
	Since the potential is time dependent, we have
	\begin{align*}
		X(s;t_l,x_l,v_l) = X(s-t_l+T_1;T_1,x_l,v_l)
	\end{align*}
	for all $0\le s \le t_l \le T_1$.\\
	By Lemma \ref{cpl}, the term \eqref{aa631} becomes
	\begin{align}
		&C_{R,\Phi}\sum_{i_1}^{M_1}\sum_{I_3}^{(M_3)^3} \int_0^{t} \mathbf{1}_{\mathcal{P}_{i_1}^{T_1}}(s-t_l+T_1) \exp\left\{-e^{-\|\Phi\|_\infty}\nu_0(t-s)\right\} \|h(s)\|_{L^\infty_{x,v}} \nonumber \\ 
		& \times \int_{|v| \le R}e^{-\frac{|v|^2}{8}}\int_{\prod_{j=1}^{l-1}\mathcal{V}_j}\int_{|v_l| \le R}\mathbf{1}_{\mathcal{P}_{I_3}^v}(v_l) e^{-\frac{|v_l|^2}{16}} \nonumber \\ \label{aa632}
		&\times \left(\int_{|\eta|\le R}(1+|\eta|)^{-2\beta+4}|h(s, X(s-t_l+T_1;T_1,x_l,v_l),\eta)|^2 d\eta \right)^{\frac{1}{2}} dv_l \left\{\prod_{j=1}^{l-1}d\sigma_j\right\}dvds .
	\end{align}
	From Lemma \ref{cpl}, we have the following partitions:
	\begin{align*}
		&\left\{(s-t_l+T_1,x_l,v_l)\in \mathcal{P}_{i_1}^{T_1}\cross \mathcal{P}_{I_2}^{\Omega}\cross \mathcal{P}_{I_3}^{v} : \det\left(\frac{dX}{dv_l}(s-t_l+T_1;T_1,x_l,v_l)\right)=0\right\}\\
		& \subset \bigcup_{j=1}^3 \left\{(s-t_l+T_1,x_l,v_l)\in \mathcal{P}_{i_1}^{T_1}\cross \mathcal{P}_{I_2}^{\Omega}\cross \mathcal{P}_{I_3}^{v} :s-t_l+T_1\in \left(t_{j,i_1,I_2,I_3}-\frac{\tilde{\epsilon}}{4M_1},t_{j,i_1,I_2,I_3}+\frac{\tilde{\epsilon}}{4M_1}\right)\right\}.
	\end{align*}
	Thus for each $i_1$,$I_2$, and $I_3$, we split $\mathbf{1}_{\mathcal{P}_{i_1}^{T_1}}(s-t_l+T_1)$ as
	\begin{align} \label{aa633}
		&\mathbf{1}_{\mathcal{P}_{i_1}^{T_1}}(s-t_l+T_1)\mathbf{1}_{\cup_{j=1}^3(t_{j,i_1,I_2,I_3}-\frac{\tilde{\epsilon}}{4M_1},t_{j,i_1,I_2,I_3}+\frac{\tilde{\epsilon}}{4M_1})}(s-t_l+T_1)\\ \label{aa634}
		& +\mathbf{1}_{\mathcal{P}_{i_1}^{T_1}}(s-t_l+T_1)\left\{1-\mathbf{1}_{\cup_{j=1}^3(t_{j,i_1,I_2,I_3}-\frac{\tilde{\epsilon}}{4M_1},t_{j,i_1,I_2,I_3}+\frac{\tilde{\epsilon}}{4M_1})}(s-t_l+T_1)\right\}.
	\end{align}
	\bigskip
	$\mathbf{Case\ 4 \ (\romannumeral 1)\ of\ J_{22} :}$ The integration \eqref{aa632} corresponding to \eqref{aa633} is bounded by
	\begin{align}
		&C_{R,\Phi}\sum_{i_1}^{M_1}\sum_{I_3}^{(M_3)^3} \sum_{j=1}^3 \underbrace{\int_0^{T_1} \mathbf{1}_{\mathcal{P}_{i_1}^{T_1}}(s-t_l+T_1)  \mathbf{1}_{(t_{j,i_1,I_2,I_3}-\frac{\tilde{\epsilon}}{4M_1},t_{j,i_1,I_2,I_3}+\frac{\tilde{\epsilon}}{4M_1})}(s-t_l+T_1)}_{(*8)}  \nonumber \\ 
		& \times  \|h(s)\|_{L^\infty_{x,v}}\int_{|v| \le R}e^{-\frac{|v|^2}{8}}\int_{\prod_{j=1}^{l-1}\mathcal{V}_j}\int_{|v_l| \le R}\mathbf{1}_{\mathcal{P}_{I_3}^v}(v_l) e^{-\frac{|v_l|^2}{16}} \nonumber \\ \label{aa635}
		&\times \left(\int_{|\eta|\le R}(1+|\eta|)^{-2\beta+4}|h(s, X(s-t_l+T_1;T_1,x_l,v_l),\eta)|^2 d\eta \right)^{\frac{1}{2}} dv_l \left\{\prod_{j=1}^{l-1}d\sigma_j\right\}dvds.
	\end{align}
	Here, $(*8)$ is bounded by
	\begin{align} \label{qq28}
		\int_0^{T_1} \mathbf{1}_{\mathcal{P}_{i_1}^{T_1}}(s-t_l+T_1)\mathbf{1}_{(t_{j,i_1,I_2,I_3}-\frac{\tilde{\epsilon}}{4M_1},t_{j,i_1,I_2,I_3}+\frac{\tilde{\epsilon}}{4M_1})}(s-t_l+T_1) ds
		\le \frac{\tilde{\epsilon}}{2M_1}.
	\end{align}
	From the partition of the velocity domain $[-4R,4R]^3$ in Lemma \ref{cpl}, we have
	\begin{align} \label{qq29}
		\sum_{I^3}^{(M_3)^3} \mathbf{1}_{\mathcal{P}_{I_3}^v}(v_l)\mathbf{1}_{\{|v_l|\le R\}}(v_l) = \mathbf{1}_{\{|v_l|\le R\}}(v_l).
	\end{align}
	From \eqref{qq28} and \eqref{qq29}, \eqref{aa635} is bounded by
	\begin{align*}
		& C_{R,\Phi} \sup_{0\le s \le t} \|h(s)\|_{L^\infty_{x,v}}^2\ \sum_{i_1}^{M_1}  \int_0^{T_1} \mathbf{1}_{\mathcal{P}_{i_1}^{T_1}}(s-t+T_1)\mathbf{1}_{(t_{j,i_1,I_2,I_3}-\frac{\tilde{\epsilon}}{4M_1},t_{j,i_1,I_2,I_3}+\frac{\tilde{\epsilon}}{4M_1})}(s-t+T_1) ds\\
		&\le \tilde{\epsilon} \ C_{R,\Phi} \sup_{0\le s \le t} \|h(s)\|_{L^\infty_{x,v}}^2.
	\end{align*}
	\newline
	$\mathbf{Case\ 4 \ (\romannumeral 2)\ of  \ J_{22}:}$ The integration \eqref{aa632} corresponding to \eqref{aa634} is bounded by
	\begin{align}
		&C_{R,\Phi}\sum_{i_1}^{M_1}\sum_{I_3}^{(M_3)^3} \int_0^{t} \mathbf{1}_{\mathcal{P}_{i_1}^{T_1}}(s-t_l+T_1) \left\{1-\mathbf{1}_{\cup_{j=1}^3(t_{j,i_1,I_2,I_3}-\frac{\tilde{\epsilon}}{4M_1},t_{j,i_1,I_2,I_3}+\frac{\tilde{\epsilon}}{4M_1})}(s-t_l+T_1)\right\} \nonumber \\ 
		&\quad  \times \exp\left\{-e^{-\|\Phi\|_\infty}\nu_0(t-s)\right\} \|h(s)\|_{L^\infty_{x,v}}\int_{|v| \le R}e^{-\frac{|v|^2}{8}}\int_{\prod_{j=1}^{l-1}\mathcal{V}_j}\int_{|v_l| \le R} \nonumber \\ 
		&\quad \times \left(\int_{|\eta|\le R}|h(s, X(s-t_l+T_1;T_1,x_l,v_l),\eta)|^2 d\eta \right)^{\frac{1}{2}} dv_l \left\{\prod_{j=1}^{l-1}d\sigma_j\right\}dvds\nonumber\\
		& \le C_{R,\Phi}\sum_{i_1}^{M_1}\sum_{I_3}^{(M_3)^3} \int_0^{t} \mathbf{1}_{\mathcal{P}_{i_1}^{T_1}}(s-t_l+T_1) \left\{1-\mathbf{1}_{\cup_{j=1}^3(t_{j,i_1,I_2,I_3}-\frac{\tilde{\epsilon}}{4M_1},t_{j,i_1,I_2,I_3}+\frac{\tilde{\epsilon}}{4M_1})}(s-t_l+T_1)\right\} \nonumber \\ 
		& \quad \times \exp\left\{-e^{-\|\Phi\|_\infty}\nu_0(t-s)\right\} \|h(s)\|_{L^\infty_{x,v}}\int_{|v| \le R}e^{-\frac{|v|^2}{8}}\int_{\prod_{j=1}^{l-1}\mathcal{V}_j} \nonumber \\ \label{aa636}
		&\quad \times \left(\int_{|v_l| \le R}\int_{|\eta|\le R}|h(s, X(s-t_l+T_1;T_1,x_l,v_l),\eta)|^2 d\eta dv_l\right)^{\frac{1}{2}}  \left\{\prod_{j=1}^{l-1}d\sigma_j\right\}dvds.
		\end{align}
	By Lemma \ref{cpl}, we have made a change of variables $v_l \rightarrow y:=X(s-t_l+T_1;T_1,x_l,v_l)$ satisfying 
	\begin{align*}
		\det\left(\frac{dX}{dv_l}(s-t_l+T_1;T_1,x_l,v_l)\right)> \delta_*
	\end{align*}
	and \eqref{aa636} is bounded by
	\begin{align*}
		C_{R,\Phi,M_1,M_3,\delta_*} \int_0^{t}\exp\left\{-e^{-\|\Phi\|_\infty}\nu_0(t-s)\right\}\|h(s)\|_{L^\infty_{x,v}}\left(\int_{\Omega} \int_{|\eta|\le R}|h(s,y,\eta)|^2d\eta dy\right)^{\frac{1}{2}}ds.
	\end{align*}
	Hence from \eqref{L2cont}, \eqref{aa636} is bounded by
	\begin{align*}
		&C_{R,\Phi,M_1,M_3,\delta_*} \int_0^{t}\exp\left\{-e^{-\|\Phi\|_\infty}\nu_0(t-s)\right\}\left[\sup_{0\le s\le t} \|h(s)\|_{L^\infty_{x,v}}\mathcal{E}(F_0)^{\frac{1}{2}} + \sup_{0\le s\le t} \|h(s)\|^{\frac{3}{2}}_{L^\infty_{x,v}}\mathcal{E}(F_0)^{\frac{1}{2}}\right]ds\\
		& \le C_{R,\Phi,M_1,M_3,\delta_*}\sup_{0\le s\le t} \|h(s)\|_{L^\infty_{x,v}}\mathcal{E}(F_0)^{\frac{1}{2}} +C_{R,\Phi,M_1,M_3,\delta_*} \sup_{0\le s\le t} \|h(s)\|^{\frac{3}{2}}_{L^\infty_{x,v}}\mathcal{E}(F_0)^{\frac{1}{2}}\\
		& \le \frac{C_\Phi}{R}\left[\sup_{0\le s\le t} \|h(s)\|_{L^\infty_{x,v}}^2+\sup_{0\le s\le t} \|h(s)\|^{3}_{L^\infty_{x,v}}\right]+C_{R,\Phi,M_1,M_3,\delta_*}\mathcal{E}(F_0).
	\end{align*}
	\newline
	Summing over $1\le l \le k(\epsilon,T_1)-1$ and combining the estimates of all cases, we obtain
	\begin{align}
		&\int_{\mathbb{R}^3}e^{-\frac{|v|^2}{8}}|h(t,x,v)|dv\nonumber\\
		&\le C_6 \exp\left\{-e^{-\|\Phi\|_\infty} \frac{\nu_0}{2} t \right\}\|h_0\|_{L^\infty_{x,v}}\nonumber\\
		& \quad +C_6\left(\epsilon+\frac{C_{\Phi,\epsilon,T_1}}{R}+\tilde{\epsilon} \ C_{R,\Phi,\epsilon,T_1}  \right)\left[\sup_{0\le s\le t} \|h(s)\|_{L^\infty_{x,v}}+\sup_{0\le s\le t} \|h(s)\|_{L^\infty_{x,v}}^2+\sup_{0\le s\le t} \|h(s)\|_{L^\infty_{x,v}}^3\right]\nonumber \\ \label{ta660} 
		& \quad +C_{R,\Phi,M_1,M_3,\delta_*,\epsilon,T_1}\left[\mathcal{E}(F_0)^{\frac{1}{2}}+\mathcal{E}(F_0)\right],
	\end{align}
	where $C_6$ is a generic constant.\\
	Set $\tilde{t}:= \frac{2}{e^{-\|\Phi\|_\infty}\nu_0} \log (4C_5C_6 M_0)$ and $C_4:=4C_5C_6$. It follow that
	\begin{align*}
		C_6 \exp\left\{-e^{-\|\Phi\|_\infty} \frac{\nu_0}{2} t \right\}\|h_0\|_{L^\infty_{x,v}} \le C_6 \exp\left\{-e^{-\|\Phi\|_\infty} \frac{\nu_0}{2} t \right\}M_0 \le \frac{1}{4C_5}
	\end{align*}
	for all $t \ge \tilde{t}$. From \eqref{Ape} and the assumption $\mathcal{E}(F_0) \le \epsilon_0$, \eqref{ta660} implies that
	\begin{align}
		&\int_{\mathbb{R}^3}e^{-\frac{|v|^2}{8}}|h(t,x,v)|dv\nonumber\\ \label{ta661} 
		&\le \frac{1}{4C_5}+3C_6\left(\epsilon+\frac{C_{\Phi,\epsilon,T_1}}{R}+\tilde{\epsilon} \ C_{R,\Phi,\epsilon,T_1} \right)\bar{M}^3 +C_{R,\Phi,M_1,M_3,\delta_*,\epsilon,T_1}\left[\epsilon_0^{\frac{1}{2}}+\epsilon_0\right]
	\end{align}
	for all $\tilde{t} \le t \le T_1$. First, we take $\epsilon >0$  small enough, $R>0$ large enough, then we choose sufficiently small $\tilde{\epsilon}>0$, and we choose $\epsilon_0$ small enough so that
	\begin{align*}
		3C_6\left(\epsilon+\frac{C_{\Phi,\epsilon,T_1}}{R}+\tilde{\epsilon} \ C_{R,\Phi,\epsilon,T_1} \right)\bar{M}^3 +C_{R,\Phi,M_1,M_3,\delta_*,\epsilon,T_1}\left[\epsilon_0^{\frac{1}{2}}+\epsilon_0\right] \le \frac{1}{4C_5}.
	\end{align*}
	Therefore we conclude that
	\begin{align*}
		&\int_{\mathbb{R}^3}e^{-\frac{|v|^2}{8}}|h(t,x,v)|dv \le \frac{1}{2C_5} \quad \text{for all } (t,x) \in [\tilde{t},T_1) \cross \Omega.
	\end{align*}
	We complete the proof of this Lemma.
\end{proof}

\bigskip
\subsection{$L^\infty$ estimate}

In this subsection, we will produce the $L^\infty_{x,v}$ estimate in terms of a solution $h$ to \eqref{PBE}. In order to estimate $h$ in $L^\infty_{x,v}$, we first need to estimate the semigroup $S_{G_f}$. Thus when $\mathcal{E}(F_0)$ is sufficiently small, we will estimate the semigroup $S_{G_f}$ by using the $R(f)$ {\it estimate} \eqref{Rfestimate}.
\begin{lemma} \label{L61}
	Assume the a priori assumption \eqref{Ape}. Let $f$ be a solution of \eqref{PBE} with initial datum $f_0$ and the boundary condition \eqref{PDRBC} and $h(t,x,v)=w(x,v)f(t,x,v)$.  Then there exists a constant $C_\Phi>0$ so that if $\mathcal{E}(F_0) \le \epsilon_0$, where $\epsilon_0 = \epsilon_0(\bar{M},T_1)$ is determined in Lemma \ref{Rfest},  
	\begin{align*}
		\|S_{G_f}(t-s)h(s)\|_{L^\infty_{x,v}} \le C_\Phi \exp{e^{-\|\Phi\|_\infty}\frac{3}{4}\nu_0 \tilde{t}} \exp\left\{-e^{-\|\Phi\|_\infty} \frac{\nu_0}{4} (t-s) \right\} \|h(s)\|_{L^\infty_{x,v}}
	\end{align*}
	for all $0\le s \le t \le T_1$.
\end{lemma}
\begin{proof}
	Suppose that $\mathcal{E}(F_0) \le \epsilon_0(\bar{M},T_1)$. By Lemma \ref{Rfest}, we have
	\begin{equation} \label{LL1}
	\begin{aligned}
		R(f)(t,x,v) \ge 
		\begin{cases}
			0 & \text{if }t\in [0,\tilde{t}), \\
			\frac{1}{2}e^{-\Phi(x)}\nu(v) \quad & \text{if }t\in [\tilde{t},T_1),
		\end{cases}
	\end{aligned}
	\end{equation}
	for all $(x,v)\in \Omega\cross \mathbb{R}^3$.\\
	\newline
	$\mathbf{Case\ of \ t \in [0,\tilde{t}]:}$ We know that $t/ \rho \in [m,m+1)$ for some $m \in \{0,1,...\ , [\tilde{t}/\rho]\}$. Thus it follows from Lemma \ref{Phiest} and \eqref{rhoest} that
	\begin{align*}
		\|S_{G_f}(t)h_0\|_{L^\infty_{x,v}} &\le C_3\rho^{\frac{5}{4}}\|S_{G_f}(m\rho)h_0\|_{L^\infty_{x,v}} \le C_3 \rho^{\frac{5}{4}}\left(C_3\rho^{\frac{5}{4}}\right)^m \|h_0\|_{L^\infty_{x,v}} \le C_3 \rho^{\frac{5}{4}}\left(C_3\rho^{\frac{5}{4}}\right)^{\frac{\tilde{t}}{\rho}}\|h_0\|_{L^\infty_{x,v}}\\
		& \le C \exp{e^{-\|\Phi\|_\infty}\frac{\nu_0}{2}\tilde{t}}\|h_0\|_{L^\infty_{x,v}},
	\end{align*}
	and we derive
	\begin{align*}
		\|S_{G_f}(t-s)h(s)\|_{L^\infty_{x,v}} \le C \exp{e^{-\|\Phi\|_\infty}\frac{3}{4}\nu_0 \tilde{t}} \exp\left\{-e^{-\|\Phi\|_\infty} \frac{\nu_0}{4} t \right\} \|h_0\|_{L^\infty_{x,v}}.
	\end{align*}
	\newline
	$\mathbf{Case\ of \ t \in [\tilde{t},T_1):}$ We note that $S_{G_f}(t)h_0 = S_{G_f}(t-\tilde{t})S_{G_f}(\tilde{t})h_0$. From \eqref{LL1} and Lemma \ref{LL62}, we obtain
	\begin{align*}
		\|S_{G_f}(t)h_0\|_{L^\infty_{x,v}} &\le C \exp{-e^{-\|\Phi\|_\infty}\frac{\nu_0}{4}(t-\tilde{t})}\|S_{G_f}(\tilde{t})h_0\|_{L^\infty_{x,v}}.
	\end{align*}
	We use the previous case to get
	\begin{align*}
		\|S_{G_f}(t)h_0\|_{L^\infty_{x,v}}  \le C_{\Phi} \exp{e^{-\|\Phi\|_\infty}\frac{3}{4}\nu_0 \tilde{t}} \exp\left\{-e^{-\|\Phi\|_\infty} \frac{\nu_0}{4} t \right\}\|h_0\|_{L^\infty_{x,v}}.
	\end{align*}
	\newline
	Gathering two cases, we conclude that
	\begin{align*}
		\|S_{G_f}(t)h_0\|_{L^\infty_{x,v}} \le C_\Phi \exp{e^{-\|\Phi\|_\infty}\frac{3}{4}\nu_0 \tilde{t}} \exp\left\{-e^{-\|\Phi\|_\infty} \frac{\nu_0}{4} t \right\} \|h_0\|_{L^\infty_{x,v}}
	\end{align*}
	for all $0 \le t \le T_1$.
\end{proof}

\bigskip
The following two lemmas provides the $L^\infty$ estimate of a solution of the full perturbed Boltzmann equation \eqref{WPBE} with the boundary condition \eqref{eBEBC}. These two estimates play a crucial role in achieving our main goal. Recall the definition \ref{Backwardexit}, especially \eqref{timecycle}, \eqref{Xtrajcycle}, and \eqref{Vtrajcycle}, as well as the definition of the iterated integral \eqref{iterint}.
\begin{lemma} \label{Linfty1}
	Assume that $\mathcal{E}(F_0) \le \epsilon_0$, where $\epsilon_0 = \epsilon_0(\bar{M},T_1)$ is determined in Lemma \ref{Rfest}. Let $h(t,x,v)$ be a solution to the equation \eqref{WPBBE} with initial datum $h_0$ and the diffuse reflection boundary condition \eqref{WWPBEBC}. Let $(t,x,v) \in (0,T_1] \cross \Omega \cross \mathbb{R}^3$. Then it holds that
	\begin{align*}
		|h(t,x,v)| &\le \exp{e^{-\|\Phi\|_\infty}\frac{\nu_0}{4}\tilde{t}} \int_0^t \mathbf{1}_{\{t_1 \le s\}} \exp{-e^{-\|\Phi\|_\infty}\frac{\nu_0}{4}(t-s)}\left|K_wh(s,X(s),V(s))\right| ds\\
		&\quad + \exp{e^{-\|\Phi\|_\infty}\frac{\nu_0}{4}\tilde{t}} \int_0^t \mathbf{1}_{\{t_1 \le s\}} \exp{-e^{-\|\Phi\|_\infty}\frac{\nu_0}{4}(t-s)}\left|w\Gamma_+(f,f)(s,X(s),V(s))\right| ds\\
		&\quad +C_\Phi \exp\left\{e^{-\|\Phi\|_\infty}\nu_0\tilde{t}\right\}\exp\left\{-e^{-\|\Phi\|_\infty} \frac{\nu_0}{4} t \right\}\|h_0\|_{L^\infty_{x,v}}\\
		&\quad+ C_\Phi \exp\left\{e^{-\|\Phi\|_\infty}\nu_0\tilde{t}\right\}\left(\epsilon+\frac{C_{\epsilon,T_1}}{R}+\tilde{\epsilon} \ C_{R,\epsilon,T_1} \right)\sup_{0\le s \le t} \left[\|h(s)\|_{L^\infty_{x,v}}+\|h(s)\|_{L^\infty_{x,v}}^2+\|h(s)\|_{L^\infty_{x,v}}^3\right]\\
		&\quad+C_{\Phi,R, \delta_*,\epsilon,\tilde{\epsilon},T_1} \exp\left\{e^{-\|\Phi\|_\infty}\nu_0\tilde{t}\right\}\left[\mathcal{E}(F_0)^{\frac{1}{2}}+\mathcal{E}(F_0)\right],
	\end{align*}
	where $\epsilon>0$ and $\tilde{\epsilon}>0$ are arbitrary small, and $R>0$ is sufficiently large.
\end{lemma}
\begin{proof}
	Let $(t,x,v) \in (0,T_1] \cross \Omega \cross \mathbb{R}^3$. By Duhamel principle, \eqref{WPBBE} implies that
	\begin{align*}
		h(t,x,v) = S_{G_f}(t)h_0+\int_0^t S_{G_f}(t-s)e^{-\Phi}K_wh(s)ds + \int_0^t S_{G_f}(t-s)e^{-\frac{\Phi}{2}}w\Gamma_+(f,f)(s)ds.
	\end{align*}
	First, by Lemma \ref{L61}, we obtain
	\begin{align} \label{aa4}
		|S_{G_f}(t)h_0| \le C_\Phi \exp{e^{-\|\Phi\|_\infty}\frac{3}{4}\nu_0 \tilde{t}} \exp\left\{-e^{-\|\Phi\|_\infty} \frac{\nu_0}{4} t \right\} \|h_0\|_{L^\infty_{x,v}}.
	\end{align}
	We note that
	\begin{align*}
		&S_{G_f}(t-s)e^{-\Phi}K_wh(s)\\
		& = \mathbf{1}_{\{t_1 \le s\}} \exp{-\int_s^t R(f)(\tau,X(\tau),V(\tau))d\tau}e^{-\Phi(X(s))}K_wh(s,X(s),V(s))\\
		&\quad +\frac{\exp{-\int_{t_1}^t R(f)(\tau, X(\tau),V(\tau))d\tau}}{\tilde{w}(x_1,V(t_1))} \sum_{l=1}^{k-1}\int_{\prod_{j=1}^{k-1}\mathcal{V}_j}\mathbf{1}_{\{t_{l+1}\le s < t_l\}} e^{-\Phi(X_l(s))}
		K_wh(s,X_l(s),V_l(s))d\Sigma^f_l(s)\\
		&\quad +\frac{\exp{-\int_{t_1}^t R(f)(\tau, X(\tau),V(\tau))d\tau}}{\tilde{w}(x_1,V(t_1))}\int_{\prod_{j=1}^{k-1}\mathcal{V}_j}\mathbf{1}_{\left\{t_k>s\right\}}\left(S_{G_f}(t-s)e^{-\Phi}K_wh(s)\right)(t_k,x_k,V_{k-1}(t_k))d\Sigma^f_{k-1}(t_k)\\
		&=:I_{11}+I_{12}+I_{13},
	\end{align*}
	where
	\begin{align*}
		d\Sigma^f_l(s) &= \left\{\prod_{j=l+1}^{k-1}d\sigma_j\right\}\left\{\exp{-\int_s^{t_l} R(f)(\tau,X_l(\tau).V_l(\tau))d\tau} \tilde{w}(x_l,v_l)d\sigma_l\right\}\\
		&\quad \cross \prod_{j=1}^{l-1}\left\{\exp{-\int_{t_{j+1}}^{t_j} R(f)(\tau, X_j(\tau),V_j(\tau))d\tau}d\sigma_j\right\}.
	\end{align*}
	By Lemma \ref{Rfest}, we have
	\begin{align} \label{aa1}
		&\exp{-\int_s^t R(f)(\tau,X(\tau),V(\tau))d\tau} \le \exp{e^{-\|\Phi\|_\infty}\frac{\nu_0}{4}\tilde{t}}\exp{-e^{-\|\Phi\|_\infty}\frac{\nu_0}{4}(t-s)}, \\
		&\exp{-\int_{t_1}^t R(f)(\tau,X(\tau),V(\tau))d\tau-\int_{t_2}^{t_1} R(f)(\tau,X_1(\tau),V_1(\tau))d\tau-\cdots-\int_{s}^{t_l} R(f)(\tau,X_l(\tau),V_l(\tau))d\tau}\nonumber\\ \label{aa2}
		&\le \exp{e^{-\|\Phi\|_\infty}\frac{\nu_0}{4}\tilde{t}}\exp{-e^{-\|\Phi\|_\infty}\frac{\nu_0}{4}(t-s)}.
	\end{align}
	For $I_{11}$, using \eqref{aa1}, we obtain
	\begin{align*}
		I_{11} \le \int_0^t\mathbf{1}_{\{t_1\le s\}}\exp{e^{-\|\Phi\|_\infty}\frac{\nu_0}{4}\tilde{t}}\exp{-e^{-\|\Phi\|_\infty}\frac{\nu_0}{4}(t-s)}\left|K_wh(s,X(s),V(s))\right|ds.
	\end{align*}
	Now, let us consider $I_{13}$. From Lemma \ref{L61}, we have
	\begin{align}
		&\left|\left(S_{G_f}(t-s)e^{-\Phi}K_wh(s)\right)(t_k,x_k,V_{k-1}(t_k))\right|\nonumber\\ \label{aa3}
		& \le C_\Phi \exp{e^{-\|\Phi\|_\infty}\frac{3}{4}\nu_0 \tilde{t}}\exp{-e^{-\|\Phi\|_\infty}\frac{\nu_0}{4}(t_k-s)}\|h(s)\|_{L^\infty_{x,v}}.
	\end{align}
	By Lemma \ref{Lsmall}, \eqref{aa2}, and \eqref{aa3}, $I_{13}$ is bounded by
	\begin{align*}
		&C_\Phi  \exp{e^{-\|\Phi\|_\infty}\nu_0 \tilde{t}}\exp{-e^{-\|\Phi\|_\infty}\frac{\nu_0}{4}(t-s)}\|h(s)\|_{L^\infty_{x,v}} \int_{\prod_{j=1}^{k-1}\mathcal{V}_{j}}\mathbf{1}_{\{t_k>s\}} \tilde{w}(x_k,v_{k})d\sigma_{k-1}d\sigma_{k-2}\cdots d\sigma_1\\ 
		& \le \epsilon \ C_\Phi  \exp{e^{-\|\Phi\|_\infty}\nu_0 \tilde{t}}\exp{-e^{-\|\Phi\|_\infty}\frac{\nu_0}{4}(t-s)}\|h(s)\|_{L^\infty_{x,v}},
	\end{align*}
	where we have taken $k = k(\epsilon,T_1)+1$.\\
	Let us consider $I_{12}$. We use a scheme similar to \eqref{aa611} to obtain
	\begin{align*}
		&\int_0^tI_{12} ds\\
		& \le \frac{C_{\epsilon,T_1}}{R} \exp{e^{-\|\Phi\|_\infty}\frac{\nu_0}{4}\tilde{t}}\sup_{0\le s \le t}\|h(s)\|_{L^\infty_{x,v}}+\tilde{\epsilon} \  C_{R,\Phi,\epsilon,T_1} \exp{e^{-\|\Phi\|_\infty}\frac{\nu_0}{4}\tilde{t}}  \sup_{0\le s \le t}\|h(s)\|_{L^\infty_{x,v}}\\
		& \quad + C_{R,\Phi, \epsilon,\tilde{\epsilon},T_1}\exp{e^{-\|\Phi\|_\infty}\frac{\nu_0}{4}\tilde{t}} \left[\mathcal{E}(F_0)^{\frac{1}{2}}+\mathcal{E}(F_0)\right].
	\end{align*}
	Hence it follows that
	\begin{align}
		&\left|\int_0^tS_{G_f}(t-s)e^{-\Phi}K_wh(s)ds\right| \nonumber\\
		& \le \int_0^t\mathbf{1}_{\{t_1\le s\}}\exp{e^{-\|\Phi\|_\infty}\frac{\nu_0}{4}\tilde{t}}\exp{-e^{-\|\Phi\|_\infty}\frac{\nu_0}{4}(t-s)}\left|K_wh(s,X(s),V(s))\right|ds\nonumber\\
		& \quad + C_\Phi \left(\epsilon+\frac{C_{\epsilon,T_1}}{R}+\tilde{\epsilon} \ C_{R,\epsilon,T_1} \right)\exp{e^{-\|\Phi\|_\infty}\nu_0\tilde{t}}\sup_{0\le s\le t}\|h(s)\|_{L^\infty_{x,v}}\nonumber\\ \label{aa5}
		& \quad + C_{R, \Phi, \epsilon, \tilde{\epsilon}, T_1}\exp{e^{-\|\Phi\|_\infty}\nu_0\tilde{t}}\left[\mathcal{E}(F_0)^{\frac{1}{2}}+\mathcal{E}(F_0)\right].
	\end{align}
	We note that
	\begin{align*}
		&S_{G_f}(t-s)e^{-\frac{\Phi}{2}}w\Gamma_+(f,f)(s)\\
		& = \mathbf{1}_{\{t_1 \le s\}} \exp{-\int_s^t R(f)(\tau,X(\tau),V(\tau))d\tau}e^{-\frac{{\Phi(X(s))}}{2}}w\Gamma_+(f,f)(s,X(s),V(s))\\
		&\quad +\frac{\exp{-\int_{t_1}^t R(f)(\tau, X(\tau),V(\tau))d\tau}}{\tilde{w}(x_1,V(t_1))} \sum_{l=1}^{k-1}\int_{\prod_{j=1}^{k-1}\mathcal{V}_j}\mathbf{1}_{\{t_{l+1}\le s < t_l\}} e^{-\frac{{\Phi(X_l(s))}}{2}}
		w\Gamma_+(f,f)(s,X_l(s),V_l(s))d\Sigma^f_l(s)\\
		&\quad +\frac{\exp{-\int_{t_1}^t R(f)(\tau, X(\tau),V(\tau))d\tau}}{\tilde{w}(x_1,V(t_1))}\int_{\prod_{j=1}^{k-1}\mathcal{V}_j}\mathbf{1}_{\left\{t_k>s\right\}}\left(S_{G_f}(t-s)e^{-\frac{\Phi}{2}}w\Gamma_+(f,f)(s)\right)(t_k,x_k,V_{k-1}(t_k))\\
		&\qquad \times d\Sigma^f_{k-1}(t_k)\\
		&=:I_{21}+I_{22}+I_{23}.
	\end{align*}
	For $I_{21}$, using \eqref{aa1}, we obtain
	\begin{align*}
		I_{21} \le \int_0^t\mathbf{1}_{\{t_1\le s\}}\exp{e^{-\|\Phi\|_\infty}\frac{\nu_0}{4}\tilde{t}}\exp{-e^{-\|\Phi\|_\infty}\frac{\nu_0}{4}(t-s)}\left|w\Gamma_+(f,f)(s,X(s),V(s))\right|ds.
	\end{align*}
	Now, let us consider $I_{23}$. From Lemma \ref{L61}, we have
	\begin{align}
		&\left|\left(S_{G_f}(t-s)e^{-\frac{\Phi}{2}}w\Gamma_+(f,f)(s)\right)(t_k,x_k,V_{k-1}(t_k))\right|\nonumber\\ \label{aa3}
		& \le C_\Phi \exp{e^{-\|\Phi\|_\infty}\frac{3}{4}\nu_0 \tilde{t}}\exp{-e^{-\|\Phi\|_\infty}\frac{\nu_0}{4}(t_k-s)}\|h(s)\|_{L^\infty_{x,v}}^2.
	\end{align}
	By Lemma \ref{Lsmall}, \eqref{aa2}, and \eqref{aa3}, $I_{23}$ is bounded by
	\begin{align*}
		&C_\Phi  \exp{e^{-\|\Phi\|_\infty}\nu_0 \tilde{t}}\exp{-e^{-\|\Phi\|_\infty}\frac{\nu_0}{4}(t-s)}\|h(s)\|_{L^\infty_{x,v}}^2 \int_{\prod_{j=1}^{k-1}\mathcal{V}_{j}}\mathbf{1}_{\{t_k>s\}} \tilde{w}(x_k,v_{k})d\sigma_{k-1}d\sigma_{k-2}\cdots d\sigma_1\\
		& \le\epsilon \ C_\Phi \exp{e^{-\|\Phi\|_\infty}\nu_0 \tilde{t}}\exp{-e^{-\|\Phi\|_\infty}\frac{\nu_0}{4}(t-s)}\|h(s)\|_{L^\infty_{x,v}}^2,
	\end{align*}
	where we have taken $k = k(\epsilon,T_1)+1$.\\
	Let us consider $I_{22}$. We use a scheme similar to \eqref{aa631} to obtain
	\begin{align*}
		&\int_0^t I_{22}ds\\
		& \le \frac{C_{\epsilon,T_1}}{R} \exp{e^{-\|\Phi\|_\infty}\frac{\nu_0}{4}\tilde{t}}\sup_{0\le s \le t}\left[\|h(s)\|^2_{L^\infty_{x,v}}+\|h(s)\|^3_{L^\infty_{x,v}}\right]+\tilde{\epsilon} \ C_{R,\Phi,\epsilon,T_1} \exp{e^{-\|\Phi\|_\infty}\frac{\nu_0}{4}\tilde{t}}  \sup_{0\le s \le t}\|h(s)\|_{L^\infty_{x,v}}^2\\
		& \quad + C_{R,\Phi, \epsilon,\tilde{\epsilon},T_1}\exp{e^{-\|\Phi\|_\infty}\frac{\nu_0}{4}\tilde{t}} \mathcal{E}(F_0).
	\end{align*}
	Hence it follows that
	\begin{align}
		&\left|\int_0^tS_{G_f}(t-s)e^{-\frac{\Phi}{2}}w\Gamma_+(f,f)(s)ds\right|\nonumber\\
		& \le \int_0^t\mathbf{1}_{\{t_1\le s\}}\exp{e^{-\|\Phi\|_\infty}\frac{\nu_0}{4}\tilde{t}}\exp{-e^{-\|\Phi\|_\infty}\frac{\nu_0}{4}(t-s)}\left|w\Gamma_+(f,f)(s,X(s),V(s))\right|ds\nonumber\\
		& \quad + C_\Phi \left(\epsilon+\frac{C_{\epsilon,T_1}}{R}+\tilde{\epsilon} \ C_{R,\epsilon,T_1} \right)\exp{e^{-\|\Phi\|_\infty}\nu_0\tilde{t}}\sup_{0\le s \le t}\left[\|h(s)\|^2_{L^\infty_{x,v}}+\|h(s)\|^3_{L^\infty_{x,v}}\right]\nonumber\\ \label{aa6}
		& \quad + C_{R, \Phi, \epsilon, \tilde{\epsilon}, T_1}\exp{e^{-\|\Phi\|_\infty}\nu_0\tilde{t}}\mathcal{E}(F_0).
	\end{align}
	Combining \eqref{aa4}, \eqref{aa5}, and \eqref{aa6}, the proof of Lemma \ref{Linfty1} is completed. 
\end{proof}
\bigskip
\begin{lemma} \label{Linfty2}
	Assume that $\mathcal{E}(F_0) \le \epsilon_0$, where $\epsilon_0 = \epsilon_0(\bar{M},T_1)$ is determined in Lemma \ref{Rfest}. Let $h(t,x,v)$ be a solution to the equation \eqref{WPBBE} with initial datum $h_0$ and the diffuse reflection boundary condition \eqref{WWPBEBC}. There exists a generic constant $C_7 \ge 1$ such that
	\begin{align*}
		\|h(t)\|_{L^\infty_{x,v}} &\le C_7 \exp{2e^{-\|\Phi\|_\infty}\nu_0 \tilde{t}}\|h_0\|_{L^\infty_{x,v}}\left(1+\int_0^t\|h(s)\|_{L^\infty_{x,v}}ds\right)\exp\left\{-e^{-\|\Phi\|_\infty} \frac{\nu_0}{8} t \right\}\\
		& \quad +C_7 \exp{2e^{-\|\Phi\|_\infty}\nu_0 \tilde{t}}\left(\epsilon+\frac{C_{\epsilon,T_1}}{R}+\tilde{\epsilon} \ C_{R,\epsilon,T_1} \right) \sup_{0\le s \le t} \Big[\|h(s)\|_{L^\infty_{x,v}}+\|h(s)\|_{L^\infty_{x,v}}^2+\|h(s)\|_{L^\infty_{x,v}}^3\\
		&\qquad +\|h(s)\|_{L^\infty_{x,v}}^4+\|h(s)\|_{L^\infty_{x,v}}^5\Big]\\
		&\quad + C_{\Phi,R, \delta_*,\epsilon,\tilde{\epsilon},T_1}\exp{2e^{-\|\Phi\|_\infty}\nu_0 \tilde{t}}\left[\mathcal{E}(F_0)^{\frac{1}{2}}+\mathcal{E}(F_0)+\mathcal{E}(F_0)^2\right]
	\end{align*}
	for all $0\le t \le T_1$, where $\epsilon>0$ and $\tilde{\epsilon}>0$ are arbitrarily small, and $R>0$ is sufficiently large.
\end{lemma}

\begin{proof}
	Let $(t,x,v) \in (0,T_1] \cross \Omega \cross \mathbb{R}^3$. In Lemma \ref{Linfty1}, we have to estimate the following terms:
	\begin{align*}
		&I_1:=  \int_0^t \mathbf{1}_{\{t_1 \le s\}} \exp{-e^{-\|\Phi\|_\infty}\frac{\nu_0}{4}(t-s)}\left|K_wh(s,X(s),V(s))\right| ds,\\
		& I_2:= \int_0^t \mathbf{1}_{\{t_1 \le s\}} \exp{-e^{-\|\Phi\|_\infty}\frac{\nu_0}{4}(t-s)}\left|w\Gamma_+(f,f)(s,X(s),V(s))\right| ds.
	\end{align*}
	First, let us estimate $I_1$.
	\begin{align} \label{aa7}
		I_1 \le  \int_0^t \mathbf{1}_{\{t_1 \le s\}} \exp{-e^{-\|\Phi\|_\infty}\frac{\nu_0}{4}(t-s)}\int_{\mathbb{R}^3}|k_w(V(s),v')||h(s,X(s),v')|dv'ds.
	\end{align}
	Applying Lemma \ref{Linfty1} to $|h(s,X(s),v')|$ in \eqref{aa7}, we obtain
	\begin{align*}
		I_1 &\le C_\Phi  \exp\left\{e^{-\|\Phi\|_\infty}\nu_0\tilde{t}\right\}\int_0^t \exp{-e^{-\|\Phi\|_\infty}\frac{\nu_0}{4}(t-s)} \bigg[\exp\left\{-e^{-\|\Phi\|_\infty} \frac{\nu_0}{4} s \right\}\|h_0\|_{L^\infty_{x,v}}\\
		&\qquad +\left(\epsilon+\frac{C_{\epsilon,T_1}}{R}+C_{R,\epsilon,T_1}\ \tilde{\epsilon}\right)\sup_{0\le s \le t} \left[\|h(s)\|_{L^\infty_{x,v}}+\|h(s)\|_{L^\infty_{x,v}}^2+\|h(s)\|_{L^\infty_{x,v}}^3\right]\\
		& \qquad + C_{\Phi,R, \delta_*,\epsilon,\tilde{\epsilon},T_1}\left[\mathcal{E}(F_0)^{\frac{1}{2}}+\mathcal{E}(F_0)\right] \bigg]ds\\
		& \quad +C_\Phi \exp{e^{-\|\Phi\|_\infty}\frac{\nu_0}{4}\tilde{t}} \int_0^t \int_0^s \int_{\mathbb{R}^3}\int_{\mathbb{R}^3} \mathbf{1}_{\{t_1 \le s\}} \mathbf{1}_{\{t_1' \le s'\}}\exp{-e^{-\|\Phi\|_\infty}\frac{\nu_0}{4}(t-s')} |k_w(V(s),v')|\\
		& \qquad \times |k_w(V'(s'),v'')||h(s',X'(s'),v'')|dv''dv'ds'ds\\
		& \quad + C_\Phi \exp{e^{-\|\Phi\|_\infty}\frac{\nu_0}{4}\tilde{t}}\int_0^t \int_0^s \int_{\mathbb{R}^3}\mathbf{1}_{\{t_1 \le s\}} \mathbf{1}_{\{t_1' \le s'\}}\exp{-e^{-\|\Phi\|_\infty}\frac{\nu_0}{4}(t-s')} |k_w(V(s),v')|\\
		& \qquad \times \left|w\Gamma_+(f,f)(s',X'(s'),V'(s')\right|dv'ds'ds\\
		& =: I_{11}+I_{12}+I_{13},
	\end{align*}
	where $t_1' = t_1(s,X(s),V(s))$.\\
	By simple computation, we deduce
	\begin{align}
		I_{11} &\le  C_\Phi  \exp\left\{e^{-\|\Phi\|_\infty}\nu_0\tilde{t}\right\}  \bigg[\exp{-e^{-\|\Phi\|_\infty}\frac{\nu_0}{8}t}\|h_0\|_{L^\infty_{x,v}}\nonumber\\
		& \qquad +\left(\epsilon+\frac{C_{\epsilon,T_1}}{R}+C_{R,\epsilon,T_1}\ \tilde{\epsilon}\right)\sup_{0\le s \le t} \left[\|h(s)\|_{L^\infty_{x,v}}+\|h(s)\|_{L^\infty_{x,v}}^2+\|h(s)\|_{L^\infty_{x,v}}^3\right]\nonumber\\\label{aa8}
		& \qquad + C_{\Phi,R, \delta_*,\epsilon,\tilde{\epsilon},T_1}\left[\mathcal{E}(F_0)^{\frac{1}{2}}+\mathcal{E}(F_0)\right] \bigg].
	\end{align}
	We divide three cases to estimate $I_{12}$.\\
	\newline
	$\mathbf{Case\ 1:}$ $|v| \ge R$.\\
	From Lemma \ref{Kest}, we compute that
	\begin{align*}
		I_{12} \le \frac{C_\Phi}{R}\exp\left\{e^{-\|\Phi\|_\infty} \frac{\nu_0}{4} \tilde{t} \right\}\sup_{0\le s \le t}\|h(s)\|_{L^\infty_{x,v}}.
	\end{align*}
	\newline
	$\mathbf{Case\ 2:}$ $|v| \le R$, $|v'| \ge 2R$ or $|v'|\le 2R$, $|v''| \ge 3R$ with $R \gg 2\sqrt{2\|\Phi\|_\infty}$.\\
	Note that either $|v-v'| \ge R$ or $|v'-v''| \ge R$. From \eqref{sese}, either one of the followings holds:
	\begin{align*}
		&|V(s)-v'| \ge |v-v'|-|V(s)-v| \ge R-\frac{R}{2}=\frac{R}{2},\\
		&|V'(s')-v''| \ge |v'-v''|-|V'(s')-v'| \ge R-\frac{R}{2}=\frac{R}{2}.
	\end{align*}
	Then we have either one of the followings:
	\begin{equation} \label{aa9}
	\begin{aligned}
		&|k_w(V(s),v')| \le e^{-\frac{R^2}{64}}|k_w(V(s),v')|e^{\frac{1}{16}|V(s)-v'|^2},\\
		&|k_w(V'(s'),v'')| \le e^{-\frac{R^2}{64}}|k_w(V'(s'),v'')|e^{\frac{1}{16}|V'(s')-v''|^2}.
	\end{aligned}
	\end{equation}
	This yields from Lemma \ref{Kest},
	\begin{equation} \label{aa10}
	\begin{aligned}
		&\int_{|v'| \ge 2R} |k_w(V(s),v')|e^{\frac{1}{16}|V(s)-v'|^2}dv' < C,\\
		&\int_{|v''| \ge 3R} |k_w(V'(s'),v'')|e^{\frac{1}{16}|V'(s')-v''|^2}dv'' < C,
	\end{aligned}
	\end{equation}
	for some constant $C$.\\
	It follows from \eqref{aa9} and \eqref{aa10} that
	\begin{align*}
		I_{12} \le C_\Phi e^{-\frac{R^2}{64}}\exp{e^{-\|\Phi\|_\infty} \frac{\nu_0}{4}\tilde{t}}\sup_{0\le s \le t}\|h(s)\|_{L^\infty_{x,v}}.
	\end{align*}
	\newline
	$\mathbf{Case\ 3:}$ $|v| \le R$, $|v'| \le 2R$, $|v''| \le 3R$\\
	Since $k_w(v,v')$ has possible integrable singularity of $\frac{1}{|v-v'|}$, we can choose smooth function $k_R(v,v')$ with compact support such that
	\begin{align*}
		\sup_{|v|\le 3R}\int_{|v'| \le 3R} \left|k_R(v,v')-k_w(v,v')\right|dv' \le \frac{1}{R}.
	\end{align*}
	We split 
	\begin{align*}
		k_w(V(s),v')k_w(V'(s'),v'') &= \left\{k_w(V(s),v')-k_R(V(s),v')\right\}k_w(V'(s'),v'')\\&\quad + \{k_w(V'(s'),v'')-k_R(V'(s'),v'')\}k_R(V(s),v')\\
		&\quad +k_R(V(s),v')k_R(V'(s'),v'').
	\end{align*}
	Then $I_{12}$ in this case is bounded by 
	\begin{align}
		&\frac{C_\Phi}{R}\exp{e^{-\|\Phi\|_\infty} \frac{\nu_0}{4}\tilde{t}}\sup_{0 \le s \le t}\|h(s')\|_{L^\infty_{x,v}} \nonumber \\ \label{aa11} 
		&\quad + C_{R,\Phi}\exp{e^{-\|\Phi\|_\infty} \frac{\nu_0}{4}\tilde{t}}\int_0^t \int_0^s \exp\left\{-e^{-\|\Phi\|_\infty}\frac{\nu_0}{4}(t-s')\right\} \int_{|v'|\le 2R, |v''|\le 3R}|h(s',X'(s'),v'')|dv''dv'ds'ds,
	\end{align}
	where we have used the fact $|k_R(V(s),v')||k_R(V'(s'),v'')| \le C_R$.\\
	In this case, we recall that $X'(s') = X(s';s,X(s;t,x,v),v')$.
	Since the potential is time dependent, we have
	\begin{align*}
		X(s';s,X(s;t,x,v),v') = X(s'-s+T_1;T_1,X(T_1;t-s+T_1,x,v),v')
	\end{align*}
	for all $0\le s' \le s \le t$.\\
	By Lemma \ref{cpl}, the term \eqref{aa11} becomes
	\begin{align} 
		&C_{R,\Phi}\exp{e^{-\|\Phi\|_\infty} \frac{\nu_0}{4}\tilde{t}}\sum_{i_1}^{M_1}\sum_{I_2}^{(M_2)^3}\sum_{I_3}^{(M_3)^3} \int_0^t \mathbf{1}_{\{X(T_1;t-s+T_1,x,v)\in \mathcal{P}_{I_2}^{\Omega}\}}(s)\int_0^s\mathbf{1}_{\mathcal{P}_{i_1}^{T_1}}(s'-s+T_1) \exp\left\{-e^{-\|\Phi\|_\infty}\frac{\nu_0}{4}(t-s')\right\} \nonumber \\ \label{aa12}
		& \quad \times \int_{|v'|\le 2R, |v''|\le 3R}\mathbf{1}_{\mathcal{P}_{I_3}^v}(v')|h(s',X(s'-s+T_1;T_1,X(T_1;t-s+T_1,x,v),v'),v'')|dv''dv'ds'ds.
	\end{align}
	From Lemma \ref{cpl}, we have the following partitions:
	\begin{align*}
		&\bigg\{(s'-s+T_1,X(T_1;t-s+T_1,x,v),v')\in \mathcal{P}_{i_1}^{T_1}\cross \mathcal{P}_{I_2}^{\Omega}\cross \mathcal{P}_{I_3}^{v}\\
		&\qquad : \det\left(\frac{dX}{dv'}(s'-s+T_1;T_1,X(T_1;t-s+T_1,x,v),v')\right)=0\bigg\}\\
		& \subset \bigcup_{j=1}^3 \bigg\{(s'-s+T_1,X(T_1;t-s+T_1,x,v),v')\in \mathcal{P}_{i_1}^{T_1}\cross \mathcal{P}_{I_2}^{\Omega}\cross \mathcal{P}_{I_3}^{v}\\
		& \qquad :s'-s+T_1\in \left(t_{j,i_1,I_2,I_3}-\frac{\tilde{\epsilon}}{4M_1},t_{j,i_1,I_2,I_3}+\frac{\tilde{\epsilon}}{4M_1}\right)\bigg\}.
	\end{align*}
	Thus for each $i_1$,$I_2$, and $I_3$, we split $\mathbf{1}_{\mathcal{P}_{i_1}^{T_1}}(s'-s+T_1)$ as
	\begin{align} \label{aa13}
		&\mathbf{1}_{\mathcal{P}_{i_1}^{T_1}}(s'-s+T_1)\mathbf{1}_{\cup_{j=1}^3(t_{j,i_1,I_2,I_3}-\frac{\tilde{\epsilon}}{4M_1},t_{j,i_1,I_2,I_3}+\frac{\tilde{\epsilon}}{4M_1})}(s'-s+T_1)\\ \label{aa14}
		& +\mathbf{1}_{\mathcal{P}_{i_1}^{T_1}}(s'-s+T_1)\left\{1-\mathbf{1}_{\cup_{j=1}^3(t_{j,i_1,I_2,I_3}-\frac{\tilde{\epsilon}}{4M_1},t_{j,i_1,I_2,I_3}+\frac{\tilde{\epsilon}}{4M_1})}(s'-s+T_1)\right\}.
	\end{align}
	$\mathbf{Case\ 3 \ (\romannumeral 1):}$ The integration \eqref{aa12} corresponding to \eqref{aa13} is bounded by
	\begin{align}
		&C_{R,\Phi}\exp{e^{-\|\Phi\|_\infty} \frac{\nu_0}{4}\tilde{t}}\sum_{i_1}^{M_1}\sum_{I_2}^{(M_2)^3}\sum_{I_3}^{(M_3)^3} \sum_{j=1}^3 \int_0^t \mathbf{1}_{\{X(T_1;t-s+T_1,x,v)\in \mathcal{P}_{I_2}^{\Omega}\}}(s) \nonumber \\
		&\quad \times \underbrace{\int_0^s\mathbf{1}_{\mathcal{P}_{i_1}^{T_1}}(s'-s+T_1)\mathbf{1}_{(t_{j,i_1,I_2,I_3}-\frac{\tilde{\epsilon}}{4M_1},t_{j,i_1,I_2,I_3}+\frac{\tilde{\epsilon}}{4M_1})}(s'-s+T_1)}_{(*9)} \exp\left\{-e^{-\|\Phi\|_\infty} \frac{\nu_0}{4} (t-s') \right\} \nonumber \\ \label{aa15}
		&\quad \times \int_{|v'|\le 2R}\mathbf{1}_{\mathcal{P}_{I_3}^v}(v')\int_{|v''|\le 3R}|h(s',X(s'-s+T_1;T_1,X(T_1;t-s+T_1,x,v),v'),v'')|dv''dv'ds'ds.
	\end{align}
	Here, the term $(*9)$ is bounded by
	\begin{align} \label{qq41}
		\int_0^s \mathbf{1}_{\mathcal{P}_{i_1}^{T_1}}(s'-s+T_1)\mathbf{1}_{(t_{j,i_1,I_2,I_3}-\frac{\tilde{\epsilon}}{4M_1},t_{j,i_1,I_2,I_3}+\frac{\tilde{\epsilon}}{4M_1})}(s'-s+T_1)ds'  \le \frac{\tilde{\epsilon}}{2M_1}.
	\end{align}
	From the partition of the time interval $[0,T_1]$ and the velocity domain $[-4R,4R]^3$ in Lemma \ref{cpl}, we have
	\begin{equation} \label{qq42}
	\begin{aligned}
		&\sum_{I_2}^{(M_2)^3}\mathbf{1}_{\{X(t-s+T_1,x,v)\in \mathcal{P}_{I_2}^{\Omega}\}}(s) \le \mathbf{1}_{\{0\le s \le T_1\}}(s),\\
		&\sum_{I_3}^{(M_3)^3} \mathbf{1}_{\mathcal{P}_{I_3}^v}(v')\mathbf{1}_{\{|v'|\le 2R\}}(v') = \mathbf{1}_{\{|v'|\le 2R\}}(v').
	\end{aligned}
	\end{equation}
	From \eqref{qq41} and \eqref{qq42}, \eqref{aa15} is bounded by
	\begin{align*}
		\tilde{\epsilon} \ C_{R,\Phi}\exp{e^{-\|\Phi\|_\infty} \frac{\nu_0}{4}\tilde{t}} \sup_{0\le s \le t} \|h(s)\|_{L^\infty_{x,v}}.
	\end{align*}
	\newline
	$\mathbf{Case\ 3 \ (\romannumeral 2):}$ The integration \eqref{aa12} corresponding to \eqref{aa14} is bounded by
	\begin{align}
		&C_{R,\Phi}\exp{e^{-\|\Phi\|_\infty} \frac{\nu_0}{4}\tilde{t}}\sum_{i_1}^{M_1}\sum_{I_2}^{(M_2)^3}\sum_{I_3}^{(M_3)^3} \int_0^t \mathbf{1}_{\{X(T_1;t-s+T_1,x,v)\in \mathcal{P}_{I_2}^{\Omega}\}}(s)\int_0^s\mathbf{1}_{\mathcal{P}_{i_1}^{T_1}}(s'-s+T_1) \nonumber \\
		& \times \left\{1-\mathbf{1}_{\cup_{j=1}^3(t_{j,i_1,I_2,I_3}-\frac{\tilde{\epsilon}}{4M_1},t_{j,i_1,I_2,I_3}+\frac{\tilde{\epsilon}}{4M_1})}(s'-s+T_1)\right\}\exp\left\{-e^{-\|\Phi\|_\infty} \frac{\nu_0}{4} (t-s') \right\} \nonumber \\ \label{aa16}
		&\times \underbrace{\int_{|v'|\le 2R}\mathbf{1}_{\mathcal{P}_{I_3}^v}(v')\int_{|v''|\le 3R}|h(s',X(s'-s+T_1;T_1,X(T_1;t-s+T_1,x,v),v'),v'')|dv''dv'}_{(\# 9)}ds'ds.
	\end{align}
	By Lemma \ref{cpl}, we have made a change of variables $v' \rightarrow y:=X(s'-s+T_1;T_1,X(T_1;t-s+T_1,x,v),v')$ satisfying 
	\begin{align*}
		\det\left(\frac{dX}{dv'}(s'-s+T_1;T_1,X(T_1;t-s+T_1,x,v),v')\right)> \delta_*
	\end{align*}
	and the term $(\# 9)$ is bounded by
	\begin{align*}
		&\int_{|v'|\le 2R}\int_{|v''|\le 3R}|h(s',X(s'-s+T_1;T_1,X(T_1;t-s+T_1,x,v),v'),v'')|dv''dv'\\
		& \le \frac{C_{R,\Phi}}{\delta_*} \left(\int_{\Omega} \int_{|v''|\le 3R}|h(s',y,v'')|^2dv''dy\right)^{\frac{1}{2}}.
	\end{align*}
	Thus \eqref{aa16} is bounded by
	\begin{align*}
		C_{R,\Phi,M_1,M_2,M_3,\delta_*} \exp{e^{-\|\Phi\|_\infty} \frac{\nu_0}{4}\tilde{t}}\int_0^{t}\int_0^s\exp\left\{-e^{-\|\Phi\|_\infty}\frac{\nu_0}{4}(t-s')\right\}\left(\int_{\Omega} \int_{|v''|\le 3R}|h(s',y,v'')|^2dv''dy\right)^{\frac{1}{2}}ds'ds.
	\end{align*}
	Hence from \eqref{L2cont}, \eqref{aa16} is bounded by
	\begin{align*}
		\frac{C_\Phi}{R} \exp{e^{-\|\Phi\|_\infty} \frac{\nu_0}{4}\tilde{t}}\sup_{0\le s\le t} \|h(s)\|_{L^\infty_{x,v}} + C_{R,\Phi,M_1,M_2,M_3,\delta_*}\exp{e^{-\|\Phi\|_\infty} \frac{\nu_0}{4}\tilde{t}}\left[\mathcal{E}(F_0)^{\frac{1}{2}}+\mathcal{E}(F_0)\right].
	\end{align*}
	\newline
	Next, we divide four cases to estimate $I_{13}$. By Lemma \ref{Gam+est}, we have
	\begin{align*}
		I_{13} &\le C_\Phi \exp{e^{-\|\Phi\|_\infty} \frac{\nu_0}{4}\tilde{t}} \int_0^t\int_0^s\int_{\mathbb{R}^3}\mathbf{1}_{\{t_1 \le s\}} \mathbf{1}_{\{t_1' \le s'\}}\exp{-e^{-\|\Phi\|_\infty}\frac{\nu_0}{4}(t-s')} |k_w(V(s),v')|\\
		& \qquad \times \sup_{0\le s \le t}\|h(s)\|_{L^\infty_{x,v}}\left(\int_{\mathbb{R}^3}(1+|\eta|)^{-2\beta+4}|h(s',X'(s'),\eta)|^2d\eta\right)^{\frac{1}{2}} dv'ds'ds.
	\end{align*}
	\newline
	$\mathbf{Case\ 1:}$ $|v| \ge R$.\\
	From Lemma \ref{Kest}, we compute that
	\begin{align*}
		I_{13} \le \frac{C_\Phi}{R}\exp\left\{e^{-\|\Phi\|_\infty} \frac{\nu_0}{4} \tilde{t} \right\}\sup_{0\le s \le t}\|h(s)\|_{L^\infty_{x,v}}^2.
	\end{align*}
	\newline
	$\mathbf{Case\ 2:}$ $|v| \le R$, $|v'| \ge 2R$.\\
	Note that $|v-v'| \ge R$. From \eqref{sese}, it holds that
	\begin{align*}
		&|V(s)-v'| \ge |v-v'|-|V(s)-v| \ge R-\frac{R}{2}=\frac{R}{2}.
	\end{align*}
	Then we have 
	\begin{equation} \label{aa17}
	\begin{aligned}
		&|k_w(V(s),v')| \le e^{-\frac{R^2}{64}}|k_w(V(s),v')|e^{\frac{1}{16}|V(s)-v'|^2}.
	\end{aligned}
	\end{equation}
	This yields from Lemma \ref{Kest},
	\begin{equation} \label{aa18}
	\begin{aligned}
		&\int_{|v'| \ge 2R} |k_w(V(s),v')|e^{\frac{1}{16}|V(s)-v'|^2}dv' < C,
	\end{aligned}
	\end{equation}
	for some constant $C>0$.\\
	It follows from \eqref{aa17} and \eqref{aa18} that
	\begin{align*}
		I_{13} \le C_\Phi e^{-\frac{R^2}{64}}\exp{e^{-\|\Phi\|_\infty} \frac{\nu_0}{4}\tilde{t}}\sup_{0\le s \le t}\|h(s)\|_{L^\infty_{x,v}}^2.
	\end{align*}
	\newline
	$\mathbf{Case\ 3:}$ $|v| \le R$, $|v'| \le 2R$, $|\eta| \ge R$.\\
	It follows from Lemma \ref{Kest} that
	\begin{align*}
		I_{13} &\le C_\Phi \exp{e^{-\|\Phi\|_\infty} \frac{\nu_0}{4}\tilde{t}}\sup_{0\le s \le t}\|h(s)\|_{L^\infty_{x,v}}^2\left(\int_{|\eta|\ge R}(1+|\eta|)^{-2\beta+4}d\eta\right)^{\frac{1}{2}}\\
		& \le \frac{C_\Phi}{R} \exp{e^{-\|\Phi\|_\infty} \frac{\nu_0}{4}\tilde{t}}\sup_{0\le s \le t}\|h(s)\|_{L^\infty_{x,v}}^2.
	\end{align*}
	\newline
	$\mathbf{Case\ 4:}$ $|v| \le R$, $|v'| \le 2R$, $|\eta| \le R$.\\
	Since $k_w(v,v')$ has possible integrable singularity of $\frac{1}{|v-v'|}$, we can choose smooth function $k_R(v,v')$ with compact support such that
	\begin{align*}
		\sup_{|v|\le 2R}\int_{|v'| \le 2R} \left|k_R(v,v')-k_w(v,v')\right|dv' \le \frac{1}{R}.
	\end{align*}
	We split 
	\begin{align*}
		k_w(V(s),v') = \{k_w(V(s),v') - k_R(V(s),v')\}+k_R(V(s),v').
	\end{align*}
	Then $I_{13}$ in this case is bounded by
	\begin{align}
		&\frac{C_\Phi}{R}\exp{e^{-\|\Phi\|_\infty} \frac{\nu_0}{4}\tilde{t}}\sup_{0 \le s \le t}\|h(s)\|_{L^\infty_{x,v}}^2 \nonumber \\  
		&\quad + C_{R,\Phi}\exp{e^{-\|\Phi\|_\infty} \frac{\nu_0}{4}\tilde{t}}\sup_{0\le s \le t} \|h(s)\|_{L^\infty_{x,v}}\int_0^t \int_0^s \exp\left\{-e^{-\|\Phi\|_\infty}\frac{\nu_0}{4}(t-s')\right\}\nonumber\\ \label{aa19}
		& \qquad \times \int_{|v'|\le 2R} \left(\int_{|\eta|\le R}(1+|\eta|)^{-2\beta+4}|h(s',X'(s'),\eta)|^2d\eta\right)^{\frac{1}{2}}dv'ds'ds,
	\end{align}
	where we have used the fact $|k_R(V(s),v')| \le C_R$.\\
	In this case, we recall that $X'(s') = X(s';s,X(s;t,x,v),v')$.
	Since the potential is time dependent, we have
	\begin{align*}
		X(s';s,X(s;t,x,v),v') = X(s'-s+T_1;T_1,X(T_1;t-s+T_1,x,v),v')
	\end{align*}
	for all $0\le s' \le s \le t$.\\
	By Lemma \ref{cpl}, the term \eqref{aa19} becomes
	\begin{align} 
		&C_{R,\Phi}\exp{e^{-\|\Phi\|_\infty} \frac{\nu_0}{4}\tilde{t}}\sup_{0\le s \le t} \|h(s)\|_{L^\infty_{x,v}}\sum_{i_1}^{M_1}\sum_{I_2}^{(M_2)^3}\sum_{I_3}^{(M_3)^3} \int_0^t \mathbf{1}_{\{X(T_1;t-s+T_1,x,v)\in \mathcal{P}_{I_2}^{\Omega}\}}(s) \nonumber \\ 
		& \quad \times \int_0^s\mathbf{1}_{\mathcal{P}_{i_1}^{T_1}}(s'-s+T_1) \exp\left\{-e^{-\|\Phi\|_\infty}\frac{\nu_0}{4}(t-s')\right\} \int_{|v'|\le 2R}\mathbf{1}_{\mathcal{P}_{I_3}^v}(v')\nonumber\\ \label{aa20}
		& \quad \times \left(\int_{|\eta|\le R}|h(s',X(s'-s+T_1;T_1,X(T_1;t-s+T_1,x,v),v'),\eta)|^2d\eta\right)^{\frac{1}{2}}dv'ds'ds.
	\end{align}
	From Lemma \ref{cpl}, we have the following partitions:
	\begin{align*}
		&\bigg\{(s'-s+T_1,X(T_1;t-s+T_1,x,v),v')\in \mathcal{P}_{i_1}^{T_1}\cross \mathcal{P}_{I_2}^{\Omega}\cross \mathcal{P}_{I_3}^{v}\\
		&\qquad : \det\left(\frac{dX}{dv'}(s'-s+T_1;T_1,X(T_1;t-s+T_1,x,v),v')\right)=0\bigg\}\\
		& \subset \bigcup_{j=1}^3 \bigg\{(s'-s+T_1,X(T_1;t-s+T_1,x,v),v')\in \mathcal{P}_{i_1}^{T_1}\cross \mathcal{P}_{I_2}^{\Omega}\cross \mathcal{P}_{I_3}^{v}\\
		& \qquad :s'-s+T_1\in \left(t_{j,i_1,I_2,I_3}-\frac{\tilde{\epsilon}}{4M_1},t_{j,i_1,I_2,I_3}+\frac{\tilde{\epsilon}}{4M_1}\right)\bigg\}.
	\end{align*}
	Thus for each $i_1$,$I_2$, and $I_3$, we split $\mathbf{1}_{\mathcal{P}_{i_1}^{T_1}}(s'-s+T_1)$ as
	\begin{align} \label{aa21}
		&\mathbf{1}_{\mathcal{P}_{i_1}^{T_1}}(s'-s+T_1)\mathbf{1}_{\cup_{j=1}^3(t_{j,i_1,I_2,I_3}-\frac{\tilde{\epsilon}}{4M_1},t_{j,i_1,I_2,I_3}+\frac{\tilde{\epsilon}}{4M_1})}(s'-s+T_1)\\ \label{aa22}
		& +\mathbf{1}_{\mathcal{P}_{i_1}^{T_1}}(s'-s+T_1)\left\{1-\mathbf{1}_{\cup_{j=1}^3(t_{j,i_1,I_2,I_3}-\frac{\tilde{\epsilon}}{4M_1},t_{j,i_1,I_2,I_3}+\frac{\tilde{\epsilon}}{4M_1})}(s'-s+T_1)\right\}.
	\end{align}
	\newline
	$\mathbf{Case\ 4 \ (\romannumeral 1):}$ The integration \eqref{aa20} corresponding to \eqref{aa21} is bounded by
	\begin{align}
		&C_{R,\Phi}\exp{e^{-\|\Phi\|_\infty} \frac{\nu_0}{4}\tilde{t}}\sup_{0\le s \le t} \|h(s)\|_{L^\infty_{x,v}}\sum_{i_1}^{M_1}\sum_{I_2}^{(M_2)^3}\sum_{I_3}^{(M_3)^3} \sum_{j=1}^3 \int_0^t \mathbf{1}_{\{X(T_1;t-s+T_1,x,v)\in \mathcal{P}_{I_2}^{\Omega}\}}(s) \nonumber \\
		&\quad \times  \underbrace{\int_0^s\mathbf{1}_{\mathcal{P}_{i_1}^{T_1}}(s'-s+T_1)\mathbf{1}_{(t_{j,i_1,I_2,I_3}-\frac{\tilde{\epsilon}}{4M_1},t_{j,i_1,I_2,I_3}+\frac{\tilde{\epsilon}}{4M_1})}(s'-s+T_1)}_{(*10)} \exp\left\{-e^{-\|\Phi\|_\infty} \frac{\nu_0}{4} (t-s') \right\} \nonumber \\ \label{aa23}
		&\quad \times \int_{|v'|\le 2R}\mathbf{1}_{\mathcal{P}_{I_3}^v}(v')\left(\int_{|\eta|\le R}|h(s',X(s'-s+T_1;T_1,X(T_1;t-s+T_1,x,v),v'),\eta)|^2d\eta\right)^{\frac{1}{2}}dv'ds'ds.
	\end{align}
	Here, $(*10)$ is bounded by
	\begin{align} \label{qq43}
		\int_0^s \mathbf{1}_{\mathcal{P}_{i_1}^{T_1}}(s'-s+T_1)\mathbf{1}_{(t_{j,i_1,I_2,I_3}-\frac{\tilde{\epsilon}}{4M_1},t_{j,i_1,I_2,I_3}+\frac{\tilde{\epsilon}}{4M_1})}(s'-s+T_1)ds'
		 \le \frac{\tilde{\epsilon}}{2M_1}.
	\end{align}
	From the partition of the time interval $[0,T_1]$ and the velocity domain $[-4R,4R]^3$ in Lemma \ref{cpl}, we have
	\begin{equation} \label{qq44}
	\begin{aligned}
		&\sum_{I_2}^{(M_2)^3}\mathbf{1}_{\{X(t-s+T_1,x,v)\in \mathcal{P}_{I_2}^{\Omega}\}}(s) \le \mathbf{1}_{\{0\le s \le T_1\}}(s),\\
		&\sum_{I_3}^{(M_3)^3} \mathbf{1}_{\mathcal{P}_{I_3}^v}(v')\mathbf{1}_{\{|v'|\le 2R\}}(v') = \mathbf{1}_{\{|v'|\le 2R\}}(v').
	\end{aligned}
	\end{equation}
	From \eqref{qq43} and \eqref{qq44}, \eqref{aa23} is bounded by
	\begin{align*}
		\tilde{\epsilon} \ C_{R,\Phi}\exp{e^{-\|\Phi\|_\infty} \frac{\nu_0}{4}\tilde{t}} \sup_{0\le s \le t} \|h(s)\|_{L^\infty_{x,v}}^2.
	\end{align*}
	\newline
	$\mathbf{Case\ 4 \ (\romannumeral 2):}$ The integration \eqref{aa20} corresponding to \eqref{aa22} is bounded by
	\begin{align}
		&C_{R,\Phi}\exp{e^{-\|\Phi\|_\infty} \frac{\nu_0}{4}\tilde{t}}\sup_{0\le s \le t} \|h(s)\|_{L^\infty_{x,v}}\sum_{i_1}^{M_1}\sum_{I_2}^{(M_2)^3}\sum_{I_3}^{(M_3)^3} \int_0^t \mathbf{1}_{\{X(T_1;t-s+T_1,x,v)\in \mathcal{P}_{I_2}^{\Omega}\}}(s)\int_0^s\mathbf{1}_{\mathcal{P}_{i_1}^{T_1}}(s'-s+T_1) \nonumber \\
		& \times \left\{1-\mathbf{1}_{\cup_{j=1}^3(t_{j,i_1,I_2,I_3}-\frac{\tilde{\epsilon}}{4M_1},t_{j,i_1,I_2,I_3}+\frac{\tilde{\epsilon}}{4M_1})}(s'-s+T_1)\right\}\exp\left\{-e^{-\|\Phi\|_\infty} \frac{\nu_0}{4} (t-s') \right\} \nonumber \\ \label{aa24}
		&\times \left(\int_{|v'|\le 2R}\int_{|\eta|\le R}|h(s',X(s'-s+T_1;T_1,X(T_1;t-s+T_1,x,v),v'),\eta)|^2d\eta dv'\right)^{\frac{1}{2}}ds'ds.
	\end{align}
	By Lemma \ref{cpl}, we have made a change of variables $v' \rightarrow y:=X(s'-s+T_1;T_1,X(T_1;t-s+T_1,x,v),v')$ satisfying 
	\begin{align*}
		\det\left(\frac{dX}{dv'}(s'-s+T_1;T_1,X(T_1;t-s+T_1,x,v),v')\right)> \delta_*
	\end{align*}
	and  \eqref{aa24} is bounded by
	\begin{align*}
		&C_{R,\Phi,M_1,M_2,M_3,\delta_*} \exp{e^{-\|\Phi\|_\infty} \frac{\nu_0}{4}\tilde{t}}\sup_{0\le s \le t} \|h(s)\|_{L^\infty_{x,v}}\int_0^{t}\int_0^s\exp\left\{-e^{-\|\Phi\|_\infty}\frac{\nu_0}{4}(t-s')\right\}\\
		& \quad \times \left(\int_{\Omega} \int_{|\eta|\le R}|h(s',y,\eta)|^2d\eta dy\right)^{\frac{1}{2}}ds'ds.
	\end{align*}
	Hence from \eqref{L2cont}, \eqref{aa24} is bounded by
	\begin{align*}
		\frac{C_\Phi}{R} \exp{e^{-\|\Phi\|_\infty} \frac{\nu_0}{4}\tilde{t}}\sup_{0\le s\le t} \left[\|h(s)\|_{L^\infty_{x,v}}^2+ \|h(s)\|^{3}_{L^\infty_{x,v}}\right] + C_{R,\Phi,M_1,M_2,M_3,\delta_*}\exp{e^{-\|\Phi\|_\infty} \frac{\nu_0}{4}\tilde{t}} \mathcal{E}(F_0).
	\end{align*}
	\newline
	Combining all cases of $I_1$, we get
	\begin{align}
		&\exp{e^{-\|\Phi\|_\infty} \frac{\nu_0}{4}\tilde{t}}I_1 \nonumber\\
		&\le C_\Phi \exp{e^{-\|\Phi\|_\infty} \frac{5}{4}\nu_0\tilde{t}}\exp\left\{-e^{-\|\Phi\|_\infty} \frac{\nu_0}{8} t \right\}\|h_0\|_{L^\infty_{x,v}}\nonumber\\
		& \quad +C_\Phi \exp{e^{-\|\Phi\|_\infty} \frac{5}{4}\nu_0\tilde{t}}\left(\epsilon+\frac{C_{\epsilon,T_1}}{R}+C_{R,\epsilon,T_1}\ \tilde{\epsilon}\right)\sup_{0\le s \le t} \Big[\|h(s)\|_{L^\infty_{x,v}}+\|h(s)\|_{L^\infty_{x,v}}^2+\|h(s)\|_{L^\infty_{x,v}}^3\Big]\nonumber\\ \label{aa25}
		& \quad + C_{\Phi,R, \delta_*,\epsilon,\tilde{\epsilon},T_1} \exp{e^{-\|\Phi\|_\infty} \frac{5}{4}\nu_0\tilde{t}}\left[\mathcal{E}(F_0)^{\frac{1}{2}}+\mathcal{E}(F_0)\right].
	\end{align}
	\newline
	It remains to estimate $I_2$. By Lemma \ref{Gam+est}, we have
	\begin{align} \label{aa26}
		I_2 \le C_\Phi \int_0^t \mathbf{1}_{\{t_1 \le s\}} \exp{-e^{-\|\Phi\|_\infty}\frac{\nu_0}{4}(t-s)}\|h(s)\|_{L^\infty_{x,v}}\left(\int_{\mathbb{R}^3}(1+|\eta|)^{-2\beta+4}|h(s,X(s),\eta)|^2d\eta\right)^{\frac{1}{2}} ds.
	\end{align}
	\newline
	$\mathbf{Case\ 1:}$ $|\eta| \ge R$.\\
	It is straightforward to get
	\begin{align*}
		I_2 \le \frac{C_\Phi}{R} \sup_{0\le s \le t}\|h(s)\|_{L^\infty_{x,v}}^2.
	\end{align*}
	\newline
	$\mathbf{Case\ 2:}$ $|\eta| \le R$.\\
	Applying Lemma \ref{Linfty1} to $|h(s,X(s), \eta)|$ in \eqref{aa26}, we deduce
	\begin{align*}
		&\left(\int_{|\eta|\le R}(1+|\eta|)^{-2\beta+4}|h(s,X(s),\eta)|^2d\eta\right)^{\frac{1}{2}}\\
		& \le C_\Phi \exp\left\{e^{-\|\Phi\|_\infty}\nu_0\tilde{t}\right\}\exp\left\{-e^{-\|\Phi\|_\infty} \frac{\nu_0}{4} s \right\}\|h_0\|_{L^\infty_{x,v}}\\
		&\quad+ C_\Phi \exp\left\{e^{-\|\Phi\|_\infty}\nu_0\tilde{t}\right\}\left(\epsilon+\frac{C_{\epsilon,T_1}}{R}+C_{R,\epsilon,T_1}\ \tilde{\epsilon}\right)\sup_{0\le s \le t} \left[\|h(s)\|_{L^\infty_{x,v}}+\|h(s)\|_{L^\infty_{x,v}}^2+\|h(s)\|_{L^\infty_{x,v}}^3\right]\\
		&\quad+C_{\Phi,R, \delta_*,\epsilon,\tilde{\epsilon},T_1} \exp\left\{e^{-\|\Phi\|_\infty}\nu_0\tilde{t}\right\}\left[\mathcal{E}(F_0)^{\frac{1}{2}}+\mathcal{E}(F_0)\right]\\
		&\quad + \exp{e^{-\|\Phi\|_\infty}\frac{\nu_0}{4}\tilde{t}} \\
		& \qquad  \times \left(\int_{|\eta|\le R} (1+|\eta|)^{-2\beta +4}\left|\int_0^s \mathbf{1}_{\{t_1' \le s'\}} \exp{-e^{-\|\Phi\|_\infty}\frac{\nu_0}{4}(s-s')}\left|K_wh(s',X'(s'),\eta(s'))\right| ds'\right|^2d\eta \right)^{\frac{1}{2}}\\
		&\quad + \exp{e^{-\|\Phi\|_\infty}\frac{\nu_0}{4}\tilde{t}} \\
		& \qquad  \times\left(\int_{|\eta|\le R} (1+|\eta|)^{-2\beta +4}\left|\int_0^s \mathbf{1}_{\{t_1' \le s'\}} \exp{-e^{-\|\Phi\|_\infty}\frac{\nu_0}{4}(s-s')}\left|w\Gamma_+(f,f)(s',X'(s'),\eta(s'))\right| ds'\right|^2d\eta \right)^{\frac{1}{2}},
	\end{align*}
	where $\eta(s') = \eta - \int^s_{s'}\nabla_x \Phi(X(\tau))d\tau$.\\
	Applying the above inequality to \eqref{aa26} and using the Cauchy-Schwarz inequality and Lemma \ref{Gam+est}, $I_2$ is bounded by
	\begin{align}
		&C_\Phi \exp\left\{e^{-\|\Phi\|_\infty}\nu_0\tilde{t}\right\}\exp\left\{-e^{-\|\Phi\|_\infty} \frac{\nu_0}{4} t \right\}\|h_0\|_{L^\infty_{x,v}}\int_0^t\|h(s)\|_{L^\infty_{x,v}}ds \nonumber\\
		&+ C_\Phi \exp\left\{e^{-\|\Phi\|_\infty}\nu_0\tilde{t}\right\}\left(\epsilon+\frac{C_{\epsilon,T_1}}{R}+C_{R,\epsilon,T_1}\ \tilde{\epsilon}\right)\sup_{0\le s \le t} \left[\|h(s)\|_{L^\infty_{x,v}}^2+\|h(s)\|_{L^\infty_{x,v}}^3+\|h(s)\|_{L^\infty_{x,v}}^4\right] \nonumber\\
		&+C_{\Phi,R, \delta_*,\epsilon,\tilde{\epsilon},T_1} \exp\left\{e^{-\|\Phi\|_\infty}\nu_0\tilde{t}\right\}\left[\mathcal{E}(F_0)+\mathcal{E}(F_0)^2\right]\nonumber\\
		&+ C_\Phi \exp{e^{-\|\Phi\|_\infty}\frac{\nu_0}{4}\tilde{t}}\int_0^t \mathbf{1}_{\{t_1 \le s\}} \exp{-e^{-\|\Phi\|_\infty}\frac{\nu_0}{4}(t-s)}\|h(s)\|_{L^\infty_{x,v}}\Bigg(\int_0^s \mathbf{1}_{\{t_1' \le s'\}} \exp{-e^{-\|\Phi\|_\infty}\frac{\nu_0}{4}(s-s')}\nonumber\\ \label{aa27}
		& \quad \times \int_{|\eta|\le R} (1+|\eta|)^{-2\beta +4}\left|\int_{\mathbb{R}^3}k_w(\eta(s'),v'')h(s',X'(s'),v'')dv''\right|^2 d\eta ds' \Bigg)^{\frac{1}{2}}ds \\ 
		& +C_\Phi \exp{e^{-\|\Phi\|_\infty}\frac{\nu_0}{4}\tilde{t}}\int_0^t \mathbf{1}_{\{t_1 \le s\}} \exp{-e^{-\|\Phi\|_\infty}\frac{\nu_0}{4}(t-s)}\|h(s)\|_{L^\infty_{x,v}}\Bigg(\int_0^s \mathbf{1}_{\{t_1' \le s'\}} \exp{-e^{-\|\Phi\|_\infty}\frac{\nu_0}{4}(s-s')} \nonumber\\ \label{aa28}
		& \quad \times \|h(s')\|_{L^\infty_{x,v}}^2\int_{|\eta|\le R} (1+|\eta|)^{-2\beta +4}\int_{\mathbb{R}^3}(1+|v''|)^{-2\beta+4}|h(s',X'(s'),v'')|^2 dv'' d\eta ds' \Bigg)^{\frac{1}{2}}ds \\
		& =: I_{21}+I_{22}+I_{23}+I_{24}+I_{25}. \nonumber
	\end{align}
	From now on, we estimate the terms $I_{24}$ and $I_{25}$. We decompose the term $I_{24}$ into two cases $|v''| \le 2R$ and $|v''| \ge 2R$. From \eqref{aa17} and \eqref{aa18}, we derive
	\begin{align*}
		&\Bigg(\int_0^s \mathbf{1}_{\{t_1' \le s'\}} \exp{-e^{-\|\Phi\|_\infty}\frac{\nu_0}{4}(s-s')}\int_{|\eta|\le R} (1+|\eta|)^{-2\beta +4}\left|\int_{\mathbb{R}^3}k_w(\eta(s'),v'')h(s',X'(s'),v'')dv''\right|^2 d\eta ds' \Bigg)^{\frac{1}{2}}\\
		& \le C_\Phi e^{-\frac{R^2}{64}}\sup_{0\le s \le t}\|h(s)\|_{L^\infty_{x,v}}\\
		& \quad + \Bigg(\int_0^s \mathbf{1}_{\{t_1' \le s'\}} \exp{-e^{-\|\Phi\|_\infty}\frac{\nu_0}{4}(s-s')}\int_{|\eta|\le R} (1+|\eta|)^{-2\beta +4}\left|\int_{|v''|\le 2R}k_w(\eta(s'),v'')h(s',X'(s'),v'')dv''\right|^2 d\eta ds' \Bigg)^{\frac{1}{2}}\\
		& \le C_\Phi e^{-\frac{R^2}{64}}\sup_{0\le s \le t}\|h(s)\|_{L^\infty_{x,v}}\\
		& \quad +C_\Phi \Bigg(\int_0^s \mathbf{1}_{\{t_1' \le s'\}} \exp{-e^{-\|\Phi\|_\infty}\frac{\nu_0}{4}(s-s')}\int_{|\eta|\le R} \int_{|v''|\le 2R}|h(s',X'(s'),v'')|^2dv'' d\eta ds' \Bigg)^{\frac{1}{2}},
	\end{align*}
	where we have used the Cauchy-Schwarz inequality and $\int_{\mathbb{R}^3}|k_w(\eta(s'),v'')|^2 dv''$ is finite.\\
	Using the Cauchy-Schwarz inequality and the above estimate, the term $I_{24}$ is bounded by 
	\begin{align*}
		&C_\Phi e^{-\frac{R^2}{64}}\exp{e^{-\|\Phi\|_\infty}\frac{\nu_0}{4}\tilde{t}}\sup_{0\le s \le t}\|h(s)\|_{L^\infty_{x,v}}^2\\
		&+C_\Phi \exp{e^{-\|\Phi\|_\infty}\frac{\nu_0}{4}\tilde{t}}\Bigg(\int_0^t \mathbf{1}_{\{t_1 \le s\}} \exp{-e^{-\|\Phi\|_\infty}\frac{\nu_0}{2}(t-s)}\|h(s)\|_{L^\infty_{x,v}}^2\int_0^s \mathbf{1}_{\{t_1' \le s'\}} \exp{-e^{-\|\Phi\|_\infty}\frac{\nu_0}{4}(s-s')}\\
		&\quad \times \int_{|\eta|\le R} \int_{|v''|\le 2R}|h(s',X'(s'),v'')|^2dv'' d\eta ds' ds\Bigg)^{\frac{1}{2}}.
	\end{align*}
	By Lemma \ref{cpl} and a similar scheme to estimate \eqref{aa19}, we can bound the term $I_{24}$ by
	\begin{align*}
		C_\Phi \exp{e^{-\|\Phi\|_\infty}\frac{\nu_0}{4}\tilde{t}} \left(\frac{1}{R}+C_{R}\ \tilde{\epsilon}\right)\sup_{0\le s \le t}\left[\|h(s)\|^2_{L^\infty_{x,v}}+\|h(s)\|^3_{L^\infty_{x,v}}\right] + C_{R,\Phi,T_1,\tilde{\epsilon},\delta_*}\exp{e^{-\|\Phi\|_\infty}\frac{\nu_0}{4}\tilde{t}} \mathcal{E}(F_0).
	\end{align*}
	 We decompose the term $I_{25}$ into two cases $|v''| \le 2R$ and $|v''| \ge 2R$. Then we have
	\begin{align*}
		&\Bigg(\int_0^s \mathbf{1}_{\{t_1' \le s'\}} \exp{-e^{-\|\Phi\|_\infty}\frac{\nu_0}{4}(s-s')}\|h(s')\|_{L^\infty_{x,v}}^2 \int_{|\eta|\le R} (1+|\eta|)^{-2\beta +4}\int_{\mathbb{R}^3}(1+|v''|)^{-2\beta+4} \\
		& \quad \times |h(s',X'(s'),v'')|^2 dv'' d\eta ds' \Bigg)^{\frac{1}{2}}\\
		& \le \frac{C_\Phi}{R} \sup_{0\le s \le t}\|h(s)\|_{L^\infty_{x,v}}^2\\
		& \quad +C \Bigg(\int_0^s \mathbf{1}_{\{t_1' \le s'\}} \exp{-e^{-\|\Phi\|_\infty}\frac{\nu_0}{4}(s-s')}\|h(s')\|_{L^\infty_{x,v}}^2 \int_{|\eta|\le R} \int_{|v''|\le 2R}|h(s',X'(s'),v'')|^2 dv'' d\eta ds' \Bigg)^{\frac{1}{2}}
	\end{align*}
	Using the Cauchy-Schwarz inequality and the above estimate, the term $I_{25}$ is bounded by 
	\begin{align*}
		&\frac{C_\Phi}{R} \exp{e^{-\|\Phi\|_\infty}\frac{\nu_0}{4}\tilde{t}}\sup_{0\le s \le t}\|h(s)\|_{L^\infty_{x,v}}^3\\
		&  +C_\Phi \exp{e^{-\|\Phi\|_\infty}\frac{\nu_0}{4}\tilde{t}}\Bigg(\int_0^t \mathbf{1}_{\{t_1 \le s\}} \exp{-e^{-\|\Phi\|_\infty}\frac{\nu_0}{2}(t-s)}\|h(s)\|_{L^\infty_{x,v}}^2\int_0^s \mathbf{1}_{\{t_1' \le s'\}} \exp{-e^{-\|\Phi\|_\infty}\frac{\nu_0}{4}(s-s')}\\
		& \quad \times \|h(s')\|_{L^\infty_{x,v}}^2 \int_{|\eta|\le R} \int_{|v''|\le 2R}|h(s',X'(s'),v'')|^2 dv'' d\eta ds' ds\Bigg)^{\frac{1}{2}}.
	\end{align*}
	By Lemma \ref{cpl} and a similar scheme to estimate \eqref{aa19}, we can bound the term $I_{25}$ by
	\begin{align*}
		&C_\Phi \exp{e^{-\|\Phi\|_\infty}\frac{\nu_0}{4}\tilde{t}} \left(\frac{1}{R}+C_{R}\ \tilde{\epsilon}\right)\sup_{0\le s \le t}\left[\|h(s)\|^3_{L^\infty_{x,v}}+\|h(s)\|^4_{L^\infty_{x,v}}+\|h(s)\|^5_{L^\infty_{x,v}}\right]\\
		& + C_{R,\Phi,T_1,\tilde{\epsilon},\delta_*}\exp{e^{-\|\Phi\|_\infty}\frac{\nu_0}{4}\tilde{t}} \mathcal{E}(F_0).
	\end{align*}
	Combining all cases of $I_2$, we get
	\begin{align}
		&\exp{e^{-\|\Phi\|_\infty} \frac{\nu_0}{4}\tilde{t}}I_2 \nonumber\\
		&\le C_\Phi \exp{e^{-\|\Phi\|_\infty} \frac{5}{4}\nu_0\tilde{t}}\exp\left\{-e^{-\|\Phi\|_\infty} \frac{\nu_0}{8} t \right\}\|h_0\|_{L^\infty_{x,v}}\int_0^t\|h(s)\|_{L^\infty_{x,v}}ds\nonumber\\
		& \quad +C_\Phi \exp{e^{-\|\Phi\|_\infty} \frac{5}{4}\nu_0\tilde{t}}\left(\epsilon+\frac{C_{\epsilon,T_1}}{R}+C_{R,\epsilon,T_1}\ \tilde{\epsilon}\right)\sup_{0\le s \le t} \Big[\|h(s)\|_{L^\infty_{x,v}}^2+\|h(s)\|_{L^\infty_{x,v}}^3+\|h(s)\|_{L^\infty_{x,v}}^4+\|h(s)\|_{L^\infty_{x,v}}^5\Big]\nonumber\\ \label{aa29}
		& \quad + C_{\Phi,R, \delta_*,\epsilon,\tilde{\epsilon},T_1} \exp{e^{-\|\Phi\|_\infty} \frac{5}{4}\nu_0\tilde{t}}\left[\mathcal{E}(F_0)+\mathcal{E}(F_0)^2\right].
	\end{align}
	\newline
	Combining $I_1$, $I_2$, and Lemma \ref{Linfty1}, we therfore conclude that
	\begin{align*}
		|h(t,x,v)| &\le C_7 \exp{2e^{-\|\Phi\|_\infty}\nu_0 \tilde{t}}\|h_0\|_{L^\infty_{x,v}}\left(1+\int_0^t\|h(s)\|_{L^\infty_{x,v}}ds\right)\exp\left\{-e^{-\|\Phi\|_\infty} \frac{\nu_0}{8} t \right\}\\
		& \quad +C_7 \exp{2e^{-\|\Phi\|_\infty}\nu_0 \tilde{t}}\left(\epsilon+\frac{C_{\epsilon,T_1}}{R}+C_{R,\epsilon,T_1}\ \tilde{\epsilon}\right) \sup_{0\le s \le t} \Big[\|h(s)\|_{L^\infty_{x,v}}+\|h(s)\|_{L^\infty_{x,v}}^2+\|h(s)\|_{L^\infty_{x,v}}^3\\
		&\qquad +\|h(s)\|_{L^\infty_{x,v}}^4+\|h(s)\|_{L^\infty_{x,v}}^5\Big]\\
		&\quad + C_{\Phi,R, \delta_*,\epsilon,\tilde{\epsilon},T_1}\exp{2e^{-\|\Phi\|_\infty}\nu_0 \tilde{t}}\left[\mathcal{E}(F_0)^{\frac{1}{2}}+\mathcal{E}(F_0)+\mathcal{E}(F_0)^2\right]
	\end{align*}
	for some constant $C_7$.
	\end{proof}
	
	\bigskip
	
	\subsection{Nonlinear Asymptotic Stability of Large Amplitude solution} \label{NASLAS}
	\begin{proof}[\textbf{Proof of Theorem \ref{mainresult2}}]
		Take $C_8 := \max\left\{C_0, C_3 \rho^{\frac{5}{4}}, C_7 \right\} >1$, and let
		\begin{align} \label{aa30}
			\bar{M} := 8 C_8^2 C_4^4 M_0^5 \exp{\frac{8}{\nu_0e^{-\|\Phi\|_\infty}}C_8 C_4^4 M_0^5}.
		\end{align}
		and
		\begin{align} \label{aa31}
			T_1 := \frac{16}{\nu_0e^{-\|\Phi\|_\infty}}\left(\log\bar{M}+|\log\delta_0|\right).
		\end{align}
		Assume that $\mathcal{E}(F_0) \le \epsilon_0(\bar{M},T_1)$, which is determined by Lemma \ref{Rfest}.
		By the a priori assumption \eqref{Ape} and Lemma \ref{Linfty2}, we get
		\begin{align} \label{aa32}
			\|h(t)\|_{L^\infty_{x,v}} \le C_8 C_4^4 M_0^5 \left(1+\int_0^t \|h(s)\|_{L^\infty_{x,v}}ds\right)\exp{-e^{-\|\Phi\|_\infty} \frac{\nu_0}{8}t}+E \quad \text{for all } 0\le t \le T_1,
		\end{align}
		where
		\begin{align*}
			E:&= C_8 C_4^4 M_0^4 \bigg\{\left(\epsilon +\frac{C_{\epsilon,T_1}}{R}+C_{R,\Phi, \epsilon, T_1}\ \tilde{\epsilon}\right)\left[\bar{M}+\bar{M}^2+\bar{M}^3+\bar{M}^4+\bar{M}^5\right]\\
			& \qquad  +C_{R, \Phi, \delta_*, \epsilon,T_1 ,\tilde{\epsilon}} \left[\mathcal{E}(F_0)^{\frac{1}{2}}+\mathcal{E}(F_0)+\mathcal{E}(F_0)^2 \right] \bigg\}.
		\end{align*}
		We define
		\begin{align*}
			G(t) := 1+\int_0^t \|h(s)\|_{L^\infty_{x,v}}ds.
		\end{align*}
		Then the inequality \eqref{aa32} becomes
		\begin{align} \label{aa33}
			G'(t) \le C_8 C_4^4 M_0^5 \exp{-e^{-\|\Phi\|_\infty}\frac{\nu_0}{8}t} G(t) + E.
		\end{align}
		By Gr\"{o}nwall's inequality, \eqref{aa33} implies that
		\begin{equation} \label{aa34}
		\begin{aligned}
			G(t) &\le (1+Et) \exp{\frac{8}{\nu_0 e^{-\|\Phi\|_\infty}}C_8 C_4^4 M_0^5 \left(1-\exp{-e^{-\|\Phi\|_\infty}\frac{\nu_0}{8}t}\right) }\\
			& \le (1+Et)\exp{\frac{8}{\nu_0 e^{-\|\Phi\|_\infty}}C_8 C_4^4 M_0^5 }
		\end{aligned}
		\end{equation}
		for all $0\le t \le T_1$.\\
		Substituting \eqref{aa34} into \eqref{aa32}, we deduce for $0\le t \le T_1$
		\begin{align*}
			\|h(t)\|_{L^\infty_{x,v}} &\le C_8 C_4^4 M_0^5 \exp{\frac{8}{\nu_0 e^{-\|\Phi\|_\infty}}C_8 C_4^4 M_0^5 }(1+Et)\exp{-e^{-\|\Phi\|_\infty} \frac{\nu_0}{8}t}+E\\
			& \le \frac{1}{8C_8} \bar{M} (1+Et)\exp{-e^{-\|\Phi\|_\infty} \frac{\nu_0}{8}t}+E\\
			& \le \frac{1}{8C_8} \bar{M} \left(1+\frac{16}{\nu_0 e^{-\|\Phi\|_\infty}}E \right)\exp{-e^{-\|\Phi\|_\infty} \frac{\nu_0}{16}t} +E.
		\end{align*}
		We first choose $\epsilon >0$ small enough, then choose $R>0$ sufficiently large, and take $\tilde{\epsilon}>0$ small enough and assume $\mathcal{E}(F_0) \le \epsilon_1(\delta, M_0)$ with small enough $\epsilon_1(\delta, M_0)$ so that
		\begin{align*}
			E \le \min\left\{\frac{\nu_0e^{-\|\Phi\|_\infty}}{32} , \frac{\delta_0}{4}\right\},
		\end{align*}
		and it follows that
		\begin{align} \label{aa35}
			\|h(t)\|_{L^\infty_{x,v}} \le \frac{3}{16C_8}\bar{M}\exp{-e^{-\|\Phi\|_\infty} \frac{\nu_0}{16}t} + \frac{\delta_0}{4} \le \frac{1}{2C_8} \bar{M}
		\end{align}
		for all $0\le t \le T_1$.\\
		Hence we have shown the a priori assumption over $t \in [0,T_1]$ if $\mathcal{E}(F_0) \le \bar{\epsilon}_0:= \min\left\{\epsilon_0, \epsilon_1 \right\}$.\\
		\indent We claim that a solution to the Boltzmann equation \eqref{BE} extends into time interval $[0,T_1]$. From Theorem \ref{LocExist}, there exists the Boltzmann solution $F(t) \ge 0$ to \eqref{BE} on $[0,\hat{t}_0]$ such that
		\begin{align} \label{aa36}
			\sup_{0\le t \le  \hat{t}_0} \|h(t)\|_{L^\infty_{x,v}} \le 2C_8 \|h_0\|_{L^\infty_{x,v}} \le \frac{1}{2C_8} \bar{M}.
		\end{align}
		We define $t_*:=\left(\hat{C}_\rho\left[1+(2C_8)^{-1}\bar{M}\right]\right)^{-1}>0$, where $\hat{C}_\rho$ is a constant in Theorem \ref{LocExist}. Taking $t=\hat{t}_0$ as the initial time, it follows from \eqref{aa36} and Theorem \ref{LocExist} that we can extend the Boltzmann equation solution $F(t) \ge 0$ into time interval $[0, \hat{t}_0+t_*]$ satisfying
		\begin{align*}
			\sup_{\hat{t}_0 \le t \le \hat{t}_0+t_*} \|h(t)\|_{L^\infty_{x,v}} \le 2C_8 \|h(\hat{t}_0)\|_{L^\infty_{x,v}} \le \bar{M}.
		\end{align*}
		Thus we have
		\begin{align} \label{aa37}
			\sup_{0 \le t \le \hat{t}_0+t_*} \|h(t)\|_{L^\infty_{x,v}} \le \bar{M}.
		\end{align}
		Note that \eqref{aa37} means $h(t)$ satisfies the a priori assumption \eqref{Ape} over $[0,\hat{t}_0+t_*]$.
		From \eqref{aa35}, we can obtain
		\begin{align*}
			\sup_{0 \le t \le \hat{t}_0+t_*} \|h(t)\|_{L^\infty_{x,v}} \le\frac{1}{2C_8}\bar{M}.
		\end{align*}
		Repeating the same process for finite times, we can derive that there exists the Boltzmann equation solution $F(t)\ge 0$ on the time interval $[0,T_1]$ such that
		\begin{align*}
			\sup_{0\le t \le T_1} \|h(t)\|_{L^\infty_{x,v}} \le \frac{1}{2C_8} \bar{M}.
		\end{align*}
		Let us consider the case $[T_1,\infty)$. From \eqref{aa35}, we get
		\begin{align*}
			\|h(T_1)\|_{L^\infty_{x,v}} \le \frac{3}{16C_8}\bar{M}\exp{-e^{-\|\Phi\|_\infty} \frac{\nu_0}{16}T_1} + \frac{\delta_0}{4} \le \frac{3 \delta_0}{16C_8}+\frac{\delta_0}{4} < \frac{\delta_0}{2}.
		\end{align*}
		Taking $t=T_1$ as the initial time and using Theorem \ref{LT}, we conclude that there exists the Boltzmann equation solution $F(t)\ge 0$ on $[0,\infty)$. Therefore, we have proven the global existence and uniqueness of the Boltzmann equation \eqref{BE}.\\
		\indent It remains to show the exponential decay of the Boltzmann solution $f(t)$ in $L_{x,v}^\infty(w)$ space. By Theorem \ref{mainresult1}, for all $t \ge T_1$,
		\begin{align} \label{aa38}
			\|h(t)\|_{L^\infty_{x,v}} \le C_0 \|h(T_1)\|_{L^\infty_{x,v}} e^{-\lambda_0(t-T_1)} \le C_0 \delta_0 e^{-\lambda_0(t-T_1)}.
		\end{align}
		 Taking $\tilde{C}_L := 8C_8^3C_4^4$ and $\lambda_L:= \min \left\{\lambda_0,e^{-\|\Phi\|_\infty}\frac{\nu_0}{16}\right\}$, it follows from \eqref{aa35} and \eqref{aa38} that 
		 \begin{align*}
		 	\|h(t)\|_{L^\infty_{x,v}} &\le \max\left\{\frac{1}{2},C_0\right\}\bar{M}e^{-\lambda_L t} \le C_8 \bar{M} e^{-\lambda_L t}\\
		 	& \le 8 C_8^3 C_4^4 M_0^5 \exp{\frac{8}{\nu_0e^{-\|\Phi\|_\infty}}C_8 C_4^4 M_0^5}e^{-\lambda_L t}\\
		 	& \le \tilde{C}_L M_0^5\exp{\frac{\tilde{C}_L M_0^5}{\nu_0e^{-\|\Phi\|_\infty}} }e^{-\lambda_L t}
		 \end{align*}
		 for all $t \ge 0$.
	\end{proof}
	
\bigskip

\section{Appendix} \label{Appendix}
In this section, we present the proof of Lemma \ref{coer}. By choosing suitable test functions in a weak formulation and using the elliptic estimate, we will demonstrate this lemma.
\begin{proof}[\textbf{Proof of Lemma \ref{coer}}]
	We will choose suitable test functions $\psi=\psi(t,x,v) \in H_{x,v}^1$. We can deduce from the equation \eqref{SLBE} and the Green's identity that
	\begin{equation} \label{dwwf}
	\begin{aligned}
		&\int_{\Omega \cross \mathbb{R}^3} \psi(t)f(t)dxdv - \int_{\Omega \cross \mathbb{R}^3} \psi(0)f(0)dxdv\\
		& = \int_0^t\frac{d}{ds}\left(\int_{\Omega \cross \mathbb{R}^3}\psi(s)f(s)dxdv\right)ds\\
		& = \int_0^t \int_{\Omega \cross \mathbb{R}^3} f \left(v\cdot \nabla_x \psi \right)dxdvds -\int_0^t\int_{\gamma} \psi f \{n(x) \cdot v\}dS(x)dvds-\int_0^t \int_{\Omega \cross \mathbb{R}^3}e^{-\Phi(x)}L(f) \psi dxdvds\\
		& \quad +\int_0^t \int_{\Omega \cross \mathbb{R}^3}g\psi dxdvds +\int_0^t \int_{\Omega \cross \mathbb{R}^3} f (\partial_t \psi)dxdvds - \int_0^t\int_{\Omega \cross \mathbb{R}^3} f\left(\nabla_x \Phi(x) \cdot \nabla_v \psi\right)dxdvds.
	\end{aligned}
	\end{equation}
	We decompose $f=P_{L}(f)+\left(I-P_{L}\right)(f)$.\\
	From the fact $L\left(P_L(f)\right) = 0$, we obtain the weak form of the equation
	\begin{equation} \label{wfwf}
	\begin{aligned}
		&-\int_0^t \int_{\Omega \cross \mathbb{R}^3}\left(v\cdot \nabla_x \psi \right) P_{L}(f) dxdvds\\
		& =\int_{\Omega \cross \mathbb{R}^3} \psi(0) f(0) dxdv - \int_{\Omega \cross \mathbb{R}^3} \psi(t)f(t) dxdv+\int_0^t \int_{\Omega \cross \mathbb{R}^3} \left(v\cdot \nabla_x \psi \right) (I-P_{L})(f) dxdvds\\
		& \quad -\int_0^t \int_{\Omega \cross \mathbb{R}^3} e^{-\Phi(x)}L\left(\left(I-P_{L}\right)(f)\right) \psi dxdvds -\int_0^t \int_{\gamma} \psi f \{n(x) \cdot v\}	dS(x)dvds\\
		& \quad +\int_0^t \int_{\Omega \cross \mathbb{R}^3} f\left(\partial_t\psi\right)dxdvds +\int_0^t \int_{\Omega \cross \mathbb{R}^3} \psi g dxdvds\\
		& \quad -\int_0^t \int_{\Omega \cross \mathbb{R}^3} (\nabla_x \Phi(x) \cdot \nabla_v \psi) (I-P_{L})(f) dxdvds-\int_0^t \int_{\Omega \cross \mathbb{R}^3} (\nabla_x \Phi(x) \cdot \nabla_v \psi) P_{L}(f) dxdvds.
	\end{aligned}
	\end{equation}
	In \eqref{LPro}, we denote $P_L(f)(t,x,v) = a(t,x) \mu^{\frac{1}{2}}(v) + b(t,x)\cdot v\mu^{\frac{1}{2}}(v) + c(t,x)\frac{|v|^2-3}{\sqrt{6}}\mu^{\frac{1}{2}}(v)$. From now on, we will derive the estimates for $a(t,x),b(t,x)$, and $c(t,x)$.\\
	\newline
	$\mathbf{Estimate \;of \; c.}$\\
	We can choose the test function
	\begin{align*}
		\psi = \psi_c(t,x,v) = \left(|v|^2-\beta_c\right)\mu^{\frac{1}{2}}(v)v\cdot \nabla_x \phi_c(t,x),
	\end{align*}
	where the function $\phi_c$ satisfies 
	\begin{align*}
		\begin{cases}
		-\Delta_x \phi_c(t,x) = c(t,x)\\
		\phi_c |_{\partial \Omega}=0
	\end{cases}
	\end{align*}
	and $\beta_c>0$ is chosen such that
	\begin{align*}
		\int_{\mathbb{R}^3} \left(|v|^2-\beta_c\right)\mu(v)v_i^2dv=0 \quad \text{for} \quad i=1,2,3.
	\end{align*}
	We can get $\beta_c = 5$ from the simple computation.\\
	From the standard elliptic estimate, we get $\|\phi_c(t)\|_{H_x^2} \lesssim \|c(t)\|_{L_x^2}$ for all $t\ge 0$.\\
	First, we will deduce the estimate for the first term of (RHS) to the weak form \eqref{wfwf}.\\
	Let $G_f^{c}(s) = -\int_{\Omega \cross \mathbb{R}^3} \psi(s)f(s)dxdv$. Using the H\"{o}lder inequality and the elliptic estimate,
	\begin{align*}
		\left|G_f^c(s)\right| \lesssim  \|f(s)\|_{L^2_{x,v}}\|\phi_c(s)\|_{H_x^2} \lesssim  \|f(s)\|_{L^2_{x,v}}\|c(s)\|_{L_x^2}\lesssim  \|f(s)\|_{L^2_{x,v}}^2.
	\end{align*}
	Next, we will deduce the estimate for the second term of (RHS) to the weak form \eqref{wfwf}.\\
	Using the H\"{o}lder inequality, we obtain
	\begin{align*}
		&\left|\int_{\Omega \cross \mathbb{R}^3} \left(v\cdot \nabla_x \psi \right) (I-P_{L})(f)(t) dxdv\right|\\
		& \le \sum_{i,j=1}^3\left(\int_{\Omega \cross \mathbb{R}^3}\left|\left(|v|^2-\beta_c\right)\mu^{\frac{1}{2}}(v)v_iv_j\partial_{ij}\phi_c(t,x)\right|^2dxdv\right)^{\frac{1}{2}}\left\|(I-P_{L})(f)(t)\right\|_{L^2_{x,v}}\\
		& \lesssim\|\phi_c(t)\|_{H_x^2}\left\|(I-P_{L})(f)(t)\right\|_{L^2_{x,v}}\\
		& \lesssim \|c(t)\|_{L^2_x}\left\|(I-P_{L})(f)(t)\right\|_{L^2_{x,v}},
	\end{align*}
	where $\int_{\mathbb{R}^3}\left|\left(|v|^2-\beta_c\right)\mu^{\frac{1}{2}}(v)v_iv_j\right|^2dv$ is finite.\\
	Thus, we deduce that
	\begin{align} \label{c1}
		\left|\int_0^t \int_{\Omega \cross \mathbb{R}^3} \left(v\cdot \nabla_x \psi \right) (I-P_{L})(f) dxdvds\right| \lesssim \int_0^t \|c(s)\|_{L^2_x}\left\|(I-P_{L})(f)(s)\right\|_{L^2_{x,v}}ds.
	\end{align}
	Third, we will deduce the estimate for the third term of (RHS) to the weak form \eqref{wfwf}.\\
	From $L =\nu(v)-K$, the third term of (RHS) is bounded by
	\begin{align*}
		&\left|\int_{\Omega \cross \mathbb{R}^3}\psi\nu(v)(I-P_{L})(f)(t)dxdv\right| + \left|\int_{\Omega \cross \mathbb{R}^3}\psi K\left((I-P_{L})(f)(t)\right)dxdv\right|\\
		&=: I_1+I_2.
	\end{align*}
	We can easily get
	\begin{align*}
		I_1 \lesssim \|c(t)\|_{L_x^2}\left\|(I-P_{L})(f)(t)\right\|_{L^2_{x,v}}.
	\end{align*}
	Since $K$ is bounded in $L^2_{x,v}$, we use H\"older inequality to get
	\begin{align*}
		I_2 \lesssim \|\phi_c(t)\|_{H_x^2}\left\|(I-P_{L})(f)(t)\right\|_{L^2_{x,v}} \lesssim \|c(t)\|_{L^2_x}\left\|(I-P_{L})(f)(t)\right\|_{L^2_{x,v}}.
	\end{align*}
	Gathering $I_1$, $I_2$ and integrating from $0$ to $t$, we obtain
	\begin{align} \label{c2}
	\left|\int_0^t\int_{\Omega \cross \mathbb{R}^3} \psi e^{-\Phi(x)} L \left[(I-P_{L})(f)\right]dxdvds\right| \lesssim \int_0^t \|c(s)\|_{L^2_x}\left\|(I-P_{L})(f)(s)\right\|_{L^2_{x,v}}ds.
	\end{align}
	Fourth, we will deduce the estimate for the fourth term of (RHS) to the weak form \eqref{wfwf}. We can decompose the fourth term of (RHS) into two terms:
	\begin{align*}
		\int_{\partial \Omega \cross \mathbb{R}^3}\psi f\{n(x) \cdot v\}dS(x)dv = \int_{\gamma_+}\psi f\{n(x) \cdot v\}dS(x)dv +\int_{\gamma_-}\psi f\{n(x) \cdot v\}dS(x)dv.
	\end{align*}
	Decomposing $f = P_\gamma f + (I-P_\gamma)f$, we get
	\begin{align*}
		\int_{\partial \Omega \cross \mathbb{R}^3} \psi f \{n(x) \cdot v\}dS(x)dv =\int_{\gamma_+} \psi \left[(I-P_\gamma)f\right]\{n(x) \cdot v\}dS(x)dv + \int_{\gamma} \psi \left(P_\gamma f\right)\{n(x) \cdot v\}dS(x)dv.
	\end{align*}
	Setting $z(t,x) =c_\mu \int_{n(x) \cdot v' >0}f(x,v')\mu^{\frac{1}{2}}(v')\{n(x)\cdot v'\}dv'$, we obtain 
	\begin{align*}
		\int_\gamma \psi (P_\gamma f)\{n(x) \cdot v\}dS(x)dv
		&=\sum_{i=1}^3 \left(\int_{\mathbb{R}^3} \left(|v|^2-\beta_c\right)\mu(v)v_i^2 dv\right)\left(\int_{\partial \Omega}\partial_{x_i} \phi_c(t,x)z(t,x)n_i(x)dS(x)\right)=0,
	\end{align*}
	where we have used the oddness for integral and the definition of $\beta_c$.\\
	Then we can simplify
	\begin{align*}
		\int_{\partial \Omega \cross \mathbb{R}^3} \psi f \{n(x) \cdot v\}dS(x)dv
		&\le \sum_{i,k=1}^3\left(\int_{\gamma_+} \left|\left(|v|^2-\beta_c\right)\mu^{\frac{1}{2}}(v)v_iv_k n_k(x)\partial_{x_i} \phi_c(t,x)\right|^2 dS(x)dv\right)^{\frac{1}{2}}\|(I-P_\gamma)f(t)\|_{L^2_{\gamma_+}}\\
		&\lesssim \left\|\partial_{x_i} \phi_c(t)\right\|_{L^2(\partial\Omega)}\|(I-P_\gamma)f(t)\|_{L^2_{\gamma_+}},
	\end{align*}
	where $\int_{\mathbb{R}^3}\left|\left(|v|^2-\beta_c\right)\mu^{\frac{1}{2}}(v)v_iv_k\right|^2dv$ is finite.\\
	By the trace theorem, the above is bounded by
	\begin{align*}
		\left\|\phi_c(t)\right\|_{H^2_x}\|(I-P_\gamma)f(t)\|_{L^2_{\gamma_+}} \lesssim \|c(t)\|_{L^2_x}\|(I-P_\gamma)f(t)\|_{L^2_{\gamma_+}}.
	\end{align*}
	Thus, we deduce
	\begin{align} \label{c3}
		\left|\int_0^t\int_{\gamma}\psi f\{n(x)\cdot v\}dS(x)dvds\right|\lesssim \int_0^t \|c(s)\|_{L^2_x}\|(I-P_\gamma)f(s)\|_{L^2_{\gamma_+}}ds.
	\end{align}
	Fifth, we will deduce the estimate for the sixth term of (RHS) to the weak form \eqref{wfwf}.
	We apply the Cauchy-Schwarz inequality and the elliptic estimate to obtain
	\begin{equation} \label{c4}
	\begin{aligned} 
		\left|\int_{\Omega \cross \mathbb{R}^3}\psi g dxdv\right| \lesssim \|c(t)\|_{L^2_x}\left\|g(t)\right\|_{L^2_{x,v}}.
	\end{aligned}
	\end{equation}
	Sixth, we will deduce the estimate for (LHS) to the weak form \eqref{wfwf}.
	From the construction for $\beta_c$ and the oddness of integration in $v$, we deduce that
	\begin{align*}
		\int_0^t \int_{\Omega \cross \mathbb{R}^3}\left(v\cdot \nabla_x \psi \right) P_{L}(f) dxdvds= \sum_{i=1}^3 \int_0^t \left(\int_{\mathbb{R}^3} \left(|v|^2-\beta_c\right)\mu(v)v_i^2\frac{|v|^2-3}{\sqrt{6}}dv\right)\left(\int_{\Omega}\left[\partial_{ii}\phi_c(s,x)\right] c(s,x)dx\right)ds.
	\end{align*}
	Here, for all $i=1,2,3$,
	\begin{align*}
		\int_{\mathbb{R}^3} \left(|v|^2-\beta_c\right)\mu(v)v_i^2\frac{|v|^2-3}{\sqrt{6}}dv =\frac{1}{3\sqrt{6}}(7\cdot5\cdot3-8\cdot5\cdot3+15\cdot3)(2\pi)^{\frac{3}{2}}=:A>0.
	\end{align*}
	This yields
	\begin{align} \label{c5}
		-\int_0^t \int_{\Omega \cross \mathbb{R}^3}\left(v\cdot \nabla_x \psi \right) P_{L}(f) dxdvds &=-A\sum_{i=1}^3 \int_0^t \left(\int_{\Omega}\left[\partial_{ii}\phi_c(s,x)\right] c(s,x)dx\right)ds =A\int_0^t\|c(s)\|_{L_x^2}^2ds. 
	\end{align}
	Seventh, we will deduce the estimate for the seventh term of (RHS) to the weak form \eqref{wfwf}. Using the H\"older inequality, we obtain
	\begin{align*}
		 &\left|\int_{\Omega \cross \mathbb{R}^3} (\nabla_x \Phi(x) \cdot \nabla_v \psi) (I-P_{L})(f)(t) dxdv\right|\\
		 &\le \sum_{i,j=1}^3\left(\int_{\Omega \cross \mathbb{R}^3}\left|\partial_{x_i} \Phi(x) \partial_{x_j}\phi_c(t,x)\partial_{v_i}\left( (|v|^2-\beta_c)\mu^{\frac{1}{2}}(v)v_j\right)\right|^2dxdv\right)^{\frac{1}{2}}\left\|(I-P_{L})(f)(t)\right\|_{L^2_{x,v}}\\
		 &\le C_{\Phi} \|\phi_c(t)\|_{H_x^2} \left\|(I-P_{L})(f)(t)\right\|_{L^2_{x,v}}\\
		 &\le C_{\Phi} \|c(t)\|_{L^2_x} \left\|(I-P_{L})(f)(t)\right\|_{L^2_{x,v}},
	\end{align*}
	where $\int_{\mathbb{R}^3}\left| \partial_{v_i}\left( (|v|^2-\beta_c)\mu^{\frac{1}{2}}(v)v_j\right)\right|^2dv$ is finite.\\
	Thus, we deduce that
	\begin{align} \label{c6}
		\left|\int_0^t \int_{\Omega \cross \mathbb{R}^3} \left(\nabla_x \Phi(x) \cdot \nabla_v \psi\right) (I-P_{L})(f)(s) dxdvds\right| \le C_{\Phi}\int_0^t \|c(s)\|_{L^2_x}\left\|(I-P_{L})(f)(s)\right\|_{L^2_{x,v}}ds.
	\end{align}
	Eighth, we will deduce the estimate for the eighth term of (RHS) to the weak form \eqref{wfwf}. Using the H\"older inequality, we obtain
	\begin{align*}
		&\left|\int_{\Omega \cross \mathbb{R}^3} (\nabla_x \Phi(x) \cdot \nabla_v \psi) P_{L}(f)(t) dxdv\right|\\ 
		&\le C_{\Phi}\sum_{i,j=1}^3\left(\int_{\Omega}\left| \partial_{x_j}\phi_c(t,x) \right|^2 \int_{\mathbb{R}^3}\left| \partial_{v_i}\left( (|v|^2-\beta_c)\mu^{\frac{1}{2}}(v)v_j\right)\right|^2dvdx\right)^{\frac{1}{2}}\left\|P_{L}(f)(t)\right\|_{L^2_{x,v}}\\
		&\le C_{\Phi}\|wf(t)\|_{L^\infty_{x,v}} \left\|P_{L}(f)(t)\right\|_{L^2_{x,v}}^2.
	\end{align*}
	Thus, we deduce that
	\begin{align} \label{c7}
		\left|\int_0^t \int_{\Omega \cross \mathbb{R}^3} \left(\nabla_x \Phi(x) \cdot \nabla_v \psi\right) P_{L}(f)(s) dxdvds\right| \le C_{\Phi}\int_0^t \|wf(s)\|_{L^\infty_{x,v}} \left\|P_{L}(f)(s)\right\|_{L^2_{x,v}}^2ds.
	\end{align}
	Lastly, we will deduce the estimate for the fifth term of (RHS) to the weak form \eqref{wfwf}.\\
	We decompose $f=P_{L}(f)+\left(I-P_{L}\right)(f)$ to get
	\begin{align*}
		&\int_0^t \int_{\Omega \cross \mathbb{R}^3} f\left(\partial_t\psi\right)dxdvds\\ 
		&=\sum_{i=1}^3\int_0^t\int_{\Omega \cross \mathbb{R}^3}\left(|v|^2-\beta_c\right)\mu(v)v_i\left[\partial_t\partial_{x_i}\phi_c\right]\left[a(s,x)+b(s,x)\cdot v+c(s,x)\frac{|v|^2-3}{\sqrt{6}}\right]dxdvds\\
		&\quad +\sum_{i=1}^3\int_0^t\int_{\Omega \cross \mathbb{R}^3} \left(|v|^2-\beta_c\right)\mu^{\frac{1}{2}}(v)v_i\left[\partial_t\partial_{x_i}\phi_c\right](I-P_{L})(f)(s)dxdvds.
	\end{align*}
	From the construction for $\beta_c$ and the oddness of integration in $v$, the above expression becomes
	\begin{align*}
		\sum_{i=1}^3\int_0^t\int_{\Omega \cross \mathbb{R}^3} \left(|v|^2-\beta_c\right)\mu^{\frac{1}{2}}(v)v_i\left[\partial_t\partial_{x_i}\phi_c\right](I-P_{L})(f)(s)dxdvds.
	\end{align*}
	Lately, we will demonstrate the estimate of $\nabla_x \partial_t \phi_c$ as following :
	\begin{align*}
		\left\|\nabla_x \partial_t \phi_c(t)\right\|_{L^2_x} \lesssim \|b(t)\|_{L_x^2}+\|(I-P_{L})(f)(t)\|_{L_{x,v}^2}+\|g(t)\|_{L^2_{x,v}}.
	\end{align*}
	By the estimate of $\nabla_x \partial_t \phi_c$, we get
	\begin{equation} \label{c8}
	\begin{aligned}
		&\int_0^t\int_{\Omega \cross \mathbb{R}^3} \left(|v|^2-\beta_c\right)\mu^{\frac{1}{2}}(v)v_i\left[\partial_t\partial_{x_i}\phi_c\right](I-P_{L})(f)(s)dxdvds \\
		& \lesssim \int_0^t \left\|\nabla_x \partial_t \phi_c(s)\right\|_{L^2_x}\|(I-P_{L})(f)(s)\|_{L_{x,v}^2}ds\\
		& \le \epsilon \int_0^t\|b(s)\|_{L_x^2}^2ds+C(\epsilon)\int_0^t\|(I-P_{L})(f)(s)\|_{L_{x,v}^2}^2ds + C'(\epsilon) \int_0^t\|g(s)\|_{L^2_{x,v}}^2ds, 
	\end{aligned}
	\end{equation}
	where we have used the Young's inequality.\\
	Gathering \eqref{c1}, \eqref{c2}, \eqref{c3}, \eqref{c4}, \eqref{c5}, \eqref{c6}, \eqref{c7} and \eqref{c8}, we obtain
	\begin{align*}
		A\int_0^t\|c(s)\|_{L_x^2}^2ds 
		&\le  G_f^c(t)-G_f^c(0) +\epsilon\int_0^t \|c(s)\|_{L^2_x}^2ds + \epsilon \int_0^t\|b(s)\|_{L_x^2}^2ds\\
		&\quad +C_1(\epsilon)\int_0^t\|(I-P_{L})(f)(s)\|_{L_{x,v}^2}^2ds+C_2(\epsilon)\int_0^t\left\|(I-P_{\gamma})f(s)\right\|_{L^2_{\gamma_+}}^2ds\\ 
		&\quad + C_3(\epsilon) \int_0^t\|g(s)\|_{L^2_{x,v}}^2ds +C_{\Phi}\int_0^t \|wf(s)\|_{L^\infty_{x,v}} \left\|P_{L}(f)(s)\right\|_{L^2_{x,v}}^2ds ,
	\end{align*}
	where we have used the Young's inequality.\\
	Thus, choosing sufficiently small $\epsilon >0$, we conclude that
	\begin{equation} \label{co1}
	\begin{aligned}
		\int_0^t\|c(s)\|_{L_x^2}^2ds &\le G_f^c(t)-G_f^c(0) +\epsilon \int_0^t\|b(s)\|_{L_x^2}^2ds \\
		&\quad +C_1(\epsilon)\int_0^t\|(I-P_{L})(f)(s)\|_{L_{x,v}^2}^2ds +C_2(\epsilon)\int_0^t\left\|(I-P_{\gamma})f(s)\right\|_{L^2_{\gamma_+}}^2ds  \\ 
		&\quad +C_3(\epsilon) \int_0^t\|g(s)\|_{L^2_{x,v}}^2ds+C_{\Phi}\int_0^t \|wf(s)\|_{L^\infty_{x,v}} \left\|P_{L}(f)(s)\right\|_{L^2_{x,v}}^2ds . 
	\end{aligned}
	\end{equation}
	\newline
	$\mathbf{Estimate \;of \; \nabla_x\partial_t\phi_c.}$\\
	We consider the weak formulation over $[t,t+\epsilon]$:
	\begin{equation} \label{dwf}
	\begin{aligned} 
		&\int_{\Omega \cross \mathbb{R}^3} \psi(x,v)f(t+\epsilon)dxdv - \int_{\Omega \cross \mathbb{R}^3} \psi(x,v)f(t)dxdv\\
		& = \int_t^{t+\epsilon} \int_{\Omega \cross \mathbb{R}^3} f\left(v\cdot \nabla_x \psi \right)dxdvds -\int_t^{t+\epsilon}\int_{\gamma} \psi f \{n(x) \cdot v\}dS(x)dvds\\
		&\quad -\int_t^{t+\epsilon} \int_{\Omega \cross \mathbb{R}^3}e^{-\Phi(x)}L(f) \psi dxdvds +\int_t^{t+\epsilon} \int_{\Omega \cross \mathbb{R}^3}g\psi dxdvds\\
		&\quad +\int_t^{t+\epsilon} \int_{\Omega \cross \mathbb{R}^3} f(\partial_t \psi)dxdvds -\int_t^{t+\epsilon} \int_{\Omega \cross \mathbb{R}^3} (\nabla_x \Phi(x) \cdot \nabla_v \psi) f dxdvds.
	\end{aligned}
	\end{equation}
	We choose the test function
	\begin{align*}
		\psi = \psi(x,v) = \phi(x)\frac{|v|^2-3}{\sqrt{6}}\mu^{\frac{1}{2}}(v),
	\end{align*}
	where $\phi(x)$ depends only on $x$.\\
	Note that
	\begin{align*}
		&\int_{\mathbb{R}^3} \mu(v) \frac{|v|^2-3}{\sqrt{6}}dv=0, \quad \int_{\mathbb{R}^3} \mu(v) v_iv_j\frac{|v|^2-3}{\sqrt{6}}dv = \frac{\sqrt{6}}{3}(2\pi)^{\frac{3}{2}}\delta_{ij}, \quad \int_{\mathbb{R}^3} \mu(v)\left(\frac{|v|^2-3}{\sqrt{6}}\right)^2dv = (2\pi)^{\frac{3}{2}},
	\end{align*}
	where $\delta_{ij}$ is 1 if $i=j$ and 0 otherwise.\\
	First of all, we will deduce the estimate for (LHS) to the weak formulation \eqref{dwf}.
	\begin{align*}
		\int_{\Omega \cross \mathbb{R}^3} \phi(x) \frac{|v|^2-3}{\sqrt{6}}\mu^{\frac{1}{2}}(v)f(t+\epsilon)dxdv = \int_{\Omega}\phi(x)c(t+\epsilon,x)dx, \quad \int_{\Omega \cross \mathbb{R}^3} \psi(x,v)f(t)dxdv = \int_{\Omega} \phi(x)c(t,x)dx.
	\end{align*}
	Next, we will deduce the estimate for the first term of (RHS) to the weak formulation \eqref{dwf}. We decompose $f=P_{L}(f)+\left(I-P_{L}\right)(f)$ to get
	\begin{align*}
		\int_{t}^{t+\epsilon} \int_{\Omega \cross \mathbb{R}^3}f(v \cdot \nabla_x\psi)dxdvds &= \sum_{i=1}^3	\int_t^{t+\epsilon}\left(\int_\Omega\frac{\sqrt{6}}{3}(2\pi)^{\frac{3}{2}} \partial_{x_i} \phi(x) b_i(s,x)dx\right)ds\\
		&\quad + \int_{t}^{t+\epsilon} \int_{\Omega \cross \mathbb{R}^3} v \cdot \nabla_x \phi(x)\frac{|v|^2-3}{\sqrt{6}}\mu^{\frac{1}{2}}(v)\left(I-P_{L}\right)(f)(s)dxdvds,
	\end{align*}
	where we have used the oddness of integration in $v$.\\
	Third, we will deduce the estimate for the third term of (RHS) to the weak formulation \eqref{dwf}. Since $L$	is self-adjoint, it holds that
	\begin{align*}
		\int_{\mathbb{R}^3} L(f)\frac{|v|^2-3}{\sqrt{6}} \mu^{\frac{1}{2}}(v)dv=0,
	\end{align*}
	and we obtain
	\begin{align*}
		\int_t^{t+\epsilon}\int_{\Omega \cross \mathbb{R}^3} e^{-\Phi(x)}L(f)\psi dxdvds = 0.
	\end{align*}
	Fourth, we easily get
	\begin{align*}
		\int_t^{t+\epsilon}\int_{\Omega \cross \mathbb{R}^3}(\partial_t \psi)fdxdvds=0
	\end{align*}
	since $\psi$ is independent of $t$.\\
	Fifth, we will deduce the estimate for second term of (RHS) to the weak formulation \eqref{dwf}.
	For fixed $t>0$, define $\phi(x) = \Phi_c(x)$ with
	\begin{align*}
		\begin{cases}
			-\Delta_x \Phi_c(x) = \partial_t c(t,x)\\
			\Phi_c|_{\partial \Omega} = 0.
		\end{cases}
	\end{align*}
	Then we have for fixed t,
	\begin{align*}
		\Phi_c (x) = -\Delta_x ^{-1}\partial_t c(t,x) = \partial_t \phi_c(t,x).
	\end{align*}
	From the fact $\Phi_c|_{\partial \Omega} = 0$, the second term of (RHS) to the weak formulation becomes
	\begin{align*}
		\int_t^{t+\epsilon}\int_{\gamma} \psi f \{n(x) \cdot v\}dS(x)dvds=0.
	\end{align*}
	Sixth, we will deduce the estimate for sixth term of (RHS) to the weak formulation \eqref{dwf}. We decompose $f = P_L(f) + (I-P_L)(f)$ to get
	\begin{align*}
		\int_t^{t+\epsilon} \int_{\Omega \cross \mathbb{R}^3} (\nabla_x \Phi(x) \cdot \nabla_v \psi) f dxdvds
		& = \sum_{i=1}^3 \int_t^{t+\epsilon} \left(\int_{\Omega} \frac{\sqrt{6}}{3}(2\pi)^{\frac{3}{2}}\phi(x)\partial_{x_i}\Phi(x)b_i(s,x)dx\right)ds\\
		& \quad + \int_t^{t+\epsilon} \int_{\Omega \cross \mathbb{R}^3} (\nabla_x \Phi(x) \cdot \nabla_v \psi) (I-P_L)(f)(s) dxdvds,
	\end{align*}
	where we have used the oddness of integration in $v$.\\
 	Combing the above process and taking the difference quotient in \eqref{dwf}, for all $t \ge 0$,
	\begin{align*}
		\int_{\Omega} \phi(x) \partial_t c(t,x) dx &= \frac{\sqrt{6}}{3}(2\pi)^{\frac{3}{2}}\int_{\Omega}b(t,x)\cdot \nabla_x \phi(x) dx+ \int_{\Omega \cross \mathbb{R}^3} v \cdot \nabla_x \phi(x) \frac{|v|^2-3}{\sqrt{6}}\mu^{\frac{1}{2}}(v) \left(I-P_{L}\right)(f)(t)dxdv\\
		&\quad + \int_{\Omega \cross \mathbb{R}^3}g(t,x,v) \phi(x)\frac{|v|^2-3}{\sqrt{6}}\mu^{\frac{1}{2}}(v)dxdv + \frac{\sqrt{6}}{3}(2\pi)^{\frac{3}{2}}\int_{\Omega}\phi(x) b(t,x)\cdot \nabla_x\Phi(x)dx\\
		&\quad + \int_{\Omega \cross \mathbb{R}^3} (\nabla_x \Phi(x) \cdot \nabla_v \psi) (I-P_L)(f)(t) dxdv.
	\end{align*}
	From the above equality, for all $t \ge 0$,
	\begin{align*}
		\|\nabla_x \partial_t \phi_c(t)\|_{L^2_x}^2 &= \int_\Omega \left|\nabla_x \Phi_c(x)\right|^2 dx = -\int_\Omega \Phi_c(x) \left(\Delta_x \Phi_c(x)\right)dx= \int_\Omega \phi(x) \partial_t c(t,x) dx\\
		& \lesssim \epsilon \left\{\|\nabla_x \Phi_c\|_{L^2_x}^2+\|\Phi_c\|_{L^2_x}^2\right\} + \|b(t)\|_{L^2_x}^2+\| \left(I-P_{L}\right)(f)(t)\|_{L^2_{x,v}}^2+\|g(t)\|_{L^2_{x,v}}^2,
	\end{align*}
	where we have used the integration by parts and the Young's inequality.\\
	We use the Poincar\'e inequality to obtain
	\begin{align*}
		\|\nabla_x \partial_t \phi_c(t)\|_{L^2_x}^2 \lesssim \epsilon \|\nabla_x \Phi_c\|_{L^2_x}^2 + \|b(t)\|_{L^2_x}^2+\| \left(I-P_{L}\right)(f)(t)\|_{L^2_{x,v}}^2+\|g(t)\|_{L^2_{x,v}}^2.
	\end{align*}
	For sufficiently small $\epsilon >0$, we have for all $t \ge 0$,
	\begin{align*}
		\|\nabla_x \partial_t \phi_c(t)\|_{L^2_x} \lesssim \|b(t)\|_{L^2_x}+\| \left(I-P_{L}\right)(f)(t)\|_{L^2_{x,v}}+\|g(t)\|_{L^2_{x,v}}.
	\end{align*}
	\newline
	$\mathbf{Estimate \;of \; b.}$\\
	Firstly, we will estimate $(\partial_{x_i} \partial_{x_j} \Delta_x^{-1}b_j)b_i$ for $i,j =1,2,3$. Fix $i,j$.\\
	We choose the test function
	\begin{align*}
		\psi = \psi_b^{i,j}(t,x,v) = (v_i^2-\beta_b)\mu^{\frac{1}{2}}(v) \partial_{x_j}\phi_b^j(t,x),
	\end{align*}
	where
	\begin{align*}
		\begin{cases}
			-\Delta_x\phi_b^j(t,x) = b_j(t,x)\\
			\phi_b^j|_{\partial \Omega}=0
		\end{cases}
	\end{align*}
	and $\beta_b>0$ is chosen such that for all $i=1,2,3$,
	\begin{align*}
		\int_{\mathbb{R}^3}\left[(v_i)^2-\beta_b\right]\mu(v)dv = \frac{1}{3}\int_{\mathbb{R}^3}\left(|v|^2-3\beta_b\right)\mu(v)dv=0.
	\end{align*}
	Then we note that for all $i \not= k$,
	\begin{align*}
		&\int_{\mathbb{R}^3}(v_i^2-\beta_b)v_k^2\mu(v)dv = \int_{\mathbb{R}^3}(v_1^2-1)v_2^2\mu(v)dv=0,\\
		&\int_{\mathbb{R}^3}(v_i^2-\beta_b)v_i^2\mu(v)dv = 2\pi\int_{\mathbb{R}}(v_1^4-v_1^2)e^{-\frac{v_1^2}{2}}dv_1=2(2\pi)^{\frac{3}{2}}.
	\end{align*}
	From the standard elliptic estimate, we get $\|\phi_b^j(t)\|_{H_x^2} \lesssim \|b_j(t)\|_{L_x^2}$ for all $t\ge 0$.\\
	First of all, we will deduce the estimate for the first term of (RHS) to the weak form \eqref{wfwf}.\\
	Let $G_f^{b_{ij}}(s) = -\int_{\Omega \cross \mathbb{R}^3} \psi(s)f(s)dxdv$. Using the H\"{o}lder inequality and the elliptic estimate, we obtain
	\begin{align*}
		\left|G_f^{b_{ij}}(s)\right| \lesssim \|f(s)\|_{L^2_{x,v}}^2.
	\end{align*}
	Next, we will deduce the estimate for the second term of (RHS) to the weak form \eqref{wfwf}.\\
	Using the H\"{o}lder inequality and the elliptic estimate, we obtain
	\begin{align} \label{b1}
		\left|\int_0^t \int_{\Omega \cross \mathbb{R}^3} \left(v\cdot \nabla_x \psi \right) (I-P_{L})(f) dxdvds\right| \lesssim \int_0^t \|b(s)\|_{L^2_x}\left\|(I-P_{L})(f)(s)\right\|_{L^2_{x,v}}ds.
	\end{align}
	Third, we will deduce the estimate for the third term of (RHS) to the weak form \eqref{wfwf}.\\
	From $L =\nu(v)-K$, the third term of (RHS) is bounded by
	\begin{align*}
		&\left|\int_{\Omega \cross \mathbb{R}^3}\psi\nu(v)(I-P_{L})(f)(t)dxdv\right| + \left|\int_{\Omega \cross \mathbb{R}^3}\psi K\left((I-P_{L})(f)(t)\right)dxdv\right| \lesssim \|b(t)\|_{L_x^2}\left\|(I-P_{L})(f)(t)\right\|_{L^2_{x,v}}.
	\end{align*}
	Integrating from $0$ to $t$, we obtain
	\begin{align} \label{b2}
		\left|\int_0^t\int_{\Omega \cross \mathbb{R}^3} \psi e^{-\Phi(x)}L \left[(I-P_{L})(f)\right]dxdvds\right| \lesssim \int_0^t \|b(s)\|_{L^2_x}\left\|(I-P_{L})(f)(s)\right\|_{L^2_{x,v}}ds.
	\end{align}
	Fourth, we will deduce the estimate for the fourth term of (RHS) to the weak form \eqref{wfwf}. We can decompose the fourth term of (RHS) into two terms:
	\begin{align*}
		\int_{\partial \Omega \cross \mathbb{R}^3}\psi f\{n(x) \cdot v\}dS(x)dv =\int_{\gamma_+} \psi \left[(I-P_\gamma)f\right]\{n(x) \cdot v\}dS(x)dv + \int_{\gamma} \psi \left(P_\gamma f\right)\{n(x) \cdot v\}dS(x)dv.
	\end{align*}
	Setting $z(t,x) =c_\mu \int_{n(x) \cdot v' >0}f(x,v')\mu^{\frac{1}{2}}(v')\{n(x)\cdot v'\}dv'$, we obtain 
	\begin{align*}
		\int_\gamma \psi (P_\gamma f)\{n(x) \cdot v\}dS(x)dv =\sum_{k=1}^3 \left(\int_{\mathbb{R}^3} (v_i^2-\beta_b)\mu(v)v_k dv\right)\left(\int_{\partial \Omega}\partial_{x_j}\phi_b^j(t,x)z(t,x)n_k(x)dS(x)\right)=0,
	\end{align*}
	where we have used the oddness of integration in $v$.\\
	Thus we can simplify
	\begin{align*}
		&\int_{\partial \Omega \cross \mathbb{R}^3} \psi f \{n(x) \cdot v\}dS(x)dv\\
		&\le \sum_{i,k=1}^3\left(\int_{\partial \Omega } \left|n_k(x)\partial_{x_j} \phi_b^j(t,x)\right|^2\left(\int_{\mathbb{R}^3}\left|\left(v_i^2-\beta_b\right)\mu^{\frac{1}{2}}(v)v_k\right|^2dv\right) dS(x)\right)^{\frac{1}{2}}\|(I-P_\gamma)f(t)\|_{L^2_{\gamma_+}}\\
		&\lesssim \left\|\partial_{x_j} \phi_b^j(t)\right\|_{L^2(\partial\Omega)}\|(I-P_\gamma)f(t)\|_{L^2_{\gamma_+}},
	\end{align*}
	where $\int_{\mathbb{R}^3}\left|\left(v_i^2-\beta_b\right)\mu^{\frac{1}{2}}(v)v_k\right|^2dv$ is finite.\\
	By the trace theorem, the above is bounded by
	\begin{align*}
		\left\|\phi_b^j(t)\right\|_{H^2_x}\|(I-P_\gamma)f(t)\|_{L^2_{\gamma_+}} \lesssim \|b(t)\|_{L^2_x}\|(I-P_\gamma)f(t)\|_{L^2_{\gamma_+}}.
	\end{align*}
	Thus, we deduce
	\begin{align} \label{b3}
		\left|\int_0^t\int_{\gamma}\psi f\{n(x)\cdot v\}dS(x)dvds\right|\lesssim \int_0^t \|b(s)\|_{L^2_x}\|(I-P_\gamma)f(s)\|_{L^2_{\gamma_+}}ds.
	\end{align}
	Fifth, we will deduce the estimate for the sixth term of (RHS) to the weak form \eqref{wfwf}.
	We apply the Cauchy-Schwarz inequality to obtain
	\begin{equation} \label{b4}
	\begin{aligned}
		\left|\int_{\Omega \cross \mathbb{R}^3}\psi g dxdv\right| 		
		 \lesssim \|\phi_b^j(t)\|_{H_x^2}\left\|g(t)\right\|_{L^2_{x,v}} \lesssim \|b(t)\|_{L^2_x}\left\|g(t)\right\|_{L^2_{x,v}}.
	\end{aligned}
	\end{equation}
	Sixth, we will deduce the estimate for (LHS) to the weak form \eqref{wfwf}.
	From the construction for $\beta_b$ and the oddness of integration in $v$, we deduce that
	\begin{equation} \label{b5}
	\begin{aligned}
		\int_0^t \int_{\Omega \cross \mathbb{R}^3}\left(v\cdot \nabla_x \psi \right) P_{L}(f) dxdvds
		& = \sum_{k=1}^3 \int_0^t \int_{\Omega \cross \mathbb{R}^3}\left(v_i^2-\beta_b\right)\mu(v)v_k^2 \left[\partial_{kj}\phi_b^j(s,x)\right]b_k(s,x)dxdvds\\
		& = -2(2\pi)^{\frac{3}{2}}\int_0^t \left(\int_{\Omega}\left(\partial_{ij}\Delta_x^{-1} b_j\right)(s,x) b_i(s,x)dx\right)ds.
	\end{aligned}
	\end{equation}
	Seventh, we will deduce the estimate for the seventh term of (RHS) to the weak form \eqref{wfwf}.\\
	Using the H\"older inequality and the elliptic estimate, we obtain
	\begin{align} \label{b6}
		\left|\int_0^t \int_{\Omega \cross \mathbb{R}^3} \left(\nabla_x \Phi(x) \cdot \nabla_v \psi\right) (I-P_{L})(f)(s) dxdvds\right| \le C_{\Phi}\int_0^t \|b(s)\|_{L^2_x}\left\|(I-P_{L})(f)(s)\right\|_{L^2_{x,v}}ds.
	\end{align}
	Eighth, we will deduce the estimate for the eighth term of (RHS) to the weak form \eqref{wfwf}.\\
	By a similar way in \eqref{c7}, we obtain
	\begin{align} \label{b7}
		\left|\int_0^t \int_{\Omega \cross \mathbb{R}^3} \left(\nabla_x \Phi(x) \cdot \nabla_v \psi\right) P_{L}(f)(s) dxdvds\right| \le C_{\Phi}\int_0^t\|wf(s)\|_{L^\infty_{x,v}} \left\|P_{L}(f)(s)\right\|_{L^2_{x,v}}^2ds.
	\end{align}
	Lastly, we will deduce the estimate for the fifth term of (RHS) to the weak form \eqref{wfwf}.\\
	We decompose $f=P_{L}(f)+\left(I-P_{L}\right)(f)$ to get
	\begin{align*}
		\int_0^t \int_{\Omega \cross \mathbb{R}^3} f\left(\partial_t\psi\right)dxdvds &=\frac{2}{\sqrt{6}}(2\pi)^{\frac{3}{2}}\int_0^t\int_{\Omega \cross \mathbb{R}^3} \left[\partial_t\partial_{x_j}\phi_b^j(s,x)\right]c(s,x)dxdvds\\
		&\quad +\int_0^t\int_{\Omega \cross \mathbb{R}^3} (v_i^2-\beta_b)\mu^{\frac{1}{2}}(v)\left[\partial_t\partial_{x_j}\phi_b^j\right](I-P_{L})(f)(s)dxdvds,
	\end{align*}
	where we have used the construction for $\beta_b$ and the oddness of integration in $v$.\\
	Lately, we will demonstrate the estimate of $\nabla_x \partial_t \phi_b^j$ as following :
	\begin{align*}
		\left\|\nabla_x \partial_t \phi_b^i(t)\right\|_{L^2_x} \lesssim \|a(t)\|_{L_x^2}+\|c(t)\|_{L_x^2}+\|(I-P_{L})(f)(t)\|_{L_{x,v}^2}+\|g(t)\|_{L^2_{x,v}}.
	\end{align*}
	By the estimate of $\nabla_x \partial_t \phi_b^j$, we get
	\begin{equation} \label{b8}
	\begin{aligned}
		\int_0^t \int_{\Omega \cross \mathbb{R}^3} f\left(\partial_t\psi\right)dxdvds
		& \lesssim \int_0^t \left\|\nabla_x \partial_t \phi_b^j(s)\right\|_{L_x^2} \left(\|(I-P_L)(f)(s)\|_{L_{x,v}^2} + \|c(s)\|_{L_x^2}\right)ds\\
		& \le \epsilon \int_0^t \|a(s)\|^2_{L_x^2}ds + C(\epsilon)\int_0^t\|c(s)\|^2_{L_x^2}ds + C'(\epsilon)\int_0^t \|(I-P_{L})(f)(s)\|_{L_{x,v}^2}^2ds \\
		&\quad + C''(\epsilon)\int_0^t \|g(s)\|_{L_x^2}^2ds ,
	\end{aligned}
	\end{equation}
	where we have used the Young's inequality.\\
	Gathering \eqref{b1}, \eqref{b2}, \eqref{b3}, \eqref{b4}, \eqref{b5}, \eqref{b6}, \eqref{b7}, and \eqref{b8}, we obtain
	\begin{align*}
		&\int_0^t\int_{\Omega}\left(\partial_{x_i}\partial_{x_j} \Delta_x^{-1} b_j\right)(s,x)b_i(s,x)dxds\\
		&\le G_f^{b_{ij}}(t)-G_f^{b_{ij}}(0) +\epsilon\int_0^t \|b(s)\|_{L^2_x}^2ds + \epsilon \int_0^t\|a(s)\|_{L_x^2}^2ds\\
		&\quad +C_1(\epsilon)\int_0^t\|(I-P_{L})(f)(s)\|_{L_{x,v}^2}^2ds+C_2(\epsilon)\int_0^t\left\|(I-P_{\gamma})f(s)\right\|_{L^2_{\gamma_+}}^2ds + C_3(\epsilon) \int_0^t\|g(s)\|_{L^2_{x,v}}^2ds\\
		&\quad +C_4(\epsilon)\int_0^t \|c(s)\|_{L^2_{x}}^2ds +C_{\Phi}\int_0^t \|wf(s)\|_{L^\infty_{x,v}} \left\|P_{L}(f)(s)\right\|_{L^2_{x,v}}^2ds,
	\end{align*}
	where we have used the Young's inequality.\\
	\newline
	Now, we will estimate $(\partial_{x_j}\partial_{x_j}\Delta_x^{-1}b_i)b_i$ for $i \not= j$.\\
	We choose the test function
	\begin{align*}
		\psi = \psi_b^{i,j}(t,x,v) = |v|^2v_iv_j \mu^{\frac{1}{2}}(v) \partial_{x_j}\phi_b^i(t,x),
	\end{align*}
	where
	\begin{align*}
		\begin{cases}
			-\Delta_x \phi_b^i(t,x) = b_i(t,x)\\
			\phi_b^i |_{\partial \Omega}=0
		\end{cases}.
	\end{align*}
	From the standard elliptic estimate, we get $\|\phi_b^i(t)\|_{H_x^2} \lesssim \|b_i(t)\|_{L_x^2}$ for all $t\ge 0$.\\
	First of all, we will deduce the estimate for the first term of (RHS) to the weak form \eqref{wfwf}.\\
	Let $G_f^b(s) = -\int_{\Omega \cross \mathbb{R}^3} \psi(s)f(s)dxdv$. Using the H\"{o}lder inequality and the elliptic estimate,
	\begin{align*}
		\left|G_f^{b}(s)\right|
		\lesssim  \|f(s)\|_{L^2_{x,v}}\|\phi_b^i(s)\|_{H_x^2} \lesssim \|f(s)\|_{L^2_{x,v}}\|b_i(s)\|_{L_x^2} \lesssim \|f(s)\|_{L^2_{x,v}}^2.
	\end{align*}
	Next, we can bound the second term of (RHS) to the weak form \eqref{wfwf} by
	\begin{align} \label{b9}
		\left|\int_0^t \int_{\Omega \cross \mathbb{R}^3} \left(v\cdot \nabla_x \psi \right) (I-P_{L})(f) dxdvds\right| \lesssim \int_0^t \|b(s)\|_{L^2_x}\left\|(I-P_{L})(f)(s)\right\|_{L^2_{x,v}}ds. 
	\end{align}
	Third, from a similar way in \eqref{b2}, we can bound the third term of (RHS) to the weak form \eqref{wfwf} by
	\begin{align} \label{b10}
		\left|\int_0^t\int_{\Omega \cross \mathbb{R}^3} \psi e^{-\Phi(x)}L \left[(I-P_{L})(f)\right]dxdvds\right| \lesssim \int_0^t \|b(s)\|_{L^2_x}\left\|(I-P_{L})(f)(s)\right\|_{L^2_{x,v}}ds.
	\end{align}
	Fourth, we will deduce the estimate for the fourth term of (RHS) to the weak form \eqref{wfwf}. We can decompose the fourth term of (RHS) into two terms:
	\begin{align*}
		\int_{\partial \Omega \cross \mathbb{R}^3} \psi f \{n(x) \cdot v\}dS(x)dv=\int_{\gamma_+} \psi \left[(I-P_\gamma)f\right]\{n(x) \cdot v\}dS(x)dv + \int_{\gamma} \psi \left(P_\gamma f\right)\{n(x) \cdot v\}dS(x)dv.
	\end{align*}
	Setting $z(t,x) =c_\mu \int_{n(x) \cdot v' >0}f(x,v')\mu^{\frac{1}{2}}(v')\{n(x)\cdot v'\}dv'$, we obtain 
	\begin{align*}
		\int_\gamma \psi (P_\gamma f)\{n(x) \cdot v\}dS(x)dv
		=\sum_{k=1}^3 \left(\int_{\mathbb{R}^3} |v|^2v_iv_jv_k\mu(v) dv\right)\left(\int_{\partial \Omega}\partial_{x_j}\phi_b^i(t,x)z(t,x)n_k(x)dS(x)\right)=0,
	\end{align*}
	where we have used the oddness of integration in $v$.\\
	Thus we use the trace theorem to obtain
	\begin{align} \label{b11}
		\left|\int_0^t\int_{\gamma}\psi f\{n(x)\cdot v\}dS(x)dvds\right|&\lesssim \int_0^t\left\|\partial_{x_j} \phi_b^i(s)\right\|_{L^2(\partial\Omega)}\|(I-P_\gamma)f(s)\|_{L^2_{\gamma_+}}ds\nonumber\\
		&\lesssim \int_0^t \|b(s)\|_{L^2_x}\|(I-P_\gamma)f(s)\|_{L^2_{\gamma_+}}ds.
	\end{align}
	Fifth, we will deduce the estimate for the sixth term of (RHS) to the weak form \eqref{wfwf}.
	We apply the Cauchy-Schwarz inequality to obtain
	\begin{equation} \label{b12}
	\begin{aligned}
		\left|\int_{\Omega \cross \mathbb{R}^3}\psi g dxdv\right| \lesssim \|\phi_b^i(t)\|_{H_x^2}\left\|g(t)\right\|_{L^2_{x,v}} \lesssim \|b(t)\|_{L^2_x}\left\|g(t)\right\|_{L^2_{x,v}}.
	\end{aligned}
	\end{equation}
	Sixth, we will deduce the estimate for (LHS) to the weak form \eqref{wfwf}.
	From the oddness of integration in $v$, we deduce that
	\begin{equation} \label{b13}
	\begin{aligned}
		\int_0^t \int_{\Omega \cross \mathbb{R}^3}\left(v\cdot \nabla_x \psi \right) P_{L}(f) dxdvds
		& = 7(2\pi)^{\frac{3}{2}}\int_0^t \int_{\Omega}\partial_{ij}\phi^i_b(s,x) b_j(s,x) + \partial_{jj} \phi_b^i(s,x)b_i(s,x)dxds\\
		& = -7(2\pi)^{\frac{3}{2}}\int_0^t \int_{\Omega}\left(\partial_{ij}\Delta_x^{-1} b_i\right)(s,x) b_j(s,x)+\left(\partial_{jj}\Delta_x^{-1}b_i\right)(s,x)b_i(s,x)dxds.
	\end{aligned}
	\end{equation}
	Seventh, we will deduce the estimate for the seventh term of (RHS) to the weak form \eqref{wfwf}. By a similar way in \eqref{b9}, we obtain
	\begin{align} \label{b14}
		\left|\int_0^t \int_{\Omega \cross \mathbb{R}^3} \left(\nabla_x \Phi(x) \cdot \nabla_v \psi\right) (I-P_{L})(f)(s) dxdvds\right| \le C_{\Phi}\int_0^t \|b(s)\|_{L^2_x}\left\|(I-P_{L})(f)(s)\right\|_{L^2_{x,v}}ds.
	\end{align}
	Eighth, we will deduce the estimate for the eighth term of (RHS) to the weak form \eqref{wfwf}. In a similar way in \eqref{c7}, we obtain
	\begin{align} \label{b15}
		\left|\int_0^t \int_{\Omega \cross \mathbb{R}^3} \left(\nabla_x \Phi(x) \cdot \nabla_v \psi\right) P_{L}(f)(s) dxdvds\right| \le C_{\Phi}\int_0^t\|wf(s)\|_{L^\infty_{x,v}} \left\|P_{L}(f)(s)\right\|_{L^2_{x,v}}^2ds.
	\end{align}
	Lastly, we will deduce the estimate for the fifth term of (RHS) to the weak form \eqref{wfwf}. We decompose $f=P_{L}(f)+\left(I-P_{L}\right)(f)$ to get
	\begin{align*}
		\int_0^t \int_{\Omega \cross \mathbb{R}^3} f\left(\partial_t\psi\right)dxdvds
		=\int_0^t\int_{\Omega \cross \mathbb{R}^3} |v|^2v_iv_j\mu^{\frac{1}{2}}(v)\left[\partial_t\partial_{x_j}\phi_b^i\right](I-P_{L})(f)(s)dxdvds,
	\end{align*}
	where we have used the oddness of integration in $v$.\\	
	Lately, we will demonstrate the estimate of $\nabla_x \partial_t \phi_b^i$ as following :
	\begin{align*} 
		\left\|\nabla_x \partial_t \phi_b^i(t)\right\|_{L^2_x} \lesssim \|a(t)\|_{L_x^2}+\|c(t)\|_{L_x^2}+\|(I-P_{L})(f)(t)\|_{L_{x,v}^2}+\|g(t)\|_{L^2_{x,v}}.
	\end{align*}
	By the estimate of $\nabla_x \partial_t \phi_b^i$, we get
	\begin{equation} \label{b16}
	\begin{aligned}
		\int_0^t \int_{\Omega \cross \mathbb{R}^3} f\left(\partial_t\psi\right)dxdvds
		& \lesssim \int_0^t \left\|\nabla_x \partial_t \phi_b^i(s)\right\|_{L_x^2}\|(I-P_{L})(f)(s)\|_{L_{x,v}^2} ds\\
		& \le \epsilon \int_0^t \|a(s)\|^2_{L_x^2}ds + \epsilon \int_0^t\|c(s)\|^2_{L_x^2}ds + C(\epsilon)\int_0^t \|(I-P_{L})(f)(s)\|_{L_{x,v}^2}^2ds \\
		&\quad + C'(\epsilon)\int_0^t \|g(s)\|_{L_x^2}^2ds,
	\end{aligned}
	\end{equation}
	where we have used the Young's inequality.\\
	Gathering \eqref{b9}, \eqref{b10}, \eqref{b11}, \eqref{b12}, \eqref{b13}, \eqref{b14}, \eqref{b15}, and \eqref{b16}, we obtain for all $i\not= j$,
	\begin{align*}
		&\int_0^t\int_{\Omega}\left(\partial_{ij}\Delta_x^{-1} b_i\right)(s,x) b_j(s,x)+\left(\partial_{jj}\Delta_x^{-1}b_i\right)(s,x)b_i(s,x)dxds\\ 
		&\le G_f^{b}(t)-G_f^{b}(0) +\epsilon\int_0^t \|b(s)\|_{L^2_x}^2ds + \epsilon \int_0^t\|a(s)\|_{L_x^2}^2ds + \epsilon\int_0^t \|c(s)\|_{L^2_x}^2ds\\
		&\quad +C_1(\epsilon)\int_0^t\|(I-P_{L})(f)(s)\|_{L_{x,v}^2}^2ds+C_2(\epsilon)\int_0^t\left\|(I-P_{\gamma})f(s)\right\|_{L^2_{\gamma_+}}^2ds + C_3(\epsilon) \int_0^t\|g(s)\|_{L^2_{x,v}}^2ds\\
		&\quad +C_{\Phi}\int_0^t \|wf(s)\|_{L^\infty_{x,v}} \left\|P_{L}(f)(s)\right\|_{L^2_{x,v}}^2ds,
	\end{align*}
	where we have used the Young's inequality.\\
	Combining the above estimate with the estimate of $\partial_{x_i}\partial_{x_j}\left(\Delta_x^{-1}b_j\right)b_i$, for all $i \not= j$,
	\begin{align*}
		&\int_0^t \int_{\Omega}\left(\partial_{jj}\Delta_x^{-1}b_i\right)(s,x)b_i(s,x)dxds\\
		&\le G_f^{b}(t)-G_f^{b}(0) +\epsilon\int_0^t \|b(s)\|_{L^2_x}^2ds + \epsilon \int_0^t\|a(s)\|_{L_x^2}^2ds + \left(\epsilon+C_4(\epsilon)\right)\int_0^t \|c(s)\|_{L^2_x}^2ds\\
		&\quad +C_1(\epsilon)\int_0^t\|(I-P_{L})(f)(s)\|_{L_{x,v}^2}^2ds+C_2(\epsilon)\int_0^t\left\|(I-P_{\gamma})f(s)\right\|_{L^2_{\gamma_+}}^2ds + C_3(\epsilon) \int_0^t\|g(s)\|_{L^2_{x,v}}^2ds\\
		&\quad +C_{\Phi}\int_0^t \|wf(s)\|_{L^\infty_{x,v}} \left\|P_{L}(f)(s)\right\|_{L^2_{x,v}}^2ds.
	\end{align*}
	From the estimates of $\partial_{jj}(\Delta_x^{-1}b_j)b_j$ and $\partial_{jj}\left(\Delta_x^{-1}b_i\right)b_i$, summing over $j=1,2,3$,
	\begin{equation} \label{co2}
	\begin{aligned}
		\int_0^t \|b(s)\|_{L_x^2}^2ds &\le G_f^{b}(t)-G_f^{b}(0)+\epsilon\int_0^t \|b(s)\|_{L^2_x}^2ds + \epsilon \int_0^t\|a(s)\|_{L_x^2}^2ds \\ 
		& \quad + \left(\epsilon+C_4(\epsilon)\right)\int_0^t \|c(s)\|_{L^2_x}^2ds +C_1(\epsilon)\int_0^t\|(I-P_{L})(f)(s)\|_{L_{x,v}^2}^2ds\\
		& \quad +C_2(\epsilon)\int_0^t\left\|(I-P_{\gamma})f(s)\right\|_{L^2_{\gamma_+}}^2ds + C_3(\epsilon) \int_0^t\|g(s)\|_{L^2_{x,v}}^2ds\\
		&\quad +C_{\Phi}\int_0^t \|wf(s)\|_{L^\infty_{x,v}} \left\|P_{L}(f)(s)\right\|_{L^2_{x,v}}^2ds
	\end{aligned}
	\end{equation}
	\newline
	$\mathbf{Estimate \;of \; \nabla_x\partial_t\phi_b^i.}$\\
	We consider the weak formulation over $[t,t+\epsilon]$. We choose the test function
	\begin{align*}
		\psi = \psi(x,v) = \phi(x)v_i\mu^{\frac{1}{2}}(v),
	\end{align*}
	where $\phi(x)$ depends only on $x$.\\
	We note that
	\begin{align*}
		\int_{\mathbb{R}^3}v_iv_j\mu(v)dv = (2\pi)^{\frac{3}{2}}\delta_{ij},\quad \int_{\mathbb{R}^3}v_iv_j\frac{|v|^2-3}{\sqrt{6}}\mu(v)dv = \frac{\sqrt{6}}{3}(2\pi)^{\frac{3}{2}}\delta_{ij},
	\end{align*}
	 where $\delta_{ij}$ is 1 if $i=j$ and 0 otherwise.\\
	First of all, we will deduce the estimate for (LHS) to the weak formulation \eqref{dwf}.
	\begin{align*}
		\int_{\Omega \cross \mathbb{R}^3} \phi(x) v_i\mu^{\frac{1}{2}}(v)f(t+\epsilon)dxdv = \int_{\Omega}\phi(x)b_i(t+\epsilon,x)dx, \quad \int_{\Omega \cross \mathbb{R}^3} \psi(x,v)f(t)dxdv = \int_{\Omega} \phi(x)b_i(t,x)dx.
	\end{align*}
	Next, we will deduce the estimate for the first term of (RHS) to the weak formulation \eqref{dwf}. We decompose $f=P_{L}(f)+\left(I-P_{L}\right)(f)$ to get
	\begin{align*}
		\int_{t}^{t+\epsilon} \int_{\Omega \cross \mathbb{R}^3}f(v \cdot \nabla_x\psi)dxdvds
		&=  (2\pi)^{\frac{3}{2}}\int_t^{t+\epsilon}\int_\Omega \partial_{x_i} \phi(x) \left[a(s,x)+\frac{\sqrt{6}}{3}c(s,x)\right]dxds\\
		&\quad + \int_{t}^{t+\epsilon} \int_{\Omega \cross \mathbb{R}^3} v \cdot \nabla_x \phi(x)v_i\mu^{\frac{1}{2}}(v)\left(I-P_{L}\right)(f)(s)dxdvds,
	\end{align*}
	where we have used the oddness of integration in $v$.\\
	We easily get
	\begin{align*}
		\int_t^{t+\epsilon}\int_{\Omega \cross \mathbb{R}^3} e^{-\Phi(x)}L(f)\psi dxdvds = 0, \quad \int_t^{t+\epsilon}\int_{\Omega \cross \mathbb{R}^3}(\partial_t \psi)fdxdvds=0.
	\end{align*}
	Third, we will deduce the estimate for second term of (RHS) to the weak formulation \eqref{dwf}.
	For fixed $t>0$, define $\phi(x) = \Phi_b^i(x)$ with
	\begin{align*}
		\begin{cases}
			-\Delta_x \Phi_b^i(x) = \partial_t b_i(t,x)\\
			\Phi_b^i|_{\partial \Omega} = 0
		\end{cases}.
	\end{align*}
	Then we have for fixed t,
	\begin{align*}
		\Phi_b^i (x) = -\Delta_x ^{-1}\partial_t b_i(t,x) = \partial_t \phi_b^i(t,x).
	\end{align*}
	The second term of (RHS) to the weak formulation becomes
	\begin{align*}
		\int_t^{t+\epsilon}\int_{\gamma} \psi f \{n(x) \cdot v\}dS(x)dvds= \int_t^{t+\epsilon}\int_{\mathbb{R}^3}v_i\mu^{\frac{1}{2}}(v) \left(\int_{\partial \Omega}f(s,x,v)\Phi_b^i(x) \{n(x)\cdot v\}dS(x)\right)dvds=0.
	\end{align*}
	At last, we will deduce the estimate for sixth term of (RHS) to the weak formulation \eqref{dwf}. We decompose $f = P_L(f) + (I-P_L)(f)$ to get
	\begin{align*}
		&\int_t^{t+\epsilon} \int_{\Omega \cross \mathbb{R}^3} (\nabla_x \Phi(x) \cdot \nabla_v \psi) f dxdvds\\
		& = \int_t^{t+\epsilon} \left(\int_{\Omega} \frac{1}{2}(2\pi)^{\frac{3}{2}}\phi(x)\partial_{x_i}\Phi(x)a(s,x)dx\right)ds - \int_t^{t+\epsilon} \left(\int_{\Omega} \frac{1}{\sqrt{6}}(2\pi)^{\frac{3}{2}}\phi(x)\partial_{x_i}\Phi(x)c(s,x)dx\right)ds\\
		& \quad + \int_t^{t+\epsilon} \int_{\Omega \cross \mathbb{R}^3} (\nabla_x \Phi(x) \cdot \nabla_v \psi) (I-P_L)(f)(s) dxdvds,
	\end{align*}
	where we have used the oddness of integration in $v$.\\
	Combing the above process and taking the difference quotient in \eqref{dwf}, for all $t \ge 0$,
	\begin{align*}
		\int_{\Omega} \phi(x) \partial_t b_i(t,x) dx &= (2\pi)^{\frac{3}{2}}\int_{\Omega}\partial_{x_i} \phi(x) \left[a(t,x)+\frac{\sqrt{6}}{3}c(t,x)\right] dx\\
		& \quad + \int_{\Omega \cross \mathbb{R}^3} v \cdot \nabla_x \phi(x) v_i\mu^{\frac{1}{2}}(v) \left(I-P_{L}\right)(f)(t)dxdv+ \int_{\Omega \cross \mathbb{R}^3}g(t,x,v) \phi(x)v_i\mu^{\frac{1}{2}}(v)dxdv\\
		&\quad + \frac{1}{2}(2\pi)^{\frac{3}{2}}\int_{\Omega} \phi(x)\partial_{x_i}\Phi(x)a(t,x)dx - \frac{1}{\sqrt{6}}(2\pi)^{\frac{3}{2}}\int_{\Omega} \phi(x)\partial_{x_i}\Phi(x)c(t,x)dx.
	\end{align*}
	From the above equality, for all $t \ge 0$,
	\begin{align*}
		\|\nabla_x \partial_t \phi_b^i(t)\|_{L^2_x}^2 &= \int_\Omega \left|\nabla_x \Phi_b^i(x)\right|^2 dx= -\int_\Omega \Phi_b^i(x) \left(\Delta_x \Phi_b^i(x)\right)dx= \int_\Omega \phi(x) \partial_t b_i(t,x) dx\\
		& \lesssim \epsilon \left\{\|\nabla_x \Phi_b^i\|_{L^2_x}^2+\|\Phi_b^i\|_{L^2_x}^2\right\} + \|a(t)\|_{L^2_x}^2+ \|c(t)\|_{L^2_x}^2+\| \left(I-P_{L}\right)(f)(t)\|_{L^2_{x,v}}^2+\|g(t)\|_{L^2_{x,v}}^2,
	\end{align*}
	where we have used the integration by parts and the Young's inequality.\\
	We use the Poincar\'e inequality to obtain
	\begin{align*}
		\|\nabla_x \partial_t \phi_b^i(t)\|_{L^2_x}^2 \lesssim \epsilon \|\nabla_x \Phi_b^i\|_{L^2_x}^2 + \|a(t)\|_{L^2_x}^2+\|c(t)\|_{L^2_x}^2+\| \left(I-P_{L}\right)(f)(t)\|_{L^2_{x,v}}^2+\|g(t)\|_{L^2_{x,v}}^2.
	\end{align*}
	For sufficiently small $\epsilon >0$, we have for all $t \ge 0$,
	\begin{align*}
		\|\nabla_x \partial_t \phi_b^i(t)\|_{L^2_x} \lesssim \|a(t)\|_{L^2_x}+\|c(t)\|_{L^2_x}+\| \left(I-P_{L}\right)(f)(t)\|_{L^2_{x,v}}+\|g(t)\|_{L^2_{x,v}}.
	\end{align*}
	\newline
	$\mathbf{Estimate \;of \; a.}$\\
	Since $\int_\Omega a(t,x)dx = \int_{\Omega \cross \mathbb{R}^3} \mu^{\frac{1}{2}}(v)f(t,x,v)dxdv = 0$ by the mass conservation, we can choose the test function
	\begin{align*}
		\psi = \psi_a(t,x,v) = \left(|v|^2-\beta_a\right)\mu^{\frac{1}{2}}(v)v\cdot\nabla_x \phi_a(t,x),
	\end{align*}
	where
	\begin{align*}
		\begin{cases}
			-\Delta_x \phi_a(t,x) = a(t,x)\\
			\frac{\partial}{\partial n}\phi_a |_{\partial \Omega} = 0
		\end{cases}
	\end{align*}
	and $\beta_a>0$ is chosen such that
	\begin{align*}
		\int_{\mathbb{R}^3}\left(|v|^2-\beta_a\right)\frac{|v|^2-3}{\sqrt{6}}v_i^2\mu(v)dv = 0 \quad \text{for all }i=1,2,3.
	\end{align*}
	From the standard elliptic estimate, we get $\|\phi_a(t)\|_{H_x^2} \lesssim \|a(t)\|_{L_x^2}$ for all $t\ge 0$.\\
	First, we will deduce the estimate for the first term of (RHS) to the weak form \eqref{wfwf}.\\ Let $G_f^{a}(s) = \int_{\Omega \cross \mathbb{R}^3} \psi(s)f(s)dxdv$. Then we have $\left|G_f^a(s)\right| \lesssim  \|f(s)\|_{L^2_{x,v}}^2$.\\
	Next, we will deduce the estimate for the second term of (RHS) to the weak form \eqref{wfwf}. Using the H\"{o}lder inequality and the elliptic estimate, we obtain
	\begin{align} \label{a1}
		\left|\int_0^t \int_{\Omega \cross \mathbb{R}^3} \left(v\cdot \nabla_x \psi \right) (I-P_{L})(f) dxdvds\right| \le C\int_0^t \|a(s)\|_{L^2_x}\left\|(I-P_{L})(f)(s)\right\|_{L^2_{x,v}}ds.
	\end{align}
	Third, we will deduce the estimate for the third term of (RHS) to the weak form \eqref{wfwf}. This case is similar to the case \eqref{c2}.
	We deduce that
	\begin{align} \label{a2}
		\left|\int_0^t\int_{\Omega \cross \mathbb{R}^3} \psi e^{-\Phi(x)}L \left[(I-P_{L})(f)\right]dxdvds\right| \lesssim \int_0^t \|a(s)\|_{L^2_x}\left\|(I-P_{L})(f)(s)\right\|_{L^2_{x,v}}ds.
	\end{align}
	Fourth, we will deduce the estimate for the fourth term of (RHS) to the weak form \eqref{wfwf}. We can decompose the fourth term of (RHS) into two terms:
	\begin{align*}
		\int_{\partial \Omega \cross \mathbb{R}^3}\psi f\{n(x) \cdot v\}dS(x)dv =\int_{\gamma_+} \psi \left[(I-P_\gamma)f\right]\{n(x) \cdot v\}dS(x)dv + \int_{\gamma} \psi \left(P_\gamma f\right)\{n(x) \cdot v\}dS(x)dv.
	\end{align*}
	We split $v = v_{\parallel}+v_\perp$, where $\begin{cases}
		v_{\parallel} = \{n(x) \cdot v\}n(x)\\
		v_{\perp} =v- v_{\parallel}	\end{cases}$.\\	
	Setting $z(t,x) =c_\mu e^{-\Phi(x)}\int_{n(x) \cdot v' >0}f(x,v')\mu^{\frac{1}{2}}(v')\{n(x)\cdot v'\}dv'$, we obtain 
	\begin{align*}
		\int_\gamma \psi (P_\gamma f)\{n(x) \cdot v\}dS(x)dv
		&=\sum_{i=1}^3 \int_{\gamma}\left(|v|^2-\beta_a\right)\left(\frac{\partial}{\partial n}\phi_a(t,x)\right)\mu(v) z(t,x)\left\{n(x) \cdot v \right\}^2dS(x)dv\\
		&\quad + \sum_{i,j=1}^3\int_{\partial \Omega}\partial_{x_i} \phi_a(t,x)z(t,x)n_j(x)\left(\int_{\mathbb{R}^3}\left(|v|^2-\beta_a\right)(v_\perp)_iv_j\mu(v)dv\right)dS(x)\\
		&=0,
	\end{align*}
	where we have used the oddness of integration in $v$ and $\frac{\partial}{\partial n}\phi_a(t,x)=0$.\\
	Using the trace theorem, we can simplify
	\begin{align} \label{a3}
		\left|\int_0^t\int_{\gamma}\psi f\{n(x)\cdot v\}dS(x)dvds\right|\lesssim \int_0^t \|a(s)\|_{L^2_x}\|(I-P_\gamma)f(s)\|_{L^2_{\gamma_+}}ds.
	\end{align}
	Fifth, we will deduce the estimate for the sixth term of (RHS) to the weak form \eqref{wfwf}.
	We apply the Cauchy-Schwarz inequality to obtain
	\begin{equation} \label{a4}
	\begin{aligned}
		\left|\int_{\Omega \cross \mathbb{R}^3}\psi g dxdv\right| 
		\ \lesssim \|\phi_a(t)\|_{H_x^2}\left\|g(t)\right\|_{L^2_{x,v}} \lesssim \|a(t)\|_{L^2_x}\left\|g(t)\right\|_{L^2_{x,v}}.
	\end{aligned}
	\end{equation}
	Sixth, we will deduce the estimate for (LHS) to the weak form \eqref{wfwf}.
	From the construction for $\beta_a$ and the oddness of integration in $v$, we deduce that
	\begin{equation} \label{a5}
	\begin{aligned}
		\int_0^t \int_{\Omega \cross \mathbb{R}^3}\left(v\cdot \nabla_x \psi \right) P_{L}(f) dxdvds
		= 5\int_0^t\int_{\Omega}\left(-\Delta_x\phi_a(s,x)\right)a(s,x)dxds
		= 5\int_0^t \left\|a(s)\right\|_{L_x^2}^2ds.
	\end{aligned}
	\end{equation}
	Seventh, we will deduce the estimate for the seventh term of (RHS) to the weak form \eqref{wfwf}. Using the H\"older inequality, we obtain
	\begin{align} \label{a6}
		\left|\int_0^t \int_{\Omega \cross \mathbb{R}^3} \left(\nabla_x \Phi(x) \cdot \nabla_v \psi\right) (I-P_{L})(f)(s) dxdvds\right| \le C_{\Phi}\int_0^t \|a(s)\|_{L^2_x}\left\|(I-P_{L})(f)(s)\right\|_{L^2_{x,v}}ds.
	\end{align}
	Eighth, we will deduce the estimate for the eighth term of (RHS) to the weak form \eqref{wfwf}.\\
	By a similar way in \eqref{c7}, we deduce that
	\begin{align} \label{a7}
		\left|\int_0^t \int_{\Omega \cross \mathbb{R}^3} \left(\nabla_x \Phi(x) \cdot \nabla_v \psi\right) P_{L}(f)(s) dxdvdd\right| \le C_{\Phi}\int_0^t \|wf(s)\|_{L^\infty_{x,v}} \left\|P_{L}(f)(s)\right\|_{L^2_{x,v}}^2ds.
	\end{align}
	Lastly, we will deduce the estimate for the fifth term of (RHS) to the weak form \eqref{wfwf}.\\
	We decompose $f=P_{L}(f)+\left(I-P_{L}\right)(f)$ to get
	\begin{align*}
		\int_0^t \int_{\Omega \cross \mathbb{R}^3} f\left(\partial_t\psi\right)dxdvds
		& = -5(2\pi)^{\frac{3}{2}} \sum_{i=1}^3 \int_0^t \int_{\Omega}\partial_t\partial_{x_i} \phi_a(s,x)b_i(s,x)dxds  \\
		&\quad +\sum_{i=1}^3\int_0^t\int_{\Omega \cross \mathbb{R}^3} \left(|v|^2-\beta_a\right)\mu^{\frac{1}{2}}(v)v_i\left[\partial_t\partial_{x_i}\phi_a\right](I-P_{L})(f)(s)dxdvds,
	\end{align*}
	where we have used the construction for $\beta_a$ and the oddness of integration in $v$.\\
	Lately, we will demonstrate the estimate of $\nabla_x \partial_t \phi_a$ as following :
	\begin{align*}
		\left\|\nabla_x \partial_t \phi_a(t)\right\|_{L^2_x} \lesssim \|b(t)\|_{L_x^2}+\|g(t)\|_{L^2_{x,v}}.
	\end{align*}
	By the estimate of $\nabla_x \partial_t \phi_a$, we get
	\begin{equation} \label{a8}
	\begin{aligned}
		\int_0^t \int_{\Omega \cross \mathbb{R}^3} f\left(\partial_t\psi\right)dxdvds
		& \lesssim \int_0^t \left\|\nabla_x \partial_t \phi_a(s)\right\|_{L^2_x}\left(\|(I-P_{L})(f)(s)\|_{L_{x,v}^2}+\|b(s)\|_{L^2_x}\right)ds\\		
		& \lesssim  \int_0^t\|b(s)\|_{L_x^2}^2ds+\int_0^t\|(I-P_{L})(f)(s)\|_{L_{x,v}^2}^2ds +  \int_0^t\|g(s)\|_{L^2_{x,v}}^2ds,
	\end{aligned}
	\end{equation}
	where we have used the Young's inequality.\\
	Gathering \eqref{a1}, \eqref{a2}, \eqref{a3}, \eqref{a4}, \eqref{a5}, \eqref{a6}, \eqref{a7}, and \eqref{a8}, we obtain
	\begin{align*}
		\int_0^t\|a(s)\|_{L_x^2}^2ds &\le  G_f^a(t)-G_f^a(0) +\epsilon\int_0^t \|a(s)\|_{L^2_x}^2ds + C \int_0^t\|b(s)\|_{L_x^2}^2ds\\
		&\quad +C_1(\epsilon)\int_0^t\|(I-P_{L})(f)(s)\|_{L_{x,v}^2}^2ds+C_2(\epsilon)\int_0^t\left\|(I-P_{\gamma})f(s)\right\|_{L^2_{\gamma_+}}^2ds\\ 
		&\quad + C_3(\epsilon) \int_0^t\|g(s)\|_{L^2_{x,v}}^2ds +C_{\Phi}\int_0^t \|wf(s)\|_{L^\infty_{x,v}} \left\|P_{L}(f)(s)\right\|_{L^2_{x,v}}^2ds,
	\end{align*}
	where we have used the Young's inequality.\\
	Thus, choosing sufficiently small $\epsilon >0$, we conclude that
	\begin{equation} \label{co3}
	\begin{aligned}
		\int_0^t\|a(s)\|_{L_x^2}^2ds &\le G_f^a(t)-G_f^a(0) +C \int_0^t\|b(s)\|_{L_x^2}^2ds\\
		&\quad +C_1(\epsilon)\int_0^t\|(I-P_{L})(f)(s)\|_{L_{x,v}^2}^2ds+C_2(\epsilon)\int_0^t\left\|(I-P_{\gamma})f(s)\right\|_{L^2_{\gamma_+}}^2ds\\ 
		&\quad +C_3(\epsilon) \int_0^t\|g(s)\|_{L^2_{x,v}}^2ds+C_{\Phi}\int_0^t \|wf(s)\|_{L^\infty_{x,v}} \left\|P_{L}(f)(s)\right\|_{L^2_{x,v}}^2ds.
	\end{aligned}
	\end{equation}
	\newline
	$\mathbf{Estimate \;of \; \nabla_x\partial_t\phi_a.}$\\
	We consider the weak formulation over $[t,t+\epsilon]$.
	We choose the test function
	\begin{align*}
		\psi = \psi(x,v) = \phi(x) \mu^{\frac{1}{2}}(v),
	\end{align*}
	where $\phi(x)$ depends only on $x$.\\
	First of all, we will deduce the estimate for (LHS) to the weak formulation \eqref{dwf}.
	\begin{align*}
		\int_{\Omega \cross \mathbb{R}^3} \phi(x) \mu^{\frac{1}{2}}(v)f(t+\epsilon)dxdv = \int_{\Omega}\phi(x)a(t+\epsilon,x)dx, \quad \int_{\Omega \cross \mathbb{R}^3} \psi(x,v)f(t)dxdv = \int_{\Omega} \phi(x)a(t,x)dx.
	\end{align*}
	Next, we will deduce the estimate for the first term of (RHS) to the weak formulation \eqref{dwf}. Decompose $f=P_{L}(f)+\left(I-P_{L}\right)(f)$ to get
	\begin{align*}
		\int_{t}^{t+\epsilon} \int_{\Omega \cross \mathbb{R}^3}f(v \cdot \nabla_x\psi)dxdvds= \sum_{i=1}^3	\int_t^{t+\epsilon}\left(\int_\Omega \partial_{x_i} \phi(x) b_i(s,x)dx\right)ds,
	\end{align*}
	where we have used the oddness of integration in $v$ and $\int_{\mathbb{R}^3}v_i\mu^{\frac{1}{2}}(v)\left(I-P_{L}\right)(f)(s)dv=0$.\\
	We easily get
	\begin{align*}
		\int_t^{t+\epsilon}\int_{\Omega \cross \mathbb{R}^3} e^{-\Phi(x)}L(f)\psi dxdvds = 0, \quad \int_t^{t+\epsilon}\int_{\Omega \cross \mathbb{R}^3}(\partial_t \psi)fdxdvds=0.
	\end{align*}
	Third, we will deduce the estimate for second term of (RHS) to the weak formulation \eqref{dwf}.
		The second term of (RHS) to the weak formulation becomes
	\begin{align*}
		\int_{\gamma} \psi f \{n(x) \cdot v\}dS(x)dv=\int_{\gamma_+} \psi \left[\left(I-P_\gamma\right)f\right]\{n(x) \cdot v\}dS(x)dv + \int_{\gamma} \psi \left(P_\gamma f\right)\{n(x) \cdot v\}dS(x)dv.
	\end{align*}
	By the oddness of integration in $v$, we get
	\begin{align*}
		\int_{\gamma} \psi \left(P_\gamma f\right)\{n(x) \cdot v\}dS(x)dv &= \sum_{i=1}^3 \int_{\partial \Omega} \phi(x)z(t,x)n_i(x)\left(\int_{\mathbb{R}^3}v_i\mu(v)dv\right)dS(x)=0,
	\end{align*}
	where $z(t,x) =c_\mu \int_{n(x) \cdot v' >0}f(x,v')\mu^{\frac{1}{2}}(v')\{n(x)\cdot v'\}dv'$.\\
	From the fact $(I-P_\gamma) \perp P_\gamma$, we obtain
	\begin{align*}
		\int_{\gamma_+} \psi \left[\left(I-P_\gamma\right)f\right]\{n(x) \cdot v\}dS(x)dv
		=\int_{\partial \Omega} \phi(x) \left(\int_{n(x) \cdot v >0} \mu^{\frac{1}{2}}(v) \left[\left(I-P_\gamma\right)f\right]\{n(x) \cdot v\}dv\right)dS(x)=0.
	\end{align*}
	This yields
	\begin{align*}
		&\int_t^{t+\epsilon}\int_{\gamma} \psi f \{n(x) \cdot v\}dS(x)dvds =0.
	\end{align*}
	Fourth, we will deduce the estimate for sixth term of (RHS) to the weak formulation \eqref{dwf}. We decompose $f = P_L(f) + (I-P_L)(f)$ to get
	\begin{align*}
		&\int_t^{t+\epsilon} \int_{\Omega \cross \mathbb{R}^3} (\nabla_x \Phi(x) \cdot \nabla_v \psi) f dxdvds = \sum_{i=1}^3 \int_t^{t+\epsilon} \left(\int_{\Omega} (2\pi)^{\frac{3}{2}}\phi(x)\partial_{x_i}\Phi(x)b_i(s,x)dx\right)ds,
	\end{align*}
	where we have used the oddness of integration in $v$.\\
	Combing the above process and taking the difference quotient in \eqref{dwf}, for all $t \ge 0$,
	\begin{align*}
		\int_{\Omega} \phi(x) \partial_t a(t,x) dx &= \int_{\Omega}b(t,x)\cdot \nabla_x \phi(x) dx+ \int_{\Omega \cross \mathbb{R}^3}g(t,x,v)\mu^{\frac{1}{2}}(v) \phi(x)dxdv\\
		&\quad +(2\pi)^{\frac{3}{2}}\int_{\Omega} \phi(x)b(t,x)\cdot \nabla_x \Phi(x) dx.
	\end{align*}
	From the above equality, for all $t \ge 0$,
	\begin{align*}
		&\int_\Omega \phi(x) \partial_t a(t,x) dx \lesssim \|b(t)\|_{L_x^2} \|\nabla_x \phi\|_{L_x^2}+\|g(t)\|_{L^2_{x,v}}\|\phi\|_{L^2_x},
	\end{align*}
	where we have used the Cauchy-Schwartz inequality.\\
	We use the Poincar\'e inequality to obtain
	\begin{align*}
		\int_\Omega \phi(x) \partial_t a(t,x) dx &\lesssim \left(\|b(t)\|_{L_x^2} +\|g(t)\|_{L^2_{x,v}}\right)	\|\nabla_x \phi\|_{L_x^2}\lesssim \left(\|b(t)\|_{L_x^2} +\|g(t)\|_{L^2_{x,v}}\right) \|\phi\|_{H_x^1}
	\end{align*}
	Thus we get
	\begin{align*}
		\|\partial_t a(t)\|_{(H_x^1)^*} \lesssim \|b(t)\|_{L_x^2} +\|g(t)\|_{L^2_{x,v}},
	\end{align*}
	where $(H_x^1)^*$ is dual space of $H_x^1$ with respect to the dual pair $(f,g)=\int_{\Omega}f(x)g(x)dx$ for $f \in H^1_x$ and $g \in (H^1_x)^*$.\\
	Since $\int_{\Omega} \partial_t a(t,x)dx=0$ by the mass conservation, for fixed $t>0$, define $\phi(x) = \Phi_a(x)$ with
	\begin{align*}
		\begin{cases}
			-\Delta_x \Phi_a(x) = \partial_t a(t,x)\\
			\frac{\partial}{\partial n}\Phi_a|_{\partial \Omega} = 0.
		\end{cases}
	\end{align*}
	Then we have for fixed t,
	\begin{align*}
		\Phi_a (x) = -\Delta_x ^{-1}\partial_t a(t,x) = \partial_t \phi_a(t,x).
	\end{align*}
	Hence, we use the standard elliptic estimate to obtain
	\begin{align*}
		\|\nabla_x \partial_t \phi_a(t)\|_{L^2_x} &= \|\nabla_x \Delta_x^{-1}\partial_t a(t)\|_{L^2_x}= \|\Delta_x^{-1}\partial_t a(t)\|_{H_x^1}=\|\Phi_a(t)\|_{H_x^1} \\
		& \lesssim \|\partial_t a(t)\|_{(H_x^1)^*}\lesssim \|b(t)\|_{L^2_x} +\|g(t)\|_{L^2_{x,v}}.
	\end{align*}
	\newline
	$\mathbf{Conclusion.}$\\
	From \eqref{co1}, \eqref{co2}, and \eqref{co3}, for $\eta,\delta >0$, we have
	\begin{align*}
		\int_0^t\left(\|a(s)\|_{L_x^2}^2+\eta\|b(s)\|_{L_x^2}^2+\delta\|c(s)\|_{L_x^2}^2\right)ds
		&\le G_f(t)-G_f(0) + \eta \epsilon_b\int_0^t\|a(s)\|_{L_x^2}^2ds + \left(\delta \epsilon_c + C(\epsilon_a)\right)\int_0^t\|b(s)\|_{L_x^2}^2ds \\
		&\quad + \eta \left(\epsilon_b+C(\epsilon_b)\right)\int_0^t\|c(s)\|_{L_x^2}^2ds\\
		&\quad +C\int_0^t\left\|(I-P_{L})(f)(s)\right\|_{L_{x,v}^2}^2ds + C\int_0^t \|(I-P_\gamma)f(s)\|_{L^2_{\gamma_+}}^2ds\\
		&\quad +C\int_0^t \|g(t)\|_{L^2_{x,v}}^2ds+C_{\Phi}\int_0^t \|wf(s)\|_{L^\infty_{x,v}} \left\|P_{L}(f)(s)\right\|_{L^2_{x,v}}^2ds.
	\end{align*}
	Fix $\epsilon_a >0$. First, we choose large $\eta>0$ such that $\eta > C(\epsilon_a)$, and then choose small $\epsilon_b >0$ such that $1>\eta \epsilon_b$. Next, we choose large $\delta >0$ such that $\delta > \eta\left(\epsilon_b + C(\epsilon_b)\right)$. Lastly, we choose small $\epsilon_c>0$ such that $\eta > C(\epsilon_a) + \delta \epsilon_c$. Therefore, we conclude that
	\begin{align*}
		\int_0^t\left\|P_{L}(f)(s)\right\|^2_{L^2_{x,v}}ds &\le G_f(t) - G_f(0) +C_\perp\int_0^t\left[\left\|\left(I-P_{L}\right)(f)(s)\right\|^2_{L^2_{x,v}}+\left\|\left(I-P_\gamma\right)f(s)\right\|^2_{L_{\gamma+}^2} \right] ds\\
		&\quad +C_\perp\int_0^t\left\|g(s)\right\|^2_{L^2_{x,v}}ds+C_\perp\int_0^t \|wf(s)\|_{L^\infty_{x,v}} \left\|P_{L}(f)(s)\right\|_{L^2_{x,v}}^2ds.
	\end{align*}
\end{proof}

	\noindent{\bf Data availability:} No data was used for the research described in the article.
	\newline
	
	\noindent{\bf Conflict of interest:} The authors declare that they have no conflict of interest.\newline
	
	\noindent{\bf Acknowledgement}
	J. Kim and D. Lee are supported by the National Research Foundation of Korea(NRF) grant funded by the Korea government(MSIT)(No.RS-2023-00212304 and No.RS-2023-00219980). 
	
	\nocite{}
	\bibliographystyle{abbrv}
	\bibliography{kji.bib}

\begin{thebibliography}{10}

\bibitem{AsanoAlmost}
K.~Asano.
\newblock Almost transversality theorem in the classical dynamical system. {I}.
\newblock {\em J. Math. Kyoto Univ.}, 34(1):87--94, 1994.

\bibitem{BKLY_BGK}
G.-C. Bae, G.~Ko, D.~Lee, and S.-B. Yun.
\newblock Large amplitude problem of {BGK} model: Relaxation to quadratic nonlinearity.
\newblock {\em arXiv preprint arXiv:2301.09857}, 2023.

\bibitem{Guo-Briant}
M.~Briant and Y.~Guo.
\newblock Asymptotic stability of the {B}oltzmann equation with {M}axwell boundary conditions.
\newblock {\em J. Differential Equations}, 261(12):7000--7079, 2016.

\bibitem{YCRV2021}
Y.~Cao and C.~Kim.
\newblock On {S}ome {R}ecent {P}rogress in the {V}lasov-{P}oisson-{B}oltzmann {S}ystem with {D}iffuse {R}eflection {B}oundary.
\newblock In {\em Recent Advances in Kinetic Equations and Applications}, pages 93--114, Cham, 2021. Springer International Publishing.

\bibitem{CKLVPB}
Y.~Cao, C.~Kim, and D.~Lee.
\newblock Global strong solutions of the {V}lasov-{P}oisson-{B}oltzmann system in bounded domains.
\newblock {\em Arch. Ration. Mech. Anal.}, 233(3):1027--1130, 2019.

\bibitem{HCQL2020}
H.~Chen, C.~Kim, and Q.~Li.
\newblock Local {W}ell-{P}osedness of {V}lasov–{P}oisson–{B}oltzmann {E}quation with {G}eneralized {D}iffuse {B}oundary {C}ondition.
\newblock {\em Journal of Statistical Physics}, 179:535--631, 2020.

\bibitem{DV}
L.~Desvillettes and C.~Villani.
\newblock On the trend to global equilibrium for spatially inhomogeneous kinetic systems: the {B}oltzmann equation.
\newblock {\em Invent. Math.}, 159(2):245--316, 2005.

\bibitem{DLCP1989}
R.~J. DiPerna and P.-L. Lions.
\newblock On the {C}auchy problem for {B}oltzmann equations: global existence and weak stability.
\newblock {\em Ann. of Math. (2)}, 130(2):321--366, 1989.

\bibitem{DHWY2017}
R.~Duan, F.~Huang, Y.~Wang, and T.~Yang.
\newblock Global well-posedness of the {B}oltzmann equation with large amplitude initial data.
\newblock {\em Arch. Ration. Mech. Anal.}, 225(1):375--424, 2017.

\bibitem{DHWZ2019}
R.~Duan, F.~Huang, Y.~Wang, and Z.~Zhang.
\newblock Effects of soft interaction and non-isothermal boundary upon long-time dynamics of rarefied gas.
\newblock {\em Arch. Ration. Mech. Anal.}, 234(2):925--1006, 2019.

\bibitem{DKL2020}
R.~Duan, G.~Ko, and D.~Lee.
\newblock The {B}oltzmann equation with a class of large-amplitude initial data and specular reflection boundary condition.
\newblock {\em J. Stat. Phys.}, 190(12):Paper No. 189, 46, 2023.

\bibitem{DW2019}
R.~Duan and Y.~Wang.
\newblock The {B}oltzmann equation with large-amplitude initial data in bounded domains.
\newblock {\em Adv. Math.}, 343:36--109, 2019.

\bibitem{GuoVPB}
Y.~Guo.
\newblock The {V}lasov-{P}oisson-{B}oltzmann system near {M}axwellians.
\newblock {\em Comm. Pure Appl. Math.}, 55(9):1104--1135, 2002.

\bibitem{GuoVMB}
Y.~Guo.
\newblock The {V}lasov-{M}axwell-{B}oltzmann system near {M}axwellians.
\newblock {\em Invent. Math.}, 153(3):593--630, 2003.

\bibitem{GuoQAM}
Y.~Guo.
\newblock Bounded solutions for the {B}oltzmann equation.
\newblock {\em Quart. Appl. Math.}, 68(1):143--148, 2010.

\bibitem{GDCB2010}
Y.~Guo.
\newblock Decay and continuity of the {B}oltzmann equation in bounded domains.
\newblock {\em Arch. Ration. Mech. Anal.}, 197(3):713--809, 2010.

\bibitem{GuoLandau}
Y.~Guo.
\newblock The {V}lasov-{P}oisson-{L}andau system in a periodic box.
\newblock {\em J. Amer. Math. Soc.}, 25(3):759--812, 2012.

\bibitem{SGMR2012}
Y.~Guo and R.~M. Strain.
\newblock Momentum regularity and stability of the relativistic {V}lasov-{M}axwell-{B}oltzmann system.
\newblock {\em Comm. Math. Phys.}, 310(3):649--673, 2012.

\bibitem{ChanBELP}
C.~Kim.
\newblock Boltzmann equation with a large potential in a periodic box.
\newblock {\em Comm. Partial Differential Equations}, 39(8):1393--1423, 2014.

\bibitem{KimLee}
C.~Kim and D.~Lee.
\newblock The {B}oltzmann equation with specular boundary condition in convex domains.
\newblock {\em Comm. Pure Appl. Math.}, 71(3):411--504, 2018.

\bibitem{KimLeeNonconvex}
C.~Kim and D.~Lee.
\newblock Decay of the {B}oltzmann equation with the specular boundary condition in non-convex cylindrical domains.
\newblock {\em Arch. Ration. Mech. Anal.}, 230(1):49--123, 2018.

\bibitem{KKL}
G.~Ko, C.~Kim, and D.~Lee.
\newblock Dynamical {B}illiard and a long-time behavior of the {B}oltzmann equation in general 3{D} toroidal domains.
\newblock {\em arXiv preprint arXiv:2304.04530}, 2023.

\bibitem{KLP2022}
G.~Ko, D.~Lee, and K.~Park.
\newblock The large amplitude solution of the {B}oltzmann equation with soft potential.
\newblock {\em J. Differential Equations}, 307:297--347, 2022.

\bibitem{MouhotL1}
C.~Mouhot.
\newblock Rate of convergence to equilibrium for the spatially homogeneous {B}oltzmann equation with hard potentials.
\newblock {\em Comm. Math. Phys.}, 261(3):629--672, 2006.

\bibitem{SGSR2004}
R.~M. Strain and Y.~Guo.
\newblock Stability of the relativistic {M}axwellian in a collisional plasma.
\newblock {\em Comm. Math. Phys.}, 251(2):263--320, 2004.

\bibitem{SGED2008}
R.~M. Strain and Y.~Guo.
\newblock Exponential decay for soft potentials near {M}axwellian.
\newblock {\em Arch. Ration. Mech. Anal.}, 187(2):287--339, 2008.

\bibitem{Ukai}
S.~Ukai.
\newblock On the existence of global solutions of mixed problem for non-linear {B}oltzmann equation.
\newblock {\em Proc. Japan Acad.}, 50(3):179--184, 1974.

\bibitem{WWGS2019}
G.~Wang and Y.~Wang.
\newblock Global stability of {B}oltzmann equation with large external potential for a class of large oscillation data.
\newblock {\em J. Differential Equations}, 267(6):3610--3645, 2019.

\bibitem{YSEB2018}
X.~Yang.
\newblock Stability of the {E}quilibrium to the {B}oltzmann {E}quation with {L}arge {P}otential {F}orce.
\newblock {\em Chinese Annals of Mathematics, Series B}, 39(5):805--816, 2018.

\end{thebibliography}
\end{document}